\documentclass{amsart}

\usepackage{amsmath, amssymb, amsthm, amsfonts, marginnote, stmaryrd}

\usepackage{xcolor}

\input xy
\xyoption{all}

\usepackage{hyperref}

\swapnumbers
\numberwithin{equation}{subsection}

\theoremstyle{plain}
\newtheorem{theorem}[subsubsection]{Theorem}
\newtheorem{lemma}[subsubsection]{Lemma}
\newtheorem{prop}[subsubsection]{Proposition}
\newtheorem{cor}[subsubsection]{Corollary}
\newtheorem{conj}[subsubsection]{Conjecture}
\newtheorem*{claim}{Claim}

\newtheorem{ansatz}[subsubsection]{Ansatz}

\theoremstyle{definition}
\newtheorem{defn}[subsubsection]{Definition}

\newtheorem{remark}[subsubsection]{Remark}
\newtheorem{exam}[subsubsection]{Example}

\newtheorem{ex}[subsubsection]{Example}

\newtheorem{notation}[subsubsection]{Notation}

\def\RMod{\operatorname{RMod}}
\def\LMod{\operatorname{LMod}}
\def\Bimod{\operatorname{Bimod}}

\def\AA{\mathbb{A}}
\def\BB{\mathbb{B}}
\def\CC{\mathbb{C}}
\def\DD{\mathbb{D}}

\def\GG{\mathbb{G}}

\def\LL{\mathbb{L}}

\def\QQ{\mathbb{Q}}

\def\RR{\mathbb{R}}

\def\TT{\mathbb{T}}

\def\ZZ{\mathbb{Z}}

\def\calG{\mathcal{G}}
\def\calH{\mathcal{H}}

\def\calL{\mathcal{L}}

\def\calN{\mathcal{N}}
\def\calO{\mathcal{O}}

\def\calW{\mathcal{W}}

\newcommand\cA{\mathcal{A}}
\newcommand\cB{\mathcal{B}}
\newcommand\cC{\mathcal{C}}
\newcommand\cD{\mathcal{D}}
\newcommand\cE{\mathcal{E}}
\newcommand\cF{\mathcal{F}}
\newcommand\cG{\mathcal{G}}
\newcommand\cH{\mathcal{H}}
\newcommand\cI{\mathcal{I}}
\newcommand\cJ{\mathcal{J}}
\newcommand\cK{\mathcal{K}}
\newcommand\cL{\mathcal{L}}
\newcommand\cM{\mathcal{M}}
\newcommand\cN{\mathcal{N}}
\newcommand\cO{\mathcal{O}}
\newcommand\cP{\mathcal{P}}
\newcommand\cQ{\mathcal{Q}}
\newcommand\cR{\mathcal{R}}
\newcommand\cS{\mathcal{S}}
\newcommand\cT{\mathcal{T}}
\newcommand\cU{\mathcal{U}}
\newcommand\cV{\mathcal{V}}
\newcommand\cW{\mathcal{W}}

\newcommand\cY{\mathcal{Y}}
\newcommand\cZ{\mathcal{Z}}

\def\bR{\mathbf{R}}

\newcommand\frA{\mathfrak{A}}
\newcommand\frB{\mathfrak{B}}
\newcommand\frC{\mathfrak{C}}
\newcommand\frD{\mathfrak{D}}
\newcommand\frE{\mathfrak{E}}
\newcommand\frF{\mathfrak{F}}
\newcommand\frG{\mathfrak{G}}

\newcommand\frP{\mathfrak{P}}

\newcommand\frR{\mathfrak{R}}
\newcommand\frS{\mathfrak{S}}
\newcommand\frT{\mathfrak{T}}

\newcommand\frZ{\mathfrak{Z}}

\newcommand\frb{\mathfrak{b}}

\newcommand\frg{\mathfrak{g}}

\newcommand\frh{\mathfrak{h}}
\newcommand\fri{\mathfrak{i}}
\newcommand\frj{\mathfrak{j}}

\newcommand\frl{\mathfrak{l}}
\newcommand\frakm{\mathfrak{m}}
\newcommand\frn{\mathfrak{n}}

\newcommand\frs{\mathfrak{s}}
\newcommand\frt{\mathfrak{t}}
\newcommand\fru{\mathfrak{u}}

\newcommand\frz{\mathfrak{z}}

\newcommand\tilW{\widetilde{W}}

\newcommand\tilw{\widetilde{w}}

\newcommand\Aff{\textup{Aff}}

\newcommand\Alg{\textup{Alg}}

\newcommand\bimon{\mathit{bimon}}

\newcommand{\Bun}{\textup{Bun}}

\newcommand\Gal{\textup{Gal}}

\newcommand\IC{\textup{IC}}
\newcommand\id{\textup{id}}
\renewcommand{\Im}{\textup{Im}}
\newcommand{\Ind}{\textup{Ind}}

\newcommand\Irr{\textup{Irr}}

\newcommand\Lie{\textup{Lie}}
\newcommand\Loc{\textup{Loc}}

\newcommand\Mod{\textup{Mod}}

\newcommand{\QCoh}{\textup{QCoh}}
\newcommand{\qc}{\textup{QC}}

\newcommand{\red}{\textup{red}}
\newcommand{\reg}{\textup{reg}}

\newcommand{\Res}{\textup{Res}}

\newcommand\Span{\textup{Span}}
\newcommand\Spec{\textup{Spec }}

\newcommand\St{\mathit{St}}

\newcommand\Stab{\textup{Stab}}
\newcommand\Supp{\textup{Supp}}
\newcommand\Sym{\textup{Sym}}

\newcommand\Tot{\textup{Tot}}

\newcommand{\tr}{\textup{tr}}

\newcommand{\univ}{\mathit{univ}}

\newcommand{\Vect}{\textup{Vect}}

\newcommand\Aut{\textup{Aut}}
\newcommand\Hom{\textup{Hom}}
\newcommand\End{\textup{End}}

\newcommand\Mor{\textup{Mor}}

\newcommand\uHom{\underline{\Hom}}

\newcommand\GL{\textup{GL}}
\newcommand\gl{\mathfrak{gl}}
\newcommand\PGL{\textup{PGL}}

\newcommand\SL{\textup{SL}}

\newcommand{\Gm}{\GG_m}
\def\Ga{\GG_a}

\newcommand{\ad}{\textup{ad}}
\newcommand{\Ad}{\textup{Ad}}
\renewcommand\sc{\textup{sc}}
\newcommand{\der}{\textup{der}}

\newcommand\xch{\mathbb{X}^*}
\newcommand\xcoch{\mathbb{X}_*}

\renewcommand\a\alpha
\renewcommand\b\beta
\newcommand\g\gamma
\newcommand\G\Gamma
\renewcommand\d\delta
\newcommand\D\Delta
\newcommand{\e}{\epsilon}
\newcommand{\ep}{\epsilon}
\renewcommand{\th}{\theta}

\renewcommand{\k}{\kappa}
\newcommand{\ph}{\varphi}
\newcommand{\Sig}{\Sigma}
\newcommand{\s}{\sigma}
\renewcommand{\t}{\tau}

\renewcommand{\l}{\lambda}
\renewcommand{\L}{\Lambda}
\newcommand{\om}{\omega}
\newcommand{\Om}{\Omega}
\newcommand{\io}{\iota}

\newcommand{\incl}{\hookrightarrow}
\newcommand{\isom}{\stackrel{\sim}{\to}}
\newcommand{\bij}{\leftrightarrow}
\newcommand{\surj}{\twoheadrightarrow}
\newcommand{\mt}{\mapsto}

\renewcommand{\j}[1]{\langle{#1}\rangle}
\newcommand{\wt}[1]{\widetilde{#1}}
\newcommand{\wh}[1]{\widehat{#1}}
\newcommand\quash[1]{}
\newcommand\un{\underline}
\newcommand\ov{\overline}
\newcommand\bs{\backslash}
\newcommand\ot{\otimes}
\newcommand{\op}{\oplus}
\newcommand{\tl}[1]{[\![#1]\!]}
\newcommand{\lr}[1]{(\!(#1)\!)}

\newcommand\upH{\textup{H}}

\newcommand\dG{G^{\vee}}
\newcommand\dB{B^{\vee}}

\newcommand\dH{H^{\vee}}

\newcommand\dT{T^{\vee}}
\newcommand\dL{L^{\vee}}
\newcommand\dg{\mathfrak{g}^{\vee}}

\newcommand\pt{\textup{pt}}

\newcommand\vn{\varnothing}
\newcommand\hs{\heartsuit}

\newcommand\xr{\xrightarrow}
\renewcommand\c{\circ}

\newcommand{\Wh}{\textup{Wh}}

\newcommand{\Coh}{\textup{Coh}}

\newcommand{\Sh}{\mathit{Sh}}

\newcommand{\triv}{\mathit{triv}}

\newcommand{\new}{\mathit{new}}
\newcommand{\old}{\mathit{old}}

\newcommand{\beq}{\begin{equation}}
\newcommand{\eeq}{\end{equation}}

\newcommand{\ssupp}{\mathit{ss}}

\newcommand{\oo}{\infty}

\newcommand{\ol}{\overline}

\newcommand{\colim}{\operatorname{colim}}
\newcommand{\Fun}{\operatorname{Fun}}

\newcommand\sss{\subsubsection}

\newcommand\Cat{\operatorname{Cat}}
\newcommand\Corr{\operatorname{Corr}}

\renewcommand\r{\rho}

\newcommand\bu{\bullet}

 \renewcommand{\new}{\mathit{new}}
 \renewcommand{\old}{\mathit{old}}

 \newcommand\ch{\textup{ch}}
 
 \newcommand\sne{\subsetneqq}
\newcommand\f{\frac}
\newcommand\Wa{W^{a}}
\newcommand\Fac{\textup{Fac}}
\newcommand\Crit{\textup{Crit}}
\newcommand\nb{\nabla}
\newcommand\bn{\mathbf{n}} 
 
\newcommand\pl{\partial}
\newcommand\NP{\textup{NP}}
\newcommand\opp{\textup{opp}}
\newcommand\Sets{\mathcal{S}ets}

\newcommand\sft{\subset_{ft}} 
\newcommand\GSpr{\mathrm{GSpr}}
\newcommand\lmod{\textup{-mod}}

\newcommand\bi{\mathbf{i}}
\newcommand\im{\imath}
\newcommand\str{\mathit{star}}
\newcommand\dGder{G^{\vee,\der}}
\newcommand\dGsc{G^{\vee,\sc}}
\newcommand\dLder{L^{\vee,\der}}
\renewcommand\dg{\mathfrak{g}^{\vee}}
\newcommand\dt{\mathfrak{t}^{\vee}}
\newcommand\Nil{\mathrm{Nil}}
\newcommand\Spr{\mathrm{Spr}}
\newcommand\sign{\mathit{sign}}

\newcommand{\scI}{\mathscr{I}}
\newcommand{\bw}{\mathbf{w}}

\usepackage[mathscr]{euscript}

\def\hatW{\widehat{W}}
\def\hatfrC{\widehat{\mathfrak{C}}}
\def\tilfrC{\widetilde{\mathfrak{C}}}
\def\tilfrS{\widetilde{\mathfrak{S}}}
\def\hatfrS{\widehat{\mathfrak{S}}}

\newcommand{\scrQCoh}{\mathscr{QC}\textup{oh}}

\title{Functions on the commuting stack via  Langlands duality}
\author{Penghui Li}
\address{Yau Mathematical Sciences Center\\Tsinghua University\\Beijing, China}
\email{lipenghui@mail.tsinghua.edu.cn}
\author{David Nadler}
\address{Department of Mathematics\\University of California, Berkeley\\Berkeley, CA  94720-3840}
\email{nadler@math.berkeley.edu}
\author{Zhiwei Yun}
\address{Department of Mathematics, MIT,  77 Massachusetts Ave., Cambridge, MA 02139}
\email{zyun@mit.edu}
\date{\today}
\subjclass[]{}

\dedicatory{}	
\date{\today}

\begin{document}


\begin{abstract}
For a complex reductive group, 
we construct a semi-orthogonal decomposition of  the cocenter of the universal variant of its affine Hecke category. We use this to  calculate the endomorphisms of a Whittaker object in the cocenter via a diagram organizing  parabolic induction of character sheaves. 
Assuming a universal variant of Bezrukavnikov's spectral description of the affine Hecke category, we deduce a formula for  the dg algebra of global functions on commuting stacks of complex reductive groups. 
In particular, the formula shows that the ring of invariant functions on the commuting scheme is reduced. 
\end{abstract}

\maketitle

\tableofcontents



\section{Introduction}

This paper is part of 
a broader study of the cocenter of the universal affine Hecke category.
Its main results include (i) a semi-orthogonal decomposition of the cocenter, as found in Theorem~\ref{thm:main in intro}, whose distinguished case spelled out in Theorem~\ref{thm:intro ff} proves a conjecture of~\cite{liUniformizationSemistableBundles2021}; and (ii)  the calculation of endomorphisms of a distinguished Whittaker object in the cocenter, as found in Theorem~\ref{thm:intro end in crit},
which builds on work of~\cite{liDerivedCategoriesCharactera,liDerivedCategoriesCharacterb}.
In a sequel, we will combine the results of this paper with those of~\cite{nadlerAutomorphicGluingFunctor} to identify the cocenter of the universal affine Hecke category with the genus one automorphic category in Betti Geometric Langlands, 
giving an elliptic realization of affine character sheaves.
Our aim in this paper is to show that,  assuming a universal variant of Bezrukavnikov's spectral description of the affine Hecke category as formalized in Ansatz~\ref{ans:intro univ aff equiv},
we obtain a formula, as found in Theorem~\ref{intro main cor},  for the dg algebra of global functions on the commuting stack of a reductive group. When the corresponding derived group is simply-connected, the formula was
 conjectured in~\cite[Conjecture 1]{berestRepresentationHomologyTopological2022}, but to our knowledge is  new in general.
 In particular, it shows that the ring of invariant functions on the commuting scheme is reduced. (See Remark~\ref{rem:history} for a brief discussion of prior work on this question.)

In the rest of the introduction, we will first explain the main results of the paper,  then turn to more details of this application. 

%
%
%


\subsection{Background}
Our results apply to a general reductive group $G$ over the complex numbers $\mathbb C$.
But in much of the introduction, we will make the simplifying assumption that $G$ is {\em almost simple and simply-connected}. 
To begin, we summarize known results about the universal variant of the affine Hecke category that will be used in this paper.


\sss{Universal affine Hecke category}\label{sss:intro univ Hk}
Let $B\subset G$ be a Borel subgroup, $U \subset B$ its unipotent radical, 
and $H = B/U$ the universal Cartan. 
Let $\cG = G\lr{t}$ be the loop group, $\cI\subset \cG$ the standard Iwahori subgroup corresponding to $B$,  $\cI^u\subset \cI$ its pro-unipotent radical,
and note $H \simeq \cI/\cI^u$.

Fix a maximal torus $T\subset B \subset G$. Let $W = N_G(T)/T$ be the Weyl group of $G$, $\xcoch(T) = \Hom(\GG_m, T)$ the coweight lattice of $T$, and  $\tilW = N_{\cG}(T)/T\tl{t} \simeq \xcoch(H)\rtimes W$ the extended affine Weyl group of $\cG$.
Let $W^{a}\subset \tilW$ be the affine Weyl group generated by affine simple reflections, 
and set $\Omega = \wt W/W^a\cong N_{\cG}(\cI)/\cI$.


 The universal finite Hecke category of $G$ (resp. universal affine Hecke category of $G$) is the convolution monoidal dg derived category
  of $H$-bimonodromic complexes of sheaves of $\CC$-modules
\begin{equation*}
\cH_{G} =
\Sh_\bimon(U \bs G/U)
\qquad 
\text{(resp. $\cH_{\cG} = \colim_{w\in \tilW}\Sh_\bimon(\cI^u \bs \cG_{\le w}/\cI^u)$ 
)}
\end{equation*}
where the latter is with respect to the natural ind-scheme structure  $ \cG = \colim_{w\in \tilW}  \cG_{\le w}$.

\sss{Cocenter}
We will regard  $\cH_G$ and $\cH_{\cG}$ as algebra objects in $\CC$-linear stable presentable $\oo$-categories $\St^{L}_\CC$ (where morphisms are left adjoints) and make constructions therein. It is also sometimes useful to regard 
 $\cH_G$ and $\cH_{\cG}$ as algebra objects in the   $\CC$-linear stable presentable bimodule $\oo$-category $\Bimod_{\cH_H}(\St^{L}_\CC)$ where the finite Hecke category $\cH_H = \Sh_{mon}(H)$ for the
 universal Cartan  naturally acts by left and right convolution. We will formulate our results in the absolute setting, but most hold in this relative setting as well; in fact, certain proofs are made easier by shifting to the relative setting.

Recall 
the cocenter of a monoidal category $\cA$  is the Hochschild homology category
\begin{equation*}
hh(\cA) =   \cA  \otimes_{\cA  \otimes \cA^{op}} \cA
\end{equation*}
The trace map is the natural projection
\begin{equation*}
\xymatrix{
\tr: \cA \ar[r] &  \cA  \otimes_{\cA  \otimes \cA^{op}} \cA = hh(\cA ) 
}
\end{equation*}
induced by the unit of $\cA$ in the first factor of the tensor.

The closed embedding $G\subset \cG$ induces a natural fully faithful monoidal functor $\cH_G \to \cH_{\cG}$, and in turn a natural map on cocenters we will denote by
\begin{equation}\label{intro a}
\xymatrix{a:hh(\cH_G) \ar[r] &  hh(\cH_{\cG}).}
\end{equation}

As a special case of a more general coequalization result in Section~\ref{sect:horo desc},  we show for the universal finite Hecke category $\cH_G$, there  is a natural identification of its cocenter (see Theorem~\ref{thm:CS G})
\begin{equation}\label{intro CS}
hh(\cH_G) \simeq \Sh_\cN(G/G)
\end{equation}
where $\Sh_\cN(G/G)$ is the dg derived category of complexes of  sheaves of $\CC$-modules with nilpotent singular support on the adjoint-quotient $G/G$. One can think of $\Sh_\cN(G/G)$ as a version of character sheaves on $G$.  (Similar results have been shown for fixed central character and $\cD$-modules in
 \cite{ben-zviCharacterTheoryComplex}, \cite{bezrukavnikovCharacterDmodulesDrinfeld2012} and \cite{lusztigTruncatedConvolutionCharacter}.)

A guiding goal is to arrive at a similar geometric description of  
the  cocenter $hh(\cH_{\cG})$ of the universal affine Hecke category $\cH_{\cG}$.
As a starting point, we will use a universal monodromic version of a result of Tao-Travkin~\cite{taoAffineHeckeCategory}. In Appendix~\ref{app:TT}, we provide the necessary extensions to apply their arguments  to the universal monodromic case. 

Let $I$ be the set of simple roots of $G$ and $I^a$ be the set of affine simple roots of $G$. For each proper subset  $J\sne I^a$, let $L_J \subset \cG$ be the corresponding Levihoric, and $ \cH_{L_J} \subset \cH_{\cG}$ the universal finite Hecke category of $L_J$. For $J \subset J' \sne I^a$, we have a natural diagram of  monoidal inclusions $\cH_{L_J} \subset \cH_{L_{J'}}$.

\begin{theorem}[Tao-Travkin~\cite{taoAffineHeckeCategory}, see Appendix~\ref{app:TT}]\label{thm:intro colim}
Assume $G$ is almost simple and simply-connected. Then the natural maps induce a monoidal equivalence
\begin{equation*}
\xymatrix{
\colim^{\ot}_{J \sne I^a} \cH_{L_J} \ar[r]^-\sim  & \cH_{\cG}
}
\end{equation*}
where the colimit is of monoidal categories.
\end{theorem}


\subsection{Main automorphic results}
Now we formulate our main results regarding the cocenter $hh(\cH_{\cG})$.

\sss{Whittaker objects} By the Whittaker functional on   $ \Sh_\cN(G/G)$, we mean the functor 
\begin{equation*}
\xymatrix{
\cW_{G/G}: \Sh_\cN(G/G)\ar[r] & \Mod_\CC & \cW_G(\cF) = \varphi_0 \chi_* r^!(\cF)
}
\end{equation*}
given by the right-adjoint transport $ \chi_* r^!$ across the Whittaker correspondence
\begin{equation*}
\xymatrix{
G/G & \ar[l]_-r U/U \ar[r]^-\chi & \AA^1
}
\end{equation*}
followed by vanishing cycles $\varphi_0$ for the coordinate function on $\AA^1$.
Here $\chi$ is induced from a generic character $U\to U/[U, U] \simeq \oplus_{i\in I} \AA^1_i \to \AA^1$.

By the  Whittaker object  $\Wh_{G/G} \in  \Sh_\cN(G/G)$, we mean the object corepresenting $\cW_{G/G}$ in the sense of a natural equivalence $\cW_{G/G}(\cF) \simeq \Hom(\Wh_{G/G}, \cF)$ for $\cF\in \Sh_\cN(G/G)$.

Recall the natural maps $a: \Sh_\cN(G/G) \simeq hh(\cH_G) \to hh(\cH_{\cG})$ in \eqref{intro a} and \eqref{intro CS}. For the purpose of introduction, we define the cocenter  Whittaker object $\Wh_{\cG/\cG}$ as
\beq
\Wh_{\cG/\cG} := a(\Wh_{G/G}) \in hh(\cH_\cG).
\eeq
The actual definition of $\Wh_{\cG/\cG}$ in the main text uses the notion of descended trace  (see Definition~\ref{def:desc tr}), which makes the above identity a nontrivial theorem (see Theorem~\ref{thm: dtr of whit}).


\sss{Fully faithful embedding from colimit of character sheaves to affine cocenter}
To calculate the derived endomorphism ring of $\Wh_{\cG/\cG}$, we will identify a full subcategory of $hh(\cH_{\cG})$ containing $\Wh_{\cG/\cG}$, in which the endomorphism ring is easier to compute.

Recall  the notion of the cocenter of a monoidal category  $\cA$ generalizes to the Hochschild homology category of any $\cA$-bimodule $\cM$. One defines
\begin{equation*}
hh(\cA, \cM) =   \cA  \otimes_{\cA  \otimes \cA^{op}} \cM
\end{equation*}
with trace map the natural projection
\begin{equation*}
\xymatrix{
\tr: \cM \ar[r] &  \cA  \otimes_{\cA  \otimes \cA^{op}} \cM = hh(\cA, \cM) 
}
\end{equation*}
induced by the unit of $\cA$ in the first factor of the tensor. So the cocenter $hh(\cA)$ is the Hochschild homology category of the regular $\cA$-bimodule category $\cA$.

To begin, from Theorem~\ref{thm:intro colim}, for the bimodule category $\cH_{\cG}$, we  obtain an  equivalence after passing to Hochschild homology categories
\begin{equation*}
\xymatrix{
\colim_{J \sne I^a} hh(\cH_{L_J}, \cH_{\cG}) \ar[r]^-\sim  & hh(\cH_{\cG}, \cH_{\cG}) = hh(\cH_{\cG}).
}
\end{equation*}
We can restrict the domain of this equivalence to obtain a functor relating cocenters
\begin{equation*}
\xymatrix{
\colim_{J \sne I^a} hh(\cH_{L_J}) =  \colim_{J \sne I^a} hh(\cH_{L_J}, \cH_{L_{J}})\ar[r]  & hh(\cH_{\cG}, \cH_{\cG}) = hh(\cH_{\cG}). 
}
\end{equation*}

Our first main  result is the following theorem  conjectured in~\cite[see Claim 1.12]{liUniformizationSemistableBundles2021}.

\begin{theorem}\label{thm:intro ff} Recall $G$ is assumed to be almost simple and simply-connected for simplicity.
The natural map is a fully faithful embedding
\begin{equation*}
\xymatrix{
\colim_{J \sne I^a} hh(\cH_{L_J}) =  \colim_{J \sne I^a} hh(\cH_{L_J}, \cH_{L_{J}})\ar@{^(->}[r]  & hh(\cH_{\cG}, \cH_{\cG}) = hh(\cH_{\cG}). 
}
\end{equation*}
\end{theorem}

In fact, Theorem~\ref{thm:intro ff} is the ``semistable" part of a more general result that describes a semi-orthogonal decomposition of $hh(\cH_{\cG})$ indexed by ``Harder-Narasimhan" strata.
The generalization is not needed for the specific applications of this paper, but 
is a key input to the work in progress \cite{liCocenterAffineHecke}.  We will give a precise statement of the semi-orthogonal decomposition in Section~\ref{ss:intro semi-orth}.

Recall  $\Wh_{\cG/\cG}$ is the image of $\Wh_{G/G} \in  \Sh_\cN(G/G)\simeq hh(\cH_G)$ under the map $a$ defined in \eqref{intro a}. Note $G = L_I$, for $I \subset I^a$ the finite simple roots. Let us write  $\Wh_{G, I} $ for the image of $\Wh_{G/G}$ under the natural map 
\begin{equation*}
\xymatrix{
\Sh_\cN(G/G) \simeq hh(\cH_G) = hh(\cH_{L_I}) \ar[r] & \colim_{J \sne I^a} hh(\cH_{L_J}) \simeq \colim_{J \sne I^a}\Sh_{\cN}(L_{J}/L_{J}).
}
\end{equation*}
By Theorem~\ref{thm:intro ff}, to calculate the derived endomorphism ring of $\Wh_{\cG/\cG}$, it remains to calculate the derived endomorphism ring of $\Wh_{G, I}$ as an object in $\colim_{J \sne I^a}\Sh_{\cN}(L_{J}/L_{J})$.

\sss{Calculation of endomorphisms of $\Wh_{G, I}$}\label{sss:intro cal endo}
To state the output of our calculation of $\End(\Wh_{G,I})$, we need to recall some constructions from generalized Springer theory. Here we will first see the appearance of the dual group $G^\vee$.


Let $Z(G)$ be the center of $G$, and $\Irr(Z(G))$ be the group of irreducible complex characters of $Z(G)$.  
To each $\chi \in \textup{Irr}(Z(G))$, under the generalized Springer correspondence, the local system on the regular nilpotent orbit of $G$ with central character $\chi$ appears in the parabolic induction of a cuspidal local system on a Levi subgroup $L_{\chi}\subset G$, unique up to conjugation. Set  $T_\chi$ to be the connected center of $L_\chi$ with Lie algebra $\frt_\chi$, and let $T_\chi^\vee$ denote the dual torus with Lie algebra $\frt^\vee_\chi$.   Set $W_\chi = N_G(L_\chi)/L_\chi$ to be the relative Weyl group which acts on $T_{\chi}$ and $\frt_{\chi}$,  hence also on $T^{\vee}_{\chi}$ and $\frt^{\vee}_{\chi}$. 

\begin{theorem}[See Theorem~\ref{thm:end_of_Wh}]\label{thm:intro end in crit}  There is a canonical equivalence of dg algebras
\begin{equation*}
\End(\Wh_{G, I}) \simeq
 \oplus_{\chi \in \Irr(Z(G))} \cO(T^\vee_\chi \times T^\vee_\chi \times \frt^\vee_\chi[-1])^{W_\chi}.
\end{equation*}
\end{theorem}
To explain some of the additional notation in the theorem: $\frt^{\vee}_{\chi}[-1]$ denotes the affine derived scheme with coordinate ring $\cO(\frt^{\vee}_{\chi}[-1])$ equal to the exterior algebra $
\Sym^*((\frt_\chi^{\vee})^*[1])$ generated in degree $-1$.

Note that $\Irr(Z(G))$ can be canonically identified with $\pi_{1}(\dG)$, therefore the index set of the above decomposition can be replaced with $\pi_{1}(\dG)$, which is what appeared in Theorem \ref{intro main cor}.

The proof of the theorem uses generalized Springer theory and the decomposition of character sheaves of~\cite{liDerivedCategoriesCharactera,liDerivedCategoriesCharacterb}.

Theorems \eqref{thm:intro ff} and \ref{thm:intro end in crit} together provide the following description of the derived endomorphism ring of the cocenter  Whittaker object  $\Wh_{\cG/\cG}\in hh(\cH_\cG)$.

  

\begin{theorem}[See Corollary~\ref{c:end_of_Wh}]\label{thm:intro end calc} There is an equivalence of dg algebras
\begin{equation*}
\End(\Wh_{\cG/\cG})  \simeq \oplus_{\chi \in \Irr(Z(G))} \cO(T^\vee_\chi \times T^\vee_\chi \times \frt^\vee_\chi[-1])^{W_\chi}.
\end{equation*}
\end{theorem}

%
%


%
%
%
%
%


\subsection{Parabolic character sheaves and semi-orthogonal decomposition of the cocenter}\label{ss:intro semi-orth} 
We describe here some of the  arguments that lead to the proof of Theorem \ref{thm:intro ff}.  The reader interested in applications could skip to Section~\ref{s:intro app} without any disruption. 

The results we are about to present here are not used in full strength in the proof of Theorem \ref{thm:intro ff}, but they will be used  in future work to identify the affine cocenter with the genus one Betti geometric Langlands automorphic category. 

For simplicity, we will continue to assume here $G$ is almost simple and simply-connected.


\sss{Parabolic character sheaves and Hochschild homology}
In Section~\ref{sect:horo desc}, we give a geometric realization of the Hochschild homology $hh(\cH_{L_J}, \cH_{\cG})$ in terms of Lusztig's theory of parabolic (or rather parahoric) character sheaves.

For $J \sne I^a$, let $\cP_J \subset \cG$ be the standard parahoric subgroup with Levi quotient $L_{J}$, and $\cP_J^u \subset \cP_J$ its unipotent radical.  Let $\cH_{\cG,J}$ be the full subcategory
\begin{equation*}
\cH_{\cG,J}=\Sh_\cN\left( \f{\cP_{J}^{u}\bs\cG/\cP_J^u}{\Ad(L_J)}\right)
\end{equation*}
where $\Sh_{\cN}(-)$ means sheaves whose singular support is nilpotent under the moment map for the left (or right) $L_{J}$-action when pulled back to $\cP_{J}^{u}\bs\cG/\cP_J^u$. Then $\cH_{\cG,J}$ can be viewed as a Betti version of Lusztig's parabolic character sheaves for loop groups.

\begin{theorem}[see Theorem \ref{thm:coequal J}]\label{thm:intro geom coeq J}
For $J\sne I^{a}$, there is a canonical equivalence
\begin{equation*}
hh(\cH_{L_J}, \cH_{\cG})\simeq \cH_{\cG,J}.
\end{equation*}
\end{theorem}
Moreover, for $J\subset J'\sne I^{a}$, the natural functor $hh(\cH_{L_J}, \cH_{\cG})\to hh(\cH_{L_J'}, \cH_{\cG})$ gets transported under the above equivalence to the functor $\cH_{\cG,J}\to \cH_{\cG,J'}$ given by a natural horocycle correspondence.

Under the identifications in the theorem, the diagram of full subcategories
\begin{equation*}
\xymatrix{
J \sne I^a \mapsto  hh(\cH_{L_J}, \cH_{L_J})  =\Sh_{\cN}(L_{J}/L_{J})
}
\end{equation*}
is identified with the full subcategory of $\cH_{\cG,J}$ of sheaves supported on $\f{\cP_{J}^{u}\bs \cP_{J}/\cP^{u}_{J}}{L_{J}}$, which is essentially the adjoint quotient $L_{J}/L_{J}$. The transition functors for $J\subset J'$ are given by parabolic induction. 

Theorem \ref{thm:intro geom coeq J} is a consequence of a very general result we prove in Section~\ref{sect:horo desc}, which says that for very general $\cH_{L}$-bimodules $\cM$ coming from geometry (where $L$ is a reductive group), the trace map $\cM\to hh(\cH_{L}, \cM)$ can be geometrically realized as the pull-push functor along a horocycle correspondence.

\sss{Semi-orthogonal decomposition of the cocenter}\label{sss:intro semi}

To state the semi-orthogonal decomposition of $hh(\cH_{\cG})$, we need the notion of Newton points. First, there is a Newton point map $\nu : W^{a} \to \xcoch(T)_{\QQ}^{+}$ from  the affine Weyl group $W^{a}$ to rational dominant coweights: for any $w\in W^a$, and sufficiently divisible $n$, we have $w^n\in  \xcoch(T)$, and set $\nu(w)\in \xcoch(T)_{\QQ}^{+}$ to be the rational dominant coweight so that $n\nu(w)$ and $w^n$ are in the same $W$-orbit.
The Newton point map $\nu$ is invariant under conjugation by $W^{a}$; we denote by $\NP\subset \xcoch(T)_{\QQ}^{+}$ its image. Note $\NP$ has  a natural positive coroot partial ordering but we will work with a coarser order.

Next, to each Newton point $\nu \in \NP$, we associate a finite simplicial complex\footnote{Strictly speaking, $\frB^\hs_\nu$ may  only be a simplicial complex after a barycentric subdivision: as naturally constructed, the intersection of two simplices in $\frB^\hs_\nu$ may be a union of more than one simplex.}    $\frB^\hs_\nu$.  For example, when $\nu=0$, $\frB^\hs_0$ recovers the fundamental alcove of $\frA$. For each facet $\s$ of $\frB^{\hs}_{\nu}$, we attach a category of (possibly twisted) character sheaves $\Sh_\cN(\cY_{\nu, \sigma})$. Together they form a cosheaf of categories on the poset opposite to the set of facets of $\frB^\hs_\nu$, where the transition functors are given by (twisted) parabolic induction. For example, over $\frB^\hs_0$, we recover the cosheaf 
of categories  $J\mapsto\Sh_\cN(L_J/L_{J})$ for $J \sne I^a$, with transition maps given by the usual parabolic induction.


\begin{theorem}\label{thm:main in intro}  The cocenter category $hh(\cH_{\cG})$ has a semi-orthogonal decomposition indexed by non-negative integers $n\ge0$ with the $n$-th associated graded category of the form
\begin{equation*}
hh(\cH_{\cG})_{n}=\bigoplus_{\nu\in\NP, \j{2\r,\nu}=n}hh(\cH_{\cG})_{\nu}
\end{equation*}
and
\begin{equation*}
hh(\cH_{\cG})_{\nu}\simeq \colim_{\sigma \subset \frB^\hs_\nu }\Sh_{\cN}(\cY_{\nu, \sigma})
\end{equation*}
\end{theorem}

In particular, the first filtered piece $hh(\cH_{\cG})_{0}$ is a full subcategory of $hh(\cH_{\cG})$, and the theorem identifies it with the colimit $\colim_{J\sne I^{a}}\Sh_{\cN}(L_{J}/L_{J})$. This leads to the fully faithfulness asserted in Theorem~\ref{thm:intro ff}.

Theorem \ref{thm:main in intro} is a categorification of a result of Xuhua He \cite{heCocenterPadicGroups2018} where, among other things, he gave a decomposition of the cocenter of the affine Hecke algebra indexed by Newton points.

%
%
%
%

\sss{Idea of proof of Theorem \ref{thm:main in intro}}
The essential idea of the proof of Theorem~\ref{thm:main in intro} is to perform a categorical version of Morse theory on a cosheaf on a certain topological space $\frB_{\nu}$ (or rather the poset of its facets) that encodes the combinatorics of conjugacy classes in the affine Weyl group $W^a$. 

We sketch the construction of $\frB_{\nu}$. Let $\frA$ be the apartment  associated to $T$ in the building of $\cG$. We regard $\frA$ as a labelled simplicial complex, with each facet labelled by its type $J \subsetneq I^a$. We view the labelling as a simplicial map $\frA \to \frA/W^{a} \simeq \D$ where $\D$ is the fundamental alcove whose facets are in bijection with  $J \sne I^a$.

Given a conjugacy class $[w]\subset W^a$ of an element $w \in W^a$, with centralizer $C_w \subset W^a$, we may identify $[w]\simeq W^a/C_w$   with the open alcoves (open facets, or equivalently, $J$-facets with $J =\vn$) in the quotient space $X_{[w]} = \frA/C_w$ which depends only on the conjugacy class $[w]$. In general, the  $J$-facets of $X_w$, for any $J \sne I^a$,  index the image of $\cO_{[w]}$ under the natural projection to the  adjoint quotient $W^a \to W^a/\Ad(W_J)$

For each $\nu \in \NP$, we glue together the $X_{[w]}$, for conjugacy classes of $[w]$ with Newton point $\nu$, into a simplicial complex \footnote{Same comment as above:  $\frB_\nu$ may only be a simplicial complex after a barycentric subdivision: as naturally constructed, the intersection of two simplices in $\frB_\nu$ may be a union of more than one simplex.} $\frB_{\nu}$. The gluing procedure relies on the combinatorics called {\em pieces} for the affine Weyl group (see Section~\ref{sss:comb piece}). The space $\frB^{\hs}_{\nu}$ mentioned in Section~\ref{sss:intro semi} is a subspace of $\frB_{\nu}$ which we call the {\em essential part} of $\frB_{\nu}$. 

There is a natural function $f_\nu: \frB_\nu \to \RR$ obtained  by gluing a certain quadratic function on $X_{[w]}$ introduced by He and Nie~\cite{heMinimalLengthElements2014} in their work on minimal length elements in conjugacy classes. 
The critical locus $\Crit(f_{\nu})\subset \frB_{\nu}$ is contained in the subspace $\frB^{\hs}_{\nu}\subset \frB_{\nu}$.

For example, when $\nu=0$, $\frB_{0}$ is obtained by gluing $X_{[w]}$ for all conjugacy classes $[w]$ of finite order. The critical locus of $f_{0}$, which coincides with $\frB^{\hs}_{0}$ in this case, is exactly the image of $X_{[1]}\to \frB_{0}$, which can be identified with the fundamental alcove $\D$. 

%

For each $J$, the category of parahoric character sheaves $\cH_{\cG,J}$ has a semi-orthogonal decomposition  given by the stratification of $\f{\cP_{J}^{u}\bs \cG/\cP^{u}_{J}}{L_{J}}$ by {\em geometric pieces}, which were introduced by Lusztig \cite{lusztigParabolicCharacterSheaves}. The strata in $\cH_{\cG,J}$ are indexed by $J$-facets of $\frB=\coprod_{\nu\in\NP}\frB_{\nu}$. When we try to compute the colimit $\colim_{J}\cH_{\cG,J}$, the complication is that the transition functors do not respect the semi-orthogonal decompositions. However, by performing a categorical version of Morse theory on $\frB_\nu$, we are able to show that only the pieces of $ \cH_{\cG, J}$ indexed by the essential part   $\frB_\nu^\hs  \subset \frB_\nu$ contribute to the cocenter.

There are two underlying reasons we are able to implement  categorical Morse theory (and specifically, a contraction principle) in our situation. One is that the strata categories of parahoric character sheaves attached to facets of $\frB_{\nu}$ have a certain local constancy property. This is a consequence of a geometric result proved by He \cite{heGENERALIZATIONCYCLICSHIFT}. Another is that the embedding $\frB^{\hs}_{\nu}\incl \frB_{\nu}$ is a homotopy equivalence, which uses the gradient flow of the function $f_{\nu}$. In Appendix \ref{s:str sh}, we collect general methods of calculating colimits indexed by posets, and prove a general contraction principle for cosheaves of categories (see Theorem \ref{th:contracting cosheaf}).

\subsection{Application to commuting stacks}\label{s:intro app}
Here we turn to the application of our prior results to commuting stacks for the dual group $G^\vee$. The results are proved in Section~\ref{s:comm} in the text, conditional on a  
universal variant of Bezrukavnikov's spectral description of the affine Hecke category~\cite{bezrukavnikovTwoGeometricRealizations2016} stated below as Ansatz~\ref{ans:intro univ aff equiv}.  
We expect a proof of Ansatz~\ref{ans:intro univ aff equiv}  to appear in forthcoming 
work of Chen-Dhillon-Taylor, building on Taylor's universal variant~\cite{taylor} of Soergel's Endomorphismensatz.

In this subsection, the main player is $\dG$, a  connected reductive group over $\CC$.

\sss{Spectral realization of universal affine Hecke category}

Given a Borel subgroup  $\dB\subset \dG$, 
the universal Steinberg stack is the fiber product of adjoint quotients
\begin{equation*}
\St_{\dG} = \dB/\dB \times_{\dG/\dG} \dB/\dB
\end{equation*} 
(Here the derived and naive fiber product agree.)
More geometrically, it is the moduli
\begin{equation*}
\St_{\dG}  \simeq  \Loc_{\dG, \dB }(S^1 \times [0,1], S^1 \times \{0, 1\})
\end{equation*} 
of $\dG$-local systems on the cylinder $S^1 \times [0,1]$ with $\dB$-reductions along the boundary $S^1 \times \{0,1\}$.

The universal spectral affine Hecke category is the monoidal convolution category
$
\Ind\Coh(\St_{\dG})  
$
of ind-coherent sheaves on the universal Steinberg stack.

\begin{ansatz}\label{ans:intro univ aff equiv}
There is a monoidal equivalence of universal affine Hecke categories
\begin{equation}
\label{univ aff equiv}
\xymatrix{
\Phi:\Ind\Coh(\St_{\dG}) \ar[r]^-\sim &   \cH_{\cG}
}
\end{equation} 
with the following properties:
\begin{enumerate}
\item  $\Phi$ identifies the structure sheaf $\cO_{\St_{\dG}}$ with the universal affine  Whittaker object $\Wh_{\cG}$ as coalgebra objects. (See Section~\ref{section:decended_trace_O}  for 
the coalgebra structure on $\cO_{\St_{\dG}}$, and Section~\ref{sss:aff whit} for the definition of $\Wh_{\cG}$ and its coalgebra structure.)

\item $\Phi$ is naturally an equivalence of  algebras in bimodules over $\qc(H^\vee)$. 
(Here we regard the $\cH_{H}$-bimodule $\cH_{\cG}$ as a $\qc(H^\vee)$-bimodule using the canonical monoidal equivalence $\cH_{H}\simeq  \qc(H^\vee)$.)

\end{enumerate}
\end{ansatz}

\begin{remark}
 Ansatz~\ref{ans:intro univ aff equiv} (2) is not used in this paper, we only record it for conceptual completeness.
\end{remark}

\sss{Derived commuting stacks} We will state both derived and underived assertions about commuting stacks. We deduce the  underived versions from the derived, and so begin there.

The derived commuting stack $Z^2_{\dG}$ is the derived moduli
of pairs $g_1, g_2 \in \dG$ with $g_1g_2g_1^{-1} g_2^{-1} =1$ up to conjugation.
More formally, it is the adjoint quotient of the derived fiber product
\begin{equation*}
Z^2_{\dG} \simeq  (( \dG\times \dG) \times^{\bR}_{\dG} \{1\} )/\dG 
\end{equation*} 
with respect to the
 commutator map  $\dG\times \dG\to \dG$, $(g_1, g_2) \mapsto g_1g_2g_1^{-1} g_2^{-1}$ and unit element $1\in \dG$. The underlying classical stack of $Z^{2}_{\dG}$ is $\frC^{2}_{\dG}/\dG$. More geometrically, it is the moduli
\begin{equation*}
Z^2_{\dG} \simeq  \Loc_{\dG}(T^2)
\end{equation*} 
of $\dG$-local systems on the two-torus $T^2$, where $g_1, g_2$ are the monodromies around the two factors.

Our access to the commuting stack is via a result of~\cite{ben-zviSpectralIncarnationAffine2017}:

\begin{theorem}[\cite{ben-zviSpectralIncarnationAffine2017}]
There is a canonical equivalence
\begin{equation}\label{intro cocenter equiv}
\Ind\Coh_{\cN}(Z^{2}_{\dG}) \simeq hh(\Ind\Coh(\St_{\dG})).
\end{equation}

Hence assuming Ansatz \eqref{ans:intro univ aff equiv}, there is a canonical equivalence
\begin{equation}\label{intro cocenter equiv}
\Ind\Coh_{\cN}(Z^{2}_{\dG})\simeq hh(\cH_{\cG}).
\end{equation}
\end{theorem}

Within this identification, we match  the structure sheaf $\cO_{Z^{2}_{\dG}}$ and  cocenter Whittaker object $\Wh_{\cG/\cG}$:

\begin{theorem}[See Corollary~\ref{c:OZ Wh}] Assuming Ansatz \eqref{ans:intro univ aff equiv}, under the equivalence \eqref{intro cocenter equiv},  the structure sheaf $\cO_{Z^{2}_{\dG}}$ corresponds to the cocenter Whittaker object $\Wh_{\cG/\cG}$.
\end{theorem}

Now we can transport the calculation of the dg algebra $\End(\Wh_{\cG/\cG})$ of Theorem~\ref{thm:intro end calc} to calculate the dg algebra of global functions $\cO(Z^{2}_{\dG})$, since the latter is the derived endomorphism ring of $\cO_{Z^{2}_{\dG}}$.

Following \cite{borelAlmostCommutingElements}, $Z^{2}_{\dG}$ decomposes into connected components indexed by $\pi_{1}(\dGder)$, the fundamental group of the derived group $G^{\vee,\der}$ of $\dG$. For each $c\in \pi_{1}(\dGder)$, one defines a torus $\dT_{c}$, as the abelianization of a Levi subgroup $\dL_{c}$ of $\dG$,  with an action of $W_{c}=N_{\dG}(\dL_{c})/\dL_{c}$, the relative Weyl group of $\dL_{c}$. When $c=1\in \pi_{1}(G^{\vee,\der})$ is the identity, $\dT_{c}=\dT$ is a maximal torus of $\dG$, and $W_{c}$ is the usual Weyl group $W=W(\dG,\dT)$. For more details, we refer to Section~\ref{ss:alm comm}.


\begin{theorem}[Corollary \ref{c:derived fun comm c}]\label{intro main cor} 
Let $\dG$ be a connected reductive group over $\CC$. Assume  Ansatz~\ref{ans:intro univ aff equiv} holds.  

Then there is an equivalence of dg algebras
\begin{equation*}
\cO(Z^2_{\dG} ) \simeq\bigoplus_{c\in \pi_{1}(G^{\vee,\der})}\cO(Z^2_{\dT_{c}} )^{W_{c}}\simeq 
 \bigoplus_{c\in \pi_{1}(G^{\vee,\der})} \cO(T^\vee_c\times T^\vee_c \times \frt^\vee_c[-1])^{W_c}.
\end{equation*}
\end{theorem}
Here, $\frt^{\vee}_{c}=\Lie \dT_{c}$ and  the derived scheme $\frt^\vee_c[-1]$ is defined similarly as $\frt^{\vee}_{\chi}[-1]$ (see the paragraph after Theorem \ref{thm:intro end in crit}). 


\begin{remark}
When the derived group $G^{\vee,\der}$ is simply-connected, we have a simple equality
\begin{equation*}
\cO(Z^2_{\dG} ) \simeq \cO(T^\vee\times T^\vee \times \frt^\vee[-1])^{W}.
\end{equation*}
In this case, the above isomorphism was conjectured in~\cite[Conjecture 1]{berestRepresentationHomologyTopological2022}.
\end{remark}

\sss{Classical commuting stacks}
Now we turn to underived statements.

Let $\frC^{2}_{\dG}$ be the (underived) commuting scheme of $\dG$, the (underived)  closed subscheme of $\dG\times\dG$ consisting of $(g_{1},g_{2})\in \dG\times\dG$ satisfying the equation $g_1g_2g_1^{-1} g_2^{-1} =1$. The group $\dG$ acts on $\frC^{2}_{\dG}$ by simultaneous conjugation, and the quotient stack $\frC^{2}_{\dG}/\dG$ is the classical stack underlying the derived commuting stack $Z^{2}_{\dG}$. 

As above, following \cite{borelAlmostCommutingElements}, $\frC^{2}_{\dG}$ decomposes into connected components indexed by $\pi_{1}(\dGder)$, the fundamental group of the derived group $G^{\vee,\der}$ of $\dG$. For each $c\in \pi_{1}(\dGder)$, one defines a torus $\dT_{c}$, as the abelianization of a Levi subgroup $\dL_{c}$ of $\dG$,  
and a map of stacks
\begin{equation*}
\io_{c}:  (\dT_{c}\times \dT_{c})/W_{c}\to \frC^{2}_{\dG}/\dG
\end{equation*}
wheref $W_{c}=N_{\dG}(\dL_{c})/\dL_{c}$ is the relative Weyl group of $\dL_{c}$. 
For more details, we refer to Section~\ref{ss:alm comm}.

%
%


\begin{theorem}[See Theorem~\ref{thm:function c part}]\label{thm:intro classical functions} Let $\dG$ be a connected reductive group over $\CC$. Assume  Ansatz~\ref{ans:intro univ aff equiv} holds.  

Then the maps $\io_{c}$ for $c\in \pi_{1}(\dGder)$ induce an isomorphism on rings of invariant functions
\begin{equation*}
\cO(\frC^2_{\dG} ) ^{\dG}\simeq \bigoplus_{c \in \pi_{1}(\dGder)} \cO(\dT_c \times \dT_c)^{W_c}.
\end{equation*}
In particular, $\cO(\frC^2_{\dG})^{\dG}$ is reduced.
\end{theorem}
%
%


From  Theorem~\ref{thm:intro classical functions}, we also deduce the following description of invariant functions on the Lie algebra commuting scheme 
and
 Lie algebra-Lie group  commuting scheme.

\begin{theorem}[See Theorem \ref{thm:Lie alg} and Theorem \ref{thm:mix}]\label{thm:intro Lie mix}
Let $\dG$ be a connected reductive group over $\CC$ with maximal torus $\dT \subset \dG$, and respective Lie algebras $\dt \subset \dg$.  Assume  Ansatz~\ref{ans:intro univ aff equiv} holds.

Let $\frC^{2}_{\dg}$ be the Lie algebra commuting scheme of  pairs
$(X_{1},X_{2})\in \dg\times\dg$ satisfying $\ad_{X_1}(X_2) = 0$,  and $\frC_{\dg,\dG}$  the Lie algebra-Lie group 
commuting scheme of pairs $(X,g)\in \dg\times\dG$ satisfying $\Ad_g(X) = X$.

Then restrictions to $\dt\times \dt$ and $\dt\times \dT$ respectively give isomorphisms  on  rings of invariant functions
\begin{eqnarray*}
\cO(\frC^{2}_{\dg})^{\dG}\isom \cO(\dt\times\dt)^{W},\\
\cO(\frC_{\dg,\dG})^{\dG}\isom \cO(\dt\times\dT)^{W}.
\end{eqnarray*}
\end{theorem}

In fact,  it is  possible to adapt our methods and prove Theorem~\ref{thm:intro Lie mix}
 directly. With this approach, we do not need to assume  Ansatz~\ref{ans:intro univ aff equiv},
 but can simply appeal to Bezrukavnikov's equivalence \cite{bezrukavnikovTwoGeometricRealizations2016}. 
 Since we do not take this approach in this paper, we include Ansatz~\ref{ans:intro univ aff equiv} in the statement of  Theorem~\ref{thm:intro Lie mix}.

{
\begin{remark}\label{rem:history} The question of reducedness of the scheme of commuting pairs of matrices has a long history, and it is still open. The reducedness of the ring of adjoint-invariant functions on $\frC^{2}_{\dG}, \frC^{2}_{\dg}$ and $\frC^{2}_{\dg,\dG}$ for a general reductive group $\dG$ was also an open question considered by many authors.  It was known by Joseph \cite{joseph1997harish} 
and Etingof-Ginzburg \cite{etingofSymplecticReflectionAlgebras2002} 
that the {\em reduced} ring of $\cO(\frC^{2}_{\frg^{\vee}})^{\dG}$ is isomorphic to $\cO(\dt\times\dt)^{W}$. Prior to our work, the reduceness of $\cO(\frC^{d}_{\dg})^{\dG}$ was known for classical groups (see the next paragraph). We are not aware of results on $\frC^{2}_{\dG}$ or $\frC^{2}_{\dg,\dG}$ other than type $A$. Our results settle these reducedness questions completely and uniformly across all Lie types.

One may wonder about the reducedness of $\cO(\frC^{d}_{\dg})^{\dG}$, where $\frC^{d}_{\dg}$ is the scheme of commuting $d$-tuples in $\dg$, for any $d\ge2$. This is known to be true for classical groups. For $\dg=\gl_{n}$ and $d=2$,  this was proved by 
Gan and Ginzburg \cite{ganAlmostcommutingVarietyDmodules2010};
for general $d$, this was proved independently by Domokos \cite{domokosVectorInvariantsClass2009} 
and Vaccarino \cite{vaccarinoLinearRepresentationsSymmetric2007}. 
Recently, T-H. Chen and B-C.Ng\^o \cite{chenInvariantTheoryCommuting}
proved the case of symplectic Lie algebras for any $d$. More recently, L.Song, X.Xia and J.Xu \cite{songHigherdimensionalChevalleyRestriction} proved the case of orthogonal Lie algebras for any $d$.
\end{remark}
}


\subsection{Further results}
In Section~\ref{ss:additional app}, we use similar techniques to calculate the derived endomorphism rings of two other natural objects in $hh(\cH_{\cG})$ in terms of spectral data. One of the calculations can be interpreted as a derived spherical Hecke algebra, and the other one is the endomorphism ring of the universal Eisenstein series in the genus one automorphic category.

The  methods of this paper lead not only to a calculation of the endomorphisms of objects in the cocenter but to a full description of the cocenter of the universal affine Hecke category. The primary additional inputs are:

\begin{enumerate}
\item The automorphic gluing construction under nodal degenerations of curves~\cite{nadlerAutomorphicGluingFunctor}.
\item The description of nilpotent sheaves on degree zero semistable $G$-bundles on a  genus one curve~\cite{liUniformizationSemistableBundles2021}.
\end{enumerate}

Combined with the methods of this paper, we are able to prove the following to appear in a sequel~\cite{liCocenterAffineHecke}.

For $X$  a smooth projective curve, its Betti automorphic  category
$\Sh_\cN(\Bun_G(X))$ is the dg derived category of complexes of sheaves of $\CC$-modules on the moduli 
of $G$-bundles on $X$.

\begin{conj}\label{thm:betti genus one}
For $E$  a smooth projective genus one curve,
there is a natural equivalence from the cocenter of the universal affine Hecke category to the Betti automorphic  category
\begin{equation*}
\xymatrix{
hh(\cH_\cG) \ar[r]^-\sim & \Sh_\cN(\Bun_G(E))
}
\end{equation*}
\end{conj} 

We can invoke the spectral calculations of~\cite{ben-zviSpectralIncarnationAffine2017} to deduce the Betti geometric Langlands conjecture in genus one.
Recall for $X$ be a smooth projective curve, there is a natural spectral action on its 
Betti automorphic  category
$\Sh_\cN(\Bun_G(X))$
by the tensor category $\QCoh(\Loc_{\dG} (E))$
of quasi-coherent complexes on  the moduli 
of $\dG$-local systems on $X$.

\begin{cor}[of Conjecture~\ref{thm:betti genus one}]
For $E$  a smooth projective genus one curve,
the Betti geometric Langlands conjecture holds: there is an equivalence of $\QCoh(\Loc_{\dG} (E))$-module categories
\begin{equation*}
\xymatrix{
\Ind\Coh_\cN(\Loc_{\dG} (E)) \ar[r]^-\sim & \Sh_\cN(\Bun_G(E))
}
\end{equation*}
\end{cor}

%
%
\subsection{Conventions} In the rest of the paper, we will use the following notations.

\sss{} We fix a field $k$ of characteristic $0$ as the coefficient field for our sheaves and categories.

\sss{} We denote by $\St^{L}$ (resp. $\St^{R}$) the $\infty$-category of stable presentable categories with morphisms left adjoints (resp. right adjoints). 

We denote by $\St_k^{L}$ (resp. $\St^{R}_{k}$) the $\infty$-category of stable presentable $k$-linear categories with morphisms left adjoints (resp. right adjoints). In the many body of the paper, we will be working primarily with $\St^{L}_{k}$. By a monoidal category, we will typically mean an algebra object in $\St^{L}_k$.

 \sss{}
 For a complex algebraic stack $X$, we  denote by $\Sh(X)$ the $k$-linear dg derived category of complexes of sheaves in $k$-vector spaces on $X$ under the analytic topology. We will refer to objects in $\Sh(X)$ as {\em sheaves}.

If $X$ is smooth and $\L\subset T^{*}X$ is a conical closed subset, we denote by $\Sh_{\L}(X)$ the full subcategory of $\Sh(X)$ consisting of sheaves with singular support in $\L$. In particular, using $0$ to denote the zero section, $\Sh_{0}(X)$ is the full subcategory of sheaves with locally constant cohomology sheaves.

\sss{}\label{sss: g notation} Let $G$ be a connected reductive group over $\CC$. When needed, we will choose a maximal torus $T$ and a pair of opposite Borel subgroups $B$ and $B^{-}$ containing $T$. Let $U$ and $U^{-}$ be the unipotent radicals of $B$ and $B^{-}$. The quotient $H = B/U$ (the universal Cartan) is canonically independent of the choice of $B$. Let $r=\dim H$ be the rank of $G$. Let $W=W(G,T)$ be the Weyl group of $G$.  Let $\frg, \frt, \frb, \fru,\cdots$ denote the Lie algebras of $G, T, B, U, \cdots$.

We will fix an $\Ad(G)$-invariant non-degenerate symmetric bilinear form on $\frg$ to identify $\frg$ and $\frg^{*}$.

\sss{} For an affine algebraic group $L$ over $\CC$, let $\BB L=[(\Spec \CC)/L]$ be its classifying space, regarded as an Artin stack. 

%


\subsection{Acknowledgements}
We thank David Ben-Zvi and Quoc P. Ho for inspiring discussions, and Tsao-Hsien Chen, Peter Haine, Nick Rozenblyum and James Tao for generous technical help. We thank Xuhua He especially for providing a key geometric ingredient needed in this paper in the form of \cite{heGENERALIZATIONCYCLICSHIFT}. We thank Pavel Etingof for comments on the history of the problem of reducedness for commuting schemes.

PL was partially supported by the National Natural Science Foundation of China (Grant No. 12101348).
DN was partially supported by NSF grant DMS-2101466. ZY was partially supported by the Simons Investigatorship and the Packard Fellowship.


\section{Universal affine Hecke category}

In this section, we provide more details about the universal affine Hecke category as introduced in Section~\ref{sss:intro univ Hk}. But here and in the rest of the paper, unless otherwise stated, we follow the setup of Section~\ref{sss: g notation}, and work with $G$ a general {\em  connected complex reductive group}.


\subsection{Hecke categories}

\subsubsection{Finite Hecke categories}

Consider the quotient stack $U\bs G/U$,  and its convolution diagram
\begin{equation*}
\xymatrix{
&\ar[dl]_-{p_1} U\bs G/U \times U\bs G/U \ar[dr]^-{p_2} & \ar[l]_-\d U\bs G \times^U  G/U  \ar[r]^-\pi & U\bs G/U \\
U\bs G/U  && U\bs G/U   &
}
\end{equation*}

Note $\delta$ is smooth (a base change of the diagonal $\BB U\to \BB U \times \BB U$). The map $\pi$ is given by the multiplication on $G$. It is a base change of the projection $\BB U\to \BB G$ with fibers isomorphic to $G/U$. It behaves like a proper map for $H$-monodromic sheaves on the source. More precisely, $\pi$ has a factorization
\begin{equation*}
\xymatrix{
\pi:U\bs G \times^U  G/U  \ar[r]^-{h} &  U\bs G \times^B  G/U  \ar[r]^-{\pi'}  & U\bs G/U 
}
\end{equation*}
where $h$ is an $H$-torsor (a base change of $\BB U \to \BB B$  with fibers isomorphic to $H= B/U$), 
 and $\pi'$ is proper (a base change of the projection $\BB B\to \BB G$ with fibers isomorphic to $G/B$).
 In general, for any $H$-torsor $p:E\to X$, and a sheaf  $\cF \in \Sh(E) $ that is $H$-monodromic, so locally constant along the fibers of $p$, we have a canonical isomorphism
 \beq\label{push H torsor}
 p_!\cF \simeq p_*\cF[-r]
 \eeq
 where as usual $r=\dim H$ is the rank of $G$.
Indeed, we write $H(\CC)=H_{>0}H_{c}$ where  $H_{c}$ is the compact real form of $H$ and $H_{>0}$ the neutral component of the split real form of $H$. Then $p$ factors as
\begin{equation*}
\xymatrix{
p:E \ar[r]^{p_{0}} & E/H_{>0} \ar[r]^{p_{c}} & X.
}
\end{equation*}
Now $p_{c}$ is proper and $H_{>0}$ is contractible so $\cF$ descends to $\ov\cF\in \Sh(E/H_{>0})$. We have $p_{!}\cF=p_{c!}p_{0!}\cF=p_{c*}p_{0!}(p_{0}^{*}\ov\cF)\simeq p_{c*}(\ov\cF\ot p_{0!}k)$,  and $p_{*}\cF\simeq p_{c*}\ov\cF$. The relative fundamental class of $p_{0}$ gives a canonical isomorphism $p_{0!}k\simeq k[-r]\in \Sh(E/H_{>0})$, hence a natural isomorphism $p_{!}\cF\simeq p_{*}\cF[-r]$.

Returning to the convolution diagram, for any sheaf  $\cF \in \Sh(U\bs G \times^U  G/U) $ that is $H$-monodromic for the action $t\cdot (g_{1},g_{2})=(g_{1}t^{-1},tg_{2})$ (for $t\in H$, $g_{1},g_{2}\in G$), so locally constant along the fibers of $h$, we have canonical isomorphisms
 \begin{equation*}
 \pi_!\cF\simeq  \pi'_! h_!\cF \simeq  \pi'_! h_*\cF[-r] \simeq  \pi'_* h_*\cF[-r] \simeq  \pi_*\cF[-r].
 \end{equation*}
 Thus for such sheaves, the pushforward $\pi_!$ satisfies all the base change identities of a proper map.
 
 Consider the closed embedding of the unit coset
\begin{equation*}
\xymatrix{
U\bs B/U  \ar[r]^-u & U\bs G/U. 
}
\end{equation*}
 Let $q_{H}: U\bs B/U\to H$ be the natural projection, which is a $U$-gerbe.
Let  $\exp:\frh \to H$ be the universal cover, and introduce the universal local system
\begin{equation*}
\cL_\univ =  q^{*}_{H}\exp_! \cD_{\frh} \in \Sh_0(U\bs B/U)
\end{equation*} 
where $\cD_\frh  \simeq k_{\frh}[r] \in \Sh_0(\frh)$ is the Verdier dualizing sheaf. Note that $\cL_{\univ}$ is concentrated in degree $-r$ with stalks isomorphic to the group algebra $k[\xcoch(H)]$.

\begin{defn} The {\em universal finite  Hecke category} of $G$ is the  monoidal category of 
$H$-bimonodromic sheaves on $U\bs G/U$ (under the left and right translations of $H$)
\begin{equation*}
\cH_{G} =  \Sh_{\bimon}(U\bs G/U) 
\end{equation*}
equipped with convolution
\begin{equation*}
\xymatrix{
\star: \cH_{G}\otimes \cH_{G} \ar[r] &  \cH_{G}
&
\cF_1\star \cF_2 = \pi_!\delta^*(p_1^*\cF_1 \boxtimes p_2^*\cF_2)
}\end{equation*}
and unit object $e = u_!\cL_\univ$.
\end{defn}

It is easy to see that $H$-bimonodromic sheaves on $U\bs G/U$ are exactly 
$U\times U$-equivariant sheaves with nilpotent singular support when pulled back to $G$. Therefore we also denote $\Sh_{\bimon}(U\bs G/U)$ by $\Sh_\cN(U\bs G/U)$.

\begin{ex}
For the torus $H$, we have $\cH_{H} = \Sh_0(H)$ the dg derived category of locally constant sheaves on $H$. Convolution is simply the pushforward $\cL_1 \star \cL_2 = m_!(\cL_1 \boxtimes \cL_2)$ along the multiplication map $m:H \times H\to H$, and  
the unit object is the universal local system $e = \cL_\univ = \exp_!\cD_\frh$.
\end{ex}

\begin{remark}
We can also naturally regard $\cH_{G}$ as a monoidal category in $\cH_{H}$-bimodule categories. 
\end{remark}


\sss{Loop group and parahorics}
Let $\cG = G\lr{t}$ be the loop group of $G$, $\cI\subset \cG$ the Iwahori subgroup given by the preimage of $B$ under the reduction mod $t$ map $G\tl{t}\to G$. Let $I^a$ be the set of simple (affine) roots of $\cG$ with respect to $\cI$, and $I\subset I^{a}$  the subset of simple roots of $G$ with respect to $B$.

Let $\tilW=\xcoch(H)\rtimes W$ be the extended affine Weyl group of $\cG$ and $W^{a}\subset \tilW$  the affine Weyl group generated by affine simple reflections. We think of $\tilW$ as acting on the standard apartment $\frA=\xcoch(T)_{\RR}$. For a simple reflection $s\in W^{a}$, let $\a_{s}\in I^{a}$ denote the corresponding affine simple root.

A subset $J\subset I^{a}$ is of {\em finite type} if the subgroup $W_{J}$ generated by simple reflections $s$ for $\a_{s}\in J$ is finite.  We use the notation $J\sft I^{a}$ to mean that $J$ is a finite type subset of $I^{a}$. When $G$ is almost simple, $J\sft I^{a}$ simply means that $J\sne I^{a}$.
 
Given $J \sft I^a$, let $\cP_J\subset \cG$ be the standard parahoric subgroup containing $\cI$ of type $J$, i.e., if $\cP^u_J \subset \cP_J$ denotes its pro-unipotent radical, 
and $L_J = \cP_J/\cP^u_J$ its Levi quotient, then $J$ are the simple roots of $L_J$. 
Note that $B_J = \cI/(\cI \cap \cP_J^u)$ is a Borel subgroup of $L_J$, with unipotent radical $U_J = \cI^u/(\cI^u \cap \cP_J^u)$.

When $J= \vn$, we have $\cI = \cP_\vn $ and $H = L_\vn$,  and we write $\cI^u = \cP^u_\vn$. When $J=I$, we have $\cP_{I}=G\tl{t}$, $L_{I}=G$, $B = B_I$ and $U = U_I$.

We identify $\tilW$ with $N_{\cG}(T)/T\tl{t}$. For $w\in \tilW$ and any lift $\dot w\in N_{\cG}(T)$, the subspace $\cI\dot w \cI\subset \cG$ is independent of the choices of $T$ and $\dot w$, and we denote it by $\cG_{w}$. Let $\le $ denote the Bruhat order on $W^{a}$ extended to $\tilW$ by declaring $w_{1}$ and $w_{2}$ are incomparable if they are in different cosets of $W^{a}$.  Let $\cG_{\le w}=\cup_{w'\le w}\cG(w')$. It is well-known that $\cG_{w}/\cI\subset \cG/\cI$ is isomorphic to an affine space of dimension $\ell(w)$, and its closure is $\cG_{\le w}/\cI$.

\subsubsection{Affine Hecke categories}
As in the finite-dimensional case, 
consider the quotient stack $\cI^u\bs \cG/\cI^u$,  and its convolution diagram
\begin{equation*}
\xymatrix{
&\ar[dl]_-{p_1} \cI^u\bs \cG/\cI^u \times \cI^u\bs \cG/\cI^u\ar[dr]^-{p_2} & \ar[l]_-\d \cI^u\bs \cG \times^{\cI^u} \cG/\cI^u  \ar[r]^-\pi & \cI^u\bs \cG/\cI^u \\
\cI^u\bs \cG/\cI^u  && \cI^u\bs \cG/\cI^u&
}
\end{equation*}

Note $\delta$ is pro-smooth (a base change of the diagonal $\BB\cI^u\to \BB\cI^u \times \BB\cI^u$), and $\pi$ has the similar ``almost" ind-proper property as in the finite-dimensional case for $H$-monodromic sheaves. 
More precisely, consider the factorization
\begin{equation*}
\xymatrix{
\pi:\cI^u\bs \cG \times^{\cI^u}  \cG/\cI^u  \ar[r]^-{h} &   \cI^u\bs \cG \times^B   \cG/\cI^u  \ar[r]^-{\pi'}  & \cI^u \bs \cG/\cI^u
}
\end{equation*}
where $h$ is an $H$-torsor (a base change of $\BB\cI^u \to \BB\cI$  with fibers isomorphic to $H\simeq \cI/\cI^u$), 
 and $\pi'$ is ind-proper (a base change of the projection $\BB\cI\to \BB\cG$ with fibers isomorphic to the affine flag variety $\cG/\cI$).
For any sheaf  $\cF \in \Sh(\cI^u\bs \cG \times^{\cI^u}  \cG/\cI^u) $ that is $H$-monodromic with respect to the action $t\cdot (g_{1},g_{2})=(g_{1}t^{-1},tg_{2})$, so locally constant along the fibers of $h$, we have a canonical isomorphism by \eqref{push H torsor}
\begin{equation*}
h_!\cF \simeq   h_*\cF[-r].
\end{equation*}
On the other hand, if in addition $\cF$ is supported on $\cI^u\bs \cG_{\le w_{1}} \times^{\cI^u}  \cG_{\le w_{2}} /\cI^u $ for some $w_{1},w_{2}\in \tilW$, then since $\pi'$ is proper when restricted to $\cI^u\bs \cG_{\le w_{1}} \times^{\cI}  \cG_{\le w_{2}} /\cI^u $, we have canonical isomorphisms
 \begin{equation*}
 \pi_!\cF\simeq  \pi'_! h_!\cF \simeq  \pi'_! h_*\cF[-r] \simeq  \pi'_* h_*\cF[-r] \simeq  \pi_*\cF[-r]
 \end{equation*}
Thus for $H$-monodromic sheaves $\cF$ suppored on some $\cI^u\bs \cG_{\le w_{1}} \times^{\cI^u}  \cG_{\le w_{2}} /\cI^u $, the pushforward $\pi_!$ satisfies all the base change identities of a proper map.
 
 Consider the closed embedding of the unit coset
\begin{equation*}
\xymatrix{
\cI^u \bs \cI/\cI^u  \ar[r]^-u & \cI^u\bs \cG/\cI^u. 
}
\end{equation*}
Let $q_{H}: \cI^u\bs \cI/\cI^u\to H$ be the natural projection, which is a $\cI^{u}$-gerbe. 
Introduce the universal local system
\begin{equation*}
\cL_\univ =  q^{*}_{H}\exp_! \cD_{\frh} \in \Sh_0(\cI^u\bs \cI/\cI^u)
\end{equation*} 
where $\cD_\frh  \simeq k_{\frh}[r] \in \Sh_0(\frh)$ is the Verdier dualizing sheaf.

\begin{defn}  The {\em universal affine Hecke category} of $G$ is the colimit of $H$-bimonodromic sheaves on $\cI^u \bs \cG_{\le w}/\cI^u$ for $w\in \tilW$ 
\begin{equation*}
\cH_{\cG}=\colim_{w\in \tilW}\Sh_\bimon(\cI^u \bs \cG_{\le w}/\cI^u) 
\end{equation*}
with respect to the full embeddings $\Sh_\bimon(\cI^u \bs \cG_{\le w_{1}}/\cI^u)\incl  \Sh_\bimon(\cI^u \bs \cG_{\le w_{2}}/\cI^u)$ whenever $w_{1}\le w_{2}$.

It is equipped with a monoidal structure given by convolution
\begin{equation*}
\xymatrix{
\star: \cH_{\cG}\otimes \cH_{\cG} \ar[r] &  \cH_{\cG}
&
\cF_1\star \cF_2 = \pi_!\delta^*(p_1^*\cF_1 \boxtimes p_2^*\cF_2)
}\end{equation*}
and unit object $e = u_!\cL_\univ$.
\end{defn}

\begin{ex}
For the torus $H$, we have $\cH_{\cH} = \Sh_0(H \times \xcoch(H))$ the dg derived category of locally constant sheaves. Convolution is simply the pushforward $\cL_1 \star \cL_2 = m_!(\cL_1 \boxtimes \cL_2)$ along the multiplication and addition map $m:(H \times \xcoch(H)) \times (H \times \xcoch(H)) \to H \times \xcoch(H)$, and  
the unit object is the universal local system $e = \cL_\univ = \exp_!\cD_\frh$ supported on $H\times \{0\}$.
\end{ex}

\begin{remark}
As with $\cH_{G}$, we can also naturally regard $\cH_{\cG}$ as a monoidal category in $\cH_{H}$-bimodule categories. 
\end{remark}


\subsection{Whittaker objects}\label{ss:whit}

Fix a maximal torus $T\subset B \subset G$, and let $B^-\subset G$ be the opposite Borel subgroup, 
and $U^- \subset B^-$ its unipotent radical.

\subsubsection{Finite case}\label{sss:Whit G}

Consider the diagram 
\begin{equation*}
\xymatrix{
\AA^1 & \ar[l]_-\chi U^- \ar[r]^-{r_-} & U\bs G/U 
}
\end{equation*}
where $r_-$ is induced by the inclusion $U^- \subset G$, and $\chi: U \to U/[U, U] \to \AA^1$ is a non-degenerate character (i.e., nontrivial on each simple root group).

Consider the natural factorization of $r_{-}$
\begin{equation*}
\xymatrix{
r_{-}: U^-  \ar[r]^-{i_{-}} & G/U \ar[r]^-{q} & U\bs G/U 
}
\end{equation*} 
Note $i_{-}$ is a closed embedding transverse to the $B$-orbits in $G/U$, and $q$ is smooth (a base change of $pt \to \BB U$). 
Thus for any $\cF\in \Sh(  U\bs G/U)$ that is left $H$-monodromic, so locally constant along the left $H$-orbits in $U\bs G/U$, 
we have canonical isomorphisms
\begin{equation*}
r_-^! \cF \simeq i^!_{-} q^!\cF \simeq i^*_{-}q^*\cF[2\dim r_{-}] \simeq r_-^* \cF[2\dim r_{-}]
\end{equation*}

Let  $\varphi_{\chi, 1}: \Sh(U^-) \to k\lmod$ denote the vanishing cycles at the identity $1\in U^-$ with respect  to the non-degenerate character $\chi:U^-\to \AA^1$.

\begin{defn} \label{def: univ fin whit}
\begin{enumerate}
\item The Whittaker functor is  the composition
\begin{equation*}
\xymatrix{
\cW_G: \cH_G\ar[r] &  k\lmod
&
\cW_G(\cF) = \varphi_{\chi, 1}  r_-^!\cF[-\dim r_{-}].
}\end{equation*}

 \item The universal finite Whittaker sheaf $\Wh_G \in \cH_G$ is the object corepresenting the Whittaker functor
 \begin{equation*}
\xymatrix{
\cW_G(\cF) \simeq \Hom_{\cH_G}(\Wh_G, \cF), & 
\text{for all } \cF \in \cH_G.
}\end{equation*}

\end{enumerate}
\end{defn}

\begin{remark}  The shift in the definition of $\cW_{G}$ is chosen to make  $\cW_{G}$ exact for the perverse $t$-structure on $\cH_{G}$. In particular, for the unit object $e\in \cH_{G}$,  we have $\cW_{G}(e)\cong  k[\xcoch(H)]$ is concentrated in degree $0$.
\end{remark}

\subsubsection{Coalgebra structure on  Whittaker sheaf}	
\label{sec:Whittaker_coalgebra}
Here we define the natural coalgebra structure on $\Wh_G$. 

First,  introduce the analogous Whittaker functor on $U\backslash G \times^U G/U$:
\begin{equation*}
\xymatrix{
	\cW_{U\backslash G \times^U G/U}: Sh(U\backslash G \times^U G/U) \ar[r] &  k\lmod &
\cW_{U\backslash G \times^U G/U}(\cF) = \varphi_{\chi+ \chi, 1}  r_-^!\cF[-\dim r_{-}].
}\end{equation*}
where 
\begin{equation*}
\xymatrix{
\AA^1 & \ar[l]_-{\chi+ \chi} U^- \times U^-  \ar[r]^-{r_-} & U\backslash G \times^U G/U
}
\end{equation*}

Now observe the Whittaker functor $\cW_G$ is naturally lax monoidal: for any $\cF_1, \cF_2 \in \cH_G$, we have a natural map:
\begin{equation*}
\begin{gathered}
	\cW_G(\cF_1) \otimes \cW_G(\cF_2) \simeq \cW_{G \times G}(\cF_1 \boxtimes \cF_2) \simeq  
\cW_{U\backslash G \times^U G/U}(\delta^*(\cF_1 \boxtimes \cF_2))  \\
 \longrightarrow \cW_{U\backslash G \times^U G/U}(\pi^! \pi_! \delta^*(\cF_1 \boxtimes \cF_2)) \simeq \cW_G(\cF_1 \star \cF_2)
\end{gathered}
\end{equation*}

In other words, we have a natural transformation of functors 
\begin{equation*} 
\cW_G \boxtimes \cW_G \to \cW_G \circ m : \cH_G \otimes \cH_G \to k\lmod 
\end{equation*}
This induces a map between co-representing objects:
\beq 
\label{eq:map_alpha}
\alpha : m^\ell(\Wh_G) \longrightarrow \Wh_G \otimes \Wh_G 
\eeq
By adjunction, we obtain a map:
\begin{equation*}
\beta : \Wh_G \longrightarrow \Wh_G \star \Wh_G
\end{equation*}
which defines a comultiplication on $\Wh_G$. 

The counit and  coherences can be constructed analogously.
 In fact, 
	the full additive $\infty$-subcategory of $\cH_G$ generated by the monoidal unit $e$ and Whittaker object $\Wh_G$ is closed under convolution and in fact a  classical category (i.e.~equivalent as  an $\infty$-category to a discrete category, since its Homs are  concentrated in degree $0$). The coalgebra structure on $\Wh_G$ comes from a coalgebra structure  in this classical subcategory, and thus its    coherences are  all strictly determined.

\subsubsection{Microlocal description for Whittaker sheaf}

We will not need the following microlocal interpretation but mention it for conceptual clarity. We will use an analogue for character sheaves discussed in Section~\ref{ss:FS Wh} below.

Consider the cotangent bundle $T^*(G/U)$, and note its fiber at the identity coset $U/U \in G/U$ is naturally isomorphic to $(\frg/\fru)^*$. Thus the differential $d\chi:\fru^- \to \AA^1$ gives a covector 
\begin{equation*}
\xymatrix{
\xi: \frg/\fru \ar[r] &  \frg/\frb \simeq \fru^- \ar[r]^-{d\chi} &  \AA^1.
}\end{equation*}

Let  $\Xi: \cH_G\to k\lmod$ denote the $*$-pullback along $q:G/U \to U \bs G/U$, followed by the microstalk at $\xi\in T_{U/U}^*(G/U)$ (normalized so that the microstalks  of perverse sheaves are concentrated in degree $0$).

The following is a standard calculation. For a general version in Betti Geometric Langlands, see~\cite{nadlerWhittakerFunctionalShifted}.

\begin{lemma}
The Whittaker functor is naturally isomorphic to the microstalk functor
\begin{equation*}
\xymatrix{
\cW_G \simeq \Xi: \cH_G\ar[r] &  k\lmod .
}
\end{equation*}
\end{lemma}

\subsubsection{Affine case}\label{sss:aff whit}

Consider the natural closed embedding
\begin{equation*}
\xymatrix{
i: U\bs G/U \ar[r] & \cI^u\bs \cG/\cI^u 
}
\end{equation*}
induced by the inclusion of constant loops $G\subset \cG$.

\begin{defn}  \label{def: univ whit}
\begin{enumerate}
\item The affine Whittaker functor is  the composition
\begin{equation*}
\xymatrix{
\cW_\cG: \cH_\cG\ar[r] &  k\lmod &
\cW_{\cG}(\cF) = \cW_G (i^!\cF).
}\end{equation*}

 \item The universal affine Whittaker sheaf $\Wh_\cG \in \cH_\cG$ is the object corepresenting the affine Whittaker functor
 \begin{equation*}
\xymatrix{
\cW_\cG(\cF) \simeq \Hom_{\cH_\cG}(\Wh_\cG, \cF), & 
\text{for all } \cF \in \cH_\cG.
}\end{equation*}

\end{enumerate}
\end{defn}

\begin{lemma}\label{l:const loops}
The universal affine Whittaker sheaf is the pushforward of the finite Whittaker finite sheaf
\begin{equation*}
\xymatrix{
\Wh_\cG \simeq i_!\Wh_G
}
\end{equation*}
In particular, for the unit object $e\in \cH_{\cG}$, we have $\cW_{\cG}(e)\cong k[\xcoch(H)]$.
\end{lemma}

\begin{proof}
By adjunction,
\begin{equation*}
\cW_\cG(\cF) = 
 \cW_G (i^!\cF) 
\simeq  \Hom_{\cH_G}(\Wh_G, i^!\cF) \simeq 
\Hom_{\cH_G}(i_! \Wh_G, \cF)
\end{equation*}
\end{proof}

Note that $i_!$ is monoidal, and therefore $\Wh_\cG$ has the induced coalgebra structure from $\Wh_G$.


\subsection{Langlands duality}\label{s:ld}

In Section~\ref{s:comm}, we will deduce spectral consequences of our automorphic results on the universal affine Hecke category  $\cH_{\cG}$ and universal affine Whittaker sheaf $\Wh_\cG$. 
We only introduce Langlands duality and the spectral side at this juncture, rather than in  Section~\ref{s:comm}, as a  conceptual guide 
for our automorphic arguments. With the exception of the similarly placed Section~\ref{ss:spec realization}, all of our results until Section~\ref{s:comm} will take place on the automorphic side.

Let $\dG$ be the dual group of $G$, viewed as a split group over $k$. Let $\dB\subset \dG$ be the distinguished Borel subgroup,
$U^\vee \subset \dB$ its unipotent radical, 
and $\dH = \dB/U^\vee$ its universal Cartan. We may identify the Weyl group of $\dG$ with $W$.

Recall  the canonical identification $\dB/\dB \simeq \wt \dG/\dG$ where $\wt \dG/\dG$ is the Grothendieck-Springer stack of pairs $(g, E)$ of an element $g\in \dG$, and a Borel subgroup $E \subset \dG$, such that $g\in E$, all up to conjugation. 
Recall under the above identification, the map of adjoint quotients $\dB/\dB \to \dG/\dG$ corresponds to 
the Grothendieck-Springer map $\wt \dG/\dG \to  \dG/\dG$ that forgets the Borel subgroup $E \subset \dG$.

The projection $\dB\to \dH$ factors through $\dB/\dB \to \dH$, which corresponds to the projection $\wt \dG/\dG \to  \dH$ that takes a pair $(g, E)$ to the class $[g] \in E/E^u \simeq \dH$.

\begin{defn}
\begin{enumerate}
\item
The universal Steinberg stack is the derived fiber product of adjoint quotients
\begin{equation*}
\St_{\dG} = \dB/\dB \times^{\bR}_{\dG/\dG} \dB/\dB
\end{equation*}

\item The universal spectral affine Hecke category is the monoidal convolution category
\begin{equation*}
\Ind\Coh(\St_{\dG})  = \Ind\Coh(\dB/\dB \times^{{\bR}}_{\dG/\dG} \dB/\dB)
\end{equation*} 

\end{enumerate}
\end{defn}
\begin{remark} The smallness of the Grothendieck alteration $\wt \dG\to \dG$ implies that the underived fiber product $ \dB/\dB \times_{\dG/\dG} \dB/\dB$ has the expected dimension zero. Since $\dG/\dG$ and $\dB/\dB$ are both smooth, we see that the derived structure on $\St_{\dG}$ is in fact trivial.
\end{remark}

\begin{remark}
Note we  can also naturally regard $\Ind\Coh(\St_{\dG})$ as a monoidal category in $\qc(\dH)$-bimodule categories.\end{remark}

In order to deduce spectral consequences of our automorphic results, 
in Section~\ref{s:comm}, we will assume  the following universal variant of Bezrukavnikov's spectral description of the affine Hecke category~\cite{bezrukavnikovTwoGeometricRealizations2016} 
We expect a proof  to appear in forthcoming 
work of Chen-Dhillon-Taylor, building on Taylor's universal variant~\cite{taylor} of Soergel's Endomorphismensatz.
We mention it here for  conceptual clarity, but it will not be invoked until we turn to applications in Section~\ref{s:comm}.

\begin{ansatz}\label{ans:univ aff equiv}
There is a monoidal equivalence of universal affine Hecke categories
\beq
\xymatrix{
\Phi:\Ind\Coh(\St_{\dG}) \ar[r]^-\sim &   \cH_{\cG}
}
\eeq 
with the following properties:
\begin{enumerate}
\item  $\Phi$ identifies the structure sheaf $\cO_{\St_{\dG}}$ with the universal affine Whittaker object $\Wh_{\cG}$ as coalgebra objects.
	\item $\Phi$ is naturally an equivalence of  algebras in bimodules over $\qc(H^\vee)$. (Here we regard the $\cH_{H}$-bimodule $\cH_{\cG}$ as a $\qc(H^\vee)$-bimodule using the canonical monoidal equivalence $\cH_{H}\simeq \Sh_{0}(H)\simeq \qc(H^\vee)$.)

\end{enumerate}
\end{ansatz}

\begin{remark}
 In fact, Ansatz~\ref{ans:univ aff equiv} (2) is not used in Section~\ref{s:comm}, we only record it for conceptual completeness.
\end{remark}


\subsection{Coxeter presentation}

To make calculations in the cocenter of $\cH_{\cG}$, we will use a universal monodromic version of a result of Tao-Travkin~\cite{taoAffineHeckeCategory}.
In Appendix~\ref{app:TT}, we provide the necessary extensions to apply their arguments  to the universal monodromic case. 

Let $\cG^{\c}$ be the neutral component of the loop group $\cG$, and let $\cH_{\cG^{\c}}\subset \cH_{\cG}$ be the full subcategory of objects supported on $\cG^{\c}$. Then $\cH_{\cG^{\c}}$ is a monoidal subcategory of $\cH_{\cG}$.

For each finite type subset $J\sft I^{a}$, the finite universal Hecke  category $\cH_{L_{J}}$ embeds into $\cH_{\cG^{\c}}$ as a monoidal full subcategory of sheaves supported on $\cI^{u}\bs \cP_{J}/\cI^{u}$, which is a pro-unipotent gerbe over $U_{J}\bs L_{J}/U_{J}$.  These embeddings are compatible with inclusions $J\subset J'\sft I^{a}$ in an obvious sense, and induce a monoidal functor 
\beq\label{colim fin Hk}
\xymatrix{
 \colim^{\ot}_{J \sft I^a} \cH_{L_J} \ar[r]  & \cH_{\cG^{\c}}
}
\eeq
Here we regard all the Hecke categories as  objects in  stable presentable $k$-linear categories $\St^{L}_k$
with morphisms left adjoints,
and we write $\colim^{\ot}$ to emphasize that the colimit is of monoidal categories in $\cH_H$-bimodues, i.e.~of algebra objects in $\cH_H$-bimodues in $\St^{L}_k$.

\begin{theorem}[Tao-Travkin~\cite{taoAffineHeckeCategory}, see Appendix~\ref{app:TT}]\label{thm:colim}
The functor \eqref{colim fin Hk} is a monoidal equivalence of $\cH_H$-bimodule categories.
\end{theorem}

In~\cite{taoAffineHeckeCategory}, the authors worked with the bi-$\cI$-equivariant version of the affine Hecke category and assumed $G$ was simply-connected. We  sketch the necessary modifications to their argument in Appendix~\ref{app:TT}.

\begin{remark}
Formulating a generalization of Theorem~\ref{thm:colim} for the whole category $\cH_{\cG}$ when $\cG^{\c}\sne\cG$ is more combinatorially complicated. We will not need it since Theorem~\ref{thm:colim} will suffice for our application to the cocenter of any reductive $G$. See Sect.~\ref{sss:general}.
\end{remark}

\subsection{Hochschild homology and cocenters }\label{ss:hh}

We  work in the setting of stable presentable $k$-linear categories $\St^{L}_k$ with morphisms left adjoints. All higher algebra constructions will be following~\cite{lurieHigherAlgebra2012}. 
We could also work in  the setting of
presentable $k$-linear categories $\Pr^{L}_k$ by the following well-known fact:

\begin{lemma}
	The inclusion $\St^{L}_k \to \Pr^{L}_k$ preserves colimits.
\end{lemma}
\begin{proof}
Since fully-faithful functor respect colimits, suffices to show that the colimit of stable categories in $\Pr^{L}_k$ is again stable. By the equivalence $\Pr^{L,op}_k \simeq \Pr^{R}_k$, and the fact that the forgetful functor $\Pr^{R}_k \to \Cat_{\infty}$ preserves limits (where $\Pr^{R}_k$ is the category of presentable $k$-linear categories with morphism right adjoints and $\Cat_{\infty}$ is the category of ($\infty$-)categories), suffices to show that limit of stable categories (along with exact functors) in $\Cat_{\infty}$ is stable. The latter statement follows from \cite[Theorem 1.1.4.4]{lurieHigherAlgebra2012}.
\end{proof}

\begin{defn}\label{def:hh}
Let $\cA$ be an algebra object  in $\St^{L}_k$, and $\cA^{op}$ the opposite algebra.

\begin{enumerate}
\item Let $\cM$ be an $\cA$-bimodule, i.e.~an $\cA \otimes \cA^{op}$-module.

The Hochschild homology of $\cA$ with values in $\cM$ is the tensor 
\begin{equation*}
hh(\cA, \cM) = \cA \otimes_{\cA \otimes \cA^{op}} \cM
\end{equation*}

The trace map is the natural projection
\begin{equation*}
\xymatrix{
\tr: \cM \ar[r] &   \cA \otimes_{\cA \otimes \cA^{op}} \cM = hh(\cA, \cM) 
}
\end{equation*}
induced by the unit of $\cA$ in the first factor of the tensor.

\item 
The cocenter of $\cA$ is the Hocschild homology of $\cA$ with values in the regular bimodule
\begin{equation*}
hh(\cA) =  hh(\cA, \cA) = \cA \otimes_{\cA \otimes \cA^{op}} \cA
\end{equation*}

The trace map is the natural projection
\begin{equation*}
\xymatrix{
\tr: \cA \ar[r] &   \cA \otimes_{\cA \otimes \cA^{op}} \cA = hh(\cA, \cA) = hh(\cA) 
}
\end{equation*}
induced by the unit of $\cA$ in the first factor of the tensor.

\end{enumerate}

\end{defn}

Recall one typically   calculates $hh(\cA, \cM)$ as  the geometric realization of the Hochschild complex (simplicial object)
\begin{equation*}
 \xymatrix{
\cB_\bullet(\cA, \cM) = 
  [ \cM & \ar@<-0.5ex>[l]\ar@<0.5ex>[l] \cM \otimes \cA & \ar@<-1ex>[l]  \ar[l] \ar@<1ex>[l] 
  \cM \otimes \cA \otimes \cA \cdots]
 }
 \end{equation*}
 This results from tensoring the regular $\cA$-bimodule with the bar resolution
   \begin{equation*}
 \xymatrix{
\cM 
  & \ar[l]
  [ \cM \otimes \cA & \ar@<-0.5ex>[l]\ar@<0.5ex>[l] \cM \otimes \cA\otimes \cA & \ar@<-1ex>[l]  \ar[l] \ar@<1ex>[l] 
  \cM \otimes \cA \otimes \cA\otimes \cA \cdots]
 }
 \end{equation*}
 
 Given a monoidal subcategory $\cA' \subset \cA$, there is a minor variation on  the Hochschild complex that equally well 
 calculates $hh(\cA, \cM)$. 
  Note we can regard $\cA$ as an algebra in $\cA'$-bimodules, and likewise, 
 regard $\cM$ as an $\cA$-bimodule in $\cA'$-bimodules. Then to calculate  $hh(\cA, \cM)$, we can form  the relative Hochschild complex 
\begin{equation*}
 \xymatrix{
\cB_\bullet(\cA, \cM)_{\cA'} = 
 [ hh(\cA', \cM) & \ar@<-0.5ex>[l]\ar@<0.5ex>[l] hh(\cA', \cM \otimes_{\cA'} \cA) & \ar@<-1ex>[l]  \ar[l] \ar@<1ex>[l] 
  hh(\cA', \cM \otimes_{\cA'} \cA \otimes_{\cA'} \cA) \cdots]
 }
 \end{equation*} 
 This results from tensoring the regular $\cA$-bimodule with the relative bar resolution
  inside of $\cA'$-bimodules 
   \begin{equation*}
 \xymatrix{
\cM 
  & \ar[l]
  [ \cM \otimes_{\cA'} \cA & \ar@<-0.5ex>[l]\ar@<0.5ex>[l] \cM \otimes_{\cA'} \cA\otimes_{\cA'} \cA & \ar@<-1ex>[l]  \ar[l] \ar@<1ex>[l] 
  \cM  \otimes_{\cA'}\cA \otimes_{\cA'}\cA \otimes_{\cA'} \cA  \cdots]
 }
 \end{equation*}
 In particular, we can take $\cA' =\langle 1_\cA\rangle \subset \cA$ to be the full subcategory generated by the unit.
  
%
%


\subsubsection{Hochschild homology and colimits}

On the one hand, suppose 
 $\cA$ is an algebra object  in $\St^{L}_k$, and $\cM = \colim_i \cM_i$ is a colimit diagram of $\cA$-bimodules. Then by commuting colimits, we observe that the natural map is an equivalence 
$$
\xymatrix{
\colim_i hh(\cA,  \cM_i) \ar[r]^-\sim & hh(\cA, \cM)
}$$

On the other hand, suppose $\cA = \colim \cA_i$ is a
colimit diagram of algebra objects  in $\St^{L}_k$, and $\cM$ is an $\cA$-bimodule. 
Then we have a natural map
$$
\xymatrix{
\colim_i hh(\cA_i,  \cM) \ar[r] & hh(\cA, \cM)
}$$
 which is not an equivalence in general. 

 But one can deduce it is an equivalence for {\em contractible} colimit diagrams using
 the following lemma, whose proof we learned from Nick Rozenblyum. 
 
 Some set theory preliminaries: Let $\kappa$ be an uncountable regular cardinal,
and let $\Pr^L_\kappa \subset \Pr^L$ denote the subcategory  of objects that are $\kappa$-presentable, i.e.,
generated by $\kappa$-compact objects, with functors that preserve $\kappa$-compact objects, i.e.,
whose right adjoint preserves $\kappa$-filtered colimits.

By \cite[Lemmas 5.3.2.9 and 5.3.2.11]{lurieHigherAlgebra2012}, one knows that $\Pr^L_\kappa$ is a presentable symmetric monoidal category whose tensor product commutes with colimits in
each variable. Moreover, the inclusion $\Pr^L_\kappa \to \Pr^L$ is symmetric monoidal and preserves
colimits. Note that $\Mod_k$ is a commutative algebra object in  $\Pr^L_{\kappa}$, and so the $k$-linear version  $\Pr^L_{\kappa, k} = \Mod_{\Mod_k}(\Pr^L_{\kappa})  \to \Mod_{\Mod_k}(\Pr^L)  =  \Pr^L_k$  is also symmetric monoidal and preserves
colimits. Hence the same is true on algebra objects $\Alg(\Pr^L_{\kappa, k}) \to \Alg(\Pr^L_k)$.

\begin{lemma}
Suppose $\cA = \colim \cA_i$ is a contractible colimit diagram of algebra objects  in $\Pr^{L}_k$,
 and $\cM$ is an $\cA$-module. 
 Then the natural map
 is an equivalence
$$
\xymatrix{
\colim_i \cA \otimes_{\cA_i}  \cM  \ar[r]^-\sim & \cM
}$$

\end{lemma}

\begin{proof}
Following the assumptions on set theory of \cite[Sect. 1.2.15]{lurieHigherToposTheory2009a}, we can choose a regular 
cardinal $\kappa$ such that the diagram $ \colim \cA_i\to \cA$ lies in $\Alg(\Pr^L_{\kappa, k})$
and $\cM$ lies in $\Mod_{\cA}(\Pr^L_{\kappa, k})$.
Since $\Alg(\Pr^L_{\kappa, k}) \to \Alg(\Pr^L_k)$ preserves colimits, the natural map is an equivalence $ \colim A_i \isom \cA$ when the colimit is taken
 in $\Alg(\Pr^L_{\kappa, k})$.

Now applying \cite[Corollary 4.8.5.13]{lurieHigherAlgebra2012} to the presentable category $\Pr^L_{\kappa, k}$,
we see that passing to modules
 $\Mod^*: \Alg(\Pr^L_{\kappa, k}) \to \Mod_{\Pr^L_{\kappa, k}}(\Pr^L)$
preserves contractible colimits.  In particular, the composite
 $$ 
 \Alg({\Pr}^L_{\kappa, k}) \to  \Mod_{{\Pr}^L_{\kappa, k}}({\Pr}^L) \to {\Pr}^L \simeq ({\Pr}^R)^{op} \to \Cat^{op}
 $$
preserves all contractible colimits (where the equivalence $\Pr^L \simeq (\Pr^R)^{op}$ is given by
passing to right adjoints).  Thus we conclude that
  $\Mod_*: \Alg(\Pr^L_{\kappa, k})^{op} \to \Cat$
preserves contractible limits.

This implies that the natural map 
$c: \colim_i \cA \otimes_{\cA_i} \cM \to \cM $
is an equivalence
of objects in $\Mod_\cA(\Pr^L_{\kappa, k}),$ where the colimit and relative tensor product is
taken in $\Pr^L_{\kappa, k}$.  We want to see that the the  map $c$ is also an equivalence in $\Mod_\cA(\Pr^L_{k}),$
i.e., after applying the inclusion
$\Pr^L_{\kappa, k}\to \Pr^L_{ k}$ to each term.  For this, it suffices to forget the module structure and show the natural map $c$ is an equivalence in $\Pr^L_{k}$.
Since  the inclusion
$\Pr^L_{\kappa, k}\to \Pr^L_{ k}$
is symmetric monoidal and preserves colimits, it also preserves relative tensor products.
This gives the result.
\end{proof}

\begin{cor}\label{cor: general hh of colimit of algs}
Suppose $\cA = \colim \cA_i$ is a contractible colimit diagram of algebra objects  in $\Pr^{L}_k$,
 and $\cM$ is an $\cA$-bimodule. 
 Then the natural map
 is an equivalence
$$
\xymatrix{
\colim_i hh(\cA_i, \cM)  \ar[r]^-\sim & hh(\cA, \cM)
}$$

\end{cor}

\begin{proof}
Apply the lemma to the $\cA$-bimodule $\cM$ viewed as a left $\cA$-module to see
that the natural map
of $\cA$-bimodules is an equivalence 
$$
\xymatrix{
\colim_i \cA \otimes_{\cA_i}  \cM  \ar[r]^-\sim & \cM
}$$
Then take $hh(\cA, -)$ of both sides observing the natural isomorphism $hh(\cA, \cA \otimes_{\cA_i}  \cM ) \simeq hh(\cA_i, \cM)$.
\end{proof}


\subsubsection{Descended trace}

The following general construction will  provide a useful way to characterize certain objects of the cocenter of a monoidal category. 
Let $\Delta$ denote the simplex category of non-empty finite ordered sets $[n] = \{0, \ldots, n\}$, $n\geq 0$. 

Suppose $\cA$ is a monoidal category with product denoted by $\star$. 

Given an algebra object  $a\in \cA$, let $\cL = \LMod_a(\cA)$ and $\cR = \RMod_a(\cA)$ denote the category of left and right $a$-modules in $\cA$. Let $\cB = \Bimod_a(\cA)$ denote the category of $a$-bimodules in $\cA$. Note $\cB$ is naturally monoidal with product given by

\begin{equation*}
 \xymatrix{
m \star_a n =    
 \colim_{\Delta^{op}} [ m\star n & \ar@<-0.5ex>[l]\ar@<0.5ex>[l] m \star a \star n & \ar@<-1ex>[l]  \ar[l] \ar@<1ex>[l] 
  m \star a \star a \star n  \cdots]
 }
 \end{equation*}
Similarly, $\cL$ is a $\cB \otimes \cA$-bimodule, and $\cR$ is a $\cA\ot \cB$-bimodule. In fact, $\cL$ and $\cR$ are in duality with unit and counit maps denoted by
\beq\label{eq:duality of l and r}
\xymatrix{
u:\cB \ar[r] & \cL \ot_\cA \cR
&
c: \cR \ot_{\cB} \cL \ar[r] & \cA
}
\eeq 
Here, $u$ is the inverse of the natural functor $\cL\ot_{\cA}\cR\to \cB$ (sending $x\ot y$ to $x\star y$), which is an equivalence by \cite[Prop. 4.1]{ben-zviIntegralTransformsDrinfeld2010}.

Recall the dual bimodules $\cL$ and $\cR$, with their unit and counit maps, provide a map on cocenters
\beq\label{eq:taut map on cocenters}
\xymatrix{
h:hh(\cB)  = \cB \otimes_{\cB \otimes \cB^{op}} \cB  \ar[r]^-{u'} &   (\cL \ot_\cA \cR) \otimes_{\cB \otimes \cB^{op}} \cB\simeq
  \cA \otimes_{\cA \otimes \cA^{op}} (\cR \ot_{\cB} \cL)  \ar[r]^-{c'}  &
  \cA \otimes_{\cA \otimes \cA^{op}} \cA = hh(\cA)
}
\eeq
where $u'$ is induced by $u$, and $c'$ by $c$.

To avoid confusion, let us denote the usual trace maps by
\begin{equation*}
\xymatrix{
\tr_\cA: \cA \ar[r] &   hh(\cA) 
&
\tr_\cB: \cB \ar[r] &   hh(\cB) 
}
\end{equation*}

Given  $m\in \cB$, we can take its $\cB$-trace $\tr_\cB(m)\in hh(\cB)$ then its image   
$h(\tr_\cB(m))\in hh(\cA)$.

\begin{defn}\label{def:desc tr}
Given  $m\in \cB$, 
we call 
\begin{equation*}
\ol \tr(m):= h(\tr_\cB(m)) \in hh(\cA)
\end{equation*}  
the {\em descended trace} of $m$.
\end{defn}


\subsubsection{Functoriality of descended trace}

Suppose $f:\cA \to \cA'$ is a monoidal functor of monoidal categories. Then there is an evident commutative diagram
\beq\label{eq:tr funct}
\xymatrix{
\ar[d]_-{\tr} \cA \ar[r]^-f & \cA' \ar[d]^-{\tr} \\ 
hh(\cA) \ar[r]^-{hh(f)} & hh(\cA')   
}
\eeq

Suppose in addition  $a\in \cA$ is an algebra object with image algebra object $a' = f(a)\in \cA'$.  
Then applying $f$ to bimodules provides a natural monoidal functor $F:\Bimod_a(\cA) \to \Bimod_{a'}(\cA')$.
We have an evident commutative diagram
\beq\label{eq:tr bimod funct}
\xymatrix{
\ar[d]_-{\tr} \Bimod_a(\cA) \ar[r]^-F & \Bimod_{a'}(\cA') \ar[d]^-{\tr} \\ 
\ar[d]_-h hh(\Bimod_a(\cA)) \ar[r]^-{hh(F)} & hh(\Bimod_{a'}(\cA')) \ar[d]^-h   \\
hh(\cA) \ar[r]^-{hh(f)} & hh(\cA')   
}
\eeq
where each $h$ denotes the  map~\eqref{eq:taut map on cocenters} from the cocenter of bimodules to the cocenter induced by the left and right module categories.
  
We immediately conclude:

 \begin{lemma}\label{l: dtr inv} 
 Suppose $f:\cA \to \cA'$ is a monoidal functor of monoidal categories.
 
 Suppose $a\in \cA$ is an algebra object with image algebra object $a' = f(a)\in \cA'$. 
 
 Under the induced functor $hh(f): hh(\cA) \to hh(\cA')$, we have an identification of descended traces
 \begin{equation*}
 hh(f)( \ol \tr(m)) \simeq \ol \tr({F}(m)) \qquad m \in \Bimod_a(\cA).
 \end{equation*}
 In particular, for the regular bimodule $m=a$ with image $F(m)= f(a)=a'$,  an identification of descended traces
 \begin{equation*}
 hh(f)( \ol \tr(a)) \simeq \ol \tr(a') 
 \end{equation*}

 \end{lemma}


\subsubsection{Formula for descended trace} 
 We provide here a formula for the descended trace. 

Let $\Delta_+$ denote the augmented simplex category of (possibly empty) finite ordered sets $[n]= \{0, \ldots, n\}$, $n\geq -1$.  
%
%
Given   $m\in \cB =\Bimod_{a}(\cA)$, consider its bar augmented simplicial object $m_\bullet: \Delta_+^{op} \to \cB$ given by the  assignments: $m_n = m\star a^{\star n}$, with    
 each face map a contraction via the multiplication of $a$, and each degeneracy map  an insertion of the unit of $a$. We can picture the diagram of face maps in the usual way
  \begin{equation*}
 \xymatrix{
m_\bullet = [  m & \ar[l] m\star a & \ar@<-0.5ex>[l]\ar@<0.5ex>[l] m \star a \star a & \ar@<-1ex>[l]  \ar[l] \ar@<1ex>[l] 
  m \star a \star a \star a  \cdots]
 }
 \end{equation*}
 and the augmentation descends to an equivalence on the geometric realization
   \begin{equation*}
 \xymatrix{
  m & \ar[l]_-\sim  \colim_{\Delta^{op}} [ m\star a & \ar@<-0.5ex>[l]\ar@<0.5ex>[l] m \star a \star a & \ar@<-1ex>[l]  \ar[l] \ar@<1ex>[l] 
  m \star a \star a \star a  \cdots]
 }
 \end{equation*}

 Now let us take the descended trace of  the   bar augmented simplicial object  to obtain a resolution
  \begin{equation*}
 \xymatrix{
\ol \tr (m) & \ar[l]_-\sim  \colim_{\Delta^{op}} [ \ol\tr(m\star a) & \ar@<-0.5ex>[l]\ar@<0.5ex>[l] \ol\tr(m \star a \star a) & \ar@<-1ex>[l]  \ar[l] \ar@<1ex>[l] 
  \ol\tr(m \star a \star a \star a)  \cdots]
 }
 \end{equation*}
 using that we work in the setting of continuous functors.
 
Unwinding the definitions, for any $a$-bimodule of the form $\ell \star r$ where $\ell
\in\cL = \LMod_a(\cA)$ and $r\in \cR = \RMod_a(\cA)$, we have a natural isomorphism  $\ol \tr(\ell \star r) \simeq \tr_\cA(r \star_a \ell)$. 
 Thus the above resolution gives a concrete formula for the descended trace
  \beq\label{eq:desc tr formula}
 \xymatrix{
\ol \tr (m) & \ar[l]_-\sim  \colim_{\Delta^{op}} [ \tr_\cA(m) & \ar@<-0.5ex>[l]\ar@<0.5ex>[l] \tr_\cA(m \star a) & \ar@<-1ex>[l]  \ar[l] \ar@<1ex>[l] 
  \tr_\cA(m \star a \star a)  \cdots]
 }
 \eeq
 where the face maps are
 given by
contractions via the multiplication of $a$ (and  degeneracy maps given by  an insertion of the unit of $a$). In particular, the new face map  $\tr_\cA(m\star a^{\star (n+1)}) \to  \tr_\cA(m\star a^{\star n})  $ is induced by the left multiplication of the last factor on the first available thanks to the cyclic symmetry of the trace.

\subsubsection{Descended trace for coalgebras}

One can repeat the  construction of the descended trace starting with a coalgebra $c\in \cA$ rather than an algebra $a\in \cA$. We only prefer the algebra formulation due to our lack of familiarity with ``bi-comodules". But in any case, we will only work with the  coalgebra $c$ itself.  In this case, we can take the following concrete formula as definition of the descended trace
  \begin{equation*}
 \xymatrix{
\ol \tr (c) :=   \lim_{\Delta} [ \tr_\cA(c)  \ar@<-0.5ex>[r]\ar@<0.5ex>[r] & \tr_\cA(c \star c)  \ar@<-1ex>[r]  \ar[r] \ar@<1ex>[r] &  
  \tr_\cA(c \star c \star c)  \cdots]
 }
 \end{equation*}
 where the coface maps are
 given by
expansions via the comultiplication of $c$ (and  degeneracy maps given by  an insertion of the counit of $c$). In particular, the new coface map  $\tr_\cA(c^{\star n}) \to  \tr_\cA(c^{\star n+1})  $ is induced by the left comultiplication of the last factor on the first available thanks to the cyclic symmetry of the trace.

In our applications, where some additional hypotheses hold, we can relate the above definition for coalgebras with the prior theory for algebras, and in particular take advantage of the prior recorded functoriality. 

Assume that  $\cA$ is compactly generated, and the monoidal product preserves the category of compact objects $\cA_c \subset \cA$, so that $\cA \simeq \Ind \cA_c$ as monoidal categories. Then the categories $\cL,\cR,\cB,hh(\cB), hh(\cA)$ are all compactly generated, and the functor $\ol \tr$ preserves compact objects.
 Suppose our coalgebra is compact  $c\in \cA_c$, and denote  by $a \in \cA_c^{op}$ the same object regarded as an algebra in the dual monoidal category $ \cA^\vee \simeq
 \Ind \cA_c^{op}$.  
 Suppose in addition $a \in \Bimod_a(\cA^\vee)$ is compact (for example, it is a summand of $a \star a$). Then chasing definitions, under the canonical equivalence $hh(\cA^\vee)_c \simeq hh(\cA_c^{op}) \simeq hh(\cA_c)^{op} \simeq hh(\cA)^\vee_c  $,  the colimit calculating the algebra descended trace $\ol \tr(a) \in hh(\cA^{op}_c)$ is equivalent to the limit calculating the coalgebra descended trace $\ol \tr(c) \in  hh(\cA_c)$.


\subsection{Spectral realization}\label{ss:spec realization}

As with Section~\ref{s:ld}, the results of this section will not be used until we turn in Section~\ref{s:comm}
 to spectral consequences of our automorphic results. 
We include them here  as a  conceptual guide 
for our automorphic arguments.

We return to the
universal spectral affine Hecke category
 $\Ind\Coh(\St_{\dG})$, and now consider its cocenter, and the descended trace of the structure sheaf $\cO_{\St_{\dG}}$ as a coalgebra object.
The arguments are quite general, so we will work with an abstract setup.  We stress that in this subsection, all fiber products of stacks are {\em derived fiber products}.

Let $X \to Y$ be a proper morphism between smooth stacks. Then $\Ind\Coh( X \times_Y X)$ is naturally a monoidal category via convolution: $\cF \star \cG = m_* \d_{23}^*(\cF \boxtimes \cG) $ where $\d_{23}$ is the diagonal map and $m$ the forgetful map in the correspondence 
\begin{equation*}
\xymatrix{
 X \times_Y X \times  X \times_Y X  &  \ar[l]_-{\d_{23}}   X \times_Y X \times_Y X   \ar[r]^-{m}  &  X \times_Y X
}
\end{equation*}

\begin{ex}\label{ex:BBGG}
In our application, 
we will take the induction map $X =  \dB/\dB\to  \dG/\dG = Y$ so that 
$ X \times_Y X = \St_{\dG}$.
\end{ex}


Alternatively, we could define the $!$-convolution:  $\cF \star^! \cG = m_* \d_{23}^!(\cF \boxtimes \cG) $. We will denote by $\Ind\Coh( X \times_Y X)^!$ the resulting monoidal category with  the $!$-convolution. Note that $\d_{23}$ is quasi-smooth, so $\d_{23}^!$ and $\d_{23}^*$, and hence $\star$ and $\star^!$ as well, only differ by an invertible twist.

Serre duality gives an equivalence of compact objects
$\Coh( X \times_Y X)^{op}  \simeq   \Coh( X \times_Y X)$
and hence an equivalence of their ind-completions
$\Ind\Coh( X \times_Y X)^\vee  \simeq   \Ind\Coh( X \times_Y X).
$

Note  the dual
$\Ind\Coh( X \times_Y X)^\vee  $ inherits a monoidal structure from $\Ind\Coh( X \times_Y X)$.
 The identification $\Ind\Coh( X \times_Y X)^\vee  \simeq   \Ind\Coh( X \times_Y X)$ naturally lifts to an equivalence of monoidal categories
\beq
\label{eq:Serre_duality_Hecke}
\xymatrix{
 \DD^{\textup{Serre}}_{X \times_Y X}: 	\Ind\Coh( X \times_Y X)^\vee  \ar[r]^-{\sim} &   \Ind\Coh( X \times_Y X)^!  
}
\eeq 

 Now 
denote by $LY = \Hom(S^1, Y)$ the derived loop space of $Y$, and $\Lambda_{X/Y} = \pi_* \delta^!(T^{*-1}_{X \times_Y X}) \subset T^*(LY)$ the Lagrangian defined via the correspondence
\begin{equation}\label{corr loop}
    \xymatrix{ X \times_Y X & \ar[l]_-{\delta}  (X \times_Y X) \times_{X \times X} X  \simeq  LY \times_Y X  \ar[r]^-{\pi}  &  LY  }
\end{equation}

\begin{ex}\label{ex:loop St}
 Continuing in the setup of Example~\ref{ex:BBGG}, we find the commuting stack 
\begin{equation*}
LY=L(\dG/\dG) = 
Z^2_{\dG} =   (\dG/\dG \times \dG/\dG) \times_{\dG/\dG} (\{1\}/\dG )
\simeq  (( \dG\times \dG) \times_{\dG} \{1\} )/\dG
\end{equation*} 
the derived fiber product of the commutator map $c:\dG\times \dG \to \dG$, and the inclusion of the identity $1\in \dG$, all up to conjugation. We also find the Lagrangian of nilpotent codirections $\Lambda_{X/Y} = \cN$
as calculated in  \cite{ben-zviSpectralIncarnationAffine2017}. 
\end{ex}

Thanks to \cite{ben-zviSpectralIncarnationAffine2017}, we have the following:  

\begin{theorem}\label{thm:spec tr}
		\begin{enumerate}
			\item The functor  
			$\pi_*\delta^*$ lands in $ \Ind\Coh_{\Lambda_{X/Y}}(LY)$ and fits into a commutative diagram with the trace map $\tr$ inducing a horizontal equivalence
			\begin{equation*}
			\xymatrix{
			\ar[d]_-{\tr} \Ind\Coh(X \times_Y X) \ar[dr]^-{\pi_*\delta^*}\\
			hh(\Ind\Coh(X \times_Y X) ) \ar[r]^-\sim & \Ind\Coh_{\Lambda_{X/Y}}(LY)
			}
			\end{equation*} 
			\item Similarly, the functor  
			$\pi_*\delta^!$ lands in $ \Ind\Coh_{\Lambda_{X/Y}}(LY)$ and fits into a commutative diagram with the trace map $\tr$ inducing a horizontal equivalence
			\begin{equation*}
			\xymatrix{
			\ar[d]_-{\tr} \Ind\Coh(X \times_Y X)^! \ar[dr]^-{\pi_*\delta^!}\\
			hh(\Ind\Coh(X \times_Y X)^!)  \ar[r]^-\sim & \Ind\Coh_{\Lambda_{X/Y}}(LY)
			}
			\end{equation*} 
		\end{enumerate}
		
		\end{theorem}

\begin{remark}
It follows that the equivalence on trace categories induced by \eqref{eq:Serre_duality_Hecke} is also given by Serre duality, i.e, the following diagram naturally commutes
\beq
\label{eq:Serre_duality_trace}
\xymatrix{
	hh(\Ind\Coh(X \times_Y X) ) ^\vee  \ar[rr]^{hh(\DD^{\textup{Serre}}_{X \times_Y X})} \ar[d]_-\sim && hh(\Ind\Coh(X \times_Y X)^!)  \ar[d]_-\sim  \\
	\Ind\Coh_{\Lambda_{X/Y}}(LY)^\vee   \ar[rr]^{\DD^{\textup{Serre}}_{LY}}  &&	\Ind\Coh_{\Lambda_{X/Y}}(LY)
}
\eeq 
\end{remark}

\begin{ex}\label{ex:tr O} In the above situation, we compute $\tr(\D_{*}\cO_{X})$, where $\D: X\to X\times_{Y}X$ is the diagonal map. We have a commutative diagram with a derived Cartesian square on the left
\begin{equation*}
\xymatrix{ X\ar[d]^{\D} & \ar[l]_{p_{X}} LX\ar[d]^{\ph} \ar[dr]^{Lf}\\
X\times_{Y}X & \ar[l]_-{\d} LY\times_{Y}X \ar[r]^-{\pi}& LY
}
\end{equation*}
By base change we have
\begin{equation*}
\tr(\D_{*}\cO_{X})=\pi_{*}\d^{*}\D_{*}\cO_{X}\simeq \pi_{*}\ph_{*}p_{X}^{*}\cO_{X}=(Lf)_{*} \cO_{LX}.
\end{equation*}
\end{ex}

\subsubsection{Decended trace of structure/dualizing sheaf.}
\label{section:decended_trace_O}

We continue with the above general setup.

The structure sheaf $ \cO_{X \times_Y X} $ is naturally a coalgebra object in $\Ind\Coh(X \times_Y X)$ under convolution: its comultiplication is given by the unit of adjunction
\beq
\label{eq:coalgebra_O}
\xymatrix{
	\cO_{ X \times_Y X } \ar[r] &  m_* m^*\cO_{ X \times_Y X } \simeq 	\cO_{X \times_Y X}  \star \cO_{X \times_Y X}.  
}
\eeq

Similarly, the dualizing sheaf $\omega_{X \times_Y X}$ is naturally an algebra object in
 in $\Ind\Coh(X \times_Y X)^!$
 under $!$-convolution: its multiplication is given by the counit of adjunction
\begin{equation*}
\xymatrix{
	\omega_{X \times_Y X}  \star^! \omega_{X \times_Y X} \simeq m_* m^! \omega_{X \times_Y X}  \ar[r] &  \omega_{X \times_Y X}
}
\end{equation*}

In fact, the coalgebra $\cO_{X \times_Y X}$, viewed as an algebra in $\Ind\Coh(X \times_Y X)^\vee$, and the algebra $\omega_{X \times_Y X}$ are identified under the monoidal equivalence \eqref{eq:Serre_duality_Hecke}. 

Now we have the following calculations of the descended traces of  $\cO_{X \times_Y X}$ and $\omega_{X \times_Y X}$:

\begin{prop}  
	\label{prop:descended_O}
	Assume further that $X \to Y$ is surjective.
	 \begin{enumerate}
	\item Under the equivalence of Theorem~\ref{thm:spec tr}(1)
	\begin{equation*}
	\xymatrix{
	hh(\Ind\Coh(X \times_Y X))  \ar[r]^-\sim & \Ind\Coh_{\Lambda_{X/Y}}(LY)
	}
	\end{equation*} 
	there is a canonical identification of the descended trace of the coalgebra object $\cO_{X \times_Y X}$
	with the structure sheaf
	\begin{equation*}
	\ol \tr( \cO_{X \times_Y X}) \simeq \cO_{LY} .
	\end{equation*}
	\item Similarly, under the equivalence  of Theorem~\ref{thm:spec tr}(2)
	\begin{equation*}
	\xymatrix{
	hh(\Ind\Coh(X \times_Y X)^!)  \ar[r]^-\sim & \Ind\Coh_{\Lambda_{X/Y}}(LY)
	}
	\end{equation*} 
	there is a canonical identification of the descended trace of the algebra object $\omega_{X \times_Y X}$
	with the dualizing sheaf
	\begin{equation*}
	\ol \tr( \omega_{X \times_Y X}) \simeq \omega_{LY} .
	\end{equation*}
\end{enumerate}	
\end{prop}
\begin{proof}
	We prove (2); then (1) follows from the Serre duality equivalences \eqref{eq:Serre_duality_Hecke}, \eqref{eq:Serre_duality_trace}. 
	
	For the natural map $\pi: LY \times_Y X \to LY$, we have the adjunction 
	$$  \pi_*: \Ind\Coh(LY \times_Y X) \rightleftarrows \Ind\Coh_{{\pi_*(T^{-1}_{LY \times_Y X})}}(LY): \pi^!$$
	Moreover, $\pi^!$ is conservative (because $\pi$ is proper, and by \cite[Proposition 7.8.3]{arinkinSingularSupportCoherent2015}) and preserves colimits (because $\pi$ is quasi-smooth). Also the object $\omega_{LY}$ belongs to  $\Ind\Coh_{{\pi_*(T^{-1}_{LY \times_Y X})}}(LY) $ because $\pi$ is surjective and hence $\pi_*(T^{-1}_{LY \times_Y X})$ contains the zero section, which is the singular support of $\omega_{LY}$. Therefore by Berr-Beck, the canonical resolution associated to the comonad $\pi_*\pi^!$ is an equivalence:
	\begin{equation*}
	\xymatrix{
		\omega_{LY} & \ar[l]_-\sim  \colim_{\Delta^{op}} [ \pi_*\pi^!\omega_{LY} & \ar@<-0.5ex>[l]\ar@<0.5ex>[l] 
		\pi_*\pi^!\pi_*\pi^!\omega_{LY} & \ar@<-1ex>[l]  \ar[l] \ar@<1ex>[l] 
		\pi_*\pi^!\pi_*\pi^!\pi_*\pi^!\omega_{LY} \cdots]
	}
	\end{equation*}
	By standard identities including base-change, the canonical resolution can be  identified with 
	the simplicial object 	\begin{equation*}
	\xymatrix{
		 \tr (\omega_{X \times_Y X}) & \ar@<-0.5ex>[l]\ar@<0.5ex>[l] 
		\tr(\omega_{X \times_Y X} \star^! \omega_{X \times_Y X} ) & \ar@<-1ex>[l]  \ar[l] \ar@<1ex>[l] 
		\tr(\omega_{X \times_Y X} \star^! \omega_{X \times_Y X} \star^!  \omega_{X \times_Y X} ) \cdots}
	\end{equation*}
	computing the descended trace  of $\om_{Y\times_{X}Y}$.
\end{proof}

\begin{remark} Here is a more direct proof of (2) of the theorem that in fact motivates the definition of the descended trace.

Set $\cA = \Ind\Coh(X \times_Y X)$,  $\cB = \Bimod_{ \cO_{X \times_Y X}}( \Ind\Coh(X \times_Y X))$.
By $!$-descent, we have a monoidal equivalence
$
\cB \simeq 
\QCoh(Y)
$
 where  the monoidal structure on $\QCoh(Y)$ is given by the tensor product.
Moreover, under this equivalence the regular bimodule $ \cO_{X\times_Y X} \in \cB$ corresponds to the structure sheaf
$ \cO_{Y}\in \QCoh(Y)$.

Thanks to~\cite{ben-zviIntegralTransformsDrinfeld2010}, we have a commutative diagram 
\begin{equation*}
\xymatrix{
\ar[d]_-{\tr}  \QCoh(Y) \ar[dr]^-{p^*}\\
hh( \QCoh(Y))  \ar[r]^-\sim & \QCoh(LY)
}
\end{equation*} 
where $p:LY \to Y$ is the natural base-point projection. 
 Hence we have an equivalence
\begin{equation*}
\QCoh(LY)\simeq hh(\QCoh(Y))\simeq hh(\cB)
\end{equation*}
Under this equivalence, the natural map $h: hh(\cB)\to hh( \cA)$ is identified with the inclusion $i: \QCoh(LY)  \incl \Ind\Coh_{\L_{X/Y}}(LY)$,
where we view $ \QCoh(LY) \subset  \Ind\Coh(LY)$ as ind-coherent sheaves with singular support in the zero-section.


Thus we conclude 
\begin{equation*}
\ol \tr( \cO_{X \times_Y X}) = h(\tr_\cB(  \cO_{X \times_Y X} ))  \simeq i( \tr(\cO_{Y}) )\simeq p^*\cO_{Y} \simeq  \cO_{LY}.
\end{equation*}
\end{remark}

\subsubsection{Application to Steinberg stack}\label{s:app to steinberg}

Now we specialize our prior setup to  $X \to Y$ the induction map $\dB/\dB \to \dG/\dG$
so that $X\times_Y X =  \St_{\dG}$.  By Example~\ref{ex:loop St}, we have $LY = Z^2_{\dG}$ and $\L_{X/Y}=\cN$ the nilpotent codirections in $T^{*-1}Z^2_{\dG}$.
We immediately conclude the following:


\begin{cor}\label{c:OZ OSt}
	Under the equivalence
	\beq\label{spec cocenter}
	\xymatrix{
		hh(\Ind\Coh(\St_{\dG}))  \ar[r]^-\sim & \Ind\Coh_\cN(Z^2_{\dG})
	}
	\eeq 
	there is a canonical identification of the descended trace of the coalgebra object $\cO_{\St_{\dG}}$
	with the structure sheaf
	\begin{equation*}
	\ol \tr( \cO_{\St_{\dG}}) \simeq \cO_{Z^2_{\dG}} .
	\end{equation*}
\end{cor}


\subsection{Automorphic realization}


In this subsection, we give a presentation of the cocenter of the universal Hecke category $\cH_{\cG}$ using ``partial cocenters''. The partial cocenters are then interpretated geometrically using parabolic character sheaves introduced by Lusztig.

\sss{Character sheaves as cocenter}\label{sss:CS G} 
We first review the interpretation of character sheaves on $G$ as the cocenter of $\cH_{G}$. Various versions of this statement appear in the literature, see \cite{ben-zviCharacterTheoryComplex}, \cite{bezrukavnikovCharacterDmodulesDrinfeld2012} and \cite{lusztigTruncatedConvolutionCharacter} 
We will state here a universal monodromic version.

Let $\Sh_{\cN}(G/G)$ be the full subcategory of $G$-equivariant sheaves on $G$ with nilpotent singular support. This is a universally monodromic version of character sheaves. Consider the horocycle correspondence 
\beq\label{horo diagram G}
\xymatrix{
U\bs G/U 
& \ar[l]_-{\d_G}  \f{G }{U}   \ar[r]^-{\pi_G} &  \f{G}{G} 
}
\eeq
Then we have the horocycle functor
\beq\label{horo G}
\xymatrix{
{\g:=\pi_{G!}\delta_G^*} : \cH_G = \Sh_{\bimon}(U\bs G/U)  \ar[r] & \Sh_\cN(G/G).
 }
\eeq

The following result is a special case of Theorem~\ref{thm: geom descent} which we will prove in Section~\ref{sect:horo desc}.

\begin{theorem}\label{thm:CS G} There is a canonical equivalence
\begin{equation*}
\xymatrix{hh(\cH_{G})\ar[r]^-\sim &  \Sh_\cN(G/G)}
\end{equation*}
such that the composition $\cH_{G}\xr{\tr_{G}} hh(\cH_{G})\simeq \Sh_\cN(G/G)$ is identified with $\g$.
\end{theorem}

\subsubsection{Hochschild homology under $\cH_{\cG^{\c}}$}

Recall $\cG^{\c}$ is the neutral component of the loop group $\cG$ and $\cH_{\cG^{\c}}\subset \cH_{\cG}$ is the full subcategory of objects supported on $\cG^{\c}$. 

For any finite type $J\sft I^{a}$, we have the (universal) finite Hecke category $\cH_{L_J}$ of the Levihoric $L_{J}$, which lies inside $\cH_{\cG^{\c}}$. Any $\cH_{\cG^{\c}}$-bimodule can be viewed as a $\cH_{L_{J}}$-bimodule and we can form the Hochschild homology $hh(\cH_{L_{J}}, \cM)$. For $J\subset J'$ both of finite type, we have a natural functor $hh(\cH_{L_{J}}, \cM)\to hh(\cH_{L_{J'}}, \cM)$.

\begin{cor}[of Theorem~\ref{thm:colim}]\label{cor:hh of colim}

For any $\cH_{\cG^{\c}}$-bimodule $\cM$,  the natural maps induce an equivalence
\beq\label{colim hh LJ}
\xymatrix{
\colim_{{J \sft I^a}} hh(\cH_{L_J}, \cM) \ar[r]^-\sim  & hh(\cH_{\cG^{\c}}, \cM)
}
\eeq

Moreover, for each $J \sft I^a$, the equivalence 
naturally extends to a commutative diagram
\begin{equation*}
\xymatrix{
&\ar[dl]_-{{\tr_{J}}} \ar[d] \cM \ar[dr]^-{\tr}\\
hh(\cH_{L_J}, \cM)  \ar[r] &  \colim_{J' \sft I^a} hh(\cH_{L_{J'}}, \cM) \ar[r]^-\sim  & hh(\cH_{\cG^{\c}}, \cM)
}
\end{equation*} 
where the diagonal arrows are traces, and the left horizontal arrow is the natural map. 
\end{cor}


\begin{proof}
By Theorem~\ref{thm:colim}, we have $\cH_{\cG^{\c}}\simeq \colim_{J\sft I^{a}}\cH_{L_{J}}$ where  the colimit is taken within monoidal categories.  
Applying Corollary~\ref{cor: general hh of colimit of algs}, we  deduce
\begin{equation*}
\colim_{J\sft I^{a}}hh(\cH_{L_{J}},  \cM) \isom hh(\cH_{\cG^{\c}}, \cM).
\end{equation*}
The rest of the assertions are clear.
\end{proof}


\subsubsection{Hochschild homology under $\cH_{\cG}$}\label{sss:general}
Let $G$ be a connected reductive group. Now we would like to calculate $hh(\cH_{\cG}, \cM)$ for a $\cH_{\cG}$-bimodule $\cM$ in terms of Hochschild homology under various finite Hecke categories. 

First we recall some constructions of Varshavsky.  Let $\cD^\circ$ be the poset of {finite type subsets $J \sft I^a$} under inclusion. Consider the abelian group $\Om = N_{\cG}(\cI)/\cI$ with its action on $I^a$ and induced action on $\cD^\circ$. Let $\cD$ be the category with objects $J \sft I^a$, morphisms $J \to J'$ given by $\om \in \Om$ with $\om(J)  \subset J'$, and compositions induced by multiplication in $\Om$. In other words,  $\cD$ is the groupoid $\cD^{\c}/\Om$. Note the natural (faithful but not full) functor $i:\cD^\circ \to \cD$ which  is an equivalence if and only if $G$ is simply-connected.

%

For each $\om\in\Om$ we define a monoidal auto-equivalence of $\cH_{\cG}$ as follows. Choose a lifting $\dot\om$ of $\om$ in $N_{\cG}(\cI)$. Using $\dot\om$ as the base point, we identify $\cI\dot\om\cI/\cI^{u}$ with $H$  (via $\dot\om h\bij h$). Let $C_{\dot\om}\in \cH_{\cG}$ be the extension by zero of $\cL_{\univ}$ supported on $\cI\dot\om\cI/\cI^{u}\simeq H$. Note that $C_{\dot\om^{-1}}$ is the inverse of $C_{\dot\om}$ under the monoidal structure of $\cH_{\cG}$. We get a monoidal auto-equivalence
 \begin{equation*}
\xymatrix{c_{\dot\om}: \cH_{\cG} \ar[rr]^-{C_{\dot \om}\star(-)\star C_{\dot \om^{-1}}} & & \cH_{\cG}.}
\end{equation*}
We claim that $c_{\dot\om}$ is canonically independent of the choice of the lifting $\dot\om$. Indeed, for a different lifting $\ddot\om=\dot\om h$ for some $h\in H$, we have a canonical isomorphism 
\begin{equation*}
C_{\dot\om}\simeq C_{\ddot\om}\ot_{R}\cL_{\univ,h}.
\end{equation*}
Here $R=k[\xcoch(H)]$ and $\cL_{\univ,h}$ is the stalk of $\cL_{\univ}$ at $h\in H$, an invertible $R$-module. On the other hand, 
\begin{equation*}
C_{\dot\om^{-1}}\simeq C_{\ddot\om^{-1}}\ot_{R}\cL_{\univ,h^{-1}}.
\end{equation*}
Since $\cL_{\univ,h}$ and $\cL_{\univ,h^{-1}}$ are inverse to each other as invertible $R$-modules, the operations $c_{\dot \om}=C_{\dot \om}\star(-)\star C_{\dot \om^{-1}}$ and $c_{\ddot\om}=C_{\ddot \om}\star(-)\star C_{\ddot \om^{-1}}$ are canonically identified. 
Therefore we get a canonical monoidal auto-equivalence
\begin{equation}\label{Om action HG}
c_{\om}: \cH_{\cG} \to \cH_{\cG}.
\end{equation}
The canonicity of $c_{\om}$ implies that  they together give an action of $\Om$ on $\cH_{\cG}$ as a monoidal category.

The same construction shows that:  for any $\cH_{\cG}$-bimodule $\cM$, there is a canonical action of $\Om$ on $\cM$ such that $\om\in\Om$ acts by $C_{\dot \om}\star(-)\star C_{\dot \om^{-1}}$, for any lifting $\dot \om$ of $\om$ in $N_{\cG}(\cI)$.

If $\om\in \Hom_{\cD}(J,J')$, $c_{\om}$ sends $\cH_{L_{J}}$ to $\cH_{L_{J'}}$. Therefore, the diagram of Hecke categories $J \mapsto \cH_{L_J}$, for  $J \sft I^a$, naturally extends along $i$  to a functor from $\cD$ to monoidal categories. Using these functors,  for any $\cH_\cG$-bimodule $\cM$, restriction to $\cH_{L_J}$ for  $J \sft I^a$, 
naturally extends to a functor from $\cD$ to bimodules.

%
%

\begin{prop}\label{p:hh of colim gen} Let $G$ be a reductive group.  For any $\cH_\cG$-bimodule $\cM$,  the natural maps induce an equivalence
\begin{equation*}
\xymatrix{
\colim_{\cD} hh(\cH_{L_J}, \cM) \ar[r]^-\sim  & hh(\cH_{\cG}, \cM)
}
\end{equation*}

Moreover, for each $J \sft I^a$, the equivalence 
naturally extends to a commutative diagram
\begin{equation*}
\xymatrix{
&\ar[dl]_-{\tr} \ar[d]_-{\tr_J} \cM \ar[dr]^-{\tr}\\
hh(\cH_{L_J}, \cM)  \ar[r] &  \colim_{\cD} hh(\cH_{L_J}, \cM) \ar[r]^-\sim  & hh(\cH_{\cG}, \cM)
}
\end{equation*} 
where the diagonal arrows are traces, and the left horizontal arrow is the natural map. 

\end{prop}

\begin{proof}
For any $\cH_{\cG}$-bimodule $\cM$, there is a canonical map
\begin{equation*}
\xymatrix{
c_{\cM}:\colim_{\cD} hh(\cH_{L_J},  \cM) \ar[r]  & hh(\cH_{\cG},  \cM). 
}
\end{equation*}
We need to show that this is an equivalence. It suffices to set $M = \cH_\cG \otimes \cH_\cG$ (where $ \cH_\cG$ acts on the first factor of $\cH_\cG \otimes \cH_\cG$ on the right and the second on the left) and prove  the natural map
\beq\label{eq:bimod isom}
\xymatrix{
c:\colim_{\cD} hh(\cH_{L_J},  \cH_\cG \otimes \cH_\cG)\simeq \colim_{\cD} \cH_\cG \otimes_{\cH_{L_{J}}}\cH_\cG\ar[r]  & \cH_{\cG}\simeq hh(\cH_{\cG},  \cH_\cG \otimes \cH_\cG) 
}
\eeq
is an equivalence of $\cH_\cG$-bimodules  (where the  $ \cH_\cG$-action on both sides is induced by its action on the left of the first factor of  $\cH_\cG \otimes \cH_\cG$ and  on the right of the second).
Note $c$ is induced by $\cH_\cG$-bimodule maps, so is naturally an $\cH_\cG$-bimodule map. Thus it suffices to check $c$ is an equivalence.

Consider the following commutative diagram
\begin{equation*}
\xymatrix{\colim_{\cD^{\c}}\cH_{\cG^{\c}} \otimes_{\cH_{L_{J}}} \cH_\cG\ar[d]^{\fri} \ar[r]^{\sim}&
\colim_{\cD^\circ} hh(\cH_{L_J},  \cH_{\cG^{\c}} \otimes \cH_\cG) \ar[r]^-\sim  & hh(\cH_{\cG^{\c}}, \cH_{\cG^{\c}} \otimes \cH_{\cG})\ar[d]^{\sim}  \\
\colim_{\cD}\cH_{\cG} \otimes_{\cH_{L_{J}}} \cH_\cG\ar[r]^{\sim} & \colim_{\cD} hh(\cH_{L_J},  \cH_\cG \otimes \cH_\cG) \ar[r]^-c  & hh(\cH_{\cG},  \cH_\cG \otimes \cH_\cG)}
\end{equation*}
Here the top middle arrow is an equivalence by Corollary~\ref{cor:hh of colim}, the right vertical map is the evident induction equivalence (both are equivalent to $\cH_{\cG}$), and $\fri$ is the natural map induced by $i:\cD^\circ \to \cD$ and the inclusion $\cH_{\cG^{\c}}\incl\cH_{\cG}$. Thus to show $c$ is an equivalence, it suffices to show $\fri$ is an equivalence.

Let $\Phi: \cD^{\c}\to \Cat_{\oo}$ be the functor given by $J\mapsto \cH_{\cG^\circ} \otimes_{\cH_{L_J}} \cH_\cG$. The forgetful functor $i^{!}: \Fun(\cD, \Cat_{\oo})\to \Fun(\cD^{\c}, \Cat_{\oo})$ admits a left adjoint (left Kan extension along $i$) $i_{!}: \Fun(\cD^{\c}, \Cat_{\oo})\to \Fun(\cD, \Cat_{\oo})$, so that $\colim_{\cD^{\c}}\Phi=\colim_{\cD}i_{!}\Phi$. Using this we can rewrite $\fri$ as
\begin{equation}\label{rewrite colimD}
\colim_\cD (i_! \Phi)(J)\to \colim_{\cD}  \cH_\cG  \otimes_{\cH_{L_J}} \cH_\cG
\end{equation}
induced by the termwise functor $\phi_{J}: (i_! \Phi)(J)\to \cH_\cG  \otimes_{\cH_{L_J}} \cH_\cG$ for $J\in \cD$. We show that $\phi_{J}$ is an equivalence for each $J\sft I^{a}$. Indeed, 
\begin{equation*}
(i_! \Phi)(J)=\bigoplus_{\om\in\Om} \Phi(\om(J))=\bigoplus_{\om\in\Om}\cH_{\cG^\circ} \otimes_{\cH_{L_{\om(J)}}} \cH_\cG.
\end{equation*}
We have embeddings 
\begin{equation*}
i_{\om}: \cH_{\cG^\circ} \otimes_{\cH_{L_{\om(J)}}} \cH_\cG\to \cH_{\cG}\ot_{\cH_{L_{J}}}\cH_{\cG}
\end{equation*}
given by $x\ot y\mapsto (x\star C_{\dot\om})\ot (C_{\dot\om^{-1}}\star y)$ (which is again canonically independent of the lifting $\dot\om$). The functor $\phi_{J}$ is the direct sum of $i_{\om}$
\begin{equation*}
\op i_{\om}: \bigoplus_{\om\in\Om}\cH_{\cG^\circ} \otimes_{\cH_{L_{\om(J)}}} \cH_\cG\to \cH_{\cG}\ot_{\cH_{L_{J}}}\cH_{\cG},
\end{equation*}
which is easily seen to be an equivalence.  Since  $\phi_{J}$ is an equivalence for all $J\in \cD$, \eqref{rewrite colimD} is an equivalence.

%
\end{proof}

Recall that $\Om$ acts on any $\cH_{\cG}$-bimodule $\cM$ by conjugation, and it acts on $\cH_{\cG^{\c}}$ by monoidal auto-equivalences, compatible with the bimodule structure on $\cM$. These actions  induce an action of $\Om$ on the Hochshild homology $hh(\cH_{\cG^{\c}}, \cM)$.

\begin{cor}\label{c:coin} For any $\cG$-bimodule $\cM$, there is a canonical equivalence between the $\Om$-coinvariants on $hh(\cH_{\cG^{\c}}, \cM)$ and $hh(\cH_{\cG}, \cM)$:
\begin{equation}\label{Om coin}
hh(\cH_{\cG^{\c}}, \cM)_{\Om}\isom hh(\cH_{\cG}, \cM).
\end{equation}
In particular, 
\begin{equation*}
hh(\cH_{\cG^{\c}}, \cH_{\cG})_{\Om}\simeq hh(\cH_{\cG}).
\end{equation*}
\end{cor}
\begin{proof}
Let $\cC:\cD^{\c}\to \Cat_{\infty}$ be the functor given by $\cC_{J}=hh(\cH_{L_{J}}, \cM)$. Then this functor has a canonical $\Om$-equivariant structure, hence $\colim_{J\in \cD^{\c}}\cC_{J}$ carries an action of $\Om$. There is a natural map
\begin{equation}\label{coin gen}
(\colim_{J\in \cD^{\c}} \cC_{J})_{\Om}\to\colim_{J\in \cD}\cC_{J}.
\end{equation}
By Corollary \ref{cor:hh of colim} and Proposition \ref{p:hh of colim gen}, the two sides above are equivalent to the two sides of \eqref{Om coin}. Thus it suffices to show that \eqref{coin gen} is an equivalence. Note that $\cD=\cD^{\c}/\Om$. Consider the projection $\pi: \cD\to \BB\Om$, which is a coCartesian  fibration. The left Kan extension $\pi_{!}\cC$ is $\colim_{J\in \cD^{\c}} \cC_{J}$ as a category with $\Om$-action. By \cite[Proposition 4.3.3.10]{lurieHigherToposTheory2009a}, we have
\begin{equation*}
\colim_{\cD}\cC\simeq \colim_{\BB\Om}\pi_{!}\cC\simeq (\colim_{\cD^{\c}}\cC)_{\Om}.
\end{equation*}
\end{proof}


\subsubsection{Geometry of trace}\label{sss:geom trace}
For $J \sft I^a$, define the ind-stack
\begin{equation*}
\cY_{J}:=\frac{\cP^{u}_{J}\bs\cG/\cP_J^u}{L_J},
\end{equation*}
where $\frac{\cdot}{L_{J}}$ denotes the quotient by the conjugation action of $L_{J}$. 
Consider the horocycle correspondence
\beq\label{eq:horo corr in geom of tr}
\xymatrix{
\cI^{u}\bs\cG/\cI^u & \ar[l]_-{\d_J}  \frac{\cP^{u}_{J}\bs\cG/\cP^u_{J}}{U_J\sft}  \ar[r]^-{\pi_J} &  \frac{\cP^{u}_{J}\bs\cG/\cP_J^u}{L_J}=\cY_{J}}
\eeq

 
 To simplify the notation, we will set
 \begin{equation*}
\cH_{\cG, J} =  \Sh_\cN( \cY_{J}):=\colim_{w\in \{W_{J}\bs \tilW/W_{J}\}}\Sh_\cN\left( \frac{\cP^{u}_{J}\bs \cG_{\le w}/\cP_J^u} {L_J}\right).
\end{equation*}
Here the colimit is taken over longest representatives in the  $W_{J}$-double cosets of $\tilW$ (so that $\cG_{\le w}$ is a union of $\cP_{J}$-double cosets). The notation $\Sh_{\cN}(-)$ means, viewed as sheaves on $\cP_{J}^{u}\bs\cG/\cP_J^u$,  the singular support has nilpotent image under the moment map for the $L_{J}\times L_{J}$-action by left and right translations. 

Note that when $J=\vn$, $\cH_{\cG,\vn}$ imposes an $\Ad(H)$-equivariance structure on sheaves on $\cI^{u}\bs \cG/\cI^{u}$. 

\begin{remark} We will see in Section~\ref{ss:geom pieces} that $\cH_{G,J}$ is closely related to the notion of {\em parabolic character sheaves} for the loop group $\cG$ defined by Lusztig in \cite{lusztigParabolicCharacterSheaves}. 
\end{remark}

Consider the functor  
$$
\xymatrix{
\pi_{J!}\delta_J^*:\cH_\cG=\Sh_{\bimon}(\cI^{u}\bs\cG/\cI^u)  \ar[r] & \Sh(\cY_{J}). }
$$
It is easy to check that the image of $\pi_{J!}\delta_J^*$ lands in the full subcategory $\Sh_{\cN}(\cY_{J})=\cH_{\cG,J}$ (it suffices to check on each $\cP_{J}$-double coset of $\cG$). Hence we get a functor
\begin{equation}\label{eq:horo functor in geom of tr}
\xymatrix{\pi_{J!}\delta_J^*:\cH_\cG\ar[r] & \cH_{\cG,J}.}
\end{equation}

 The following result gives a geometric interpretation of the partial cocenters $hh(\cH_{L_J}, \cH_{\cG})$. It is a special case of Theorem~\ref{thm: geom descent ind}, which we will state and prove in Section~\ref{sect:horo desc}.  

\begin{theorem} \label{thm:coequal J} For $J\sft I^{a}$, the functor $\pi_{J!}\delta_J^*$ 
 fits into a commutative diagram with the trace map $\tr$ inducing a horizontal equivalence
\begin{equation*}
\xymatrix{
\ar[d]_-{\tr}\cH_\cG \ar[dr]^-{\pi_{J!}\delta_J^*}\\
hh(\cH_{L_J}, \cH_{\cG})\ar[r]^-\sim & \cH_{\cG, J}
}
\end{equation*} 
\end{theorem}

%
%
%

Substituting Theorem~\ref{thm:coequal J} into Corollary \ref{cor:hh of colim} and Proposition \ref{p:hh of colim gen}, we immediately obtain:

\begin{cor}\label{cor:coequal J geom} We have equivalences
\begin{eqnarray*}
\colim_{J\in \cD^{\c}}   \cH_{\cG, J} \simeq  & hh(\cH_{\cG^{\c}}, \cH_{\cG}), \\
\colim_{J\in \cD}   \cH_{\cG, J} \simeq  & hh(\cH_{\cG}).
\end{eqnarray*}
Moreover, for each $J \sft I^a$, the equivalence  above naturally extends to a commutative diagram
\begin{equation*}
\xymatrix{
&\ar[dl]_-{\pi_{J!}\delta_J^*} \ar[d] \cH_\cG \ar@/^1pc/[dr]^-{\tr}\\
 \cH_{\cG, J}  \ar[r] &  \colim_{\cD}   \cH_{\cG, J'} \ar[r]^-\sim  & hh(\cH_{\cG})
}
\end{equation*} 
where the left horizontal arrow is the natural map. 
\end{cor}

\sss{Connected components}\label{sss:HG comp}
For $\om\in\Om$, let $\cH_{\cG}^{\om}$ be the full subcategory of $\cH_{\cG}$ consisting of sheaves supported on the $\om$-component of $\cG$. Similarly, define $\cH^{\om}_{\cG,J}$. Note the action of $\Om$ on $\cH_{\cG}$ preserves each $\cH^{\om}_{\cG}$. Therefore we have a decomposition by support
\begin{equation*}
hh(\cH_{\cG})=\bigoplus_{\om\in\Om}hh(\cH_{\cG})^{\om}
\end{equation*}
where
\begin{equation*}
hh(\cH_{\cG})^{\om}\simeq \colim_{J\in\cD}\cH_{\cG,J}^{\om}.
\end{equation*}


\subsection{Descended trace of  Whittaker object}\label{sss:des tr Wh} 

 Recall that $\Wh_\cG$ is naturally a coalgebra in $\cH_{\cG}$,  therefore it makes sense to take its descended trace  $\ol \tr(\Wh_\cG) \in hh(\cH_\cG)$. Let
\begin{equation*}
\Wh_{\cG/\cG}:=\ol \tr(\Wh_\cG) \in hh(\cH_\cG).
\end{equation*}
The goal of this subsection is to calculate  $\Wh_{\cG/\cG}$ in terms of character sheaves on $G$.

\sss{Reduction from $\cG$ to $G$}
Recall the monoidal functor $i_{!}: \cH_{G}\to \cH_{\cG}$.  It induces a functor by passing to  cocenters
\begin{equation}\label{define a}
\xymatrix{a: \Sh_\cN(G/G)\ar[r]^-{\sim} & hh(\cH_{G}) \ar[r]^-{hh(i_{!})} &  hh(\cH_{\cG})
}
\end{equation}
where the first equivalence is given by Theorem \ref{thm:CS G}.

Recall $I \subset I^a$ are the simple roots of $G$. The corresponding maximal parahoric $\cP_{I} \subset \cG $ is the arc group $\cG_0 = G\tl{t}$
with pro-unipotent radical $\cP^u_{I} \subset \cP_I$ the arcs based at the identity $\cG_0^u= \ker (G\tl{t} \to G)$. By writing $hh(i_{!})$ as the composition 
$$hh(\cH_{G})\to hh(\cH_{G}, \cH_{\cG})\to hh(\cH_{\cG}),$$
and using Corollary~\ref{cor:coequal J geom} (applied to $J=I$), the functor $a$ factors as the composition
\begin{equation*}
\xymatrix{a: \Sh_\cN(G/G)\ar@{^{(}->}[r]^-{i_{G/G}} & \cH_{\cG,I} \ar[r] &  hh(\cH_{\cG})
}
\end{equation*}
where $i_{G/G}$ is the full embedding given by the direct image along the natural map
$$\f{G}{G}\to \frac{ \cG^{u}_{0}\bs \cG/\cG_0^u}  {G}=\cY_{I}.$$

%
%

\begin{lemma}\label{l:Wh cG Wh G} We have the following relation between the descended traces of $\Wh_{G}$ and $\Wh_{\cG}$:
\begin{equation*}
\ov\tr(\Wh_{\cG})\simeq a(\ov\tr(\Wh_{G}))\in hh(\cH_{\cG}).
\end{equation*}
\end{lemma}
\begin{proof}
By Lemma~\ref{l:const loops}, the universal affine Whittaker sheaf is the extension by zero of the finite Whittaker finite sheaf $\cW_\cG \simeq i_{!}\Wh_G$. The statement then follows from Lemma~\ref{l: dtr inv}.
\end{proof}



\sss{Whittaker character sheaf} 

Consider the diagram
\begin{equation*}
\xymatrix{
\AA^1 & \ar[l]_-{\chi} U^{-}/U^{-} \ar[r]^-{r_{-}} & G/G
}
\end{equation*}
where $r_{-}$ is induced by the inclusion $U^{-} \subset G$, and $\chi$ is induced by the same-named non-degenerate character $\chi: U \to U/[U, U] \to \Ga=\AA^1$ used in Section~\ref{sss:Whit G}.

Let  $\varphi_{\chi, 1}: \Sh(U^{-}/U^{-}) \to k\lmod$ denote the vanishing cycles at the identity $1\in U^{-}$ with respect  to the function  $\chi:U^{-}/U^{-}\to \AA^1$.

\begin{defn} \label{def: whit char sh}
\begin{enumerate}
\item The {\em Whittaker functor} on character sheaves is  the composition
\begin{equation*}
\xymatrix{
{\cW_{G/G}}:   \Sh_\cN(G/G) \ar[r] &  k\lmod &
\cW_{G/G}(\cF) = \varphi_{\chi, 1}  r_{-}^!\cF.
}\end{equation*}

 \item The {\em Whittaker character sheaf} ${\Wh_{G/G} }\in  \Sh_\cN(G/G) $ is the object corepresenting the Whittaker functor
 \begin{equation*}
\xymatrix{
\cW_{G/G}(\cF) \simeq \Hom_{\Sh_{\cN}(G/G)}(\Wh_{G/G}, \cF), & 
\text{for all } \cF \in \Sh_\cN(G/G). 
}\end{equation*}

%
%
%
%
%
%

\end{enumerate}
\end{defn}

%
%

Now we arrive at the main result of this section.
A generalization in the context of nodal degenerations of curves will appear in~\cite{nadlerCompatibilitiesAutomorphicGluing}.

\begin{theorem}\label{thm: dtr of whit}
 \begin{enumerate}
 
 \item
 For the trace map 
\begin{equation*}
\xymatrix{
\tr_{G}=\g: \cH_G = \Sh_\cN(U\bs G/U)  \ar[r] & \Sh_\cN(G/G) \simeq hh(\cH_G)
 }
\end{equation*}
there is a canonical identification of the descended trace  of the universal finite Whittaker sheaf  $\ol \tr(\Wh_{G}) \in hh(\cH_G) $ 
with the Whittaker character sheaf $\Wh_{G/G}\in  \Sh_\cN(G/G)$.

\item  We have a canonical identification of the descended trace of $\Wh_{\cG}$:
\begin{equation*}
\Wh_{\cG/\cG}=\ov\tr(\Wh_{\cG})\simeq a(\Wh_{G/G}) \in hh(\cH_{\cG}),
\end{equation*}
where $a$ is defined in \eqref{define a}.


\end{enumerate}
\end{theorem}

\begin{proof}
 (2) follows  immediately from (1) by Lemma~\ref{l:Wh cG Wh G}.  
 
To prove (1), it will be convenient to view $\cH_G$ as an algebra in $\cH_H = \Sh_0(H)$-bimodules. Here $\Sh_0$ 
means sheaves with singular support within the zero-section, or equivalently local systems. Note $\cH_H \subset \cH_G$ is the full monoidal subcategory generated by the monoidal unit $\cH_H = \langle 1_{\cH_G}\rangle \subset \cH_G$. 

Recall $hh(\cH_G)$ is canonically independent of whether we work absolutely or in $ \cH_H$-bimodules. To be more precise, set $\cH_G^{(n)} = \cH_G \otimes  \cdots \otimes \cH_G $ ($n$ copies of  $\cH_G$) and 
$\cH_G^{(n)_H} = \cH_G \otimes_{\cH_H}  \cdots \otimes_{\cH_H} \cH_G $ ($n$ copies of  $\cH_G$). So in particular $\cH_G = \cH_G^{(1)} = \cH_G^{(1)_H}$.  Set also
$ \cH_{G, H}^{(n)} = hh(\cH_H,  \cH_G^{(n)_H})$, so  in particular 
\begin{equation*}
\cH_{G,H}= \cH_{G, H}^{(1)} = \Sh_\cN\left(\f{U\bs G/U}{H}\right)
\end{equation*}
Here as usual $\Sh_{\cN}$ means sheaves with nilpotent singular support, or equivalently monodromic sheaves with respect to the remaining $H$-action.

Then 
the natural map from the absolute to $\cH_H$-relative Hochschild  complexes induces an equivalence on colimits
\beq\label{eq:hh diags related}
 \xymatrix{
\ar[d]_-\simeq hh(\cH_G)  & \ar[l]_-\gamma
 [ \ar[d]_-{q_! = q^{(1)}_!} \cH_G  & \ar[d]_-{q^{(2)}_!} \ar@<-0.5ex>[l]\ar@<0.5ex>[l] \cH_G^{(2)} & \ar[d]_-{q^{(3)}_!} \ar@<-1ex>[l]  \ar[l] \ar@<1ex>[l] 
  \cH_G^{(3)} \cdots]
  \\
  hh(\cH_G)   
 & \ar[l]_-{\gamma_H}  [  \cH_{G, H} & \ar@<-0.5ex>[l]\ar@<0.5ex>[l]  \cH_{G, H}^{(2)} & \ar@<-1ex>[l]  \ar[l] \ar@<1ex>[l] 
   \cH_{G, H}^{(3)} \cdots]
 }
 \eeq 
 Here the maps $q_!^{(n)}$ are the $!$-pushforwards along the natural quotient maps.  
Moreover,  the augmentations are given by  $\g = {\pi_{G!}\delta_G^*}$
and $\g_H = \pi_{G,H!}\delta^*_{G,H}$ as appear in the diagram
\begin{equation*}
\xymatrix{
 \f{G}{G}  & \ar[d]  \ar[l]_-{\pi_G}  \ar[r]^-{\d_G}  \f{G}{U} & \ar[d]_-q U\bs G/U    \\
 &  \ar[ul]^-{\pi_{G,H}}   \f{G}{B}     \ar[r]^-{\d_{G,H}}  & \f{U\bs G/U}{H} 
}
\end{equation*}
with Cartesian square. Note  by base-change, $\g = {\pi_{G!}\delta_G^*}$ indeed admits the natural factorization 
\begin{equation*}
\xymatrix{
\g: \cH_{G}  \ar[r]^-{q_!} & \cH_{G,H} \ar[rr]^-{\g_H = \pi_{G,H!}\delta^*_{G,H}} &&\Sh_\cN(\f{G}{G}) 
 }
\end{equation*}

 For $1 \leq i \leq n$, let $m_{n, i}:  \cH_{G, H}^{(n)} \to  \cH_{G, H}^{(n-1)}$ denote the face
 map of the lower simplicial diagram  \eqref{eq:hh diags related} given by  convolving the  cyclically-ordered $i$th and $(i+1)$st terms. (By convention, we have $m_{1, 1} = \gamma_H$,  and $\cH_{G, H}^{(0)} = Sh_{\cN}(G/G)$.)

Note by standard base-change identities, the natural base-change map is an isomorphism $m_{n, i}^\ell m_{n, i} \simeq m_{n+1, i+1} m_{n+1, i}^\ell$, for all $1 \leq i \leq n$.
Let 
  $\g_H^{(n)} : \cH_{G, H}^{(n)} \to hh(\cH_G)$ denote the canonical map given by $\gamma_H$ applied to any
 cyclically-ordered total convolution.

Set $\Wh_G^{(n)} = \Wh_G \otimes \cdots \otimes \Wh_G$ ($n$ copies of $\Wh_G$) and $\Wh^{(n)}_{G, H} = q^{(n)}_! \Wh_G^{(n)}$. In particular, we have $\Wh_G = \Wh_G^{(1)}$ and write $\Wh_{G, H} =  \Wh_{G, H}^{(1)}$. By convention, set $\Wh_{G, H}^{(0)} = \Wh_{G/G} \in \cH_{G, H}^{(0)}$.

 The map $\alpha$ of \eqref{eq:map_alpha} gives 
\begin{equation*} 
 \alpha^H: m^\ell \Wh_{G,H}  \to \Wh_{G,H}^{(2)},
\end{equation*} 
which yields for any $n \geq 1$, and $ 1 \leq i \leq n$ a map:
\begin{equation*} 
 \alpha^H_{n,i}: m_{n+1, i}^\ell \Wh_{G, H}^{(n)} \to \Wh^{(n+1)}_{G, H} 
\end{equation*} 
Similarly, we have $\alpha^H_{0}: \gamma_H^\ell \Wh_{G/G} \to \Wh_{G,H} $  (in fact, it is constructed in the next lemma).

By adjunction, we obtain 
\begin{equation*} 
 \beta^H_{n,i}:  \Wh_{G, H}^{(n)} \to m_{n+1, i} \Wh^{(n+1)}_{G, H} 
\end{equation*} 
Note that, by construction, this is exactly the natural map induced by the coalgebra structure of $\Wh_G$ (after applying $q_!^{(n)}$).
Similarly, by adjunction, we have $\beta^H_{0}:  \Wh_{G/G} \to \gamma_H \Wh_{G,H} $.


Now by Proposition~\ref{prop:whit comp} below, $\alpha^H_{0}$ and all $\alpha^H_{n,i}$ are isomorphisms.
Thus unwinding the identities,  we conclude the canonical resolution of $\Wh_{G/G}$ given by the monad $T = \gamma_H \gamma_H^\ell$ is precisely  
 the resolution calculating the descended trace of $\Wh_G \in \cH_G$ regarded as a coalgebra:
 $$
 \xymatrix{
\Wh_{G/G}  
\isom    [\gamma_H\gamma_H^\ell \Wh_{G/G}  \ar@<-0.5ex>[r]\ar@<0.5ex>[r] & (\gamma_H\gamma_H^\ell)^2 \Wh_{G/G} \ar@<-1ex>[r]  \ar[r] \ar@<1ex>[r] & 
(\gamma_H\gamma_H^\ell)^3 \Wh_{G/G} \cdots]
}
$$
$$
\xymatrix{
\isom    [\g_H \g^{\ell}_H \Wh_{G/G}  \ar@<-0.5ex>[r]\ar@<0.5ex>[r] & \g_H^{(2)} \g^{\ell}_H { }^{(2)}  \Wh_{G/G} \ar@<-1ex>[r]  \ar[r] \ar@<1ex>[r] & 
\g_H^{(3)} \g^{\ell}_H { }^{(3)} \Wh_{G/G} \cdots]
}
$$
$$
\xymatrix{
   \simeq   [\gamma_H \Wh_{G, H}  \ar@<-0.5ex>[r]\ar@<0.5ex>[r] &  \g_H^{(2)} \Wh_{G,H}^{(2)} \ar@<-1ex>[r]  \ar[r] \ar@<1ex>[r] & 
   \g_H^{(3)}\Wh_{G, H}^{(3)} \cdots]
 }
$$
$$
\xymatrix{
   \simeq   [\gamma \Wh_{G}  \ar@<-0.5ex>[r]\ar@<0.5ex>[r] & \gamma(\Wh_{G} \star\Wh_{G})  \ar@<-1ex>[r]  \ar[r] \ar@<1ex>[r] & 
  \g(\Wh_{G} \star\Wh_{G} \star \Wh_{G})  \cdots]
 }
$$

\end{proof}

\begin{prop}\label{prop:whit comp}
The maps   $\alpha^H_{0}: \gamma_H^\ell \Wh_{G/G} \to \Wh_{G,H} $ and $ \alpha^H_{n,i}: m_{n+1, i}^\ell \Wh_{G, H}^{(n)} \to \Wh^{(n+1)}_{G, H} $, for all $n \geq 1$ and $1 \leq i \leq n$,  are isomorphisms.
\end{prop}

\begin{proof} By standard identities, it suffices to prove this for $\alpha^H_{0}$ and $  \alpha^H_{1, 1}$.
We will give an argument for $\alpha^H_{0}$; a similar but simpler argument
works for $  \alpha^H_{1, 1}$.

Set $\cW_{G,H} =\Hom(\Wh_{G,H}, -)$, $\cW_{G/G} =  \Hom(\Wh_{G/G}, -)$. 
We will prove there is a canonical isomorphism of functors 
\begin{equation*}
\xymatrix{
\cW_{G,H}(-)[-2\nu]\simeq \cW_{G/G}( \g_{H}(-)):  \cH_{G,H}  \ar[r] &   k\lmod 
}
\end{equation*}
where $\nu=\dim N$. In particular, we will then obtain $ \g_{H}^\ell(\Wh_{G/G}) \simeq \Wh_{G, H}$   by adjunction.

Consider the commutative diagram whose right hand square is Cartesian
\beq\label{eq:main diag}
\xymatrix{
\ar[d]_-{s} U^-\bs (G/U \times G/U)/H & \ar[l]_-{\wt\d}  (U^- \times \cB)/U^- \ar[d]_-r \ar[r]^-{\wt\pi} & U^-/U^- \ar[d]_-r\\
G\bs (G/U \times G/U)/H  & \ar[l]_-{\d}  (G \times \cB)/G  \ar[r]^-{\pi} & G/G
}
\eeq
Here $\cB = G/B$ is the flag variety of $G$, and $G$ acts on  $G \times \cB $ by the adjoint action on $G$ and the translation action on $\cB$. Thus $(G \times \cB)/G$ is isomorphic to the  quotient of $G$ by the adjoint action of $B$. Similarly, $U^-$ acts on $U^- \times \cB$ by the adjoint action on $G$ and the translation action on $\cB$.  From this perspective, $\pi(g, g_1) = g$, $\wt \pi(u, g_1) = u$, and $\delta(g, g_1) = (gg_1, g_1)$, $\wt \delta(u, g_1) = (ug_1, g_1)$.

Set $\wt f = \chi \circ \tilde \pi: (U^{-}\times \cB)/U^{-}\to \AA^{1}$. By base change we have
\begin{equation*}
\cW_{G/G}(\tau(\cF))\simeq \phi_{0}f_{*}r^{!}\pi_{*}\d^{*}(\cF)\simeq \phi_{0}\wt f_{*} r^{!}\d^{*}(\cF).
\end{equation*}
Since $\d$ is smooth of relative dimension $\nu=\dim N$, we have  $\d^{*}(\cF)\simeq \d^{!}\cF[-2\nu]$. Therefore
\begin{equation}\label{WhGch}
\cW_{G/G}(\tau(\cF))\simeq \phi_{0}\wt f_{*} r^{!}\d^{!}(\cF)[-2\nu]\simeq \phi_{0}\wt f_{*}\wt\d^{!} s_{}^{!}\cF[-2\nu].
\end{equation}

Now consider the leftward map $\wt \d: (U^- \times \cB)/U^-  \to U^-\bs (G/U \times G/U)/H$ at the top of ~\eqref{eq:main diag}. 

%

Stratify $\cB$ by $U^-$-orbits  $\cB = \sqcup_{w\in W} \cB^w$ where  $\cB^{1} = U^- B/B$ denotes the open $U^-$-orbit, and in general $\cB^{w}=U^{-}\dot{w}\simeq U^{-}/U^{-}_{w}$ where $U^{-}_{w}=U^{-}\cap {}^{w}B$
 where ${}^w B = \dot w B \dot w^{-1}$.

Stratify 
 $(U^- \times \cB)/U^-$ by the pullback of the $U^-$-orbits on $\cB$. So we have the strata
 $(U^- \times U^-/U^-_w)/U^-$, in particular the open stratum $(U^- \times U^-)/U^- \simeq U^-$.
Stratify 
$ U^-\bs (G/U \times G/U)/H$ by the  pullback of the $U^- \times U^-$-orbits on $\cB \times \cB$. 
So we have the strata $ U^-\bs (U^-/U^-_w \times U^-/U^-_{w'})/H$.

The map $\wt\d$ restricted to the $w$-stratum takes the form
\begin{equation*}
\xymatrix{
\wt\d_{w}: (U^- \times \cB^w)/U^-  \ar[r]  & U^-\bs (\cB^w \times_{\BB H} \cB^w).
}
\end{equation*}

We claim that for $w\ne1$, for any $\cF_{w}\in \Sh(U^-\bs (\cB^w \times_{\BB H} \cB^w))$, we have
\begin{equation}\label{Fw van}
\phi_{0}\wt f_{*}\wt\d_{w}^{!}\cF_{w}\simeq 0.
\end{equation} 
To prove the claim, 
note that $w\ne 1$ implies $U^{-}_{w}$ contains $U_{-\a_{i}}\simeq \AA^1$ for some simple root $\a_{i}$.
We are studying the correspondence
\begin{equation*}
\xymatrix{
U_w^-\bs U^-/U^-_w & \ar[l]_-{\wt \d_w} U^-/U_w^-   \ar[r]^-{\wt \pi} & U^-/U^- \ar[r]^-{f}  & \AA^1
}
\end{equation*}
where the second and third quotients are by the adjoint action. 
Let $U^-_0/U^- \subset U^-/U^-$ denote the pre-image of $0 \in \AA^1$. Then the action provides an isomorphism $U^- \simeq U_{-\alpha_i}   \times U^-_0 $ with $\wt f = f \circ \wt \pi$ given simply by the projection to $U_{-\alpha_i}  \simeq \AA^1$. 
Now $U_{-\a_{i}} \subset U^{-}_{w}$ implies that $\wt \pi_{*}\wh\d_{w}^{!}\cF_{w}$ is constant along  $U_{-\alpha_i}  \simeq \AA^1$, and 
 hence
$\phi_{0}\wh f_{*} \pi_{1*} \wh \d_{w}^{!}\cF_{w}\simeq 0 $. 

By the claim, we have
\begin{equation*}
\phi_{0}\wt f_{*}\wt\d^{!}s^!\cF \simeq \phi_{0}\wh f_{*}\wt\d^{!}_{1}\cF_{1}
\end{equation*}
 where $\cF_{1}$ is the restriction of $s^{!}\cF$ to the open stratum
 $U^-\bs (\cB^1 \times_{\BB H} \cB^1)$. Observe that the composition $s\circ \wt \delta_1$  is nothing more than the composition 
 \begin{equation*}
 \xymatrix{
 U^- \ar[r]^-{r_-} &  U\bs G/U \ar[r]^-q & (U\bs G/U)/H  
 }
 \end{equation*}
 used to define $\cW_{G, H}$.
 
 We conclude that 
\begin{equation*}
\phi_{0}\wt f_{*}\wt\d^{!}s^!\cF \simeq \phi_{0}\wh a_{*}\wt\d^{!}_{1}\cF_{1}\simeq \cW_{G, H}(\cF).
\end{equation*}
Combined with \eqref{WhGch} we get a canonical isomorphism
\begin{equation*}
\cW_{G,H}(\cF)[-2\nu]\simeq \cW_{G/G}(\tau(\cF)).
\end{equation*}
\end{proof}

\section{Horocycle descent}\label{sect:horo desc}

This section contains a proof of Theorem~\ref{thm:coequal J}, or more precisely its generalization to Theorem~\ref{thm: geom descent ind}. 


\subsection{Preliminaries}


\subsubsection{Horocycle diagrams}

Let $G$ be a connected reductive group, $B\subset G$ a Borel subgroup, $N = [B, B]$ its unipotent radical, and $H = B/N$ the universal Cartan.

Consider the basic horocycle diagram
\begin{equation}\label{abs hc}
\xymatrix{
\BB G & \ar[l]_-\epsilon \BB B \ar[r]^-\delta & \BB B \times_{\BB H} \BB B
}
\end{equation}
as a diagram over $\BB(G\times G) \simeq \BB G \times \BB G$.
Note $\e$ is smooth and proper, with fibers isomorphic to $G/B$, and $\d$ is smooth, with fibers isomorphic to~$N$.


Let $Z$ be an ind-stack  with a $G\times G$-action.  We often turn the second $G$-action as a right action on $Z$. For subgroups $G_{1}\subset G$ and $G_{2}\subset G$, we write $G_{1}\bs Z/G_{2}$ instead of $Z/(G_{1}\times G_{2})$.  

We can base-change diagram \eqref{abs hc} along $Z/(G\times G)\to \BB(G\times G)$ and get a $Z$-horocycle diagram 
\begin{equation}\label{general hc}
\xymatrix{Z/\D G & Z/\D B \ar[r]^-{\d}\ar[l]_-{\e} & (N\bs Z/N)/\D H}
\end{equation}
where we write $\Delta G, \Delta B$ to emphasize the diagonal action.

\subsubsection{Horocycle adjunctions}\label{sss:horo adj}

Let $\nu=\dim N$, and define functors
\begin{eqnarray*}
hc_{!}=\d_{!}\e^{*}[-2\nu]: \Sh(Z/\D G)\to \Sh((N\bs Z/N)/\D H)\\
ch=\e_{*}\d^{*}=\e_{!}\d^{*}: \Sh((N\bs Z/N)/\D H)\to \Sh(Z/\D G)\\
hc_{*}=\d_{*}\e^{!}: \Sh(Z/\D G)\to \Sh((N\bs Z/N)/\D H)
\end{eqnarray*}
Since $\e$ is proper and $\d$ is smooth of relative dimension $\nu$, we have adjunctions $(hc_{!},ch)$ and $(ch,hc_{*})$.

Assume $Z$ is smooth and let $\L\subset T^{*}Z$ be a closed conic $G\times G$-invariant subset. For any subgroup $G'\subset G\times G$,  consider the full subcategory $\Sh_{\L}(Z/G')\subset \Sh(Z/G')$  of $G'$-equivariant complexes $\cF$ on $Z$ with singular support satisfying $\ssupp(\cF)\subset \L$.

The following statement is a special case of \cite[Lemma 1.2]{mirkovicCharacteristicVarietiesCharacter1988}.
\begin{lemma}\label{l:hc ss} The functors $hc_{!}$ and $hc_{*}$ send $\Sh_{\L}(Z/\D G)$ to $\Sh_{\L}((N\bs Z/N)/\D H)$, and the functor $ch$ sends $\Sh_{\L}((N\bs Z/N)/\D H)$ to $\Sh_{\L}(Z/\D G)$. In particular, if we let $hc_{\L,!}: \Sh_{\L}(Z/\D G)\to\Sh_{\L}((N\bs Z/N)/\D H)$ be the restriction of $hc_{!}$, and similarly define $ch_{\L}$ and $hc_{\L,*}$, then there are adjunctions $(hc_{\L,!}, ch_{\L})$ and $(ch_{\L}, hc_{\L,*})$.
\end{lemma}

\subsubsection{Unit diagrams}

Consider the basic unit diagram
\begin{equation}\label{abs unit}
\xymatrix{
B\bs G/B \simeq \BB B \times_{\BB G} \BB B  & \ar[l]_-d \BB B \ar[r]^-p & \BB H
}
\end{equation}
as a diagram over $\BB(H\times H) \simeq \BB H \times \BB H$. Here $d$ is the relative diagonal and $p$ is the natural projection.
Note $d$ is a closed embedding and $p$ is smooth, with fibers isomorphic to $\BB N$.

Let $Z$ be an ind-stack with a $H\times H$-action.  We can base-change diagram \eqref{abs unit} along $H\bs Z/H\to \BB(H\times H)$ and get a $Z$-unit diagram 
\begin{equation}\label{general unit}
\xymatrix{
 Z' := Z \times^{H \times H} N\bs G/N  &Z \times^{ H \times  H} N\bs B/N  \simeq Z/\D B \ar[r]^-{p}\ar[l]_-{d} & Z/\Delta H
}
\end{equation}
where we write $\Delta H, \Delta B$ to emphasize the diagonal action. In forming the middle term $Z/\D B$, the action of $\D B$ factors through $\D H$.

\subsubsection{Unit adjunctions}\label{sss:unit adj}

Note $p^*: \Sh(Z/\D H) \to \Sh(Z/\D B)$ is an equivalence since the $\Delta B$-action on $Z$ factors through $\Delta H$ and the kernel is the unipotent $\Delta N$. Recall $\nu=\dim N$, so $-\nu = \dim \BB N$.
Define functors
\begin{eqnarray*}
u_{\ell}=p_{!} d^{*}[2\nu]: \Sh(Z')\to \Sh(Z/\Delta H)\\
u=d_{*} p^{*}=d_{!} p^{*}: \Sh(Z/\Delta H)\to \Sh(Z')\\
u_{r}=p_* d^!: \Sh(Z')\to \Sh(Z/\Delta H)
\end{eqnarray*}
Since $d$ is proper and $p$ is smooth of relative dimension $-\nu$, we have adjunctions $(u_{\ell},u )$ and $(u,u_{r})$.

Assume $Z$ is smooth and let $\L\subset T^{*}Z$ be a closed conic $H\times H$-invariant subset. Consider the full subcategory $\Sh_{\L}(Z/\Delta H)\subset \Sh(Z/\Delta H)$  of $\Delta H$-equivariant complexes $\cF$ on $Z$ with singular support satisfying $\ssupp(\cF)\subset \L$.
Consider as well the full subcategory $\Sh_{\L}(Z')\subset \Sh(Z')$  of $ H \times H$-equivariant complexes $\cF$ on $ Z \times N\bs G/N  $ with singular support satisfying $\ssupp(\cF)\subset \L \times \cN'$, where $\cN'\subset T^*(N\bs G/N)$ denotes the $N\times N$-reduction of $G\times \cN^* \subset G\times \frg^* \simeq T^*G$.

\begin{lemma}\label{l:unit ss} The functors $u_{\ell}$ and $u_{r}$ send $\Sh_{\L}(Z')$ to $\Sh_{\L}(Z/\Delta H)$, and the functor $u$ sends $\Sh_{\L}(Z/\Delta H)$ to $\Sh_{\L}(Z')$. In particular, if we let $u_{\L,\ell}: \Sh_{\L}(Z')\to\Sh_{\L}(Z/\D H)$ be the restriction of $u_{\ell}$, and similarly define $u_{\L}$ and $u_{\L,r}$, then there are  adjunctions $(u_{\L,\ell}, u_{\L})$ and $(u_{\L}, u_{\L,r})$.
\end{lemma}

\begin{proof}
First, we show $u$ respects the singular support conditions. Given $\cF\in Sh_{\L}(Z/\Delta H)$, viewed as a $\Delta H$-equivariant complex on $Z$, note $p^*\cF \simeq \cF \boxtimes \cF_0$ where $\cF_0$ denotes the constant sheaf on  $N\bs B/N$. Let $d_0:\BB B \to N\bs G/B$ be the closed embedding so that $d = \id \times d_0$. Then $\ssupp(d_!(\cF \boxtimes \cF_0)) = \ssupp(\cF) \times \ssupp(d_{0!}\cF_{0}) \subset  \L \times \cN'$ since $d_{0!}\cF_0$  is $H\times H$-bimonodromic hence has singular support in $\cN'$.

 Finally, we show $u_\ell, u_r$ respect the singular support conditions. 
Given $\cF\in Sh_{\L}(Z')$, view $\cF$ as an $H \times H$-equivariant complex on $Z \times N\bs G/N$. Using the estimate of singular support for pullbacks $d^{*}\cF$ and $d^{!}\cF$ as in \cite[Corollary 6.4.4, Remark 6.2.8]{kashiwaraSheavesManifoldsVolume1990}, we see that $ss(d^{*}\cF)$ and $ss(d^{!}\cF)$ are both contained in $\L$ when viewed as sheaves on $Z$. Therefore the same is true for $p_{!}d^{*}\cF$ and $p_{*}d^{!}\cF$ since $p$ is an $N$-gerbe. 
\end{proof}

\subsection{Descent for smooth stacks}
Let $Z$ be a smooth stack with a left $G\times G$-action. We will write the first $G$-action as a left action and turn the second $G$-action as a right action.

Let $\L\subset T^{*}Z$ be a closed $G\times G$-invariant  subset such that under the moment map $\mu: T^{*}Z\to \frg^{*} \times \frg^*$, we have $\mu(\L)\subset \cN^{*} \times \cN^*$ where $\cN^{*} \subset \frg^*$ denotes the nilcone in the dual to the Lie algebra.

\begin{ex}\label{ex: split}
In many situations of interest, one has $Z=Z_{-}\times Z_{+}$, $\L = \L_- \times \L_+$, where $Z_{\pm}$ are smooth stacks with $G$-action,
and
 $\L_{\pm}\subset T^{*}Z_{\pm}$ is a closed $G$-invariant  subset such that under the moment map $\mu_{\pm}: T^{*}Z_\pm\to \frg^{*}$, we have $\mu_{\pm}(\L_{\pm})\subset \cN^{*}$.
 In this case, we view the action of $G$ on $Z_{-}$ as a right action and the one on $Z_{+}$ as a left action.
 Then the $Z$-horocycle diagram \eqref{general hc} and $Z$-unit diagram \eqref{general unit} 
 take the form 
\begin{equation}\label{split hc}
\xymatrix{Z_- \times^G Z_+  & Z_- \times^B Z_+ \ar[r]^-{\d}\ar[l]_-{\e} & Z_-/N \times^H N\bs Z_+
}
\end{equation}
\begin{equation*}
\xymatrix{
 Z' = Z_- \times^{H} N\bs G/N \times^H Z_+  &Z_- \times^{  H} N\bs B/N \times^H Z_+ \simeq Z_- \times^B Z_+ \ar[r]^-{p}\ar[l]_-{d} & Z_- \times^H Z_+
}
\end{equation*}
\end{ex}

\begin{ex}\label{ex: adj}

A basic example is $Z = G$, with its natural $G\times G$-action, and $\L = G \times \cN^* \subset G\times \frg^*$.
\end{ex}


The category $\Sh_{\L}(N\bs Z/N)$ as a $\cH_{G}$-bimodule. The assumptions on $\L$ imply that any object of $\Sh_{\L}(N\bs Z/N)$  is $H\times H$-monodromic, therefore we  can also regard $\Sh_{\L}(N\bs Z/N)$ as a $\cH_{G}$-bimodule  in  $\cH_{H}=\Sh_{0}(H)$-bimodules
since  $\Sh_{0}(H) = \Mod(\End(e))\subset \cH_{G}$ is the full monoidal subcategory generated by the monoidal unit $e\in \cH_{G}$.  Note that
\begin{equation*}
\xymatrix{hh(\cH_{H}, \Sh_{\L}(N\bs Z/N))\ar[r]^-{\sim} & \Sh_{\L}((N\bs Z/N)/\D H).}
\end{equation*}

%
%
%
%

%
%
%

Here is the main technical theorem of this section.
A  generalization to ``nilpotent categorical bimodules" will appear in~\cite{nadlerCompatibilitiesAutomorphicGluing};
in particular the assumption here that $\L\subset T^{*}Z$ is Lagrangian is not necessary but allows the proof to call on the generalities of Proposition~\ref{p:cont 4 functors}.

\begin{theorem}\label{thm: geom descent}
Let $Z$ be a smooth stack with a $G\times G$-action.
Let $\L\subset T^{*}Z$ be a closed  $G\times G$-invariant  conic Lagrangian such that under the moment map $\mu: T^{*}Z\to \frg^{*} \times \frg^*$, we have $\mu(\L)\subset \cN^{*} \times \cN^*$.

 Then there is a canonical equivalence 
\begin{equation*}
\xymatrix{
hh(\cH_{G}, \Sh_{\L}(N\bs Z/N)) \ar[r]^-\sim &  \Sh_{\L}(Z/\Delta G)
}
\end{equation*}
such that the functor $ch$ defined in Section~\ref{sss:horo adj} factors as the composition
\begin{equation*}
\xymatrix{ch: \Sh_{\L}((N\bs Z/N)/\D H) \simeq  hh(\cH_{H}, \Sh_{\L}(N\bs Z/N))\ar[r] & hh(\cH_{G}, \Sh_{\L}(N\bs Z/N))\simeq\Sh_{\L}(Z/\D G).}
\end{equation*}

\end{theorem}

\begin{ex}
In the setting of Example~\ref{ex: split}, the theorem gives an equivalence
\begin{equation*}
\xymatrix{
\Sh_{\L_{+}}(Z_{-}/N)\ot_{\cH_{G}}\Sh_{\L_{-}}(N \bs Z_{+})\ar[r]^-\sim &  \Sh_{\L}(Z_- \times^G Z_+).
}
\end{equation*}
%
\end{ex}

\begin{ex}
In the setting of Example~\ref{ex: adj}, the theorem gives the equivalence
\begin{equation*}
\xymatrix{
hh(\cH_{G}) \ar[r]^-\sim &  \Sh_{\cN}(G/G)
}
\end{equation*}
that we stated in Theorem~\ref{thm:CS G}.
%
\end{ex}

\begin{remark} In \cite{nadlerCompatibilitiesAutomorphicGluing} we will prove an abstract version of Theorem \ref{thm: geom descent} where $\Sh_{\L}(Z)$ is replaced with a $\Sh(G)$-bimodule category satisfying a certain nilpotence condition reflecting that $\L$ maps to $\cN^{*}\times \cN^{*}$ under the moment map. In particular, one can remove the assumption that $\L$ is a Lagrangian in Theorem \ref{thm: geom descent}.
\end{remark}

%


 To prove Theorem~\ref{thm: geom descent}, we will identify each term of the (relative) Hochschild complex computing $hh(\cH_{G}, \Sh_{\L}(N\bs Z/N))$ as sheaves on a certain space, and augment the Hochschild complex by adding the term $\Sh_{\L}(Z/\D G)$. Finally we will apply Lurie's criterion \cite[Corollary 4.7.6.3]{lurieHigherAlgebra2012} to verify that the augmented Hochschild complex is a colimit diagram.

\subsubsection{Nonlinear diagrams}
To start, we present some general patterns for constructing the (relative) Hochschild complex for spaces, following~\cite{ben-zviNonlinearTraces}.  In the subsequent Section~\ref{sss:case of groups}, we specialize to the case we will use in the proof of Theorem~\ref{thm: geom descent}.

Let $K$ be an algebraic group. Let $\Corr^{K}$ be the category of stacks with  $K$-action with morphisms given by $K$-equivariant correspondences. 

Let $\Corr^{K\times K}$ be the monoidal category of stacks with  $K \times K$-action with morphisms given by 
$K \times K$-equivariant correspondences. The monoidal structure on $\Corr^{K\times K}$ is given by $U \star V:=U \times^{K} V$, where the quotient of $K$ uses the second action of $K$ on $U$ and the first action of $K$ on $V$. Note $K$ with its regular $K \times K$-action is the monoidal unit in $\Corr^{K\times K}$.

Note that $\Corr^{K}$ is naturally a module category for $\Corr^{K\times K}$: for $U\in\Corr^{K\times K}$ and $V\in \Corr^{K}$, we define the action of $U$ on $V$ to be $U\star V:=U \times^{K} V\in \Corr^{K}$ (quotient using the second action of $K$ on $U$ and the action of $K$ on $V$; the first action of $K$ on $U$ induces a $K$-action on $U\star V$).

Consider a diagram of stacks with Cartesian square
\beq\label{eq:abstract setup}
\xymatrix{
  X_0 \ar[d]_-{p_0} \ar[r]^-{b'} &  X \ar[r]^-f  \ar[d]^-p & Y \\
\pt \ar[r]^-b & \BB K & 
}
\eeq
In other words, we are given a stack $X_0$ with $K$-action and a $K$-invariant map $X_0\to Y$.




Given \eqref{eq:abstract setup}, observe that $A =  X_0 \times_Y  X_0$ is naturally an algebra object in $\Corr^{K\times K}$ with product given by the correspondence
\beq\label{eq:alg mult}
\xymatrix{
A \star A = (X_0\times_Y  X_0) \times^K ( X_0\times_Y  X_0)  & \ar[l]_-{\d}
 X_0\times_Y   X \times_Y  X_0 \ar[r]^-{\pi_f} &  X_0 \times_Y  X_0 = A.
}
\eeq
Here $\d$ is induced by the natural map $X = X_0/K \to  X_0 \times^K  X_0$ of the middle factors (which in turn is induced by the diagonal map $X_0\to X_0\times X_0$), and $\pi_f$ is induced by the projection of the middle factor $f: X\to Y$. The unit of $A$ is given by the correspondence
\beq\label{eq:alg unit}
\xymatrix{
K = \pt \times_{\BB K} \pt  & \ar[l]_-{\e_p}
\pt \times_{\BB K}  X \times_{\BB K} \pt \simeq X_0\times_X X_0 \ar[r]^-{\d_f} &  X_0\times_Y  X_0 = A.
}
\eeq
Here $\e_p$ is induced by $p : X\to \BB K$, and $\d_f$ is induced by  $f:X\to  Y$. Note swapping the factors of $A$ gives an equivalence with its monoidal opposite.

Given any map of stacks $q:W\to Y$, observe that $W_{0}:= X_0\times_Y W$ is naturally an $A$-module object in  $\Corr^{K}$ with action given by the correspondence
\beq\label{eq:mod mult}
\xymatrix{
A \star W_{0}= ( X_0\times_Y  X_0) \times^K (  X_0\times_Y W)  & \ar[l]_-{\delta}
 X_0\times_Y  X \times_Y W \ar[r]^-{\pi_f} &   X_{0}\times_Y W=  W_{0}
}
\eeq
where $\delta$ and $\pi_{f}$ are defined in the same as in \eqref{eq:alg mult}.

Similarly, given any map of stacks $q_1 \times q_2:W\to Y \times Y$,  $W_{0}:=X_0\times_Y W \times_{Y} X_0$ is naturally an $A$-bimodule object in $\Corr^{K\times K}$.

The following is elementary to check; we leave further details to the reader. We refer to Section~\ref{ss:hh} for the terminology of  relative Hochschild complex, which we borrow here for the monoidal category $\Corr^{K\times K}$.

\begin{lemma}\label{l:aug hoch}
Given any map of stacks $q_1 \times q_2:W\to Y \times Y$, we have an augmented simplicial object $B(A, W)_\bullet$ in $\Corr^{K\times K}$ such that:
 \begin{enumerate}
 \item The underlying simplicial object of $B(A, W)_\bullet$ is the relative Hochschild complex for the $A$-bimodule $W_{0}$ in $\Corr^{K\times K}$ (relative to the unit object $K\in \Corr^{K\times K}$ which is also an algebra object, and the unit map $K\to A$ given in \eqref{eq:alg unit} is a map of algebras):
  \begin{equation*}
 \xymatrix{
\cdots \ar@<1.5ex>[r] \ar@<0ex>[r] \ar@<-1.5ex>[r]  &  \ar@<0.75ex>[l] \ar@<-0.75ex>[l]      W_0 \times^{K \times K} A \ar@<0.75ex>[r] \ar@<-0.75ex>[r] 
 & \ar@<0ex>[l]   W_0/\Delta K 
 }
 \end{equation*}
  
 \item The augmentation map $B(A, W)_0 \to B(A, W)_{-1}$ is given by the correspondence
 \begin{equation*}
 \xymatrix{
  W_0/\Delta K =   ( X_0\times_Y W_{} \times_Y  X_0)/\Delta K & \ar[l]_-\delta W_{} \times_{Y \times Y} X
  \ar[r]^-{\pi_f} & W_{} \times_{Y \times Y} Y
 }
 \end{equation*}
where $\delta$ is induced by the natural map $ X = X_0/K \to  X_0 \times^K  X_0$, and $\pi_f$ is induced by $f:X\to Y$. 
 
\end{enumerate}
  
\end{lemma}

\subsubsection{Our case of interest}\label{sss:case of groups} We specialize the preceding constructions when the initial diagram
\eqref{eq:abstract setup} takes the form
\begin{equation*}
\xymatrix{
\BB N \ar[d]_-{p_0} \ar[r]^-{b'} & \BB B \ar[r]^-f  \ar[d]^-p & \BB G \\
\pt \ar[r]^-b & \BB H & 
}
\end{equation*}
where the maps are the embeddings $U\subset B\subset G$ and projection $B\surj H$. So here $K$ is simply the universal Cartan $H$.

Inside of $\Corr^{H\times H}$, we have the algebra $A = \BB N \times_{\BB G} \BB N \simeq N\bs G/N$ with the multiplication diagram~\eqref{eq:alg mult} given by usual convolution
\begin{equation*}
\xymatrix{
A \star A = ( N\bs G/N) \times^H ( N\bs G/N)  & \ar[l]_-{\d}
 N\bs G \times^B G/N  \ar[r]^-{\pi_f} &  N\bs G/N  = A
}
\end{equation*}
and the unit diagram \eqref{eq:alg unit} taking the form
\beq\label{H alg unit}
\xymatrix{
H  & \ar[l]_-{\e_p}
 N\bs B/N  \ar[r]^-{\d_f} &  N\bs G/ N = A.
}
\eeq
Given a stack $Z$ with $G$-action, applying the general discussion about 
$A$-modules to the induced map $q: W=G\bs Z\to \BB G$, then we have $W_0 = N\bs Z\in \Corr^{H}$ is naturally an $A$-module with action \eqref{eq:mod mult}  given by usual convolution
\begin{equation*}
\xymatrix{
A \star   W_0 = ( N\bs G/N) \times^H ( N\bs Z)  & \ar[l]_-{\delta}
N\bs G \times^B Z\ar[r]^-{\pi_f} &   N\bs Z=  W_0.
}
\end{equation*}
Similarly, given a stack $Z$ with $G\times G$-action, applying the general discussion about $A$-bimodules to the map $q_1 \times q_2:W=G\bs Z/G \to \BB G \times \BB G$, then 
 $W_0 =  N\bs Z/N \in \Corr^{H\times H}$ is naturally an $A$-bimodule object.

Now Lemma~\ref{l:aug hoch} takes the following form.

\begin{lemma}\label{l:aug hoch gps}
Given a stack $Z$ with $G\times G$-action, we have an augmented simplicial object $B(A, Z)_\bullet$ (in the notation of Lemma \ref{l:aug hoch} it should be called $B(A,W)_{\bu}$, where $W=G\bs Z/G$) in $\Corr^{H\times H}$ such that:
 \begin{enumerate}
 
 \item The underlying simplicial object of $B(A, Z)_\bullet$ is the relative Hochschild complex of the $A$-bimodule $N\bs Z/N$ (relative to the algebra map $H\to A$ in $\Corr^{H\times H}$ given by the unit diagram \eqref{H alg unit}):
  \begin{equation*}
 \xymatrix{
\cdots \ar@<1.5ex>[r] \ar@<0ex>[r] \ar@<-1.5ex>[r]  &  \ar@<0.75ex>[l] \ar@<-0.75ex>[l]      (N\bs Z/N) \times^{H \times H} (N\bs G/N) \ar@<0.75ex>[r] \ar@<-0.75ex>[r] 
 & \ar@<0ex>[l]   (N\bs Z/N) /\Delta H.
 }
 \end{equation*}
  
 \item The augmentation map $B(A, Z)_0 \to B(A, Z)_{-1}$ is given by the horocycle correspondence
 \begin{equation*}
 \xymatrix{
 (N\bs Z/N) /\Delta H & \ar[l]_-\delta Z/\Delta B
  \ar[r]^-{\pi_f} & Z/\Delta G.
 }
 \end{equation*}
 \end{enumerate}
  
\end{lemma}


\subsubsection{Augmented Hochschild complex} 
 For the next step towards the proof of  Theorem~\ref{thm: geom descent}, we shall take the categories of sheaves termwise for the augmented simplicial object $B(A, Z)_{\bullet}$ in $\Corr^{H\times H}$ provided by Lemma~\ref{l:aug hoch gps}, and impose singular support conditions to obtain the augmented Hochschild complex for the $\cH_{G}$-bimodule $\Sh_{\L}(N\bs Z/N)$.

Consider the monoidal category $\cH_{H}=\Sh_{0}(H)$. For any $X\in \Corr^{H\times H}$, $\Sh(X)$ is a bimodule for $\cH_{H}$. Thanks to \cite{gaitsgoryStudyDerivedAlgebraic2019}, passing to categories of sheaves gives a monoidal functor
\begin{equation*}
\Corr^{H\times H}\to \Bimod_{\cH_{H}}(\St)
\end{equation*}
where the target is the 2-category of stable presentable $\oo$-categories that are $\cH_{H}$-bimodules (and the monoidal structure is tensor product over the middle copy of $\cH_{H}$). For a morphism from $X$ to $Y$ in $\Corr^{H\times H}$, i.e., a $H\times H$-equivariant correspondence 
\begin{equation*}
\xymatrix{ X & C\ar[l]_{p}\ar[r]^{q} & Y
}
\end{equation*}
the functor $\Sh(X)\to \Sh(Y)$ is given by $q_{!}p^{*}$. Passing to the categories of all sheaves termwise for $B(A, Z)_{\bullet}$. We obtain an augmented simplicial object $\Sh(B(A, Z))_{\bullet}$ in $\Bimod_{\cH_{H}}(\St)$.

Now we impose singular support conditions.  Let $\L_{-1} \subset T^*(Z/\Delta G)$ be the $\Delta G$-reduction of $\L$. For $n\ge0$, $B(A, Z)_{n}$ can be written as the quotient
\begin{equation*}
B(A, Z)_{n}\simeq (Z\times G^{n})/(B\times_{H}B)^{n+1}
\end{equation*}
where for $n\ge1$, the $i$th factor ($0\le i\le n$) of $B\times_{H}B$ acts on $Z\times G^{n}$ by
\begin{equation*}
(b,b')\cdot (z,g_{1},\cdots, g_{n})=\begin{cases}(zb, b'^{-1}g_{1},\cdots, g_{n}) & i=0,\\
(z, \cdots, g_{i}b, b'^{-1}g_{i+1},\cdots, g_{n}) & 1\le i\le n-1 \\
(b'^{-1}z, g_{1},\cdots, g_{n-1},  g_{n}b) & i=n.
 \end{cases}
\end{equation*}
For $n=0$ the action of $B\times_{H}B$ acts on $Z$ is by $(b,b')\cdot z=b'^{-1}zb$.

For $n\geq 0$,  set $\L_{n}  \subset T^* B(A, Z)_{n}\simeq T^*((Z \times G^{n})/(B \times_H B)^{n+1})$ be the reduction of $\L \times (G \times \cN^*)^n$.

\begin{lemma}\label{l:Cbu} The augmented simplicial category $\Sh(B(A, Z))_{\bullet}$ restricts to an augmented simplicial object $\cC_{\bu}$ in $\Bimod_{\cH_{H}}(\St^{L}_{k})$ with terms $\cC_{n}=\Sh_{\L_{n}}(B(A,Z)_{n})$ for $n\ge-1$.  Moreover:
\begin{enumerate}
\item Let $\cM=\Sh_{\L}(N\bs Z/N)$. Then the underlying simplicial object of of $\cC_{\bu}$ is equivalent to the relative Hochschild complex $B(\cH_{G}, \cM)_\bullet$:
  \begin{equation*}
 \xymatrix{
\cdots \ar@<1.5ex>[r] \ar@<0ex>[r] \ar@<-1.5ex>[r]  &  \ar@<0.75ex>[l] \ar@<-0.75ex>[l]   
  \cM \otimes_{\cH_{H} \ot \cH_{H}^{op}} \cH_{G} 
\ar@<0.75ex>[r] \ar@<-0.75ex>[r]  & \ar@<0ex>[l]  
    \cM \ot_{\cH_{H} \ot \cH_{H}^{op}} \cH_{H}.
 }
 \end{equation*}
\item The augmentation map $\cC_{0} \to \cC_{-1}$ is given by the transform
 \begin{equation*}
 \xymatrix{
ch:\cM \ot_{\cH_{H} \ot \cH_{H}^{op}} \cH_{H}\simeq \Sh_{\L_{0}}((N\bs Z/N)/\D H) \ar[r] &  \Sh_{\L_{-1}}(Z/\Delta G)
 }
 \end{equation*}
 associated to the horocycle transform construction \eqref{general hc} applied to $Z$.
\end{enumerate}
\end{lemma}
\begin{proof} The description of $\cC_{n}$ in terms of $\cM$ and $\cH_{G}$ follows from the categorical K\"unneth formula for sheaves with prescribed singular support conditions on twisted products $X_{1}\times^{H}X_{2}$, see \cite[Lemma A.4.3]{nadlerAutomorphicGluingFunctor}. It remains to check that the prescribed singular supports are respected by the given functors in $\Sh(B(A, Z))_{\bullet}$. 

For $n\geq 0$, and the injection $\ph: [n-1]\to [n]$ whose image misses $i$, the corresponding face map of $B(A, Z)_{\bullet}$ is given by the functor $ch$ resulting from the horocycle transform construction \eqref{general hc} applied to $Z^i_n = (Z \times G^n)/B^i_n$. Here  
$B^i_n \subset (B\times_{H}B)^{n+1}$ is  the subgroup 
where we replace the $i$th  factor by the trivial group and  keep the other factors unchanged. The $G\times G$-action on $Z^i_n$ is the natural action along where the $i$th factor of $(B\times_{H}B)^{n+1}$ originally acted. 
By Lemma~\ref{l:hc ss},  the associated horocycle functors, in particular the face map  $ch$,  respect the prescribed singular support.
%

Similarly, for $n\geq 0$, and the surjection $\ph: [n+1]\to [n]$ that identifies $i$ and $i+1$, 
 the corresponding degeneracy map of $B(A, Z)_{\bullet}$  respects the singular support as follow. Observe   
 the degeneracy map is given by the unit functor $u$ (see Section~\ref{sss:unit adj})
  resulting from the unit transform construction \eqref{general unit} applied to $Z_{n, i} = (Z \times G^n)/B_{n, i}$. 
  Here  
$B_{n, i} \subset (B\times_{H}B)^{n+1}$ is  the subgroup 
where we replace the $i$th  factor by the group $N \times N$ and  keep the other factors unchanged. The $H\times H$-action on $Z_{n, i}$ is the natural action along where the $i$th factor of $(B\times_{H}B)^{n+1}$ originally acted.
%
%
By Lemma~\ref{l:unit ss},  
the associated unit functors, in particular the degeneracy map   $u$,  respect the prescribed singular support.
\end{proof}






\sss{Finish of the proof of Theorem~\ref{thm: geom descent}}\label{sss:fin descent}
 To prove Theorem~\ref{thm: geom descent}, it remains to prove that the the augmented simplicial object $\cC_{\bu}$ exhibits $\cC_{-1}=\Sh_{\L}(Z/\D G)$ as the colimit of the underlying simplicial object of $\cC_{\bu}$, i.e., the Hochschild complex $B(\cH_{G}, \cM)_{\bu}$. Equivalently, letting $\cC^{\bu}$  be the augmented cosimplicial object obtained from $\cC_{\bu}$ by passing to right adjoints (note these are available by
Lemmas~\ref{l:hc ss} and \ref{l:unit ss}), it suffices to show that $\cC^{\bu}$ exhibits $\cC^{-1}$ as the limit of the underlying cosimplicial object $\{\cC^{n}\}_{n\ge0}$. For this it suffices to check that $\cC^\bullet$ satisfies the following strong form of the criteria of \cite[Corollary 4.7.6.3]{lurieHigherAlgebra2012}: 
\begin{enumerate}
\item The augmentation map $d_{-1}^0 = hc_*:\cC^{-1} \to \cC^0$ is: (a) conservative
and (b) continuous, i.e. preserves colimits; 
\item The following commutative squares are left adjointable for any order-preserving map $\alpha:[m] \to [n]$ (where $m,n\ge-1$)
\begin{equation*}
\xymatrix{
\ar[d]_-\alpha \cC^m \ar[r]^-{d_m^0} & \cC^{m+1}\ar[d]^-{\alpha'}\\
\cC^n \ar[r]^-{d_n^0} & \cC^{n+1}
}
\end{equation*}
Here $d_{m}^{0}$ is the inclusion $[m]\to [m+1]$ that whose image misses $0$; $d_{n}^{0}$ is defined similarly; and $\a': [m+1]\to [n+1]$ is the map defined by $\a'(0)=0$ and $\a'(i+1)=\a(i)+1$ for $i\in [m]$. Left  adjointability of the above square means the face maps $d_m^0$, $d_n^0$ admit respective left adjoints $(d_m^0)^\ell$, $(d_n^0)^\ell$ (which in our case are already given by the construction of $d_m^0$, $d_n^0$ as right adjoints), and the associated base change map
is an equivalence
\begin{equation*}
\xymatrix{
(d_n^0)^\ell \circ \alpha' \ar[r]^-\sim & \alpha\circ (d_m^0)^\ell.
}
\end{equation*}


\end{enumerate}

(1a) First, we check $d_{-1}^0  = hc_{*}$ is conservative.  It suffices to show that  $ch\circ hc_{*}$  contains the identity functor as a direct summand.  This is essentially \cite[Theorem 3.6]{mirkovicCharacteristicVarietiesCharacter1988}. 

We use notation from the diagram \eqref{general hc}. By definition $ch\c hc_{*}=\e_{!}\d^{*}\d_{*}\e^{!}$. The fiber square along $\d$ can be identified with
\begin{equation*}
\xymatrix{     Z/{\D B} & (Z \times N)/\D {B} \ar[r]^-{p_{1}}\ar[l]_-{a_{1}} &   Z/\D B.
}
\end{equation*}
Here  the $\D B$-action on $N$ is by conjugation. The map $a_{1}$ is the action map of $N$ on $Z$ via $N \times \{1\} \to G \times G$,  and $p_{1}$ is the projection.  Since $\d$ is smooth of relative dimension $\nu=\dim N$, $\d^{*}\d_{*}\cong p_{1*}a_{1}^{*}\cong p_{1*}a_{1}^{!}[-2\nu]$. Hence
\begin{equation*}
ch\c hc_{*}\cong\e_{*}p_{1*}a_{1}^{!}\e^{!}[-2\nu]=p'_{1*}a'^{!}_{1}[-2\nu]
\end{equation*}
where $p'_{1}=\e\c p_{1}, a'_{1}=\e\c a_{1}$. We have a commutative  diagram
\begin{equation*}
\xymatrix{  & (Z \times N)/\D {B}  \ar[d]^-{\pi}\ar[dr]^{p'_{1}}\ar[dl]_-{a'_{1}} & \\
Z/\D G  &  (Z \times G)/\D G   \ar[r]^-{\pi_{1}}\ar[l]_-{\a_{1}}  &    Z/\D B} 
\end{equation*}
 Here the $\D G$-action (resp. $\D B$-action) on $G$ (resp. $N$) is by conjugation,  $\pi_{1}$ is projection, and $\a_{1}$ is the action map of $G$ on $Z$ via $G \times \{1\} \to G \times G$.
The map $\pi$ is the base change of the Springer resolution $ \frac{N}{\Ad(B)}\to \frac{G}{\Ad(G)}$. Hence for $\cF\in \Sh_{\L}(Z/\D G)$, we have
\begin{equation*}
ch( hc_{*}(\cF))\cong p'_{1*}a'^{!}_{1}\cF[-2\nu]\cong \pi_{1*}\pi_{*}\pi^{!}\a^{!}_{1}\cF[-2\nu]\cong \pi_{1*}\uHom(\pi_{!}k, \a^{!}_{1}\cF)[-2\nu].
\end{equation*}
The Springer sheaf contains the skyscraper sheaf at $1\in G$ as a direct summand, hence $\pi_{!}k$ contains $i_{*}k[-2\nu]$ as a direct summand, where $i:Z/\D G\incl ( Z\times G)/\D G$ corresponds to the inclusion of $1$ into $G$. Therefore $ch( hc_{*}(\cF))$ contains as a direct summand
\begin{equation*}
\pi_{1*}\uHom(i_{*}k[-2\nu], \a^{!}\cF)[-2\nu]\cong \pi_{1*}i_{*}i^{!}\a^{!}\cF\cong \cF.
\end{equation*}
We have shown that  $ch\circ hc_{*}$  contains the identity functor as a direct summand. 

(1b) follows from Proposition \ref{p:cont 4 functors}. 

(2) We will check the required left adjointability for the augmentation map $\alpha = d_{-1}^0:[-1] \to [0]$; the verification for other maps is similar.

Consider the relevant categories and functors
\begin{equation*}
\xymatrix{\Sh_{\L}(Z/\D G) \ar@<-1ex>[r]_-{hc_{*}} & \Sh_{\L}((N\bs Z/N)/\D H)\ar@<-1ex>[l]_-{ch}\ar@<1ex>[r]^-{hc_{*0}} &  \Sh_{\L}(N\bs Z/N) \otimes_{\cH_{H} \otimes \cH_{H}^{op}} \cH_{G} \ar@<1ex>[l]^-{ch_{1}}.
}
\end{equation*}
Here $ch_{1}$ is the $\cH_{G}$-action on $  \Sh_{\L}((N\bs Z/N)/\D H)$ induced by the right $G$-action on $Z$, and $hc_{*0}$ is right adjoint to the $\cH_{G}$-action on $  \Sh_{\L_{}}((N\bs Z/N)/\D H)$ induced by the left $G$-action on $Z$.
 
We seek to show the natural adjunction transformation
\begin{equation}\label{bc tau}
\xymatrix{
\t: ch_{1}\c hc_{*0}\ar[r] &  hc_{*}\c ch
}\end{equation}
 is an equivalence.  Indeed, by proper base change, $hc_{*}\c ch=\d_{*}\e^{!}\e_{!}\d^{*}$ 
can be identified with the functor $c_{*}p^{!}[-2\nu]$ constructed from the diagram
\begin{equation}\label{act1}
\xymatrix{(N\bs Z/N)/\D H & \D B\bs (Z\times  G/B) \ar[r]^-{c}\ar[l]_-{p} & (N\bs Z/N)/\D H}
\end{equation}
where in the middle term, $b\in \D B$ acts by $b(z,gB)=(bzb^{-1}, bgB)$, and the maps $p$ and $c$ are defined by $p(z,g)=z$ and $c(z, g)=gzg^{-1}$. 

On the other hand,  $ch_{1}\c hc_{*0}$ can be identified with the functor $m_{1*}m_{0}^{!}[-2\nu]$ constructed from the diagram
\begin{equation}\label{act2}
\xymatrix{(N\bs Z/N)/\D H & \frac{Z\times^{B}G}{\D B}\ar[r]^-{m_{1}}\ar[l]_-{m_{0}} & (N\bs Z/N)/\D H}
\end{equation}
where in the middle term, $b\in \D B$ acts by $b(z,g)=(bz, gb^{-1})$, and the maps $m_{0}$ and $m_{1}$ are defined by $m_{0}(z,g)=gz$ and $m_{1}(z, g)=zg$. 

We have an isomorphism between the two diagrams \eqref{act1} and \eqref{act2} that is the identity on $(N\bs Z/N)/\D H$ and on the middle terms it takes the form $(z,gB)\mapsto (g^{-1}z,g)\in \frac{Z\times^{B}G}{\D B}$. This isomorphism identifies $c_{*}p^{!}[-2\nu]$ with $m_{1*}m_{0}^{!}[-2\nu]$. One checks that $\tau$ is the composition
\begin{equation*}
ch_{1}\c hc_{*0}\simeq m_{1*}m_{0}^{!}[-2\nu] \simeq c_{*}p^{!}[-2\nu]\simeq hc_{*}\c ch.
\end{equation*}
Therefore $\t$ is an equivalence. 
This concludes the proof of Theorem~\ref{thm: geom descent}.

\subsection{Singular and ind-version} We will generalize Theorem~\ref{thm: geom descent} to the situation of stratified ind-stacks (Theorem~\ref{thm: geom descent ind}).

\sss{Setup} Let $Z$ be an ind-stack equipped with a partition into smooth locally closed substacks $Z=\sqcup_{\a\in P} Z_{\a}^{\c}$ indexed by a poset $P$. For $\a\in P$, assume $\{\b\in P|\b<\a\}$ is finite and  $Z_{\a}:=\cup_{\b\le \a} Z_{\b}^{\c}$ is closed in $Z$. Let $i_{\a}^{\c}: Z_{\a}^{\c}\to Z$ be the inclusion.

Assume $Z$ has a $G\times G$-action preserving each $Z_{\a}\subset Z$. Let $\L_{\a}\subset T^{*}Z_{\a}^\circ$ be a closed $G\times G$-invariant  conic Lagrangian such that under the moment map $\mu: T^{*}Z_{\a}^\circ\to \frg^{*} \times \frg^*$, we have $\mu(\L_{\a})\subset \cN^{*} \times \cN^*$. 

Assume for $\b < \a\in P$, the composition $(i^\circ_{\b})^* (i^\circ_{\a})_*: \Sh(Z^\circ_{\a}) \to \Sh(Z_{\b}^\circ)$ takes $\Sh_{\L_{\a}}(Z^\circ_{\a})$ to $\Sh_{\L_{\b}}(Z_{\b}^\circ)$. Define $\Sh_{\L}(Z)$ to be the full subcategory of objects $\cF$ such that $(i_{\a}^\circ)^*\cF\in  \Sh_{\L_{\a}}(Z^\circ_{\a})$, for all~$\a\in P$. 

 For any $Q\subset P$ that is locally down-closed, i.e., $Q=Q_{1}\setminus Q_{2}$ where $Q_{1}$ and $Q_{2}$ are down-closed, let $Z_{Q}=\cup_{\a\in Q}Z^{\c}_{\a}$ be the corresponding locally closed substack of $Z$, we can define the full subcategory $\Sh_{\L}(Z_{Q})\subset \Sh(Z_{Q})$ in the same way by requiring objects to have image under $(i^{\c}_{\a})^{*}$ lying in $\Sh_{\L_{\a}}(Z^\circ_{\a})$.  Now for a locally down-closed $Q$ and a downclosed subset $Q'\subset Q$ with complement $Q''=Q\setminus Q'$, the assumptions guarantee that the usual functors in the recollement diagram of $\Sh(Z_{Q}), \Sh(Z_{Q'})$ and $\Sh(Z_{Q''})$ restrict to a recollement diagram
\begin{equation*}
\xymatrix{
\Sh_{\L}(Z_{Q''}) \ar@<3ex>[r]^{j_!} \ar@<-3ex>[r]^{j_*} &  \ar[l]_{j^! = j^*}  \Sh_{\L}(Z_{Q})\ar@<3ex>[r]^{i^*}  \ar@<-3ex>[r]^{i^!}  &   \ar[l]_{i_* = i_!}   \Sh_{\L}(Z_{Q'})}
\end{equation*}
where $i: Z_{Q'}\incl Z_{Q}$ and  $j:Z_{Q''}\incl Z_{Q}$ are the closed and open inclusions. From this we see that $\Sh_{\L}(Z)$ admits a {\em stratification} indexed by $P$ in the sense of Section~\ref{sss:strat}, with strata categories  $\Sh_{\L_{\a}}(Z^{\c}_{\a})$ for $\a\in P$.

For any subgroup $G'\subset G\times G$,  let $\Sh_{\L}(Z/G')\subset \Sh(Z/G')$  denote  the full subcategory 
 of $G'$-equivariant complexes $\cF$ on $Z$ whose underlying object lies in $\Sh_{\L}(Z)$.
The assumptions on $\L$ imply that any object of $\Sh_{\L}(N\bs Z/N)$  is $H\times H$-monodromic. The $G\times G$ actions on $Z$ equip $\Sh_{\L}(N\bs Z/N)$ with a $\cH_{G}$-bimodule structure.

%
%
%
%
%
%

\begin{theorem}\label{thm: geom descent ind} With the above setup, there is a canonical equivalence of stable $\oo$-categories
\begin{equation*}
\xymatrix{
hh(\cH_{G}, \Sh_{\L}(N\bs Z/N)) \ar[r]^-\sim &  \Sh_{\L}(Z/\Delta G)
}
\end{equation*}
such that the functor $ch$ defined in Section~\ref{sss:horo adj} factors as the composition
\begin{equation*}
\xymatrix{ch: \Sh_{\L}((N\bs Z/N)/\D H) \simeq  hh(\cH_{H}, \Sh_{\L}(N\bs Z/N))\ar[r] & hh(\cH_{G}, \Sh_{\L}(N\bs Z/N))\simeq\Sh_{\L}(Z/\D G).}
\end{equation*}
\end{theorem}
\begin{proof}

 We explain that the steps in the proof of Theorem \ref{thm: geom descent} can be made to work for the ind-stack $Z$ with slight modifications. The augmented simplicial diagram in Lemma \ref{l:aug hoch gps} makes sense for the ind-stack $Z$. Now for each term $B(A,Z)_{n}=(Z\times G^{n})/(B\times_{H}B)^{n+1}$ (where $n\ge0$), we assign the full subcategory category $\cC_{n}\subset \Sh(B(A,Z)_{n})$ consisting of objects $\cF$ whose pullback to $Z^{\c}_{\a}\times G^{n}$ has singular support contained in $\L_{\a}\times (G\times \cN^{*})^{n}$, for all $i\in P$. Let $\cC_{-1}=\Sh_{\L}(Z/\Delta G)$ as is already defined. 

With this definition of $\cC_{\bu}\in \Bimod_{\cH_{H}}(\St^{L}_{k})$, we claim that the statement of Lemma \ref{l:Cbu} holds. We need to check the maps in the augmented simplicial object $\Sh(B(A,Z)_{\bu})$ preserve the subcategories $\cC_{\bu}$. 

For $\a\in P$, applying Lemma \ref{l:Cbu} to $Z_{\a}^{\c}$ together with the Lagrangian $\L_{\a}$, we get an augmented simplicial object $\cC_{\a,\bu}$ in $\Bimod_{\cH_{H}}(\St^{L}_{k})$ whose terms are $\Sh_{\L_{\a,n}}(B(A,Z^{\c}_{\a})_{n})$. Let $i^{\c}_{\a, n}: B(A,Z^{\c}_{\a})_{n}\incl B(A,Z)_{n}$ be the locally closed embedding induced by $i^{\c}_{\a}$. We claim that the functors $i^{\c}_{\a, n!}: \cC_{\a,n}\to \Sh(B(A,Z)_{n})$ induce a functor of augmented simplicial objects
\begin{equation*}
i^{\c}_{\a,\bu!}: \cC_{\a,\bu}\to \Sh(B(A,Z)_{\bu}).
\end{equation*}
Indeed, for an injection $\ph: [n-1]\to [n]$, the corresponding face maps $ch_{\a,\ph}: \cC_{\a,n}\to \cC_{\a,n-1}$ and $ch_{\ph}: \Sh(B(A,Z)_{n})\to \Sh(B(A,Z)_{n-1})$ are given by special cases of the $ch$ functor defined in Section~\ref{sss:horo adj}. The isomorphism
\begin{equation*}
ch_{\ph}\c i^{\c}_{\a,n!}\simeq i^{\c}_{\a,n-1!}\c ch_{\a,\ph}
\end{equation*}
follows from proper base change. For a surjection $\ph: [n+1]\to[n]$,  the corresponding degeneracy maps $u_{\a,\ph}: \cC_{\a,n}\to \cC_{\a,n+1}$ and $u_{\ph}: \Sh(B(A,Z)_{n})\to \Sh(B(A,Z)_{n+1})$ are given by special cases of the unit functor $u$ defined in Section~\ref{sss:unit adj}. The isomorphism
\begin{equation*}
u_{\ph}\c i^{\c}_{\a,n!}\simeq i^{\c}_{\a,n+1!}\c u_{\a,\ph}
\end{equation*}
again follows from proper base change.

Now observe that for $n\ge-1$, $\cC_{n}\subset \Sh(B(A,Z)_{n})$ is generated by the images of $i^{\c}_{\a,n!}$ (for $\a\in P$) under colimits. We have checked above that the face and degeneracy maps for $\Sh(B(A,Z)_{\bu})$ preserve the images of $i^{\c}_{\a,\bu!}$ for fixed $\a\in P$. Since the face and degeneracy maps are continuous (they are left adjoints), they preserve $\cC_{\bu}$, therefore we get an augmented simplicial object $\cC_{\bu}$, and the analog of Lemma \ref{l:Cbu} holds.

The last step is to check that the augmented cosimplicial object $\cC^{\bu}$, obtained from $\cC_{\bu}$ by passing to right adjoints, satisfies Lurie's criterion \cite[Corollary 4.7.6.3]{lurieHigherAlgebra2012} for a limit diagram. The same argument as in Section~\ref{sss:fin descent} works for the current situation without change.
\end{proof}

\section{Harder-Narasimhan subcategories}\label{s:hn subcats}


This main result of this section is Theorem~\ref{thm:main in text} giving a  recollement structure on the cocenter $hh(\cH_\cG)$ of the universal affine Hecke category $\cH_\cG$. The properties of this recollement mirror those of the recollement structure on the Betti Langlands automorphic category $\Sh_\cN(\Bun_G(E))$ for a genus one curve $E$ induced by the Harder-Narasimhan stratification of $\Bun_G(E)$. Most basically, the filtrations of both recollements are indexed  by Newton points $\NP\subset \xcoch(T)_{\QQ}^{+}$ (see Section~\ref{sss:np}).
In a sequel, we will prove that $hh(\cH_\cG)$ is equivalent to $\Sh_\cN(\Bun_G(E))$ so that
the recollements match.  In this paper, we will focus on the specific consequence stated in
Theorem~\ref{th:ff}
 that the associated graded for the minimum index $0\in \NP$ embeds fully faithfully in  $hh(\cH_\cG)$.

\subsection{Combinatorial pieces} 


Our starting point for the analysis of $hh(\cH_\cG)$ is the colimit description of  Corollary~\ref{cor:coequal J geom}:
\begin{equation*}
\xymatrix{
\colim_{\cD}   \cH_{\cG, J} \ar[r]^-\sim  & hh(\cH_{\cG})
}
\end{equation*}
We will begin with the group theory  of B\'edard and Lusztig~\cite{lusztigParabolicCharacterSheaves} 
underlying the natural decomposition
of each $ \cH_{\cG, J}$ into pieces.

\sss{Combinatorial pieces}\label{sss:comb piece}
Let $W^{a}$ be the affine Weyl group of $\cG$ and $\tilW=\xcoch(T)\rtimes W$ be the extended affine Weyl group. For $J\sft I^{a}$, let $W_{J}\subset W^{a}$ be the subgroup generated by $J$; it is the Weyl group of the Levi $L_{J}$ of the parahoric subgroup $\cP_{J}$. Let ${}^{J}\tilW$ (resp. $\tilW^{J}$) be the set of minimal length elements in the cosets $W_{J}\bs \tilW$ (resp. $\tilW/W_{J}$). Let ${}^{J}\tilW^{J'}$ be the set of minimal length elements in the double cosets $W_{J}\bs \tilW/W_{J'}$.

For a group $\G$ and a subgroup $\G' \subset \Gamma$, let $\frac{\G}{\G'}$ denote the set of $\G'$-conjugacy classes in $\G$. More generally, if $\d$ is an automorphism of $\G$, let $\frac{\G}{\Ad_{\d}(\G')}$ denote the set of orbits of $\G'$ acting on $\G$ by twisted conjugation $\g'\cdot \g=\g'\g\d(\g'^{-1})$. If $\g\in \G$, we use $\f{\G}{\Ad_{\g}(\G')}$ to mean the $\f{\G}{\Ad_{\Ad(\g)}(\G')}$.

Let us first review a combinatorial procedure of B\'edard, as found in \cite{lusztigParabolicCharacterSheaves}. Let $J\sft I^{a}$. Let $\cS_{J}$ be the set of sequences $(J_{n},J'_{n},u_{n})_{n\ge0}$ such that
\begin{enumerate}
\item $J_{0}=J$ and $J_{n}=J_{n-1}\cap \Ad(u_{0}\cdots u_{n-1})J_{n-1}$ for $n\ge1$.
\item $J'_{0}=J$ and $J'_{n}=J_{n-1}\cap \Ad(u_{0}\cdots u_{n-1})^{-1}J_{n-1}$ for $n\ge1$.
\item $u_{n}\in {}^{J'_{n}}(W_{J_{n-1}})^{J_{n}}$ for $n\ge0$.
\end{enumerate}
We call elements in $\cS_{J}$ {\em combinatorial $J$-pieces}. Note that $J_{n}$ and $J'_{n}$ stabilize for $n$ large hence $u_{n}=1$ for $n$ large. As explained in \cite[Prop.~2.5]{lusztigParabolicCharacterSheaves}, the map $(J_{n}, J'_{n},u_{n})_{n\ge0}\mapsto u_{0}u_{1}\cdots u_{m}$ (for $m\gg 0$) defines a bijection
\begin{equation}\label{SJ}
\cS_{J}\isom {}^{J}\tilW.
\end{equation}

For $u\in {}^{J}\tilW$, we denote the corresponding element in $\cS_{J}$ by 
\begin{equation*}
\f{u}{J}\in \cS_{J}.
\end{equation*}

\sss{The map $\s_{J}$}
For any $J\sft I^{a}$, we define a map
\begin{equation*}
\s_{J}: \frac{\tilW}{W_{J}}\to \cS_{J}
\end{equation*}
as follows.  For $c\in \frac{\tilW}{W_{J}}$, we set $c_{0}=c, J_{0}=J'_{0}=J$ and $u_{0}\in {}^{J}\tilW^{J}$ to be the $W_{J}$-double coset containing the $W_{J}$-conjugacy class $c$. We will inductively construct a sequence $(c_{n}, J_{n}, J'_{n}, u_{n})_{n\ge0}$ satisfying the following conditions:
\begin{enumerate}
\item $(J_{n}, J'_{n}, u_{n})_{n\ge0}\in \cS_{J}$.
\item For $n\ge1$, $c_{n}\in \frac{W_{J_{n-1}}}{\Ad_{u_{0}\cdots u_{n-1}}(W_{J'_{n}})}$ is characterized by
\begin{equation}\label{def cn}
c_{n}=\{x\in W_{J_{n-1}}|u_{0}\cdots u_{n-1}x\in c\}.
\end{equation}
\item For $n\ge1$, $u_{n}\in {}^{J'_{n}}(W_{J_{n-1}})^{J_{n}}$ is the minimal length element in the $(W_{J'_{n}}, W_{J_{n}})$-double coset of $c_{n}$.
\end{enumerate}

These conditions clearly characterize $(c_{n}, J_{n}, J'_{n}, u_{n})_{n\ge0}$ uniquely, given the initial terms $(c_{0}=0, J_{0}=J,J'_{0}=J, u_{0})$ as above. To confirm this inductive procedure is well-defined, we need to show that, given $(c_{i}, J_{i},J'_{i}, u_{i})_{0\le i\le n-1}$ satisfying the above conditions (so that $J_{n}$ and $J'_{n}$ can already be defined using $J_{n-1}$ and $u_{0}\cdots u_{n-1}$ as dictated by the requirement that $(J_{n}, J'_{n}, u_{n})_{n\ge0}\in \cS_{J}$), the set $c_{n}$ defined using \eqref{def cn} is non-empty and consists of a single $(u_{0}\cdots u_{n-1})$-twisted conjugacy class under $W_{J'_{n}}$. 

First, to see $c_{n}$ is non-empty: from the inductive hypothesis, we know that for $x\in W_{J_{n-2}}$ (we understand $W_{J_{-1}}$ to be $\tilW$),  $u_{0}\cdots u_{n-2}x\in c$ if and only if $x\in c_{n-1}$. From the construction of $u_{n-1}$, we see any $x\in c_{n-1}$ can be written as $x=au_{n-1}b$ for some $a\in W_{J'_{n-1}}$ and $b\in W_{J_{n-1}}=\Ad(u_{0}\cdots u_{n-2})W_{J'_{n-1}}$. Take any such $x=au_{n-1}b$ and
\begin{equation*}
u_{0}\cdots u_{n-2}x=u_{0}\cdots u_{n-2}au_{n-1}b=a'u_{0}\cdots u_{n-1}b
\end{equation*}
where $a'=\Ad(u_{0}\cdots u_{n-2})a\in W_{J_{n-1}}$. The above element then lies in the same $W_{J_{n-1}}$-conjugacy class as $u_{0}\cdots u_{n-1}ba'$. This implies $u_{0}\cdots u_{n-1}W_{J_{n-1}}\cap c\ne\vn$.  

Next, to see $c_{n}$ is a single $(u_{0}\cdots u_{n-1})$-twisted conjugacy class under $W_{J'_{n}}$: suppose $y,y'\in W_{J_{n-1}}$ are such that $u_{0}\cdots u_{n-1}y, u_{0}\cdots u_{n-1}y'\in c$. By the inductive hypothesis, $u_{n-1}y, u_{n-1}y'\in c_{n-1}$, hence there exists $z\in W_{J'_{n-1}}$ such that $u_{n-1}y=zu_{n-1}y'\Ad(u_{0}\cdots u_{n-2})z^{-1}$. Let $z'=u_{n-1}^{-1}zu_{n-1}$. Then
\begin{equation}\label{yy'}
y=u_{n-1}^{-1}zu_{n-1}y'\Ad(u_{0}\cdots u_{n-2})z^{-1}=z'y'\Ad(u_{0}\cdots u_{n-1})z'^{-1}.
\end{equation}
Now $z'=y\Ad(u_{0}\cdots u_{n-2})zy'^{-1}$. Note $z'\in \Ad(u_{n-1}^{-1})W_{J'_{n-1}}$ and $y\Ad(u_{0}\cdots u_{n-2})zy'^{-1}\in W_{J_{n-1}}$.  Therefore $z'\in \Ad(u_{n-1}^{-1})W_{J'_{n-1}}\cap W_{J_{n-1}}$. Since $u_{n-1}^{-1}\in {}^{J_{n-1}}(W_{J_{n-2}})^{J'_{n-1}}$, we have $\Ad(u_{n-1}^{-1})W_{J'_{n-1}}\cap W_{J_{n-1}}=W_{\Ad(u_{n -1})^{-1}J'_{n-1}\cap J_{n-1}}$. Note that $\Ad(u_{n -1})^{-1}J'_{n-1}=\Ad(u_{0}\cdots u_{n-1})^{-1}J_{n-1}$ hence $\Ad(u_{n -1}^{-1})J'_{n-1}\cap J_{n-1}=\Ad(u_{0}\cdots u_{n -1})^{-1}J_{n-1}\cap J_{n-1}=J_{n}'$. We conclude that $z'\in W_{J'_{n}}$. The equation \eqref{yy'} implies $y$ and $y'$ lie in the same $(u_{0}\cdots u_{n-1})$-twisted conjugacy class of $W_{J'_{n}}$ in $W_{J_{n-1}}$. 

This completes the definition of the map $\s_{J}$.

Composing $\s_{J}$ with the bijection \eqref{SJ} we get a map
\begin{equation*}
p_J: \frac{\tilW}{W_{J}}\xr{\s_{J}} \cS_{J}\isom {}^{J}\tilW.
\end{equation*}

\begin{lemma}\label{l:type of coset} Let $(J_{n},J'_{n}, u_{n})_{n\ge0}\in \cS_{J}$ and $u=u_{0}u_{1}\cdots u_{n}\in {}^{J}\tilW$ (for large $n$) be its image under the bijection \eqref{SJ}. Let $K$ be the stable value of $J_{n}$. 
\begin{enumerate}
\item (Compare \cite[Lemma 1.4]{heGstablePiecesWonderful2007}) $K$ is the largest subset of $J$ that is stable under $\Ad(u)$, and $K$ is also the stable value of $J'_{n}$.

We denote $K$ by $I(J,u)$.
 
\item The preimage of $\f{u}{J}$ under $\s_J$ consists of $W_{J}$-conjugacy classes of $uy$ where $y\in W_{K}$. Two elements $uy$ and $uy'$ ($y,y'\in W_{K}$) are in the same $W_{J}$-conjugacy class if and only if $y,y'$ lie in the same $u$-twisted conjugacy class of $W_{K}$, i.e., $y\mt uy$ gives a canonical bijection
\begin{equation*}
\f{uW_{K}}{W_{K}}\cong\frac{W_{K}}{\Ad_{u}(W_{K})}\isom \s_J^{-1}(\f{u}{J}).
\end{equation*}
\end{enumerate}
 \end{lemma}
\begin{proof}
(1) For $n$ large so that $J_{n-1}=K$, we have  $J_{n}=J_{n-1}\cap \Ad(u)J_{n-1}$ and $J_{n-1}=J_{n}$, hence $J_{n}$ is stable under $\Ad(u)$. Also $J'_{n}=J_{n-1}\cap \Ad(u)^{-1}J_{n-1}=J_{n-1}=K$. 

Conversely, if $K_{1}\subset J$ is stable under $\Ad(u)$, we show inductively that $K_{1}\subset J_{n}$ for all $n\ge0$, hence $K_{1}\subset K=I(J,u)$.  For $n=0$, this is clear. Assume $K_{1}\subset J_{n-1}$ for some $n\ge1$. Let $s\in K$ and let $\a_{s}$ be the corresponding simple root. For $K'\sft I^{a}$, let $\Phi_{K'}$ be the sub root system of the affine roots of $\cG$ spanned by $K'$.  Since $K_{1}$ is stable under $\Ad(u)$,  $\a_{i}=u\a_{i'}$ for some $\a_{i}\in K_{1}$. Write $u=u_{0}\cdots u_{n-1}x$ where $x\in W_{J_{n-1}}$, then $\a_{i}=u_{0}\cdots u_{n-1} \b$ where $\b=x\a_{i'}\in \Phi_{J_{n-1}}$. Since $u_{0}\cdots u_{n-1}\in \tilW^{J_{n-1}}$, $u_{0}\cdots u_{n-1}$ sends $\Phi^{+}_{J_{n-1}}$ to positive roots, therefore $\b$ is also a simple root in $\Phi_{J_{n-1}}$ for otherwise $u_{0}\cdots u_{n-1} \b$ cannot be a simple root. The equality $\a_{i}=u_{0}\cdots u_{n-1} \b$ then implies $\a_{i}\in\Ad(u_{0}\cdots u_{n-1})J_{n-1}$. Therefore $\a_{i}\in J_{n-1}\cap \Ad(u_{0}\cdots u_{n-1})J_{n-1}=J_{n}$ for any $\a_{i}\in K_{1}$, hence $K_{1}\subset J_{n}$. 

(2) is immediate from the construction of $\s_{J}$.
\end{proof}

\sss{Relevant affine subspaces}\label{sss:rel aff sp}
Let $\frA$ be the standard apartment with the action of $\tilW$. By a {\em relevant affine subspace} of $\frA$, we will mean the intersection of a set of affine root hyperplanes.  Let $\frE$ be the set of relevant affine subspaces of $\frA$, and let $\un\frE$ be the set of $W^{a}$-orbits on $\frE$. Each subset $K\sft I^{a}$ gives a relevant affine subspace $\frA(K)=\{x\in \frA|\a(x)=0, \forall \a\in K\}$. Every relevant affine subspace is $\Wa$-conjugate to one of the form $\frA(K)$. This induces a surjection $\{K\sft I^{a}\}\surj \un\frE$. This map may not be a bijection in general: for $K,K'\sft I^{a}$, $\frA(K)$ is in the $\Wa$-orbit of $\frA(K')$ if and only if there exists $w\in \Wa$ such that $wKw^{-1}=K'$.

For $E\in \frE$,  we denote its image in $\un\frE$ by $[E]$. The set $\un\frE$ is partially ordered such that $[E]\ge [E']$ if and only if $E\supset wE'$ for some $w\in \Wa$.

For $J\sft I^{a}$, let $\frE_{J}$ be the subset of  $\frE$ consisting of relevant affine subspaces that contain $\frA(J)$. Clearly $W_{J}$ acts on $\frE_{J}$. For $K\subset J$, we have $\frA(K)\in \frE_{J}$.

\begin{defn} Let $J\sft I^{a}$ and $\f{u}{J}\in \cS_{J}$. Let $K=I(J,u)\subset J$ as specified in Lemma~\ref{l:type of coset}.
\begin{enumerate}
\item We define the {\em $J$-type}  of $\f{u}{J}$ to be the element $\t_{J}(\f{u}{J})\in W_{J}\bs \frE_{J}$ given by the image of $\frA(K)$, i.e., the relevant affine space $\frA(K)$ up to $W_{J}$-action.

\item We define the {\em coarse type} of $\f{u}{J}$ to be the element $\t(\f{u}{J})\in \un \frE=\Wa\bs\frE$  given by  the image of $\frA(K)$,  i.e., the relevant affine space $\frA(K)$ up to $\Wa$-action.
\end{enumerate}
\end{defn}
Taking the coarse type of a $J$-piece defines a map
\begin{equation*}
\t\colon \ \cS:=\coprod_{J\sft I^{a}}\cS_{J}\to \un\frE. 
\end{equation*}

Next, we give a way to compute the $J$-type starting from any $W_{J}$-conjugacy class of $\tilW$. For $w\in \tilW$ and $J\sft I^{a}$, the set $\frE_{J}^{w}:=\{E\in \frE_{J}|w(E)=E\}$ is non-empty (since $\frA\in \frE_{J}^{w}$) and closed under intersection, hence has a unique minimal element which we denote by $E_{J,w}$. The next lemma explains the relation between $J$-type and $E_{J,w}$. 

\begin{lemma}\label{l:EJw}
Let $J\sft I^{a}, u\in {}^{J}\tilW$. 
\begin{enumerate}
\item Let $K=I(J,u)$, then $E_{J,uy}=\frA(K)$ for any $y\in W_{K}$.
\item Suppose $w\in \tilW$ is such that $\s_{J}(w)=\f{u}{J}\in\cS_{J}$. Then the $J$-type of $\f{u}{J}$ is the $W_{J}$-orbit of $E_{J,w}$.
\item Suppose $w, w'\in \tilW$. Then $\s_{J}(w)=\s_{J}(w')$ if and only if $w'$ is $W_{J}$-conjugate to an element $w''$ with $E_{J,w}=E_{J,w''}$ and $w|_{E_{J,w}}=w''|_{E_{J,w''}}$.
\end{enumerate}
\end{lemma}
\begin{proof} 
Let $K=I(J,u)$. Then the $J$-type of $\f{u}{J}$ is the $W_{J}$-orbit of $\frA(K)$. By Lemma \ref{l:type of coset}, $w=xuyx^{-1}$ where $y\in W_{K}$ and $x\in W_{J}$. We have $E_{J,w}=xE_{J,uy}$. Therefore it suffices to consider the case $w=uy$. Hence (2) follows from (1).

We prove (1). Since $\Ad(u)K=K$, $u$ stabilizes $\frA(K)$, hence $uy$ stabilizes $\frA(K)$. By the minimality of $E:=E_{J,uy}$, we have $E\subset \frA(K)$. We claim that $E=\frA(K)$. Let $\cF$ be the fundamental alcove. Consider the relevant affine subspace $\Span(E\cap \cF)$ spanned by $E\cap \cF$. We have $\Span(E\cap \cF)=\frA(K')$ for some $K'\sft I^{a}$. Since $\frA(J)\subset E\subset \frA(K)$, we have $K\subset K'\subset J$. Note that $\frA(K')=\Span(E\cap C_{J})$, where $C_{J}\subset\frA$ is the dominant $W_{J}$-chamber centered along $\frA(J)$. On the other hand,  since $u\in {}^{J}\tilW$, we have $u\cF\in C_{J}$,  and therefore $\Span(uyE\cap u\cF)=\Span(E\cap u\cF)\subset \Span(E\cap C_{J})=\frA(K')$. The action of $u$ sends $\Span(E\cap \cF)=\Span(yE\cap \cF)$ (note that $y\in W_{K}$ acts by identity on $\frA(K)$) isomorphically to $\Span(uyE\cap u\cF)=\Span(E\cap u\cF)$, hence we must have $\Span(E\cap u\cF)=\frA(K')$, and in particular $u$ stabilizes $\frA(K')$. Now $u$ stabilizes the affine roots $\Phi_{K'}$ spanned by $K'\subset J$, and $u\in {}^{J}\tilW$ implies that $u^{-1}$ sends positive roots $\Phi^{+}_{J}\subset \Phi_{J}$ to positive affine roots, therefore $u$ preserves $\Phi^{+}_{K'}$, and hence preserves $K'$. By the maximality of $K$ as a $u$-stable subset of $J$, we must have $K'=K$. Therefore, $E\supset \Span(E\cap \cF)=\frA(K')=\frA(K)$, forcing $E=\frA(K)$.

We prove (3). The statement is invariant under $W_{J}$-conjugation of $w$ and $w'$ separately. Therefore we may assume $w=u$. Let $K=I(J,u)$. Suppose $w'$ is $W_{J}$-conjugate to $w''$ with $\frA(K)=E_{J,w''}$ and $w|_{\frA(K)}=w''|_{\frA(K)}$. Then $w''=uy$ for some $y\in W_{K}$, hence $\s_{J}(w')=\s_{J}(w'')=\f{u}{J}$. Conversely, suppose $\s_{J}(w')=\f{u}{J}$, then by Lemma \ref{l:type of coset}, $w'$ is $W_{J}$-conjugate to $uy$ for some $y\in W_{K}$, hence $E_{J,uy}=\frA(K)$ by (1) and $uy|_{\frA(K)}=u|_{\frA(K)}$ since $y\in W_{K}$ acts by identity on $\frA(K)$.
\end{proof}

\begin{prop}\label{p:attach pieces} Let $J\subset J'\sft I^{a}$. Let $\pi^{J'}_{J}: \f{\tilW}{W_{J}}\to \f{\tilW}{W_{J'}}$ be the projection. Then there is a unique map $\d^{J'}_{J}: \cS_{J}\to \cS_{J'}$ making the following diagram commutative
\begin{equation*}
\xymatrix{ \f{\tilW}{W_{J}}\ar[d]^{\s_{J}} \ar[r]^{\pi^{J'}_{J}} & \f{\tilW}{W_{J'}}\ar[d]^{\s_{J'}}\\
\cS_{J}\ar[r]^{\d^{J'}_{J}} & \cS_{J'}}
\end{equation*}
In particular, for $J\subset J'\subset J''\sft I^{a}$, $\d^{J''}_{J'}\c\d^{J'}_{J}=\d^{J''}_{J}$.
\end{prop}
\begin{proof} Since $\s_{J}$ is surjective, $\d^{J'}_{J}$ is unique if it exists.

For the existence of $\d^{J'}_{J}$, we need to show the following. Let $w\in \tilW$ and $\s_{J}(w)=\f{u}{J}$. We need to show that $\s_{J'}(w)$ depends only on $u$ and not on $w$. By Lemma \ref{l:type of coset}, up to $W_{J}$-conjugacy we may assume $w=uy$ for some $y\in W_{K}$, where $K=I(J,u)$. We need to show that $\s_{J'}(u)=\s_{J'}(uy)$ for all $y\in W_{K}$. For this, we use the criterion in Lemma \ref{l:EJw}(3). By Lemma \ref{l:EJw}(1), $E_{J,u}=E_{J,uy}=\frA(K)$. Therefore $E_{J', u}, E_{J',uy}\subset \frA(K)$. Since $y$ acts by identity on $\frA(K)$, we have $E_{J',u}=E_{J',uy}$, and $u|_{E_{J',u}}=uy|_{E_{J',uy}}$. Therefore $\s_{J'}(u)=\s_{J'}(uy)$ by Lemma \ref{l:EJw}(3). 
\end{proof}

\begin{lemma}\label{l:ell tau under closure}
Let $J\subset J'\sft I^{a}$, $u\in {}^{J}\tilW$ and $u'\in {}^{J'}\tilW$ such that $\d^{J'}_{J}(\f{u}{J})=\f{u'}{J'}$. Then
\begin{enumerate}
\item $\ell(u)\ge\ell(u')$.
\item A representative of the $J$-type of $\f{u}{J}$ (as a relevant subspace containing $\frA(J)$, up to $W_{J}$-action) contains a representative of the $J'$-type of $\frac{u'}{J'}$. In particular, $\t(\f{u}{J})\ge \t(\frac{u'}{J'})$ under the partial order on $\un\frE$ defined in Section~\ref{sss:rel aff sp}.
\end{enumerate}
\end{lemma}
\begin{proof}
(1) By \cite[Theorem 2.5]{heMinimalLengthElements2014}, $u'$ has minimal length among $\s_{J'}^{-1}(\f{u'}{J'})$. Since $u\in \s_{J'}^{-1}(\f{u'}{J'})$ by construction, we see that $\ell(u)\ge\ell(u')$.

(2) By Lemma \ref{l:EJw}, the $J$-type of $\f{u}{J}$ is represented by $E_{J,u}$, and the $J'$-type of $\f{u'}{J'}$ is represented by $E_{J',u}$. Clearly  $E_{J',u}\subset E_{J,u}$.
\end{proof}

\begin{defn} Let $J\subset J'\sft I^{a}$. Let $\f{u}{J}\in \cS_{J}$ and  $\d^{J'}_{J}(\f{u}{J})=\f{u'}{J'}\in \cS_{J'}$. 
\begin{enumerate}
\item We say that $\frac{u}{J}$ is {\em quasi-$J'$-reduced} if $\ell(u)=\ell(u')$.
\item We say that $\frac{u}{J}$ is {\em $J'$-reduced} if $\ell(u)=\ell(u')$, and $\t(\f{u}{J})=\t(\f{u'}{J'})$. 
\end{enumerate}
\end{defn}

%
%
%

\sss{Newton point}\label{sss:np} Recall the {\em Newton point} of $w\in \tilW$ is a point $\nu(w)\in \xcoch(T)_{\QQ}^{+}$ (rational dominant cone) characterized by the following property:   for sufficiently divisible $n$, $w^{n}\in \xcoch(T)$ is in the same $W$-orbit as the translation element given by $n\nu(w)$. The Newton point is constant on each $\tilW$-conjugacy class. Therefore we have a map
\begin{equation*}
\nu: \frac{\tilW}{\tilW}\to \xcoch(T)_{\QQ}^{+}.
\end{equation*}
Let $\NP\subset \xcoch(T)_{\QQ}^{+}$ be the image of this map. 

 We also have the natural projection
\begin{equation*}
\k: \tilW\to\tilW/W^{a}=:\Om
\end{equation*}
that factors through the conjugacy classes of $\tilW$ because $\Om$ is abelian. Combining $\nu$ and $\k$ we get the enhanced Newton map
\begin{equation*}
\wt\nu: \frac{\tilW}{\tilW}\to \NP\times \Om.
\end{equation*}
Let $\wt\NP\subset \NP\times\Om$ be the image of $\wt\nu$.

\begin{lemma}\label{l:nu} For $J\sft I^{a}$, there is a unique map $\wt\nu_{J}: \cS_{J}\to \wt\NP$ such that the following diagram is commutative
\begin{equation}\label{nuJ}
\xymatrix{ \frac{\tilW}{W_{J}}\ar[d]\ar[r]^{\s_{J}} & \cS_{J}\ar[d]^{\wt\nu_{J}}\\
\frac{\tilW}{\tilW}\ar[r]^{\wt\nu} & \wt\NP
}
\end{equation}
where the left vertical map is the natural projection. Moreover, for $J\subset J'\sft I^{a}$, we have $\nu_{J'}\c\d^{J'}_{J}=\wt\nu_{J}: \cS_{J}\to \wt\NP$.
\end{lemma}
\begin{proof} Since $\s_{J}$ is surjective, the uniqueness of $\wt\nu_{J}$ is clear: it sends $\f{u}{J}$ to $\wt\nu(u)$,where $u\in {}^{J}\tilW$. To show the diagram \eqref{nuJ} is commutative, by Lemma \ref{l:type of coset},  it suffices to show that $\nu(uy)=\nu(u)$ and $\k(uy)=\k(u)$ for all $y\in W_{K}$, where $K=I(J,u)$. Since $W_{K}\subset W^{a}$, we have $\k(uy)=\k(u)\in \Om$. Now we show $\nu(uy)=\nu(u)$. For any $n\ge1$ we have $(uy)^{n}=\Ad(u)y\cdot \Ad(u^{2})y\cdots \Ad(u^{n})y\cdot u^{n}$. Each $\Ad(u^{i})y\in W_{K}$ since $K$ is stable under $\Ad(u)$. Let $m$ be the order of $\Ad(u)$ on $K$. Then if $n$ is divisible by $m|W_{K}|$, $\Ad(u)y\cdot \Ad(u^{2})y\cdots \Ad(u^{n})y=(\Ad(u)y\cdot \Ad(u^{2})y\cdots \Ad(u^{m})y)^{n/m}\in (W_{K})^{n/m}=\{1\}$. Therefore $(uy)^{n}=u^{n}$ for $n$ sufficiently divisible. This implies $\nu(uy)=\nu(u)$.
\end{proof}

\sss{Straight elements}
Following Krammer \cite{krammerConjugacyProblemCoxeter2009}, one says $w\in \tilW$ is {\em straight} if $\ell(w^{n})=n\ell(w)$ for all $n\ge1$. A conjugacy class $c\in \f{\tilW}{W_{a}}$ is called straight if it contains a straight element. 

\begin{lemma}\label{l:leng bound}
Let $w\in \tilW$ with Newton point $\nu\in \NP$. Then $\ell(w)\ge\j{2\r,\nu}$, and the equality holds if and only if $w$ is a straight element. 
\end{lemma}
\begin{proof}
For any translation element $t_{\l}\in \tilW$ corresponding to a dominant $\l\in \xcoch(T)$, we have $\ell(t_{\l})=\j{2\rho, \l^{+}}$ for the unique dominant element $\l^{+}$ in the $W$-orbit of $\l$. Therefore, if $w\in \tilW$ has Newton point $\nu\in \NP$, $w^{n}$ is conjugate to $t_{n\nu}$ for sufficiently divisible $n$, hence $\ell(w^{n})=\j{2\r, n\nu}$. Therefore $\ell(w)\ge \f{1}{n}\ell(w^{n})=\j{2\rho, \nu}$, and equality holds if and only if $w$ is a straight element. 
\end{proof}

By \cite[Theorem 3.3]{heMinimalLengthElements2014}, there is a unique straight $W^{a}$-conjugacy class for each enhanced Newton point $\wt\nu\in \wt\NP$.

\subsection{The space $\frB$} We will introduce a topological space $\frB$ obtained by gluing copies of quotients of the apartment in the building of $\cG$ and passing to quotients, using the combinatorial pieces for the affine Weyl group. The space $\frB$ will serve as an organizational tool of subquotient categories of $\cH_{\cG,J}$ that we study in the next section.

When $G$ is almost simple, $\frB$ will be a {\em $\Delta$-complex} (a weaker notion than a simplicial complex), which is a union of simplices where the intersection of two simplices is not necessarily a common face. In general, $\frB$ will be a union of poly-simplices (product of simplices).


\sss{$\cD^{\c}$-sets}\label{sss:Pset} Recall from Section~\ref{sss:general} that $\cD^{\c}$ denotes the poset of finite type subsets $J\sft I^{a}$ under inclusion. A $\cD^{\c}$-set $X$ is a functor
\begin{equation*}
X: \cD^{\c}\to \Sets.
\end{equation*}
In other words, it is an assignment $J\mapsto X_{J}\in\Sets$ for each $J\sft I^{a}$, and a map $X_{J}\to X_{J'}$ whenever $J\subset J'$, compatible with three-term inclusions. 
 
For a $\cD^{\c}$-set $X$, $\Tot(X):=\coprod_{J\sft I^{a}}X_{J}$ is a poset whose relations are of the form $x_{J}\le y_{J'}$, if $J\subset J'\sft I^{a}$ and $x_{J}\in X_{J}$ maps to $y_{J'}\in X_{J'}$ under the map $X_{J}\to X_{J'}$.

For a $\cD^{\c}$-set $X$, we shall define a topological space $|X|$ as follows. For $J\in\cD^{\c}$, we let $|\D_{J}|\subset \frA$ be the standard facet in the (reduced) apartment for $T$ indexed by $J$ (so that for $J=\vn$, $|\D_{J}|$ is the fundamental alcove in $\frA$). Then $|\D_{J}|$ is isomorphic to a product of simplices with codimension $\#J$ in $\frA$. Let $|X|$ be the space obtained as quotient of $\sqcup_{J\sft I^{a}}X_{J}\times |\D_{J}|$ by the relation $(x_{J}, t)\sim (y_{J'}, t)$ where $J\subset J'$, $x_{J}\mapsto y_{J'}$ under $X_{J}\to X_{J'}$ and $t\in |\D_{J}|\subset |\D_{J'}|$. The image of $\{x_{J}\}\times |\D_{J}|$ in $|X|$  (for $x_{J}\in X_{J}$) is called a {\em $J$-facet} of $|X|$; they are parametrized by $X_{J}$.

When $G$ is almost simple of rank $r$, each $J$-facet of $|X|$ is a simplex of dimension $r-\#J$, and $|X|$ is a $\D$-complex. In general, $|X|$ is a union of poly-simplices. 

The set $\Tot(X)$ is the set of all facets of $|X|$. The partial order on $\Tot(X)$ is {\em the opposite} of the closure order of faces.


%

\begin{exam}
\begin{enumerate}
\item Let $\d$ be the $\cD^{\c}$-set given by the constant functor valued in the singleton set. The geometric realization $|\d|$ can be identified with the fundamental alcove in $\frA$.
\item The standard apartment $\frA$ of $\cG$ gives rise to a $\cD^{\c}$-set $\Fac(\frA)$: $\Fac(\frA)_{J}$ is the set of $J$-facets in $\frA$, and the transition maps $\Fac(\frA)_{J}\to \Fac(\frA)_{J'}$ sends a $J$-facet $F$ to the unique $J'$-facet in its closure. Then we have a canonical homeomorphism $|\Fac(\frA)|\cong \frA$ respecting the poly-simplicial structures.

\item The same construction of (2) applies to any closed subset $E\subset \frA$ that is a union of facets. It gives a $\cD^{\c}$-subset $\Fac(E)\subset \Fac(\frA)$, and $|\Fac(E)|$ is identified with $E$ as a subspace of $|\Fac(\frA)|\cong \frA$.
\end{enumerate}

\end{exam}

\begin{remark}\label{r:sd geom real} For a $\cD^{\c}$-set $X$, the poset $\Tot(X)$ also has a geometric realization $|\Tot(X)|$ (see Section~ \ref{sss:poset}). Then $|\Tot(X)|$ is always a simplicial complex, and it is a subdivision of $|X|$. When $G$ is almost simple, $|\Tot(X)|$ is the barycentric subdivision of $|X|$, i.e., $|\Tot(X)|\cong sd(|X|)$ as simplicial complexes.
\end{remark}

\sss{Construction of $\frB$} 

We define a $\cD^{\c}$-set $\cS$ whose value at $J\sft I^{a}$ is $\cS_{J}$, and for  $J\subset J'\sft I^{a}$, the transition map is $\d^{J'}_{J}: \frB_{J}=\cS_{J}\to \cS_{J'}=\frB_{J'}$ defined in Proposition \ref{p:attach pieces}.

Let $\frB=|\cS|$ be the geometric realization of $\frB$. For $\f{u}{J}\in\cS_{J}$, we denote by $\frB(\f{u}{J})$ the corresponding $J$-facet of $\frB$, with interior $\frB(\f{u}{J})^{\c}$.


\begin{remark} For any finite or affine Weyl group $W$ with an automorphism $\s$ preserving the simple reflections, we have the notion of $\s$-twisted $J$-pieces and one can similarly define $\frB(W; \s)$. In our setting, the normal slice to a facet $\frB(\f{u}{J})$ in $\frB$ is isomorphic to $\frB(W_{K}; u)$ for the finite Weyl group $W_{K}$ (where $K=I(J,u)$) under the $u$-twisted conjugation action of $W_{K}$.
\end{remark}

\sss{Functions on $\frB$} We have attached three invariants to each element in $\Tot(\cS)=\coprod_{J\sft I^{a}}\cS_{J}$:
\begin{enumerate}
\item The length function $\ell: \Tot(\cS)\to \ZZ_{\ge0}$.
\item The coarse type $\t: \Tot(\cS)\to \un\frE$.
\item The enhanced Newton map $\wt\nu: \Tot(\cS)\to \wt\NP\subset \NP\times\Om$.
\end{enumerate}

We may consider these functions as piece-wise continuous functions on the geometric realization $\frB$:
\begin{enumerate}
\item $\ell: \frB\to \ZZ_{\ge0}$. 
\item $\t: \frB\to \un\frE$.
\item $\wt\nu: \frB\to \wt\NP$.
\end{enumerate}
such that the value of $\ell, \t$ and $\nu$ on  $\frB(\f{u}{J})^{\c}$ is $\ell(\f{u}{J}), \t(\f{u}{J})$ and $\nu_{J}(\f{u}{J})$ respectively. 

By Lemma \ref{l:ell tau under closure}, the functions $\ell$ and $\t$ on $\frB$ are both lower semicontinuous (non-increasing under specialization). By Lemma \ref{l:nu}, the function $\wt\nu$ is locally constant on $\frB$. We have decompositions of the $\cD^{\c}$-set $\cS$ and the space $\frB$:
\begin{equation*}
\cS=\coprod_{\wt\nu\in \wt\NP}\cS_{\wt\nu},\quad \frB=\coprod_{\wt\nu\in \wt\NP}\frB_{\wt\nu}
\end{equation*}
where $\cS_{\wt\nu, J}$ is the set of $\f{u}{J}$ such that $\wt\nu(\f{u}{J})=\wt\nu$, and $\frB_{\wt\nu}=|\cS_{\wt\nu}|$. Note each $\frB_{\wt\nu}\subset \frB$ is open and closed.

\sss{Essential part}
Fix $\wt\nu=(\nu,\om)\in \wt\NP$. By Lemma \ref{l:leng bound}, any piece $\f{u}{J}$ that appears in $\frB_{\wt\nu}$ satisfies $\ell(u)\ge \j{2\rho,\nu}$. Let $\cS^{\hs}_{\wt\nu,J}\subset \cS_{J}$ be the subset consisting of $\f{u}{J}$ with $\wt\nu(u)=\wt\nu$ and $\ell(u)=\j{2\r,\nu}$. By Lemma \ref{l:leng bound} and Lemma \ref{l:ell tau under closure}, we see that the assignment $J\mapsto \cS^{\hs}_{\wt\nu,J}$ defines a $\cD^{\c}$-subset $\cS_{\wt\nu}^{\hs}$ of $\cS_{\wt\nu}$. Let $\frB^{\hs}_{\wt\nu}=|\cS_{\wt\nu}^{\hs}|$ and call it the {\em essential part} of $\frB_{\wt\nu}$. By Lemma \ref{l:leng bound}, $\frB_{\wt\nu}^{\hs}$ is the closure of the maximal facets of $\frB_{\wt\nu}$ indexed by straight elements.


\begin{exam}\label{ex: SL2 B} For $G=\SL_{2}$, $\frB$ is a graph. Since $G$ is simply-connected, $\Om=0$ and $\wt\NP=\NP$. The connected components of $\frB$ are indexed by the possible Newton points $\nu\in\ZZ_{\ge0}$. Note that $\Wa=\j{s_{0},s_{1}}$. Let $T_{n}=(s_{1}s_{0})^{n}$ for $n\in \ZZ$. 

For $\nu>0$, $\frB_{\nu}$ is a cycle with two nodes and two edges: 
\begin{equation*}
\xymatrix{ \f{T_{-n}}{\{s_{1}\}}  \ar@/^/@{-}[r]^-{\f{T_{n}}{\vn}}\ar@/_/@{-}[r]_-{\f{T_{-n}}{\vn}} & \f{T_{n}}{\{s_{0}\}}
}
\end{equation*}
There is only one conjugacy class with Newton point $n>0$, namely that of $T_{n}$. In this case, we have $\frB^{\hs}_{\nu}=\frB_{\nu}$.

For $\nu=0$, the graph $\frB_{\nu}$ is an infinite linear tree: 
\begin{equation*}
\xymatrix{\f{1}{\{s_{1}\}}\ar@{-}[r]^-{\f{s_{1}}{\vn}} \ar@{-}[d]_-{\f{1}{\vn}}& 
\f{s_{1}}{\{s_{0}\}} \ar@{-}[r]^-{\f{s_{0}s_{1}s_{0}}{\vn}} & \f{s_{0}s_{1}s_{0}}{\{s_{1}\}} \ar@{-}[r]^-{\f{s_{1}s_{0}s_{1}s_{0}s_{1}}{\vn}} &  \cdots\\
\f{1}{\{s_{0}\}} \ar@{-}[r]^-{\f{s_{0}}{\vn}} & \f{s_{0}}{\{s_{1}\}} \ar@{-}[r]^-{\f{s_{1}s_{0}s_{1}}{\vn}} & \f{s_{1}s_{0}s_{1}}{\{s_{0}\}} \ar@{-}[r]^-{\f{s_{0}s_{1}s_{0}s_{1}s_{0}}{\vn}}& \cdots }
\end{equation*}
We draw it in three segments to reflect that there are three conjugacy classes in $W^{a}$ with Newton point $0$: the identity conjugacy class (corresponding to the vertical edge), the conjugacy class of $s_{1}$ (upper ray) and the conjugacy class of $s_{0}$ (lower ray). In this case, $\frB^{\hs}_{\nu}$ consists of the edge labelled $\f{1}{\vn}$ and its two end points. 
\end{exam}

\begin{exam} Consider the case $G=\PGL_{2}$. Now $\tilW\cong T_{1/2}^{\ZZ}\rtimes\j{s_{1}}$, where $T_{1/2}^{2}=T_{1}$ in the notation of Example \ref{ex: SL2 B}. Let $\om=s_{1}T_{1/2}$ be the length zero element in $\tilW-W^{a}$. Then for $n\ge1$ odd, $T_{n/2}=s_{1}s_{0}\cdots s_{1}\om$ (length $n$). We have $\NP=\frac{1}{2}\ZZ_{\ge0}$. The set $\wt\NP\subset \NP\times \Om\cong \frac{1}{2}\ZZ_{\ge0}\times \ZZ/2\ZZ$ is
\begin{equation*}
\wt\NP=\{(n/2,n\textup{ mod } 2)|n\in\ZZ_{\ge0}\}\cup \{(0,1)\}.
\end{equation*}

For $\wt\nu=(n,0)\in \wt\NP$, $\frB_{\wt\nu}$ is the same as $\frB_{n}$ described in Example \ref{ex: SL2 B}. 

For $\wt\nu=(n/2, 1)$ where $n\ge1$ odd, $\frB_{\wt\nu}$ is of the form
\begin{equation*}
\xymatrix{ \f{T_{-n/2}}{\{s_{1}\}}  \ar@/^/@{-}[r]^-{\f{T_{n/2}}{\vn}}\ar@/_/@{-}[r]_-{\f{T_{n/2}}{\vn}} & \f{T_{-n/2}}{\{s_{0}\}}
}
\end{equation*}
There is only one conjugacy class with Newton point $n/2>0$, namely that of $T_{n/2}$. In this case we have $\frB_{\wt\nu}^{\hs}=\frB_{\wt\nu}$.

For $\wt\nu=(0,1)$, $\frB_{\wt\nu}$ is an infinite linear tree
\begin{equation*}
\xymatrix{ \cdots \ar@{-}[r] & \f{s_{0}\om s_{0}}{\{s_{1}\}}\ar@{-}[r]^-{\f{s_{0}\om s_{0}}{\vn}} & \f{\om}{\{s_{0}\}}\ar@{-}[r]^-{\f{\om}{\vn}} & \f{\om}{\{s_{1}\}}\ar@{-}[r]^-{\f{s_{1}\om s_{1}}{\vn}} & \f{s_{1}\om s_{1}}{\{s_{0}\}}\ar@{-}[r] & \cdots}
\end{equation*}
There is only one $W^{a}$-conjugacy class in $\tilW$ with enhanced Newton point $\wt\nu=(0,1)$, namely the conjugacy class of $\om$.  In this case, $\frB^{\hs}_{\wt\nu}$ consists of the edge labelled $\f{\om}{\vn}$ and its two end points. 
\end{exam}

\sss{Linear structure on $\frB$}
We shall describe $\frB$ as glued from copies of quotients of the apartment $\frA$.
  
For a relevant affine subspace $E\in \frE$, let $W^{E}\subset \Aff(E)$ be the affine transformations of $E$ of the form $w|_{E}$ for some $w\in\Stab_{\tilW}(E)$. Let $W_{E}\subset \Wa$ be the subgroup that fixes $E$ pointwise (this is a parabolic subgroup of $\Wa$). Note that we have a short exact sequence
\begin{equation}\label{WE}
1\to W_{E}\to \Stab_{\tilW}(E)\to W^{E}\to 1.
\end{equation}

Consider the category $\wt\cA$ whose objects are pairs $(E, \wt w)$ where $E\in \frE$ and $\wt w\in \Stab_{\tilW}(E)$; its morphisms are defined by
\begin{equation}\label{mor wt A}
\Mor_{\wt \cA}((E,\wt w), (E',\wt w'))=\{g\in \Wa|gE\subset E', \wt w'|_{gE}=g\c (\wt w|_{E})\c g^{-1}\in\Aff(gE)\}
\end{equation}
with the evident composition. 

Let $\cA$ be the category with objects pairs $(E, w)$ where $E\in \frE$ and $w\in W^{E}$, and morphisms defined in the way similar to \eqref{mor wt A}
\begin{equation*}
\Mor_{\cA}((E,w), (E',w'))=\{g\in \Wa|gE\subset E', w'|_{gE}=g\c  w\c g^{-1}\in\Aff(gE)\}.
\end{equation*}

The obvious functor  $\wt\cA\to \cA$ sending $(E,\wt w)$ to $(E,w)$, where $w\in W^{E}$ is the image of $\wt w$, is an equivalence. The category $\wt\cA$ will only play an auxiliary role.

For each $(E,\wt w)\in \wt \cA$, we define a map of $\cD^{\c}$-sets
\begin{equation*}
\ph_{E,\wt w}: \Fac(E)\to \cS
\end{equation*}
as follows. It sends a $J$-facet $F=xF_{J}$ of $E$ (where $F_{J}$ is the standard $J$-facet, $x\in \Wa/W_{J}$) to $x^{-1}\wt wx$, well-defined up to $W_{J}$-conjugacy. We let $\ph_{E,\wt w}(F)$ to be $\s_{J}(x^{-1}\wt wx)\in \cS_{J}$.

\begin{lemma}\label{l:ph comp}
\begin{enumerate}
\item For $(E,\wt w)\in \wt\cA$, the map $\ph_{E,\wt w}$ depends only on the image of $\wt w$ in $W^{E}$. In particular, for $(E,w)\in \cA$, there is a well-defined map  of $\cD^{\c}$-sets
\begin{equation*}
\ph_{E,w}: \Fac(E)\to \cS.
\end{equation*}
\item Suppose $g: (E,w)\to (E',w')$ is a morphism in $\cA$, then
\begin{equation*}
\ph_{E,w}=(\ph_{E',w'}|_{\Fac(gE)})\c g: \Fac(E)\to \cS.
\end{equation*}
\end{enumerate}
\end{lemma}
\begin{proof}
We shall prove the following statement which implies both (1) and (2). Let $g: (E,\wt w)\to (E', \wt w')$ be a morphism in $\wt\cA$, then
\begin{equation}\label{wt ph comp g}
\ph_{E,\wt w}=(\ph_{E',\wt w'}|_{\Fac(gE)})\c g: \Fac(E)\to \cS.
\end{equation}
To see \eqref{wt ph comp g} implies (1), take $E=E'$ and suppose $\wt w, \wt w'\in \Stab_{\tilW}(E)$ both have the same image $w\in W^{E}$. Then $g=1$ gives an isomorphism $g=1: (E,\wt w)\isom (E', \wt w')$ in $\wt \cA$. In this case, \eqref{wt ph comp g} reads $\ph_{E,\wt w}=\ph_{E,\wt w'}$, which implies (1). It is easy to see that \eqref{wt ph comp g} also implies (2).

Let us prove \eqref{wt ph comp g}.
Let $F=xF_{J}\subset E$ be a $J$-facet. On the one hand, $\ph_{E,\wt w}(F)=\s_{J}(x^{-1}\wt wx)\in \cS_{J}$. On the other hand, $gF=gx F_{J}\subset E'$ is a $J$-facet of $E'$, and $\ph_{E',\wt w'}(g(F))= \s_{J}((gx)^{-1}\wt w'gx)=\s_{J}(x^{-1}(g^{-1}\wt w'g)x)$. Since $g: (E,\wt w)\to (E',\wt w')$ is a morphism in $\wt\cA$, we have $g^{-1}\wt w'g\in \Stab_{\tilW}(E)$ and it has the same image as $\wt w$ in $W^{E}$. By the exact sequence \eqref{WE} we can write $g^{-1}\wt w'g=\wt wy$ for some $y\in W_{E}$. We reduce to showing
\begin{equation}\label{same image under sJ}
\s_{J}(x^{-1}\wt wx)=\s_{J}(x^{-1}\wt wyx).
\end{equation}
For this we use the criterion in Lemma \ref{l:EJw}(3). Let $E_{1}=E_{J, x^{-1}\wt wx}$ and $E_{2}=E_{J, x^{-1}\wt wyx}$. Then $E_{1}$ is the minimal relevant affine subspace that contains $\frA(J)$ and stable under $x^{-1}\wt wx$. Now $x^{-1}E$ contains $\frA(J)$ and is stable under $x^{-1}\wt wx$, hence $E_{1}\subset x^{-1}E$. Since $y\in W_{E}$, $x^{-1}E$ is stable under $x^{-1}yx$, hence also stable under $x^{-1}\wt wyx$. This implies $E_{2}\subset x^{-1}E$. Moreover, because $y$ fixes $E$ pointwise, the actions of $x^{-1}\wt wx$ and $x^{-1}\wt wyx$ on $x^{-1}E$ are the same. Therefore both $E_{1}$ and $E_{2}$ are equal to the minimal relevant affine subspace that contains $\frA(J)$ and contained in $x^{-1}E$, and stable under $x^{-1}\wt wx|_{x^{-1}E}=x^{-1}\wt wyx|_{x^{-1}E}$. By Lemma \ref{l:EJw}(3), \eqref{same image under sJ} holds.
\end{proof}


By the above lemma, $\ph_{E,w}$ is invariant under $\Aut_{\cA}(E,w)$. Using the transition maps $g:E\to E'$ for $g\in \Mor_{\cA}((E,w), (E',w'))$, we may form the colimit in the category of $\cD^{\c}$-sets
\begin{equation*}
\colim_{(E,w)\in \cA} \Fac(E).
\end{equation*}
Lemma \ref{l:ph comp} implies that the maps $\{\ph_{E,w}\}$ together induce a map of $\cD^{\c}$-sets
\begin{equation*}
\ph: \colim_{(E,w)\in \cA} \Fac(E)\to \cS.
\end{equation*}
Taking geometric realizations, we get a map
\begin{equation*}
|\ph|: \colim_{(E,w)\in \cA} E\to\frB.
\end{equation*}

\begin{lemma}\label{l:glue apts} The map $\ph$ is an isomorphism of $\cD^{\c}$-sets. Consequently, the map $|\ph|$ is a  homeomorphism respecting the facet structures. 
\end{lemma}
\begin{proof}
We need to show for each $J\sft I^{a}$, the map on the set of $J$-facets 
\begin{equation*}
\ph_{J}: \colim_{(E,w)\in \cA} \Fac_{J}(E)\to \cS_{J}
\end{equation*}
is a bijection.

To see $\ph_{J}$ is surjective, note that for $E=\frA$ and $w\in \tilW=W^{E}$, $\ph_{\frA,w}$ is the composition of $\wt\ph_{\frA,w}: \Fac_{J}(\frA)=W^{a}/W_{J}\to \f{\tilW}{W_{J}}$ (sending $x\mapsto x^{-1}wx$) and $\s_{J}$. As $w$ runs over all of $\tilW$, the images of $\wt\ph_{\frA,w}$ cover all of $\f{\tilW}{W_{J}}$, therefore the images of $\ph_{\frA,w}$ cover all of $\cS_{J}$ since $\s_{J}$ is surjective. 

Now we show $\ph_{J}$ is injective. Let $(E_{1},w_{1}), (E_{2}, w_{2})\in \cA$ and $F_{i}=x_{i}F_{J}\in\Fac_{J}(E_{i})$ for $i=1,2$ such that $\ph_{J}(F_{1})=\ph_{J}(F_{2})\in \cS_{J}$. This means $\s_{J}(x_{1}^{-1}w_{1}x_{1})=\s_{J}(x_{2}^{-1}w_{2}x_{2})$. Now we invoke Lemma \ref{l:EJw}(3). Upon right multiplying $x_{2}$ by an element in $W_{J}$, we may assume $E_{J, x_{1}^{-1}w_{1}x_{1}}=E_{J, x_{2}^{-1}w_{2}x_{2}}$, which we denote by $E$, and that $x_{1}^{-1}w_{1}x_{1}|_{E}=x_{2}^{-1}w_{2}x_{2}|_{E}$, which we denote by $w\in W^{E}$. Then we have morphisms $x_{1}: (E,w)\to (E_{1}, w_{1})$ and $x_{2}: (E,w)\to (E_{2},w_{2})$ in $\cA$, and these morphisms send $F_{J}$ (which is a facet of $E$, because by definition $E$ contains $\frA(J)$) to $F_{1}$ and $F_{2}$ respectively. This means $F_{1}$ and $F_{2}$ are already identified in the colimit $\colim_{(E,w)\in \cA} \Fac_{J}(E)$.  This proves $\ph_{J}$ is injective.
\end{proof}

\begin{remark} Lemma \ref{l:glue apts} tells us that $\frB$ can be constructed as follows: for each $W^{a}$-conjugacy class in $\tilW$, take a representative $w$ and form the quotient space $\frA/C_{W^{a}}(w)$. Then $\frB$ is obtained from the disjoint union of $\frA/C_{W^{a}}(w)$ ($w$ running over representatives of $W^{a}$-conjugacy classes in $\tilW$) by gluing along relevant subspaces of dimension less that that of $\frA$.
\end{remark}

\sss{He-Nie function} Following the idea of He and Nie \cite[]{heMinimalLengthElements2014}, we now define a function $f: \frB\to \RR_{\ge0}$ as follows. For $(E,w)\in \cA$, we have the function $f_{E, w}: E\to \RR_{\ge0}$ defined by $x\mapsto \|x-wx\|^{2}$, using a fixed $W$-invariant positive definitive quadratic form $\|\cdot\|^{2}$ on $V=\xcoch(T)_{\RR}$. For any morphism $g: (E,w)\to (E',w')$ in $\cA$, one checks that 
\begin{equation*}
f_{E, w}=f_{E',w'}\c g: E\to \RR_{\ge0}.
\end{equation*}
Therefore $\{f_{E,w}\}_{(E,w)\in \cA}$ gives a piecewise smooth function on $\colim_{(E,w)\in \cA}E\cong \frB$, which we denote by 
$$f: \frB\to \RR_{\ge0}.$$ 

Moreover, using the quadratic form $\|\cdot\|^{2}$ restricted to $E$, the differential $df_{E,w}$ turns into a gradient vector field $\nb f_{E,w}$ on $E$. Since $f_{E,w}$ is quadratic, $\nb f_{E,w}$ is a linear vector field. The next lemma shows that for varying $(E,w)\in \cA$, the vector fields $\nb f_{E,w}$ assemble to a piecewise-linear continuous vector field $\nb f$ on $\frB$. 

\begin{lemma}\label{l:gradf}
If $g: (E,w)\to (E', w')$ is a morphism in $\cA$, then $g_{*}$ takes the vector field $\nb f_{E,w}$ on $E$ to the vector field $\nb f_{E',w'}|_{gE}$ on $gE$.
\end{lemma}
\begin{proof} Let $(E,w)\in \cA$. Let $V_{E}\subset \xcoch(T)_{\RR}$ be the vector space parallel to $E$, so that $w-1$ is a map $E\to V_{E}$. Let $\ov w\in \End(V_{E})$ be the linear part of $w$. Direct calculation shows 
\begin{equation}\label{nbf}
(\nb f_{E,w})(x)=2(\ov w-1)^{*}(w-1)(x), \quad \forall x\in E.
\end{equation}
Here $(\ov w-1)^{*}\in \End(V_{E})$ is the adjoint of $(\ov w-1)\in \End(V_{E})$ under the quadratic form $\|\cdot \|^{2}$ on $V_{E}$.

Replacing $(E,w)$ by $(gE, gwg^{-1})$, we may assume that $E\subset E'$, $w'(E)=E$, $w'|_{E}=w$ and $g$ is the identity element. Since $w'$ preserves $E$, the endomorphism $\ov w-1$ of $V_{E'}$ preserves $V_{E}$, and so is its adjoint. Hence by \eqref{nbf}, $\nb f_{E'
,w'}|_{E}=\nb f_{E,w}$.
\end{proof}

Let $\Crit(f)\subset \frB$ be the vanishing locus of $\nb f$. Let $\Crit(f)_{\wt\nu}=\Crit(f)\cap \frB_{\wt\nu}$.

\begin{lemma}\label{l:ess contains crit}
For any $\wt\nu\in\wt\NP$, the critical locus $\Crit(f)_{\wt\nu}$ is contained in the essential part $\frB^{\hs}_{\wt\nu}$.
\end{lemma}
\begin{proof}
Let $x\in \Crit(f)_{\wt\nu}$, then there exists some $(E,w)\in \cA$ (where $w\in W^{E}$) such that $x$ is the image of a critical point $y$ of $f_{E,w}$ under $|\ph_{E,w}|: E\to \frB$. We choose such a $(E,w)$ with $E$ minimal. Choose a connected component $C$ of $\frA-\cup_{E\subset H}H$ (remove all affine root hyperplanes that contain $E$). Let $\tilw\in \Stab_{\tilW}(E)$ be the lifting of $w$ such that $\tilw(C)=C$. Then $g=1$ gives a morphism $(E,w)\to (\frA,\tilw)$ in $\cA$. We have $y\in \Crit(f_{E,w})=E\cap \Crit(f_{\frA,\tilw})$. In particular, $\nu=\nu(\tilw)$. Apply \cite[Lemma 2.6]{heMinimalLengthElements2014} to the subspace $\Crit(f_{E,w})$ of $\Crit(f_{\frA,\tilw})$, and an alcove $A\subset C$ that contains $y$ in its closure, we conclude that $\tilw_{A}:=pos(A,\tilw A)$ (relative position, which is in the same $W^{a}$-conjugacy class of $\tilw$) has length $\j{2\r, \nu}$. The image of $A$ under $|\ph_{\frA,\tilw}|: \frA\to \frB$ is the maximal facet indexed by $\tilw_{A}\in \cS_{\vn}=\tilW$. Since $\ell(\tilw_{A})=\j{2\r, \nu}$, the closure of the maximal facet $\frB(\f{\tilw_{A}}{\vn})$ is in the essential part $\frB^{\hs}_{\wt\nu}$. On the other hand, $y\in \ov A$ implies $x$ is in the closure of $\frB(\f{\tilw_{A}}{\vn})$, hence $x\in \frB^{\hs}_{\wt\nu}$.
\end{proof}

\begin{remark} It is likely that $\frB_{\wt\nu}^{\hs}$ is the smallest $\cD^{\c}$-subset of $\frB_{\wt\nu}$ whose geometric realization contains $\Crit(f)_{\wt\nu}$. In other words, it should be true that for any straight element $w$, the critical locus of $f_{\frA, w}: \frA\to \RR$ intersects the interior of the fundamental alcove.
\end{remark}

Below we prove a topological property of certain subsets of $\frB_{\wt\nu}$. It is a key ingredient in the proof of Theorem \ref{thm:main in text}.  Fix $\wt\nu=(\nu,\om)\in \wt\NP$.

\begin{defn}\label{def:downward} Let $\wt\nu\in\wt\NP$. A $\cD^{\c}$-subset $\cS'\subset \cS_{\wt\nu}$ is a called {\em downward} if it satisfies:
\begin{itemize}
\item $\Tot(\cS')$ is finite.
\item $\cS'$ contains $\cS^{\hs}_{\wt\nu}$.
\item Let $\f{u}{J}\in \cS_{\wt\nu,J}, \f{u'}{J'}\in \cS_{\wt\nu,J'}$ satisfy $\ell(u)\le \ell(u')$ and $\t(\f{u}{J})\le \t(\f{u'}{J'})$. If $\f{u'}{J'}\in \cS'_{J'}$, then $\f{u}{J}\in \cS'_{J}$.
\end{itemize}
A subspace $\frB'\subset \frB_{\wt\nu}$ is called {\em downward} if it is of the form $|\cS'|$ for a downward $\cD^{\c}$-subset $\cS'$ of $\cS_{\wt\nu}$. 
\end{defn}

\begin{exam}\label{ex:downward}
\begin{enumerate}
\item $\frB^{\hs}_{\wt\nu}$ is the smallest downward subspace of $\frB_{\wt\nu}$. 
\item Let $n\ge\j{2\r, \nu}$ and $[E]\in \un\frE$. Let $\cS_{\wt\nu,\le(n,[E])}\subset \cS_{\wt\nu}$ be the $\cD^{\c}$-subset consisting of $\f{u}{J}\in \cS_{\wt\nu,J}$ such that $\ell(u)\le n$ and $\t(\f{u}{J})\le[E]$. The fact that this is a $\cD^{\c}$-subset follows from Lemma \ref{l:ell tau under closure}.

When $n>\j{2\r, \nu}$, $\cS_{\wt\nu,\le(n,[E])}$ is downward. When $n=\j{2\r, \nu}$ and $[E]=[\frA]$, we have $\cS_{\wt\nu,\le(\j{2\r,\nu},[\frA])}=\cS^{\hs}_{\wt\nu}$. We denote $\frB_{\wt\nu,\le(n,[E])}=|\cS_{\wt\nu,\le(n,[E])}|$.

\item By definition, any downward subspace $\frB'\subset \frB_{\wt\nu}$ is a finite union of the form
\begin{equation}\label {downclosed union}
\frB'=\frB^{\hs}_{\wt\nu}\cup\left(\bigcup_{i}\frB_{\wt\nu, \le(n_{i}, [E_{i}])}\right)
\end{equation}
where $n_{i}>\j{2\r, \nu}$.
\end{enumerate}
\end{exam}

By Lemma \ref{l:ess contains crit}, for any downward $\frB'_{\wt\nu}\subset \frB_{\wt\nu}$, we have $\Crit(f)_{\wt\nu}\subset \frB^{\hs}_{\wt\nu}\subset \frB'_{\wt\nu}$.


\begin{prop}\label{p:contractible} Fix $\wt\nu=(\nu,\om)\in \wt\NP$. Let  $\frB'\subset \frB_{\wt\nu}$ be a downward subspace. Then the inclusion $\Crit(f)_{\wt\nu}\incl \frB'$ admits a deformation retract. In particular, both inclusions $\Crit(f)_{\wt\nu}\subset \frB^{\hs}_{\wt\nu}\subset \frB'$ are homotopy equivalences.
\end{prop}
\begin{proof} 
Abbreviate $\Crit(f)_{\wt\nu}$ by $C_{\wt\nu}$. By Lemma \ref{l:gradf}, the gradient flow of $f$ is well-defined as a flow  $\Phi_{t}$ on $\frB_{\wt\nu}$, for $t\in \RR$. The deformation retract will be constructed using the flow $\Phi_{t}$.

\begin{claim} The downward subspace $\frB'$ is stable under the  flow $\Phi_{t}$,  for  $t\le 0$ non-positive.
\end{claim}
\begin{proof}[Proof of Claim] Since $\frB'$ is a union of the form \eqref{downclosed union}, it suffices to show that $\frB^{\hs}_{\wt\nu}$ and $\frB_{\wt\nu, \le (n,[E])}$ (where $n>\j{2\r, \nu}$ )are stable under the flow $\{\Phi_{t}\}_{t\ge0}$.  Since $\frB^{\hs}_{\wt\nu}=\frB_{\wt\nu, \le (\j{2\r,\nu},[\frA])}$, we suffices to show that $\frB_{\wt\nu, \le (n,[E])}$ is stable under  the flow $\{\Phi_{t}\}_{t\ge0}$ whenever $n\ge\j{2\r,\nu}$.

First consider the case $E=\frA$. Since every point of $\frB_{\wt\nu}$ lies in the image of  $|\ph_{\frA, w}|: \frA\to \frB$ for some $w\in \tilW$ with $\wt\nu(w)=\wt\nu$, and the flow stays inside the image of $|\ph_{\frA, w}|$, we only need to prove the same statement for $\frA$ with respect to the flow defined by $\nb f_{\frA, w}$. Let $A_{0}\subset \frA$ be the (closed) fundamental alcove. For any alcove $A=yA_{0}$, let $w_{A}=y^{-1}wy$. Let $\frA_{w,\le n}\subset \frA$ be the union of alcoves $A$ such that $\ell(w_{A})\le n$. Then $\frA_{w,\le n}=|\ph_{\frA,w}|^{-1}(\frB_{\wt\nu, \le(n, [\frA])})$. We only need to show that $\frA_{w,\le n}$ is stable under the flow $\Phi_{t}$ of $\nb f_{\frA,w}$ for $t\le 0$. Let $\frZ\subset \frA$ be the union of all $H\cap H'$ where $H$ and $H'$ run over distinct affine root hyperplanes. Let $U
\subset \frA_{w,\le n}$ be the subset of $x\in \frA_{w,\le n}$ such that $\Phi_{t}(x)\notin \frZ$ for all $t\le 0$. Then $U$ is dense in $\frA_{w,\le n}$ (using that $\frZ$ has codimension two in $\frA$, so $\cup_{t\ge0}\Phi_{t}(\frZ)$ has codimension at least one in $\frA$).  Therefore it is enough to show that $\Phi_{t}(x)\in \frA_{w,\le n}$ for $x\in U$ and $t\le0$. Suppose this is not the case, then for some  $x\in U$ and some $t\le0$,  $\Phi_{t}(x)$ lies in an alcove $A$ with $\ell(w_{A})>n$. Let $t_{0}$ be the supremum of such $t$. Then $\Phi_{t_{0}+\e}(x)\in A$ with $\ell(w_{A})\le n$ for small $\ep>0$ and $\Phi_{t_{0}-\e}(x)\in A'$ with $\ell(w_{A'})> n$ for small $\ep>0$. Moreover, $x_{0}:=\Phi_{t_{0}}(x)$ is on the common face $A\cap A'$ of $A$ and $A'$, and does not lie on any other affine root hyperplane. Let $\bn$ be a normal vector of $H$ that points to $A'$. Then 
\begin{equation}\label{normal derivative}
\j{\nb f_{\frA,w}(x_{0}), \bn}\ge0.
\end{equation}
We have $w_{A'}=sw_{A}s$ where $s$ is the simple reflection determined by hyperplane $H=\Span(A\cap A')$. Hence $\ell(w_{A'})=\ell(w_{A})+2$. Applying \cite[Lemma 2.1]{heMinimalLengthElements2014}, we see that $\j{\nb f_{\frA,w}(x_{0}), \bn}>0$ (note that $x_{0}=\Phi_{t_{0}}(x)$ is a regular point of $A\cap A'$ since $x_{0}\notin \frZ$).  This contradicts \eqref{normal derivative}.

Now we consider the case of a general $[E]\in\un\frE$. For any $x\in\frB_{\wt\nu, \le (n,[E])}$ we may find $(E,w)\in \cA$ (with  $E\in [E]$) such that $x$ lies in the image of $|\ph_{E,w}|: E\to \frB$. Hence it is enough to prove the analogous statement for $E$ with respect to the gradient flow of $f_{E,w}$. Let $\wt w\in \tilW$ be a lifting of $w$. Then $|\ph_{E,w}|$ is the composition $E\incl \frA\xr{|\ph_{\frA,\wt w}|}\frB$. In particular,  $E\cap |\ph_{E,w}|^{-1}(\frB_{\wt\nu, \le(n,[E])})=\frA\cap  |\ph_{\frA,\tilw}|^{-1}(\frB_{\wt\nu, \le(n,[E])})$. Moreover, the gradient flow of $f_{E,w}$ is the same as the gradient flow of $f_{\frA,\wt w}$ restricted to $E$ by Lemma \ref{l:gradf}. Therefore the case $[E]=[\frA]$ proved above implies the case of a general $[E]$.
\end{proof}



Now for any $x\in \frB_{\wt\nu}$, $\lim_{t\to-\infty}\Phi_{t}(x)\in C_{\wt\nu}$. Indeed, we may assume $x$ lies in the image of  $|\ph_{\frA, w}|: \frA\to \frB_{\wt\nu}$ for some $w\in \tilW$, and the corresponding statement is \cite[Lemma 2.3]{heMinimalLengthElements2014}. Moreover, the calculation in {\em loc.cit.} shows that the flow is contracting in a neighborhood of $C_{\wt\nu}$ as $t\to -\infty$. This implies that the map $H:[0,1)\times \frB'\to \frB'$ given by $H(s,x)=\Phi_{\log(1-s)}(x)$ can be extended to a continuous function $H: [0,1]\times \frB'\to \frB'$ by letting $H(1,x)=\lim_{t\to -\infty}\Phi_{t}(x)\in C_{\wt\nu}$. Then $H$ gives a deformation retract from $\frB'$ to $C_{\wt\nu}$. This proves the proposition.
\end{proof}

\subsection{Geometric pieces and sheaves on them}\label{ss:geom pieces}

Here we turn to the geometry indexed by the prior combinatorics, in particular sheaves on geometric pieces and the natural functors between them.

\sss{Cyclic reduction}
Consider the following general situation. Let $G$ be an algebraic group, and $P, P'\subset G$ two parabolic subgroups with unipotent radicals $P^{u}$ and $P'^{u}$ and Levi quotients $L$ and $L'$ respectively.  Let $x\in G$. Let $\d: L'\isom L$ be an isomorphism. Let $Q'=\Im(P\cap \Ad(x^{-1})P'\to L)\subset L$, $Q=\d(\Im(P'\cap \Ad(x)P\to L'))\subset L$. Then $Q,Q'$ are parabolic subgroups of $L$. Let $Q^{u}$ and $Q'^{u}$ be the unipotent radicals of $Q$ and $Q'$ and $M$ and $M'$ be their Levi quotients. Then $M'=(P\cap \Ad(x^{-1})P')^{\red}$ (where $(-)^{\red}$ denotes Levi quotient).  We have a canonical isomorphism
\begin{equation*}
\d x: M'=(P\cap \Ad(x^{-1})P')^{\red}\xr{\Ad(x)}(P'\cap \Ad(x)P)^{\red}\cong \Im(P'\cap \Ad(x)P\to L')\xr{\d} M.
\end{equation*}

\begin{lemma}[Cyclic reduction]\label{l:cyc red}
There is a canonical map
\begin{equation}\label{e:cyc red}
\f{P'^{u}\bs (P'xP)/P^{u}}{\Ad_{\d }(L')}
\to \f{Q'^{u}\bs L/Q^{u}}{\Ad_{\d x}(M')} 
\end{equation}
sending $p'xp$ to $\ov p\d(\ov p')$ (where $\ov p$ is the image of $p$ under $P\surj L$; similar for $\ov p'$). This map is a gerbe for the unipotent group $P'^{u}\cap \Ad(x)P^{u}$.

We will refer to the above map~\eqref{e:cyc red}  as a {\em cyclic reduction}.
\end{lemma}
\begin{proof}
The left $P'$-translation and right $P$-translation on $P'xP$ gives a $P'\times P$-equivariant isomorphism
\begin{equation}\label{PxP}
P'\times^{P'\cap \Ad(x)P}P\isom P'xP,  \quad  (p',p)\mapsto p'xp
\end{equation}
where the action of $P'\cap \Ad(x)P$ on $P$ is by $\Ad(x^{-1}): P'\cap \Ad(x)P\isom \Ad(x^{-1})P'\cap P$ and left translation.

Now quotienting both sides of \eqref{PxP} by left $P'^{u}$, right $P^{u}$ and $\Ad_{\d}(L')$-actions, we get isomorphisms
\begin{equation}\label{Lq1}
\f{L'\times^{P'\cap \Ad(x)P}L}{\Ad_{\d}(L')}\xleftarrow{\sim}\f{P'^{u}\bs P'\times^{P'\cap \Ad(x)P}P/P^{u}}{\Ad_{\d}(L')}\isom \f{P'^{u}\bs P'xP/P^{u}}{\Ad_{\d}(L')}.
\end{equation}
Here the action on $P'\cap \Ad(x)P$ on $L'$  is via the projection $P'\cap \Ad(x)P\to L'$, whose image is $\d^{-1}(Q)\subset L'$, and right translation on $L'$. Similarly, the action on $P'\cap \Ad(x)P$ on $L$  is via the projection $P'\cap \Ad(x)P\isom \Ad(x^{-1})P'\cap P\to Q'\subset L$ and left translation. The normal subgroup $P'^{u}\cap \Ad(x)P^{u}$ acts trivially on $L'\times L$, and
\begin{equation}\label{PQ}
(P'\cap \Ad(x)P)/(P'^{u}\cap \Ad(x)P^{u})\cong \d^{-1}(Q)\times_{M'}Q'
\end{equation}
where the projection $\d^{-1}(Q)\to M'$ is the composition $\d^{-1}(Q)\xr{\d} Q\surj M\xr{(\d x)^{-1}}M'$.

Using \eqref{PQ} we get a $P'^{u}\cap \Ad(x)P^{u}$-gerbe
\begin{equation*}
\f{L'\times^{P'\cap \Ad(x)P}L}{\Ad_{\d}(L')}\to \f{L'\times^{\d^{-1}(Q)\times_{M'}Q'}L}{\Ad_{\d}(L')}=\f{(L'/\d^{-1}(Q^{u}))\times^{M'}(Q'^{u}\bs L)}{\Ad_{\d}(L')}.
\end{equation*}
Via the map $(\ell',\ell)\mapsto \ell\d(\ell')$ (for $\ell\in L, \ell'\in L'$), we have an isomorphism
\begin{equation*}
\f{(L'/\d^{-1}(Q^{u}))\times^{M'}(Q'^{u}\bs L)}{\Ad_{\d}(L')}\isom\f{Q'^{u}\bs L/Q^{u}}{\Ad_{\d x}(M')}.
\end{equation*}
Composing all these isomorphisms we get a $P'^{u}\cap \Ad(x)P^{u}$-gerbe
\begin{equation*}
\f{P'^{u}\bs (P'xP)/P^{u}}{\Ad_{\d }(L')}
\to\f{Q'^{u}\bs L/Q^{u}}{\Ad_{\d x}(M')}.
\end{equation*}
\end{proof}

Lemma \ref{l:cyc red} gives an equivalence of categories by pullback
\begin{equation}\label{Sh cyc red}
\Sh(\f{Q'^{u}\bs L/Q^{u}}{\Ad_{\d x}(M')})\simeq \Sh(\f{P'^{u}\bs (P'xP)/P^{u}}{\Ad_{\d }(L')}).
\end{equation}

\sss{Nilpotent sheaves under cyclic reduction}\label{sss:nilp cyc red}
Next we show that sheaves with nilpotent singular support correspond to each other under the above equivalence. To simplify the statements, we introduce the following terminology: For a group $H$ acting on a smooth variety $X$, and $\cF\in \Sh(X)$, we say $\cF$ is $H$-nilpotent if $\mu_{H}(SS(\cF))$ lies in the nilpotent cone of $\frh^{*}=(\Lie H)^{*}$. When there is ambiguity as to how $H$ acts on $X$, we will specify $\cF$ is $H$-nilpotent for which $H$-action.

Let
\begin{equation*}
\Sh_{\cN}(\f{P'^{u}\bs (P'xP)/P^{u}}{\Ad_{\d }(L')})\subset \Sh(\f{P'^{u}\bs (P'xP)/P^{u}}{\Ad_{\d }(L')})
\end{equation*}
be the full subcategory consisting of objects that are $L'$-nilpotent for left translation by $L'$ (equivalently, $L$-nilpotent for the right translation). Similarly define the full subcategory
\begin{equation*}
\Sh_{\cN}(\f{Q'^{u}\bs L/Q^{u}}{\Ad_{\d x}(M')})\subset \Sh(\f{Q'^{u}\bs L/Q^{u}}{\Ad_{\d x}(M')})
\end{equation*}
using either $M'$-nilpotence for the left translation or $M$-nilpotenve for the right translation.

\begin{lemma}\label{l:ShN cyc red} Under the equivalence \eqref{Sh cyc red}, the full subcategories $\Sh_{\cN}(\f{P'^{u}\bs (P'xP)/P^{u}}{\Ad_{\d }(L')})$ and $\Sh_{\cN}(\f{Q'^{u}\bs L/Q^{u}}{\Ad_{\d x}(M')})$ correspond. In particular, pullback via the cyclic reduction map induces an equivalence
\begin{equation*}
\Sh_{\cN}(\f{Q'^{u}\bs L/Q^{u}}{\Ad_{\d x}(M')})\simeq \Sh_{\cN}(\f{P'^{u}\bs (P'xP)/P^{u}}{\Ad_{\d }(L')}).
\end{equation*}
\end{lemma}
\begin{proof}
Choose a section to $P\surj L$ and realize $L$ as a Levi subgroup of $P$. Similarly realize $L'$ as a subgroup of $P'$, $M$ as a subgroup of $Q'$ and $M$ as a subgroup of $Q$. 

We have a commutative diagram
\begin{equation*}
\xymatrix{ P'\times P \ar[d]^{\pi_{x}}\ar[r]^{p'\times p}& L'\times L\ar[r]^{m_{\d}} & L\ar[d]^{\pi}\\
\f{P'^{u}\bs (P'xP)/P^{u}}{\Ad_{\d }(L')} \ar[rr]^{c} &&\f{Q'^{u}\bs L/Q^{u}}{\Ad_{\d x}(M')})}
\end{equation*}
where the bottom arrow is the cyclic reduction map \eqref{e:cyc red}, $p,p'$ are the projections and $m_{\d}(\ell',\ell)=\ell\d(\ell')$, $\pi_{x}(g',g)=g'xg$, and $\pi$ is the quotient map.

Let $\cF\in \Sh(\f{Q'^{u}\bs L/Q^{u}}{\Ad_{\d x}(M')})$, and $\wt\cF=\pi^{*}\cF$ be its pullback to $L$. Let $\cK=c^{*}\cF\in\Sh(\f{P'^{u}\bs P'xP/P^{u}}{\Ad_{\d}(L')})$,
 and $\wt\cK=\pi_{x}^{*}\cK$ be its pullback to $P'\times P$. By definition, $\cF\in \Sh_{\cN}(\f{Q'^{u}\bs L/Q^{u}}{\Ad_{\d x}(M')})$ if and only if $\wt\cF$ is $M'$-nilpotent for the left translation of $M'$ on $L$; $\cK\in \Sh_{\cN}(\f{P'^{u}\bs P'xP/P^{u}}{\Ad_{\d}(L')})$ if and only if $\wt\cK$ is $L$-nilpotent for the right translation actions on the $P$-factor of $P'\times P$. Therefore, we reduce to prove that the following are equivalent:
 \begin{enumerate}
\item $\wt\cF$ is $M'$-nilpotent for the left translation on $L$;
\item $\wt\cK=(p'\times p)^{*}m_{\d}^{*}\wt\cF$  is $L$-nilpotent  for the right translation on the $P$-factor.
\end{enumerate}
Since $p'\times p$ is smooth and equivariant for the right translation of $L$, (2) is equivalent to
\begin{enumerate}
\item[(3)] $m_{\d}^{*}\wt\cF$ is $L$-nilpotent for the right translation on the $L$-factor.
\end{enumerate}
It is easy to see that on $L'\times L$, $L$-nilpotence with respect to the left translation is equivalent to $L$-nilpotence with respect to the right translation. Also $m_{\d}$ is equivariant with respect to the left translation action of $L$ (but not the the right translation), therefore $(3)$ is equivalent to
\begin{enumerate}
\item[(4)] $\wt\cF$ is $L$-nilpotent for the left translation.
\end{enumerate}
It remains to prove (1)$\iff$(4).  Let $\mu_{L}: T^{*}L\to \frl^{*}$ be the moment map for the left transaltion of $L$. Recall $\wt \cF$ is $Q'^{u}$-equivariant for the left translation action, so $\mu_{L}(SS(\wt\cF))\in \frn_{Q'}^{\bot}$ (here $\frn_{Q'}=\Lie Q'^{u}$). Then (1) means the image of $\mu_{L}(SS(\wt\cF))$ under the projection $\frn_{Q'}^{\bot}\to \frakm'^{*}$ is nilpotent, which is equivalent to saying that $\mu_{L}(SS(\wt\cF))$ lies in the nilpotent cone of $\frn_{Q'}^{\bot}$, which is (4).
\end{proof}

\sss{Geometric pieces}\label{sss:geom piece}
Recall from Section~\ref{sss:geom trace} that
\begin{equation*}
\cY_{J}=\frac{\cP^{u}_{J}\bs \cG/\cP_{J}^{u}}{L_{J}}
\end{equation*}
regarded as an ind-stack over $k$. For $J\sft I^{a}$, Lusztig \cite[\S3]{lusztigParabolicCharacterSheaves} defined a stratification of $\cY_{J}$ indexed by $\cS_{J}$: for $\f{u}{J}\in \cS_{J}$, $\cY(\f{u}{J})$ is the locally closed substack of $\cY_{J}$ defined as the image of the projection 
\begin{equation*}
\cY(\f{u}{J})=\Im(\cI u\cI \subset \cG\to \cY_{J}).
\end{equation*}
We shall call $\cY(\f{u}{J})$ {\em geometric $J$-pieces}. For each $w\in \tilW$ and any lifting $\dot w\in N_{\cG}(T)$, $\dot w\in \cY(\frac{u}{J})$ if and only if $\s_J(\dot w)=\f{u}{J}$. 


Below we recall an inductive construction of $\cY(\f{u}{J})$ that leads to a description of it in terms of a twisted adjoint quotient, also due to Lusztig.

Let $(J_{n}, J'_{n}, u_{n})\in \cS_{J}$ corresponding to $u\in {}^{J}\tilW$. We describe the geometric piece $\cY(\f{u}{J})$ using cyclic reduction.

Choose a lifting $\dot u_{n}$ for each $u_{n}$ in $N_{L_{J_{n-1}}}(T)$ (here $L_{J_{-1}}$ is understood to be $\cG$). For $n\ge0$, let $P'_{n+1}\subset L_{J_{n}}$ be the standard parabolic subgroup whose Levi is $L_{J'_{n+1}}$; let  $P_{n+1}\subset L_{J_{n}}$ be the standard parabolic subgroup whose Levi is $L_{J_{n+1}}$. For $n\ge0$ define
\begin{equation*}
\cZ_{n}(u_{0},\cdots, u_{n})=\f{P'^{u}_{n+1}\bs L_{J_{n}}/P^{u}_{n+1}}{\Ad_{\dot u_{0}\cdots \dot u_{n}}(L_{J'_{n+1}})}.
\end{equation*}
If we define $L_{J_{-1}}=\cG$, $P'_{0}=P_{0}=\cP_{J}$, then we can define $\cZ_{-1}$ using the above formula, and have
$$\cZ_{-1}=\f{\cP^{u}_{J}\bs \cG/\cP_{J}^{u}}{L_{J}}=\cY_{J}.$$
We have maps
\begin{equation*}
\xymatrix{\cZ_{n}(u_{0},\cdots, u_{n}) & \cZ_{n}(u_{0},\cdots, u_{n+1}):=\f{P'^{u}_{n}\bs P'_{n}u_{n+1}P_{n}/P^{u}_{n}}{\Ad_{\dot u_{0}\cdots \dot u_{n}}(L_{J'_{n+1}})}\ar[r]\ar@{_{(}->}[l] & \cZ_{n+1}(u_{0},\cdots, u_{n+1})}
\end{equation*}
where the second map is the cyclic reduction in Lemma \ref{l:cyc red}. The geometric piece $\cY(\f{u}{J})$ is defined to be the fiber product for $n\gg0$
\begin{equation}\label{defn piece}
\cY(\f{u}{J}):=\cZ_{-1}(u_{0})\times_{\cZ_{0}(u_{0})}\cZ_{0}(u_{0}, u_{1})\times_{\cZ_{1}(u_{0},u_{1})}\cdots \times_{\cZ_{n-1}(u_{0},\cdots, u_{n-1})}\cZ_{n-1}(u_{0},\cdots, u_{n}).
\end{equation}

The ``shape'' of the geometric piece $\cY(\f{u}{J})$ is described by Lusztig:
\begin{prop}[Lusztig {\cite[3.14]{lusztigParabolicCharacterSheaves}}] Let $(J_{n},J'_{n}, u_{n})_{n\ge0}\in \cS_{J}$. Choose liftings $\dot u_{n}\in N_{L_{J_{n-1}}}(T)$ for $n\ge0$ such that  $\dot u_{n}=1$ whenever $u_{n}=1$ (as usual, we understand $L_{J_{-1}}=\cG$).  Let $\dot u=\dot u_{0}\cdots \dot u_{n}$ for $n\gg0$.  Let $K=I(J,u)\subset J$. Then there is a canonical map
\begin{equation*}
\pi_{J, u}:\cY(\f{u}{J})\to \f{L_{K}}{\Ad_{\dot u}(L_{K})}
\end{equation*}
which is an iterated gerbe for (pro-)unipotent groups. 

One can eliminate the choice of $\dot u$ by writing the right side as $\f{uL_{K}}{L_{K}}$.
\end{prop}
\begin{proof}
Let $n\gg0$ such that $J_{n}=K$. In \eqref{defn piece},  projection to the last factor followed by cyclic reduction
\begin{equation*}
\cY(\f{u}{J})\to \cZ_{n-1}(u_{0},\cdots, u_{n})\to \cZ_{n}= \f{L_{K}}{\Ad_{\dot u}(L_{K})}
\end{equation*}
gives the desired map, which is an iterated gerbe for (pro-)unipotent groups. 
\end{proof}


\begin{cor}\label{c:sh geom piece} For each piece $\f{u}{J}$ with $K=I(J,u)$, pullback along $\pi_{J,u}$ induces an equivalence of categories
\begin{equation}\label{sh geom piece}
\pi^{*}_{J,u}: \Sh(\f{uL_{K}}{L_{K}})\simeq \Sh(\f{L_{K}}{\Ad_{\dot u}(L_{K})})\isom\Sh(\cY(\f{u}{J})).
\end{equation}

\end{cor}

%

\sss{Partial order}
Let $\ge$ be the Bruhat partial order on $\tilW$. In \cite[3.8, 3.9, 3.13]{heGstablePiecesWonderful2007}, a partial order $\ge_{J}$ on ${}^{J}\tilW$ is defined as follows. For $u,u'\in {}^{J}\tilW$, define $u\ge_{J}u'$ if there is a $u''$ in the same $W_{J}$-conjugacy class of $u'$ such that $u\ge u''$. 

Let $w,w'\in \tilW$ be in the same $W_{J}$-conjugacy class. Recall from {\em loc. cit.} that $w'$ is said to be obtained from $w$ by a $J$-cyclic shift if $\ell(w')=\ell(w)$ and $w'=sws$ for some simple reflection $s\in J$ such that either $\ell(sw)=\ell(w)-1$ or $\ell(ws)=\ell(w)-1$. Denote $w\sim_{J}w'$ if $w'$ can be obtained from $w$ by a sequence of $J$-cyclic shifts.   It is shown in {\em loc. cit.} that $u\ge_{J}u'$ if and only if there exists $u''\sim_{J}u'$ such that $u\ge u''$.

\begin{theorem}[He, {\em loc. cit.} Theorem 4.5]\label{th:He closure} For $u\in {}^{J}\tilW$, the closure of $\cY(\f{u}{J})$ is the union of $\cY(\f{u'}{J})$ for $u\ge_{J}u'$.
\end{theorem}

\sss{The functor $\ch^{J'}_{J}$} 
Let $J\subset J'\sft I^{a}$. Then the image of $\cP_{J}$ under the projection $\cP_{J'}\to L_{J'}$ is a parabolic subgroup $P^{J'}_{J}\subset L_{J'}$. Consider the diagram 
\begin{equation}\label{ind diagram Y}
\xymatrix{ \cY_{J}=\frac{\cP_{J}^{u}\bs \cG/\cP^{u}_{J}}{L_{J}} & \frac{\cP_{J'}^{u}\bs \cG/\cP^{u}_{J'}}{P^{J'}_{J}}\ar[l]_-{q^{J'}_{J}}\ar[r]^-{p^{J'}_{J}} & \frac{\cP_{J'}^{u}\bs \cG/\cP^{u}_{J'}}{L_{J'}}=\cY_{J'}
}
\end{equation}
Define 
\begin{equation*}
\ch^{J'}_{J}=(p^{J'}_{J})_{!}(q^{J'}_{J})^{*}: \Sh(\cY_{J})\to \Sh(\cY_{J'}).
\end{equation*}


Finally, we recall the following  key geometric input we will need  to understand the natural transforms between  sheaves on different geometric pieces. We are grateful to Xuhua He for  providing it  in a recent paper~\cite{heGENERALIZATIONCYCLICSHIFT}. Note that the sheaves involved in the following theorem need not have nilpotent singular support.

\begin{theorem}[He, \cite{heGENERALIZATIONCYCLICSHIFT}]\label{th:ch} Let $J\subset J'\sft I^{a}$, $u\in {}^{J}\tilW$ and $\f{u'}{J'}=\d^{J'}_{J}(\f{u}{J})$. Let $K=I(J,u)\subset J$ and $K'=I(J',u')\subset J'$. Let $i_{J,u}: \cY(\f{u}{J})\incl \cY_{J}$ and $i_{J',u'}: \cY(\f{u'}{J'})\incl \cY_{J'}$ be the inclusions. 
\begin{enumerate}
\item Suppose $\f{u}{J}$ is not quasi-$J'$-reduced (namely $\ell(u)>\ell(u')$). Then the image of $\ch^{J'}_{J}\c i_{J,u!}: \Sh(\cY(\f{u}{J}))\to \Sh(\cY_{J'})$ is supported on the closed substack $\cY_{J', <\ell(u)}=\cup_{\ell(v)<\ell(u)}\cY(\f{v}{J'})$.

\item\label{th ch:qred} Suppose $\f{u}{J}$ is quasi-$J'$-reduced (recall this means $\ell(u)=\ell(u')$). Then there is a unique $x\in W_{J'}\cap \tilW^{K'}$ such that $x^{-1}ux=u'$ and $K_{1}:=x^{-1}(K)\subset K'$ (note that $u'(K_{1})=K_{1}$). 

Let
\begin{equation*}
\Ad(x^{-1}): \Sh(\f{uL_{K}}{L_{K}})\to \Sh(\f{u'L_{K_{1}}}{L_{K_{1}}})
\end{equation*}
be the functor induced by conjugation by $\dot x^{-1}$, where $\dot x\in N_{L_{J'}}(T)$ is any lifting of $x$. \footnote{Two liftings of $x$ differ by multiplication by $t\in T\subset L_{K}$, therefore the resulting functors $\Sh(\f{uL_{K}}{L_{K}})\to \Sh(\f{u'L_{K_{1}}}{L_{K_{1}}})$ are canonically isomorphic.}

Consider also the induction functor
\begin{equation*}
{}^{u'}\Ind^{K'}_{K_{1}}=b^{K'}_{K_{1}, !}(a^{K'}_{K_{1}})^{*}: \Sh(\f{u'L_{K_{1}}}{L_{K_{1}}})\to \Sh(\f{u'L_{K'}}{L_{K'}})
\end{equation*}
given by the correspondence
\begin{equation*}
\xymatrix{\f{u'L_{K_{1}}}{L_{K_{1}}} & \f{u'P^{K'}_{K_{1}}}{P^{K'}_{K_{1}}}\ar[l]_-{a^{K'}_{K_{1}}}\ar[r]^-{b^{K'}_{K_{1}}} & \f{u'L_{K'}}{L_{K'}}}.
\end{equation*}
Then the outer square of the following diagram is commutative
\begin{equation*}
\xymatrix{\Sh(\f{uL_{K}}{L_{K}})\ar[r]^-{\Ad(x^{-1})} \ar[d]^{\pi_{J,u}^{*}}_{\wr} & \Sh(\f{u'L_{K_{1}}}{L_{K_{1}}}) \ar[r]^-{{}^{u'}\Ind^{K'}_{K_{1}}} & \Sh(\f{u'L_{K'}}{L_{K'}})\ar[d]^{\pi_{J',u'}^{*}}_{\wr}\\
\Sh(\cY(\f{u}{J}))\ar@{-->}[rr]^-{\g^{J',u'}_{J,u}}\ar[d]^{i_{J,u!}} && \Sh(\cY(\f{u'}{J'}))\ar[d]^{i_{J',u'!}}\\
\Sh(\cY_{J})\ar[rr]^-{\ch^{J'}_{J}} && \Sh(\cY_{J'})
}
\end{equation*}
The same is true when $i_{J,u!}$ and $i_{J',u'!}$ are replaced with $i_{J,u*}$ and $i_{J',u'*}$ respectively.

We define the functor $\g^{J',u'}_{J,u}:\Sh(\cY(\f{u}{J}))\to \Sh(\cY(\f{u'}{J'}))$ as the composition
\begin{equation}\label{ind on pieces}
\g^{J',u'}_{J,u}: \Sh(\cY(\f{u}{J}))\xr{(\pi^{*}_{J,u})^{-1}}\Sh(\f{uL_{K}}{L_{K}})\xr{\Ad(x^{-1})}\Sh(\f{u'L_{K'}}{L_{K'}})\xr{\pi^{*}_{J',u'}}\Sh(\cY(\f{u'}{J'})).
\end{equation}

\item\label{th ch:red} In particular, when $\f{u}{J}$ is $J'$-reduced (which implies $K_{1}=K'$ in the above notation), then $\g^{J',u'}_{J,u}$ is an equivalence.
\end{enumerate}
\end{theorem}
\begin{proof}
In the following we will quote results from \cite{heGENERALIZATIONCYCLICSHIFT} where the author primarily works with a finite dimensional reductive group rather than a loop group. However, \cite[4.5]{heGENERALIZATIONCYCLICSHIFT} it is explained that the results there extend to loop groups in a straightforward way.

(1) is  proved in \cite[4.2(a)]{heGENERALIZATIONCYCLICSHIFT}.

(2) From the diagram of the proof of \cite[Theorem 4.3]{heGENERALIZATIONCYCLICSHIFT}, we see that (using notation from \eqref{ind diagram Y})
\begin{equation}\label{supp control}
p^{J'}_{J}(q_{J}^{J'})^{-1}(\cY(\f{u}{J}))\subset \cY(\f{u'}{J'}).
\end{equation}
Using proper base change and the fact that $p^{J'}_{J}$ is proper, \eqref{supp control} implies that for any $\cF\in \Sh(\cY(\f{u}{J}))$, $\ch_{J}^{J'}i_{J,u!}\cF$ is a $!$-extension from $\cY(\f{u'}{J'})$. Therefore, to prove the statement, it suffices to show that $i^{*}_{J',u'}\ch_{J}^{J'}i_{J,u!}\cF$ is canonically isomorphic to $\g^{J',u'}_{J,u}\cF$. This is exactly the statement of \cite[Theorem 4.3]{heGENERALIZATIONCYCLICSHIFT}.

For the $*$-version, using \eqref{supp control} and the fact that $q^{J'}_{J}$ is smooth, we see that for any $\cF\in \Sh(\cY(\f{u}{J}))$, $\ch_{J}^{J'}i_{J,u*}\cF$ is a $*$-extension from $\cY(\f{u'}{J'})$. Therefore, it suffices to show that $i^{*}_{J',u'}\ch_{J}^{J'}i_{J,u*}\cF$ is canonically isomorphic to $\g^{J',u'}_{J,u}\cF$. This is proved in the same way as the calculations towards the end of the proof of \cite[Theorem 4.3]{heGENERALIZATIONCYCLICSHIFT}, using that $q^{J'}_{J}$ is smooth to justify the base change steps.

(3) follows directly from (2).
\end{proof}

\sss{Nilpotent sheaves} For $J\sft I^{a}$, recall that $\cH_{\cG,J}=\Sh_{\cN}(\cY_{J})$.  For $J\subset J'\sne I^{a}$, it is standard to check that $\ch^{J'}_{J}$ sends $\cH_{\cG,J}$ to $\cH_{\cG,J'}$.

For each geometric piece $\f{u}{J}$, we have the equivalence $\Sh(\cY(\f{u}{J}))\simeq \Sh(\f{uL_{K}}{L_{K}})$ given in Corollary \ref{c:sh geom piece}. Consider the subcategory $\Sh_{\cN}(\f{uL_{K}}{L_{K}})\subset \Sh(\f{uL_{K}}{L_{K}})$ consisting of sheaves whose pullback to $uL_{K}$ is $L_{K}$-nilpotent for the left transaltion (equivalently, $L_{K}$-nilpotent for the right translation), using the terminology introduced in Section~\ref{sss:nilp cyc red}. We define
\begin{equation*}
\Sh_{\cN}(\cY(\f{u}{J}))\subset \Sh(\cY(\f{u}{J}))
\end{equation*}
to be the full category corresponding to $\Sh_{\cN}(\f{uL_{K}}{L_{K}})\cong \Sh_{\cN}(\f{L_{K}}{\Ad_{\dot u}(L_{K})})$ under the equivalence \eqref{sh geom piece}.

\begin{prop}\label{p:SSN on pieces}
For $\f{u}{J}\in \cS_{J}$ and $i_{J,u}: \cY(\f{u}{J})\incl \cY_{J}$ be the inclusion. Then:
\begin{enumerate}
\item The category $\cH_{\cG,J}=\Sh_{\cN}(\cY_{J})$ consists of objects $\cF\in \Sh(\cY_{J})$ such that $i_{J,u}^{*}\cF\in \Sh_{\cN}(\cY(\f{u}{J}))$ for all $\f{u}{J}\in \cS_{J}$. Alternatively, $\Sh_{\cN}(\cY_{J})$ consists of objects $\cF\in \Sh(\cY_{J})$ such that $i_{J,u}^{*}\cF\in \Sh_{\cN}(\cY(\f{u}{J}))$ for all $\f{u}{J}\in \cS_{J}$. 
\item The functors $i_{J,u!}$ and $i_{J,u*}$ send $\Sh_{\cN}(\cY(\f{u}{J}))$ (defined above) to $\Sh_{\cN}(\cY_{J})=\cH_{\cG,J}$.
\end{enumerate}
\end{prop}
\begin{proof}
We prove (1) and (2) follows. For (1), we give the argument for the $*$-pullback statement, and the $!$-version is similar. Recall the notation $\cZ_{n}(u_{0},\cdots, u_{n})$ and $\cZ_{n}(u_{0},\cdots, u_{n+1})$ introduced in Section~\ref{sss:geom piece}, for $n\ge-1$, and $(u_{0},u_{1},\cdots )$ the sequence of elements in $\tilW$ that appear in any combinatorial piece. Define $\Sh_{\cN}(\cZ_{n}(u_{0},\cdots, u_{n}))$ and $\Sh_{\cN}(\cZ_{n}(u_{0},\cdots, u_{n+1}))$ using left $L_{J'_{n+1}}$-nilpotence. Then $\cH_{\cG,J}=\Sh_{\cN}(\cZ_{-1})$.

Fix $n\ge-1$ and $(u_{0},\cdots, u_{n})$ that appear as the first $n+1$ terms of a combinatorial piece (so that $J_{0}=J, J_{1}, J'_{1},\cdots, J_{n+1}, J'_{n+1}$ are defined according the recipe in Section~\ref{sss:comb piece}). We have a stratification
\begin{equation*}
\cZ_{n}(u_{0},\cdots, u_{n})=\bigcup_{u_{n+1}\in {}^{J'_{n+1}}W_{J_{n}}^{J_{n}}}\cZ_{n}(u_{0},\cdots, u_{n},u_{n+1}).
\end{equation*}
Convention: $W_{-1}:=\tilW$. We also have the cyclic reduction maps
\begin{equation*}
c_{u_{n+1}}: \cZ_{n}(u_{0},\cdots, u_{n},u_{n+1})\to \cZ_{n+1}(u_{0},\cdots, u_{n},u_{n+1}).
\end{equation*}
\begin{claim} For any $n\ge-1$ , $\cF\in \Sh(\cZ_{n}(u_{0},\cdots, u_{n}))$ lies in $\Sh_{\cN}(\cZ_{n}(u_{0},\cdots, u_{n}))$ if and only if for all $u_{n+1}\in {}^{J'_{n+1}}W_{J_{n}}^{J_{n}}$, the sheaf $\cF_{u_{n+1}}^{\flat}\in \Sh(\cZ_{n+1}(u_{0},\cdots, u_{n},u_{n+1}))$ corresponding to $\cF|_{\cZ_{n}(u_{0},\cdots, u_{n},u_{n+1})}$ under cyclic reduction (i.e., $c_{u_{n+1}}^{*}\cF_{u_{n+1}}^{\flat}\cong \cF|_{\cZ_{n}(u_{0},\cdots, u_{n},u_{n+1})}$) lies in $\Sh_{\cN}(\cZ_{n+1}(u_{0},\cdots, u_{n+1}))$. 
\end{claim}
\begin{proof}[Proof of Claim]
Indeed, let $\wt\cZ_{n}(u_{0},\cdots, u_{n})=P'^{u}_{n+1}\bs L_{J_{n}}/P^{u}_{n+1}$, and let $\wt\cZ_{n}(u_{0},\cdots, u_{n+1})$ be the preimage of $\cZ_{n}(u_{0},\cdots, u_{n+1})$ in $\wt\cZ_{n}(u_{0},\cdots, u_{n})$. Then $\wt\cZ_{n}(u_{0},\cdots, u_{n+1})$ for various $u_{n+1}$ give a stratification of $\wt\cZ_{n}(u_{0},\cdots, u_{n})$ that is stable under the left translation by $L_{J'_{n+1}}$. By definition, $\cF\in \Sh_{\cN}(\cZ_{n}(u_{0},\cdots, u_{n}))$ if and only if its pullback $\wt\cF$ on $\wt\cZ_{n}(u_{0},\cdots, u_{n})$ is left $L_{J_{n+1}'}$-nilpotent, which happens if and only if $\wt\cF|_{\wt\cZ_{n}(u_{0},\cdots, u_{n+1})}$ is  left $L_{J_{n+1}'}$-nilpotent for all $u_{n+1}$, if and only if $\cF|_{\cZ_{n}(u_{0},\cdots, u_{n+1})}\in \Sh_{\cN}(\cZ_{n}(u_{0},\cdots, u_{n+1}))$ for all $u_{n+1}$. By Lemma \ref{l:ShN cyc red}, $\cF|_{\cZ_{n}(u_{0},\cdots, u_{n+1})}\in \Sh_{\cN}(\cZ_{n}(u_{0},\cdots, u_{n+1}))$ if and only if $\cF_{u_{n+1}}^{\flat}\in \Sh_{\cN}(\cZ_{n+1}(u_{0},\cdots, u_{n+1}))$. This proves the claim.
\end{proof}
We continue with the proof of the proposition. Note that for each combinatorial piece $(J_{n}, J'_{n}, u_{n})\in \cS_{J}$, $J_{n}$  will stabilizes for $n\ge r$, where $r$ is the semisimple rank of $G$. Each $\f{u}{J}\in\cS_{J}$ gives an $(r+1)$-tuple $(u_{0},\cdots, u_{r})$, and by the above remark, $u$ is determined by $(u_{0},\cdots, u_{r})$. For each $u\in \cS_{J}$, we define $\cF^{\flat}_{u}\in \Sh_{\cN}(\cZ_{r}(u_{0},\cdots, u_{r}))$ to be the result of restricting and apply cyclic reduction successively. In other words, under the unipotent gerbe $\pi_{J,u}: \cY(\f{u}{J})\to \cZ_{r}(u_{0},\cdots, u_{r})$, $\pi_{J,u}^{*}\cF^{\flat}_{u}\simeq \cF|_{\cY(\f{u}{J})}$. Using the Claim in the previous paragraph repeatedly for $n=-1,0,\cdots, r-1$, we conclude that for $\cF\in \Sh(\cY_{J})$, $\cF\in \Sh_{\cN}(\cY_{J})$ if and only if $\cF_{u}^{\flat}\in \Sh(\cZ_{r}(u_{0},\cdots, u_{r}))$ lies in $\Sh_{\cN}(\cZ_{r}(u_{0},\cdots, u_{r}))$, for all $J$-pieces $\f{u}{J}$. By definition, the latter is the same as saying $\cF|_{\cY(\f{u}{J})}=i_{J,u}^{*}\cF$ lies in  $\Sh_{\cN}(\cY(\f{u}{J}))$ for all $\f{u}{J}\in \cS_{J}$.
\end{proof}

For any locally closed substack $Z\subset \cY_{J}$ that is a union of geometric $J$-pieces, we define $\Sh_{\cN}(Z)\subset \Sh(Z)$ to be the full subcategory consisting of $\cF\in \Sh(Z)$ such that $i_{J,u}^{*}\cF\in \Sh_{\cN}(\cY(\f{u}{J}))$ whenever $\f{u}{J}\subset Z$. By Proposition \ref{p:SSN on pieces}, for $Z=\cY_{J}$, this definition of $\Sh_{\cN}(\cY_{J})$ coincides with the old one which is $\cH_{\cG,J}$.

\begin{cor}\label{c:nilp sh recoll YJ} The full subcategories $\Sh_{\cN}(Z)$ for closed unions of geometric pieces $Z\subset \cY_{J}$ give a stratification structure on $\cH_{\cG, J}$ indexed by the poset $(\cS_{J}, \le_{J})$ (in the sense recalled in Section~\ref{sss:strat}), such that the strata category corresponding to $\f{u}{J}\in \cS_{J}$ is $\Sh_{\cN}(\cY(\f{u}{J}))$. 
\end{cor}


\subsection{Semi-orthogonal decomposition of the cocenter}\label{ss:cocenter recoll} 

In this section, we prove Theorem \ref{thm:intro ff} and its generalization Theorem \ref{thm:main in intro} which describes the cocenter $hh(\cH_{\cG})$ up to taking ``associated graded'' indexed by enhanced Newton points.

\sss{Colimit of character sheaves}\label{sss:colim CS} Let $J\sft I^{a}$. 
By Lemma \ref{c:sh geom piece}, we have a canonical equivalence $\Sh(\cY(\f{1}{J}))\simeq \Sh(L_{J}/L_{J})$. 
By definition, we have 
$$\Sh_{\cN}(\cY(\f{1}{J}))\simeq \Sh_{\cN}(L_{J}/L_{J}).$$ We remark that $\Sh_{\cN}(L_{J}/L_{J})$ can be viewed as a version of the category of character sheaves on $L_{J}$ allowing sheaves with infinite-dimensional stalks. 

By Proposition \ref{p:SSN on pieces}, the closed embedding $i_{J,1}: \cY(\f{1}{J})\incl \cY_{J} $ gives a fully faithful embedding
\begin{equation*}
\io_{J}=i_{J,1*}: \Sh_{\cN}(L_{J}/L_{J})\simeq\Sh_{\cN}(\cY(\f{1}{J})) \incl \cH_{\cG,J}.
\end{equation*}
For $J\subset J'\sft I^{a}$, there is the induction functor of character sheaves defined by Lusztig:
\begin{equation*}
\Ind_{J}^{J'}: \Sh_{\cN}(L_{J}/L_{J})\to \Sh_{\cN}(L_{J'}/L_{J'}).
\end{equation*}
Via the embeddings $\io_{J}$ and $\io_{J'}$, the induction functor is intertwined with $\ch^{J'}_{J}: \cH_{\cG,J}\to \cH_{\cG,J'}$. Therefore they induce a functor on colimits
\begin{equation}\label{io CS to colim}
\io: \colim_{\cD}\Sh_{\cN}(L_{J}/L_{J})\to \colim_{\cD}\cH_{\cG,J}
\end{equation}
where $\cD$ is the groupoid $\cD^{\c}/\Om$ introduced in Section~\ref{sss:general}. Combining the above functor with the equivalence to $hh(\cH_{\cG})$ given by Corollary \ref{cor:coequal J geom}, we get
\begin{equation}\label{CS to hh}
\colim_{J\in \cD}\Sh_{\cN}(L_{J}/L_{J})\to hh(\cH_{\cG}).
\end{equation}

\begin{theorem}\label{th:ff} The functor $\io$ is fully faithful.
\end{theorem}
This is the specific statement we seek for  this paper; it is an easy consequence of a general theorem describing all of $hh(\cH_{\cG})$ up to ``taking associated graded'', which we will state and prove next.  Below we will take the main step towards such a description of $hh(\cH_{\cG})$ by considering $hh(\cH_{\cG^{\c}}, \cH_{\cG})$ first.

\sss{Semi-orthogonal decomposition of $hh(\cH_{\cG^{\c}}, \cH_{\cG})$} 
In this section, we arrive at our first approximation to the main goal, as encapsulated in Theorem~\ref{th:semi-orth hh G neutral}.  

Instead of the cocenter $hh(\cH_{\cG})$, we consider $hh(\cH_{\cG^{\c}}, \cH_{\cG})$, which is equivalent to $\colim_{J\sft I^{a}}\cH_{\cG,J}$ by Corollary \ref{cor:coequal J geom}. By the discussion in Section~\ref{sss:HG comp}, $\cH_{\cG}$ decomposes into the direct sum of $\cH_{\cG^{\c}}$-bimodules $\cH^{\om}_{\cG}$ according to the connected components of $\cG$ indexed by $\om \in \Om$, we have a decomposition
\begin{equation*}
hh(\cH_{\cG^{\c}}, \cH_{\cG})=\bigoplus_{\om\in \Om}hh(\cH_{\cG^{\c}}, \cH^{\om}_{\cG}).
\end{equation*}

For $\wt\nu\in\wt\NP$, recall the essential part $\frB^{\hs}_{\wt\nu}\subset \frB_{\wt\nu}$, whose $J$-facets are indexed by the subset $\cS^{\hs}_{J,\wt\nu}\subset \cS_{J}$ consisting of $\f{u}{J}$ such that $\wt\nu(u)=\wt\nu$ and $\ell(u)=\j{2\r, \nu}$. Recall $\Tot(\cS^{\hs}_{\wt\nu})=\coprod_{J\sft I^{a}}\cS^{\hs}_{J,\wt\nu}$ is a poset as defined in Section~\ref{sss:Pset}, and the order is opposite to the closure order of facets in $\frB^{\hs}_{\wt\nu}$. If $\f{u}{J}\le\f{u'}{J'}$ in $\Tot(\cS^{\hs}_{\wt\nu})$, we have the functor 
\begin{equation*}
\g^{J',u'}_{J,u}: \Sh_{\cN}(\cY(\f{u}{J}))\to \Sh_{\cN}(\cY(\f{u'}{J'})).
\end{equation*}
defined in Theorem \ref{th:ch} (see \eqref{ind on pieces}, and here we restrict to sheaves with nilpotent singular support).

Using the functors $\g^{J',u'}_{J,u}$ we may form the colimit
\begin{equation}\label{nu part hh}
\colim_{\f{u}{J}\in \Tot(\cS_{\wt\nu}^{\hs})}\Sh_{\cN}(\cY(\f{u}{J})).
\end{equation}
For $\wt\nu=(0,0)$, $\Tot(\cS_{\wt\nu}^{\hs})$ consists of $\f{1}{J}$ for $J\sft I^{a}$, and can be identified with the poset $\cD^{\c}$. In this case, the above colimit is the same as $\colim_{J\sft I^{a}}\Sh_{\cN}(L_{J}/L_{J})$ using the induction functors.

\begin{theorem}\label{th:semi-orth hh G neutral}  Fix $\om\in \Om$. Then $hh(\cH_{\cG^{\c}}, \cH^{\om}_{\cG})$ admits a semi-orthogonal decomposition indexed by non-negative integers
\begin{equation*}
hh(\cH_{\cG^{\c}}, \cH_{\cG}^{\om})_{0}\incl hh(\cH_{\cG^{\c}},\cH_{\cG}^{\om})_{\le 1}\incl\cdots hh(\cH_{\cG^{\c}},\cH^{\om}_{\cG})_{\le n}\incl\cdots\incl hh(\cH_{\cG^{\c}},\cH_{\cG}^{\om})=\bigcup_{n\ge0}hh(\cH_{\cG^{\c}},\cH_{\cG}^{\om})_{\le n}.
\end{equation*}
In particular, each inclusion $hh(\cH_{\cG^{\c}},\cH_{\cG}^{\om})_{\le n}\incl hh(\cH_{\cG^{\c}},\cH_{\cG}^{\om})$ extends to a recollement.

For each $n\ge0$, the $n$-th associated graded category $hh(\cH_{\cG^{\c}}, \cH^{\om}_{\cG})_{n}$ has the following description: it is the direct sum
\begin{equation*}
hh(\cH_{\cG^{\c}}, \cH^{\om}_{\cG})_{n}\simeq\bigoplus_{\wt\nu=(\nu,\om)\in\wt\NP, \j{2\r,\nu}=n}hh(\cH_{\cG^{\c}}, \cH_{\cG}^{\om})_{\nu},
\end{equation*}
and for $\wt\nu=(\nu,\om)\in\wt\NP$, 
\begin{equation*}
hh(\cH_{\cG^{\c}},\cH^{\om}_{\cG})_{\nu}=\colim_{\f{u}{J}\in \Tot(\cS^{\hs}_{\wt\nu})}\Sh_{\cN}(\cY(\f{u}{J})).
\end{equation*}

\end{theorem}
\begin{proof}
In the argument we shall treat $hh(\cH_{\cG^{\c}}, \cH_{\cG})$ as a whole. It is clear that the resulting semi-orthogonal decomposition induces a semi-orthogonal decomposition for each summand $hh(\cH_{\cG^{\c}}, \cH^{\om}_{\cG})$.

Abbreviate $\cH_{\cG, J}$ by $\cC_{J}$. Let $\cY_{J, \le n}$ be the union of geometric pieces $\cY(\f{u}{J})$ with $\ell(u)\le n$. By Theorem \ref{th:He closure}, $\cY_{J,\le n}$ is closed in $\cY_{J}$. Let $\cC_{J,\le n}=\Sh_{\cN}(\cY_{J,\le n})$ (the meaning of $\Sh_{\cN}$ is defined in the paragraph preceding Corollary \ref{c:nilp sh recoll YJ}).  By Theorem \ref{th:ch}(1)(2), for $J\subset J'\sft I^{a}$, the functors $\ch_{J}^{J'}$ send $\cC_{J,\le n}$ to $\cC_{J',\le n}$. Therefore we may form the colimit
\begin{equation*}
\cC_{\le n}=\colim_{J\sft I^{a}}\cC_{J, \le n}
\end{equation*}
using the functors $\ch^{J'}_{J}$. 

For $\wt\nu\in \wt\NP$, define $\cY_{J,\wt\nu,\hs}$ to be the union of $\cY(\f{u}{J})$ where $\wt\nu(u)=\wt\nu$ and $\ell(u)=\j{2\r, \nu}$. Let $\cY_{J,\le n, \hs}$ be the substack of $\cY_{J,\le n}$
\begin{equation*}
\cY_{J,\le n, \hs}=\cY_{J,\le n-1}\cup(\bigcup_{\wt\nu=(\nu,\om)\in \wt\NP, \j{2\r,\nu}=n}\cY_{J,\nu,\hs}).
\end{equation*}
Again by Theorem \ref{th:He closure}, $\cY_{J,\le n, \hs}$ is closed in $\cY_{J}$. Define $\cC_{J,\le n, \hs}=\Sh_{\cN}(\cY_{J,\le n, \hs})$. Theorem \ref{th:ch}(1)(2), $\ch^{J'}_{J}$ sends $\cC_{J,\le n, \hs}$ to $\cC_{J',\le n,\hs}$.  Therefore we may form the colimit
\begin{equation*}
\cC_{\le n,\hs}=\colim_{J\sft I^{a}}\cC_{J, \le n,\hs}
\end{equation*}
using the functors $\ch^{J'}_{J}$.  

For each $J\sft I^{a}$, we have natural inclusions
\begin{equation*}
\cC_{J,\le 0,\hs}\incl \cC_{J,\le 0}\incl \cC_{J,\le 1,\hs}\incl\cdots\incl \cC_{J,\le n-1}\incl \cC_{J,\le n,\hs}\incl \cC_{J,\le n}\incl \cdots
\end{equation*}
These inclusions are compatible with the functors $\ch^{J'}_{J}$, hence we get functors between the colimits over $J$:
\begin{equation*}
\cC_{\le 0,\hs}\xr{\k_{0}}\cC_{\le 0}\xr{i_{0}}\cC_{\le 1,\hs}\xr{\k_{1}}\cdots\to\cC_{\le n-1}\xr{i_{n}} \cC_{\le n,\hs}\xr{\k_{n}} \cC_{\le n}\to \cdots
\end{equation*}
For each $J$, we have $\cC_{J}\simeq\colim_{n}\cC_{J,\le n}$. Commuting the order of taking colimits, we have
\begin{equation*}
hh(\cH_{\cG})\simeq \colim_{J\in\cD^{\c}}\cC_{J}\simeq \colim_{n}\cC_{\le n}.
\end{equation*}
The assertion of the theorem will follow from the two claims below:
\begin{enumerate}
\item For $n\ge0$, the functor $i_{n}: \cC_{\le n-1}\to \cC_{\le n,\hs}$ is fully faithful and extends to a recollement diagram
\begin{equation}\label{rec C n hs}
\xymatrix{\cC_{n,\hs}    \ar@<3ex>[r]^{j_{n!}}  \ar@<-3ex>[r]^{j_{n*}} 
 &   \ar[l]_{j^{!}_{n}=j^{*}_{n}} \cC_{\le n,\hs} \ar@<3ex>[r]^{i_{n}^{*}} \ar@<-3ex>[r]^{i_{n}^{!}}   &      \ar[l]_{i_{n}} \cC_{\le n-1} }
\end{equation}
where the category $\cC_{n,\hs}$ is canonically equivalent to the direct sum of $hh(\cH_{\cG^{\c}}, \cH^{\om}_{\cG})_{\nu}$ defined using \eqref{nu part hh}, for $\om\in\Om$ and $\j{2\r, \nu}=n$.

\item For $n\ge0$, the functor $\k_{n}: \cC_{\le n, \hs}\to \cC_{\le n}$ is an equivalence.
\end{enumerate}

We first prove Claim (1). Let $\wt\NP_{n}$ be the set of $\wt\nu=(\nu,\om)\in \wt\NP$ such that $\j{2\r, \nu}=n$. For $J\sft I^{a}$, let 
$$\cY_{J,n,\hs}=\coprod_{\wt\nu\in \wt\NP_{n}}\cY_{J,\wt\nu}=\coprod_{\wt\nu\in \wt\NP_{n}}\left(\coprod_{\f{u}{J}\in \cS_{J,\wt\nu}^{\hs}}\cY(\f{u}{J})\right).$$
Let $\cC_{J,n,\hs}=\Sh_{\cN}(\cY_{J,n,\hs})$. The decomposition of $\cY_{J,n,\hs}$ above  gives a decomposition
\begin{equation}\label{CJn}
\cC_{J,n,\hs}=\bigoplus_{\wt\nu\in \wt\NP_{n}}\cC_{J,\wt\nu,\hs}=\bigoplus_{\wt\nu\in\wt\NP_{n}}\left(\bigoplus_{\f{u}{J}\in \cS_{J,\wt\nu}^{\hs}}\Sh_{\cN}(\cY(\f{u}{J}))\right).
\end{equation}
Then $\cC_{J, \le n, \hs}$ carries a recollement structure
\begin{equation}\label{CJn hs recoll}
\xymatrix{\cC_{J, n,\hs}  \ar@<3ex>[r]^{j_{J, n!}}  \ar@<-3ex>[r]^{j_{J, n*}} & \ar[l]_{j^{!}_{J, n}=j^{*}_{J, n}} \cC_{J, \le n,\hs} \ar@<3ex>[r]^{i_{J, n}^{*}} \ar@<-3ex>[r]^{i_{J, n}^{!}}  & \ar[l]_{i_{J, n}}        \cC_{J, \le n-1}   
}
\end{equation}

For $J\subset J'\sft I^{a}$,  $\f{u}{J}\in\cS^{\hs}_{\wt\nu,J}$, $\f{u}{J}$ is quasi-$J'$-reduced since $\ell(u)=\ell(u')=\j{2\r, \nu}$ if $\f{u'}{J'}=\s_{J}^{J'}(\f{u}{J})$. By  Theorem \ref{th:ch}\eqref{th ch:qred}, the functor $\ch^{J'}_{J}$ respects the recollement structure on $\cC_{J, \le n,\hs}$ and induces the functor $\op_{\f{u}{J}\in \cS^{\hs}_{J,\wt\nu}}\g^{J',u'}_{J,u}: \cC_{J, n,\hs, \wt\nu)}\to \cC_{J', n,\hs}$. By Proposition \ref{p:recoll} of Appendix \ref{s:str sh}, we conclude that that the colimit $\cC_{\le n,\hs}=\colim_{J\sft I^{a}}\cC_{J, \le n,\hs}$ also has a recollement structure by taking termwise colimits of \eqref{CJn hs recoll}. This gives the recollement diagram \eqref{rec C n hs}. Moreover, the category $\cC_{n,\hs}$ in \eqref{rec C n hs} is the direct sum over $\wt\nu\in \wt\NP_{n}$ of $\cC_{\wt\nu,\hs}:=\colim_{J\sft I^{a}}\cC_{J, \wt\nu, \hs}$. By \eqref{CJn} and the description of $\ch^{J'}_{J}$ on $\cC_{J, n,\hs}$ in terms of $\g^{J',u'}_{J,u}$, we conclude that $\cC_{\nu,\hs}$ is canonically equivalent to \eqref{nu part hh}.

Now we prove Claim (2). We fix $n\in\ZZ_{\ge0}$. Let $\Sig_{<n}\subset \un\frE\times\wt\NP$ be the set of pairs $([E],\wt\nu)\in\un\frE\times\wt\NP$ such that $\j{2\r, \nu}<n$. Then $\Sig_{<n}$ has a partial order $([E], \wt\nu)\le ([E'], \wt\nu)$ if $[E]\le [E']$ in $\un\frE$ (no comparison if the Newton points are different). We extend this partial order to a total order on $\Sig_{<n}$ and denote it by $\preceq$.

For $J\sft I^{a}$, let $\cS_{J,(n,[E],\wt\nu)}$ be the set of combinatorial pieces $\f{u}{J}\in \cS_{J}$ such that $\ell(u)=n, \t(\f{u}{J})=[E]$ and $\wt\nu(u)=\wt\nu$. For $([E],\wt\nu)\in \Sig_{<n}$, define a finite subset of $\cS_{J}$ by
\begin{equation*}
\cS_{J,\preceq(n, [E], \wt\nu)}=\cS_{J, \le n, \hs}\cup \left(\bigcup_{\Sig_{<n}\ni ([E'],\wt\nu')\preceq([E], \wt\nu)}\cS_{J,(n,[E'],\wt\nu')}\right).
\end{equation*}
By Lemma \ref{l:ell tau under closure}, for $J\subset J'\sft I^{a}$, the map $\d^{J'}_{J}$ send $\cS_{J,\preceq(n, [E], \wt\nu)}$ to $\cS_{J',\preceq(n, [E], \wt\nu)}$. Therefore the assignment $J\mapsto \cS_{J,\preceq(n, [E], \wt\nu)}$ gives a $\cD^{\c}$-subset $\cS_{\preceq(n, [E], \wt\nu)}$ of $\cS_{\wt\nu}$. Let 
\begin{equation*}
\frB_{\preceq(n, [E], \wt\nu)}=|\cS_{\preceq(n, [E], \wt\nu)}|\subset \frB_{\wt\nu}.
\end{equation*}
Let $\cS_{J,\prec(n, [E], \wt\nu)}=\cS_{J,\preceq(n, [E], \wt\nu)}\setminus\cS_{J,(n,[E],\wt\nu)}$. Thus we get a $\cD^{\c}$-set $\cS_{\prec(n, [E], \wt\nu)}$ and its geometric realization $\frB_{\prec(n, [E], \wt\nu)}$.

Let $\cY_{J, \preceq(n, [E], \wt\nu)}\subset \cY_{J}$ be the union of geometric pieces $\cY(\f{u}{J})$ where $\f{u}{J}\in \cS_{J,\preceq(n, [E], \wt\nu)}$. By Theorem \ref{th:He closure}, $\cY_{J, \preceq(n, [E], \wt\nu)}$ is a closed substack of $\cY_{J}$. Similarly we define the closed substack $\cY_{J, \prec(n, [E], \wt\nu)}\subset \cY_{J}$. The category of nilpotent sheaves
\begin{equation*}
\cC_{J, \preceq(n, [E], \wt\nu)}:=\Sh_{\cN}(\cY_{J, \preceq(n, [E], \wt\nu)}), \quad \cC_{J, \prec(n, [E], \wt\nu)}:=\Sh_{\cN}(\cY_{J, \prec(n, [E], \wt\nu)})
\end{equation*}
are defined as in the paragraph preceding Corollary \ref{c:nilp sh recoll YJ}. Theorem \ref{th:ch} implies that $\ch^{J'}_{J}$ sends $\cC_{J, \preceq(n, [E], \wt\nu)}$ to $\cC_{J', \preceq(n, [E], \wt\nu)}$. We form the colimit
\begin{equation*}
\cC_{\preceq(n,[E], \wt\nu)}:=\colim_{J\sft I^{a}}\cC_{J, \preceq(n, [E], \wt\nu)}.
\end{equation*}
Similarly define $\cC_{\prec(n,[E], \wt\nu)}$. Since $\preceq$ is a total order, $\cC_{\prec(n,[E], \wt\nu)}=\cC_{\preceq(n,[E'],\wt\nu')}$ if $([E'],\wt\nu')$ is a predecessor of $([E], \wt\nu)$, or if $([E],\wt\nu)$ is the minimal element of $\Sig_{<n}$, $\cC_{\prec(n,[E], \wt\nu)}\simeq \cC_{\le n, \hs}$.

Since $\cC_{J, \le n}=\colim_{([E], \wt\nu)\in \Sig_{<n}}\cC_{J, \preceq(n,[E], \wt\nu)}$, taking colimit over $J$ we get
\begin{equation*}
\cC_{\le n}\simeq\colim_{([E], \wt\nu)\in \Sig_{<n}}\cC_{\preceq(n,[E], \wt\nu)}.
\end{equation*}
To show that $\cC_{\le n, \hs}\to \cC_{\le n}$ is an equivalence, it suffices to show that for each $([E],\wt\nu)\in\Sig_{<n}$, the natural functor
\begin{equation*}
\k_{[E], \wt\nu}: \cC_{\prec(n,[E], \wt\nu)}\to \cC_{\preceq(n,[E], \wt\nu)}
\end{equation*}
is an equivalence. For this we apply the contraction principle for cosheaves on posets that we proved in Theorem \ref{th:contracting cosheaf}, in the form of Corollary \ref{c:contracting cosheaf}.

To set the stage for applying Corollary \ref{c:contracting cosheaf}, we define a poset $\cI$ whose underlying set is 
$$\cI:=\cD^{\c}\sqcup \Tot(\cS_{(n,[E],\wt\nu)}).$$
Let $f: \cI\to \cD^{\c}$ be the projection that is the identity on $\cD^{\c}$ and the natural projection on $\Tot(\cS_{(n,[E],\wt\nu)})$. Let $s:\cD^{\c}\incl \cI$ be the inclusion. For the partial order, we include all order relations from  $\cD^{\c}$ and $\Tot(\cS_{(n,[E],\wt\nu)})$, and take the partial order generated by the following extra relations:
\begin{itemize}
\item For any $J\in \cD^{\c}$ and $\f{u}{J}\in \cS_{J,(n,[E],\wt\nu)}$, we declare $s(J)<\f{u}{J}$.
\item For any $J\subset J'$ in $\cD^{\c}$, $\f{u}{J}\in \cS_{J,(n,[E],\wt\nu)}$ such that $\d_{J}^{J'}(\f{u}{J})\in \cS_{J', \prec(n,[E],\wt\nu)}$, we declare $\f{u}{J}<s(J')$.
\end{itemize}
By construction, $f$ is a coCartesian map of posets, and every fiber $f^{-1}(J)$ has a unique minimal element $s(J)$. 

We next define a cosheaf $\cF$ on $\cI$ in the sense of Section~\ref{sss:cosh}.
\begin{itemize}
\item For $J\in \cD^{\c}$, let $\cF_{s(J)}=\cC_{J, \prec(n,[E],\wt\nu)}=\Sh_{\cN}(\cY_{J, \prec(n,[E],\wt\nu)})$.
\item For $\f{u}{J}\in \cS_{J,(n,[E],\wt\nu)}$, let $\cF_{\f{u}{J}}=\Sh_{\cN}(\cY_{J, \prec(n,[E],\wt\nu)}\cup \cY(\f{u}{J}))$. Note that $\cY_{\prec(n,[E],\wt\nu)}\cup \cY(\f{u}{J})$ is closed in $\cY_{J}$. 
\end{itemize}
The transition functors are defined as follows:
\begin{itemize}
\item For $J\subset J'\in\cD^{\c}$, the functor $\cF_{s(J)}\to \cF_{s(J')}$ is given by $\ch^{J'}_{J}$.

\item For $J\in \cD^{\c}$ and $\f{u}{J}\in \cS_{J,(n,[E],\wt\nu)}$, the functor $\cF_{s(J)}\to \cF_{\f{u}{J}}$ is $\io_{*}:\Sh_{\cN}(\cY_{J, \prec(n,[E],\wt\nu)})\incl \Sh_{\cN}(\cY_{J, \prec(n,[E],\wt\nu)}\cup \cY(\f{u}{J}))$ for the closed embedding $\io: \cY_{J, \prec(n,[E],\wt\nu)}\incl \cY_{J, \prec(n,[E],\wt\nu)}\cup\cY(\f{u}{J})$.

\item For $\f{u}{J}\le \f{u'}{J'}$ in $\Tot(\cS_{(n,[E],\wt\nu)})$ (which means $\d^{J'}_{J}(\f{u}{J})=\f{u'}{J'}$), the functor $\cF_{\f{u}{J}}\to \cF_{\f{u'}{J'}}$ is given by $\ch^{J'}_{J}$. Here we use Theorem \ref{th:ch}(2) to verify that $\ch^{J'}_{J}$ indeed sends sheaves supported on $\cY(\f{u}{J})$ to sheaves supported on $\cY(\f{u'}{J'})$ since $\ell(u)=\ell(u')=n$.

\item For $\f{u}{J}< s(J')$, where $J\subset J'$,  $\f{u}{J}\in\cS_{J,(n,[E],\wt\nu)}$ and $\d^{J'}_{J}(\f{u}{J})\in \cS_{J',\prec(n,[E],\wt\nu)}$, the functor $\cF_{\f{u}{J}}\to \cF_{s(J')}$ is again given by $\ch^{J'}_{J}$. Here we again use Theorem \ref{th:ch}(1)(2) to verify that $\ch^{J'}_{J}$ indeed sends sheaves supported on $\cY(\f{u}{J})$ either to sheaves supported on $\cY_{J',\le n-1}\subset \cY_{J', \prec(n,[E],\wt\nu)}$ if $\ell(u')<\ell(u)=n$, or to sheaves supported on $\cY(\f{u'}{J'})\subset \cY_{J', \prec(n,[E],\wt\nu)}$ if $\ell(u')=\ell(u)=n$.
\end{itemize}

By construction, $s^{*}\cF$ is the cosheaf on $\cD^{\c}$ whose value at $J$ is $\cC_{J, \prec(n,[E],\wt\nu)}$ and transition functors are given by $\ch^{J'}_{J}$. Therefore
\begin{equation}\label{sF}
\cC_{\prec(n,[E],\wt\nu)}=\colim_{\cD^{\c}}s^{*}\cF.
\end{equation}

On the other hand, we claim that for any $J\in\cD^{\c}$, $(\int_{f}\cF)_{J}\simeq \cC_{J,\preceq(n,[E],\wt\nu)}$ and this assignment extends to an equivalence of cosheaves on $\cD^{\c}$ (for the definition of $\int_{f}\cF$ for a coCartesian map $f$, see Section~\ref{sss:push cosh}). Indeed, since $f^{-1}(J)=\{s(J)\}\cup \cS_{J,(n,[E],\wt\nu)}$ as a poset has $s(J)$ as the minimal element and no extra relations besides those already in $\cS_{J,(n,[E],\wt\nu)}$, $(\int_{f}\cF)_{J}$ is the pushout of $\cF_{s(J)}$ diagonally embedded into $\op\cF_{\f{u}{J}}$ (sum over all $\f{u}{J}\in \cS_{J,(n,[E], \wt\nu)}$). In other words, $(\int_{f}\cF)_{J}$ is the pushout of $\cC_{J,\prec(n,[E],\wt\nu)}$ diagonally embedded into $\op\Sh_{\cN}(\cY_{J,\prec(n,[E],\wt\nu)}\cup\cY(\f{u}{J}))$ for all $\f{u}{J}\in \cS_{J,(n,[E], \wt\nu)}$. The latter is precisely $\cC_{J,\preceq(n,[E],\wt\nu)}$ because of the open-closed decomposition 
\begin{equation*}
\cY_{J,\preceq(n,[E],\wt\nu)}=\left(\coprod_{\f{u}{J}\in \cS_{J,(n,[E], \wt\nu)}}\cY(\f{u}{J})\right)\sqcup\cY_{J,\prec(n,[E],\wt\nu)}.
\end{equation*}
The transition maps $(\int_{f}\cF)_{J}\to (\int_{f}\cF)_{J'}$ are given by $\ch^{J'}_{J}$, therefore we get
\begin{equation}\label{f!F}
\colim_{\cD^{\c}}(\int_{f}\cF)_{J}\simeq \colim_{\cD^{\c}}\cC_{J,\preceq(n,[E],\wt\nu)}=\cC_{\preceq(n,[E],\wt\nu)}.
\end{equation}

Now we would like to apply Corollary \ref{c:contracting cosheaf} to the situation of $f:\cI\to \cD^{\c}$ with section $s$, and the cosheaf $\cF$ on $\cI$. The conclusion would be that the natural map $\colim_{\cD^{\c}}s^{*}\cF\to \colim_{\cD^{\c}}\int_{f}\cF$ is an equivalence. Combining with \eqref{sF} and \eqref{f!F}, we conclude that 
\begin{equation*}
\k_{[E], \wt\nu}: \cC_{\prec(n,[E], \wt\nu)}=\colim_{\cD^{\c}}s^{*}\cF\to \colim_{\cD^{\c}}\int_{f}\cF=\cC_{\preceq(n,[E], \wt\nu)}
\end{equation*}
is an equivalence, proving Claim (2).

It thus remains to check that the conditions for applying Theorem \ref{th:contracting cosheaf} are satisfied.

The first condition: we define $\cL$ to be the cosheaf on $\Tot(\cS_{(n,[E],\wt\nu)})=\cI\setminus s(\cJ)$ whose value at $\f{u}{J}$ is $\Sh_{\cN}(\cY(\f{u}{J}))$, and the transition functors are given by $\ch^{J'}_{J}$. We have a natural map $\b: \cF|_{\Tot(\cS_{(n,[E],\wt\nu)})}\to \cL$ termwise given by open restriction $\Sh_{\cN}(\cY_{J,\prec(n,[E], \wt\nu)}\cup\cY(\f{u}{J}))\to \Sh_{\cN}(\cY(\f{u}{J}))$.

With this definition of $\cL$, we need to check that
\begin{equation*}
\xymatrix{\cL & \ar[l]_-{\b} \cF|_{\Tot(\cS_{(n,[E],\wt\nu)})} & \ar[l]_-{\a} (s_{!}s^{*}\cF)|_{\Tot(\cS_{(n,[E],\wt\nu)})}}
\end{equation*}
is a recollement of cosheaves on $\Tot(\cS_{(n,[E],\wt\nu)})$.
By Lemma~\ref{l:pull is s!}, $s_{!}s^{*}\cF\simeq f^{*}s^{*}\cF$, therefore we can replace the rightmost term above by $(f^{*}s^{*}\cF)|_{\Tot(\cS_{(n,[E],\wt\nu)})}$. Termwise at $\f{u}{J}\in \cS_{J,(n,[E],\wt\nu)}$, they fit into a recollement
\begin{equation}\label{recYuJ}
\xymatrix{\Sh_{\cN}(\cY(\f{u}{J})) \ar@<1ex>[r]\ar@<-1ex>[r]& \ar[l]\Sh_{\cN}(\cY_{J,\prec(n,[E], \wt\nu)}\cup\cY(\f{u}{J})) \ar@<1ex>[r]\ar@<-1ex>[r]& \ar[l]\Sh_{\cN}(\cY_{J,\prec(n,[E], \wt\nu))}}
\end{equation}
by the open-closed decomposition $\cY(\f{u}{J})\cup\cY_{J,\prec(n,[E], \wt\nu)}$. All functors above are continuous by Proposition~\ref{p:cont 4 functors}, therefore \eqref{recYuJ} is a recollement in $\St^{L}_{k}$. For $\f{u}{J}\le \f{u'}{J'}$ in $\Tot(\cS_{(n,[E],\wt\nu)})$,  $\ch^{J'}_{J}$ induces a morphism of between the recollement  \eqref{recYuJ} and its counterpart for $\f{u'}{J'}$, which follows from Theorem \ref{th:ch}(2). This verifies the first condition.

The second condition: by Theorem \ref{th:ch}(3), $\cL$ is locally constant. 

The third condition: we need to check that $|s(\cD^{\c})|\subset |\cI|$ is a homotopy equivalence. Let $\cK:=\Tot(\cS_{\preceq(n,[E], \wt\nu)})$; $\cK':=\Tot(\cS_{\prec(n,[E], \wt\nu)})$. 
Then $\cI\setminus \cD^{\c}=\Tot(\cS_{(n,[E], \wt\nu)})=\cK\setminus \cK'$. For any $\cD^{\c}$-set $X$, let $\cD^{\c}\lhd X$ be the poset whose underlying set is $\cD^{\c}\sqcup  X$ with the added relation $J<x$ for any $x\in X_{J}$. Then $|\cD^{\c}\lhd X|$ is homeomorphic to the mapping cylinder $[0,1]\times |X|\coprod_{\{0\}\times |X|} |\cD^{\c}|$ of the projection $|X|\to |\cD^{\c}|$. Now we can write $\cI$ as a pushout of posets
\begin{equation*}
\cI\cong (\cD^{\c}\lhd \cK)\coprod_{\cD^{\c}\lhd \cK'}\cD^{\c}
\end{equation*}
where the map $\cD^{\c}\lhd \cK'\to \cD^{\c}$ is the identity on $\cD^{\c}$ and the projection on $\cK'$. Under this presentation, the inclusion $s: \cD^{\c}\incl\cI$ corresponds to the inclusion of the second factor $\cD^{\c}$. Taking geometric realizations we get
\begin{equation*}
|\cI|\cong ([0,1]\times |\cK|)\coprod_{\{0\}\times |\cK|\cup [0,1]\times |\cK'|}|\cD^{\c}|.
\end{equation*}
To show $|\cD^{\c}|\incl |\cI|$ is a homotopy equivalence, it suffice to show $\{0\}\times |\cK|\cup [0,1]\times |\cK'|\incl [0,1]\times |\cK|$ is a homotopy equivalence, or $|\cK'|\incl |\cK|$ is a homotopy equivalence. 

By Remark \ref{r:sd geom real}, $|\cK|$ is a subdivision of $\frB_{\wt\nu,\preceq(n,[E])}$, and $|\cK'|$ is a subdivision of $\frB_{\wt\nu,\prec(n,[E])}$. Therefore it suffices to check that the inclusion $\frB_{\wt\nu,\prec(n,[E])}\incl\frB_{\wt\nu,\preceq(n,[E])}$ is a homotopy equivalence. Both $\frB_{\wt\nu,\prec(n,[E])}$ and $\frB_{\wt\nu,\preceq(n,[E])}$ are downward subsets of $\frB_{\wt\nu}$ in the sense of Definition \ref{def:downward} (see Example \ref{ex:downward}). Therefore by Proposition \ref{p:contractible}, the inclusions $\Crit(f)_{\wt\nu}\incl |\frB_{\wt\nu,\prec(n,[E])}|$ and $\Crit(f)_{\wt\nu}\incl \frB_{\wt\nu,\preceq(n,[E])}$ are both homotopy equivalences, hence the inclusion $\frB_{\wt\nu,\prec(n,[E])}\incl \frB_{\preceq(n,[E])}$ is also a homotopy equivalence. This verifies the second condition and completes the proof of the theorem.

\end{proof}

\begin{remark} We expect that $hh(\cH_{\cG})$ has a recollement structure indexed by the stronger poset $(\wt\NP, \le)$.  Under Conjecture \ref{thm:betti genus one}, this recollement structure should correspond to the Harder-Narasimhan stratification on $\Bun_{G}(E)$, which is also indexed by $\wt\NP$.
\end{remark}

\sss{Semi-orthogonal decomposition of $hh(\cH_{\cG})$}
\label{sec:semi_orth_decomp}
Now consider $hh(\cH_{\cG})$. By the discussion in Section~\ref{sss:HG comp}, $hh(\cH_{\cG})$ again decomposes into a direct sum $hh(\cH_{\cG})^{\om}$ indexed by $\om\in \Om$.

Identify $\Om$ with length zero elements in $\tilW$, then $\Om$ acts on $\tilW$ by conjugation. The length function $\ell: \tilW\to \ZZ_{\ge0}$ and the enhanced Newton point function $\wt\nu: \tilW\to \wt\NP$ are $\Om$-invariant. Therefore for $J\sft I^{a}$ and $\om\in\Om$, we have a bijection
\begin{equation*}
d_{\om}: \cS_{J}\isom \cS_{\om(J)}
\end{equation*}
sending $\f{u}{J}$ to $\f{\om u\om^{-1}}{\om(J)}$.

For $\wt\nu\in\wt\NP$,  the map $d_{\om}$ induces a bijection 
\begin{equation*}
\cS_{J,\wt\nu}^{\hs}\isom \cS_{\om(J), \wt\nu}^{\hs}.
\end{equation*}
This induces an action of $\Om$ on the $\cD^{\c}$-set $\cS^{\hs}_{\wt\nu}$, and on the total set $\Tot(\cS^{\hs}_{\wt\nu})$. We can form the groupoid $[\Tot(\cS^{\hs}_{\wt\nu})/\Om]$ whose objects are $\Tot(\cS^{\hs}_{\wt\nu})$ and morphisms between $\f{u}{J}$ and $\f{u'}{J'}\in \Tot(\cS^{\hs}_{\wt\nu})$ is the set of $\om\in\Om$ such that $d_{\om}(\f{u}{J})=\f{u'}{J'}$.   

The action of $c_{\om}$ on $\cH_{\cG}$ (see Section~\ref{Om action HG}) induces an equivalence $\cH_{\cG, J}\to \cH_{\cG, \om(J)}$, which restricts to an equivalence $\Sh_{\cN}(\cY(\f{u}{J}))\isom \Sh_{\cN}(\cY(d_{\om}(\f{u}{J})))$.  Therefore, the assignment $\f{u}{J}\mapsto \Sh_{\cN}(\cY(\f{u}{J}))$ also extends to a functor
\begin{equation*}
\Sh_{\cN}(\cY(-)): [\Tot(\cS^{\hs}_{\wt\nu})/\Om] \to \St^{L}, \quad \f{u}{J}\mapsto \Sh_{\cN}(\cY(\f{u}{J})).
\end{equation*}
We can then form the colimit
\begin{equation*}
\colim_{\f{u}{J}\in [\Tot(\cS^{\hs}_{\wt\nu})/\Om]}\Sh_{\cN}(\cY(\f{u}{J})).
\end{equation*}
For $\wt\nu=(0,0)$, $[\Tot(\cS_{\wt\nu}^{\hs})/\Om]$ can be identified with $\cD$. In this case, the above colimit is the same as $\colim_{J\in\cD}\Sh_{\cN}(L_{J}/L_{J})$ under the induction functors.

\begin{theorem}\label{thm:main in text} 
For each $\om\in\Om$, the category $hh(\cH_{\cG})^{\om}$ admits a semi-orthogonal decomposition indexed by non-negative integers
\begin{equation*}
hh(\cH_{\cG})^{\om}_{0}\incl hh(\cH_{\cG})^{\om}_{\le 1}\incl\cdots hh(\cH_{\cG})^{\om}_{\le n}\incl\cdots\incl hh(\cH_{\cG})^{\om}=\bigcup_{n\ge0}hh(\cH_{\cG})^{\om}_{\le n}.
\end{equation*}
In particular, each inclusion $hh(\cH_{\cG})^{\om}_{\le n}\incl hh(\cH_{\cG})^{\om}$ extends to a recollement.

For $n\ge0$, the $n$-th associated graded category $hh(\cH_{\cG})_{n}$ has the following description: it is the direct sum
\begin{equation*}
hh(\cH_{\cG})^{\om}_{n}\simeq\bigoplus_{\wt\nu=(\nu,\om)\in \wt\NP, \j{2\r,\nu}=n}hh(\cH_{\cG})_{\wt\nu}
\end{equation*}
where, for $\wt\nu\in \wt\NP$,
\begin{equation*}
hh(\cH_{\cG})_{\wt\nu}= \colim_{\f{u}{J}\in [\Tot(\cS^{\hs}_{\wt\nu})/\Om]}\Sh_{\cN}(\cY(\f{u}{J})).
\end{equation*}
In particular, the functor $\io$ from \eqref{io CS to colim} is identified with the embedding $hh(\cH_{\cG})_{0}\simeq\colim_{J\in\cD}\Sh_{\cN}(L_{J}/L_{J})\incl hh(\cH_{\cG})^{0}$. In particular, $\io$ is fully faithful.
\end{theorem}
\begin{proof}
Again we will prove the statement for the whole $hh(\cH_{\cG})$, which implies the statement for each summand $hh(\cH_{\cG})^{\om}$. By Corollary \ref{c:coin} we have $hh(\cH_{\cG})\simeq hh(\cH_{\cG^{\c}}, \cH_{\cG})_{\Om}$. Each filtration step $hh(\cH_{\cG^{\c}}, \cH_{\cG})_{\le n}$ in $hh(\cH_{\cG^{\c}}, \cH_{\cG})$ constructed in Theorem \ref{th:semi-orth hh G neutral} is stable under $\Om$. By Proposition \ref{p:recoll}, taking $\Om$-coinvariants (which is the same as taking colimits over $\BB\Om$) gives a semi-orthogonal decomposition of $hh(\cH_{\cG})$ with filtration pieces
\begin{equation*}
hh(\cH_{\cG})_{\le n}\simeq(hh(\cH_{\cG^{\c}}, \cH_{\cG})_{\le n})_{\Om}
\end{equation*}
and associated graded category
\begin{equation*}
hh(\cH_{\cG})_{n}\simeq(hh(\cH_{\cG^{\c}}, \cH_{\cG})_{n})_{\Om}
\end{equation*}
By the description of $hh(\cH_{\cG^{\c}}, \cH_{\cG})_{n}$ given in Theorem \ref{th:semi-orth hh G neutral}, we see $hh(\cH_{\cG})_{n}$ is a direct sum over those $\wt\nu=(\nu,\om)\in \wt\NP$ (such that $\j{2\r,\nu}=n$) of
\begin{equation*}
(hh(\cH_{\cG^{\c}}, \cH_{\cG}^{\om})_{\nu})_{\Om}\simeq \left(\colim_{\f{u}{J}\in \Tot(\cS^{\hs}_{\wt\nu})}\Sh_{\cN}(\cY(\f{u}{J}))\right)_{\Om}.
\end{equation*}
The same argument as in the proof of Corollary \ref{c:coin} allows us to rewrite the right side as a colimit over the groupoid $[\Tot(\cS^{\hs}_{\wt\nu})/\Om]$.
\end{proof}


\section{Endomorphisms of Whittaker functional}

Throughout this section, we assume the coefficient field $k$ is algebraically closed and $\mathrm{char}(k)=0$.

\begin{notation}
	\label{section:notations}
	Fix $\tau $ a complex number with imaginary part $\Im(\tau) \neq 0$. Let $G$ be a reductive group  over $\CC$ with maximal torus $T$  and $\frt=\Lie T$. In this section, we shall adopt the following notations:
	\begin{enumerate}
		\item 
		\begin{itemize}
			\item $\Phi$ the set of roots of $(G,T)$;
			\item $\widetilde{\Phi}=\ZZ \times \Phi$;
			\item $\widehat{\Phi}=\ZZ \times \ZZ \times \Phi$.
		\end{itemize}
		
		\item  Each $\alpha \in \Phi,\widetilde{\Phi}$ or $\widehat{\Phi}$ defines an affine linear function on $\frt$ as follows
		\begin{itemize} 
			\item $\alpha \in \Phi,$ defines a linear functional $\alpha:\frt \to \CC$;
			\item  $\alpha=(n,\ov\alpha) \in \widetilde{\Phi}$ defines an affine linear function $\alpha:=\ov\alpha+n$;
			\item $\alpha=(n,m,\ov\alpha) \in \widehat{\Phi}$ defines an affine linear function $\alpha:= \ov\alpha+n+m\tau$.
		\end{itemize}
		
		Under the above notations, let $H_\alpha= \{x \in \frt: \alpha(x)=0\}$, and $s_\alpha$ the reflection about $H_\alpha$ defined by $s_{\a}(x)=x-\a(x)\ov\a^{\vee}$.
						
		For  $\alpha \in \Phi$, put $\frg_\alpha \subset \frg$ the root space, and for $\alpha=(n,\ov\a) \in \widetilde{\Phi}$ or $\a=(n,m,\ov\a)\in\widehat{\Phi},$ put $\frg_\alpha=\frg_{\ov\alpha}$. 
		
		\item  For $J \subset \Phi,\widetilde{\Phi}$ or $\widehat{\Phi}$, 	
		put $\epsilon_J =  \cap_{\alpha \in S} H_\alpha$,  and define the following sets of affine subspaces of $\frt$
		\begin{itemize} 
			\item $\frS = \{ \epsilon_J :  J \subset \Phi \}$;
			\item $\tilfrS=\{\epsilon_J: J \subset \widetilde{\Phi}\}$ \footnote{Note that by definition $\tilfrS$ is  in natural bijection with the set $\frE$ of relevant affine subspaces in the standard apartment $\frA=\frt_{\RR}$, see Section~\ref{sss:rel aff sp}.};
			\item $\hatfrS=\{\epsilon_J: J \subset \widehat{\Phi}\}$.
		\end{itemize}
		\item 	\begin{itemize}
			\item $W$ the Weyl group of $(G,T)$, it naturally acts $\frt$;
			\item $\widetilde{W} = \xcoch(T) \rtimes W$, it acts on $\frt$ via $(\lambda,w) \cdot x = \lambda+w(x)$;
			\item $\widehat{W}= (\xcoch(T) \times \xcoch(T) )\rtimes W$ (with the diagonal action of $W$), it acts on $\frt$ via $(\lambda_1,\lambda_2,w)\cdot x = \lambda_1 +\tau \lambda_2 +w(x)$.
		\end{itemize}
		\item   \begin{itemize}
			\item For $\epsilon \in \frS$, put $\Phi_\epsilon:=\{ \alpha \in \Phi: \alpha(\epsilon) =0 \}$, and $W_\epsilon=\j{s_\alpha: \alpha  \in \Phi_{\e} } \subset W$;
			\item For $\epsilon \in \tilfrS$, put $\Phi_\epsilon:=\{ \alpha \in \widetilde{\Phi}: \alpha(\epsilon) =0 \}$,  and $W_\epsilon=\j{s_\alpha: \alpha \in \Phi_{\e} } \subset \tilW$;
			\item For $\epsilon \in \hatfrS$, put $\Phi_\epsilon:=\{ \alpha \in \widehat{\Phi}: \alpha(\epsilon) =0 \}$,  and $W_\epsilon=\j{s_\alpha: \alpha\in \Phi_{\e} }\subset \hatW$.
		\end{itemize}
		
		For $\epsilon \in \frS, \widetilde{\frS}$, or $\widehat{\frS}$, 
		put $L_\epsilon $ the connected reductive subgroup of $G$  containing $T$ with roots $\Phi_\epsilon$,  so that it has Weyl group $W_\epsilon$.
		\item  
		\begin{itemize}
			\item For $\epsilon \in \frS,$ put $W^{\epsilon}= N_W(W_\epsilon)/W_\epsilon$;
			\item For $\epsilon \in \tilfrS,$ put $\tilW^{\epsilon}= 			 			
			N_{\tilW}(W_{\epsilon})/W_\epsilon$;
			\item For $\epsilon \in \hatfrS,$ put $\hatW^{\epsilon}= 			 			
			N_{\hatW}(W_{\epsilon})/W_\epsilon.$
		\end{itemize}	
		\item  Put $C_G$ the set of isomorphism classes of cuspidal sheaves on $\calN_G/G$, where $\cN_{G}\subset \frg$ is the nilpotent cone of $G$. For definition, see \cite[\S2]{lusztigFourierTransformsSemisimple1987a}. 
		\begin{itemize}
			\item 	 $\frC=\{(\epsilon,B,F): \epsilon \in \frS, F \in C_{L_\epsilon}, B \supset T$  a Borel subgroup of $L_\epsilon $\};
			\item 	 $\tilfrC=\{(\epsilon,B,F): \epsilon \in \tilfrS, F \in C_{L_\epsilon}, B \supset T$  a Borel subgroup of $L_\epsilon $\};
			\item 	 $\hatfrC=\{(\epsilon,B,F): \epsilon \in \hatfrS, F \in C_{L_\epsilon}, B \supset T$  a Borel subgroup of $L_\epsilon $\}.
		\end{itemize}
		\item We also use the notation $W_G,\frS_G, \frC_G,$ etc, to emphasis the dependence on $G$.
		\item For $J \subset \Phi,\widetilde{\Phi}$ or $\widehat{\Phi}$, 
		put $L_J=L_{\epsilon_J}, W_J=W_{\epsilon_J},  W^{J}=W^{\epsilon_J}, \frz_J=\xcoch(Z(L_{J})^{\c})\ot_{\ZZ}k$, $C_J=C_{{L_{J}}}$, $\tilfrC_J=\tilfrC_{L_J}$, etc.
		
	\end{enumerate}
\end{notation}

\subsection{Combinatorial descriptions of character sheaves} In this section, we shall recall the combinatorial description of the dg category of character sheaves in \cite{liDerivedCategoriesCharactera} (for simply-connected group), and \cite{liDerivedCategoriesCharacterb} (for reductive group).  Using the same idea, we give a combinatorial description of the colimit category $\colim_{J\in \cD}\Sh_{\cN}(L_{J}/L_{J})$, the latter may be thought of as an elliptic version of character sheaves. The main results and the relationship between various versions of character sheaves is summarized in Corollary \ref{cor:induction_and_alpha}.

\subsubsection{Sheaves of categories on affine spaces} Let $\Lambda$ be a free $\ZZ$-module of finite rank, and $\Lambda_\CC$ be complexified affine space. Let $\frF$ be a locally finite set of complex affine subspaces of $\Lambda_\CC$. For any $\epsilon \in \frF$, denote $\overline{\epsilon}$ the linear space parallel to $\epsilon$, and $\frz_\epsilon= (\overline{\epsilon} \cap \Lambda) \otimes_\ZZ k$ the finite dimensional affine space over $k$. Put $\scrQCoh_\epsilon$ the constant sheaf of categories on $\epsilon$ with value $\QCoh(\frz^*_\epsilon[-1])$. For a map of sets $f: \frG \to \frF$, denote $\scrQCoh_{\frG}=\prod_{c \in \frG} \scrQCoh_{f(c)}$, viewed as a sheaf of categories on $\L_{\CC}$. Let $K$ be a discrete group of affine linear transformations acting properly discontinuously on $\Lambda_\CC$, assume that the linear part of $K$ preserves $\Lambda$, and $\frF$ is stable under the $K$-action. Assume also that $K$ acts on $\frG$, such that $f$ is $K$-equivariant, then 
$\scrQCoh_{\frG}$ is naturally a $K$-equivariant sheaf of categories on $\Lambda_\CC$. Therefore, for any $K$-invariant open subset $U\subset \Lambda_\CC$, we have natural $K$-action on $\Gamma(U,\scrQCoh_{\frG})$. Denote the invariant category by $\Gamma(U,\scrQCoh_{\frG})^K$.

\begin{remark}
	\label{rmk:invariant_category}
	Let $A$ be an dg-algebra over $k$, and $K$ be a discrete group acting strictly on $A$.
	The \textit{smash product} $k[K]\#A$ is by definition a dg algebra whose underlying dg vector space is
	 $k[K] \otimes A$, and the multiplication is given by 
	 $(w \otimes a) \cdot (w' \otimes a') = (ww' \otimes w'^{-1}(a)a'),$
	  for all $w \in K$ and $ a \in A$. There is natural equivalence of dg categories (see e.g. \cite[Section 2.2]{liDerivedCategoriesCharactera}): 
	\begin{equation*}
	(k[K]\#A) \textup{-mod}  \simeq   (A \textup{-mod})^K 
	\end{equation*}
	
	Using smash product, we can write $\Gamma(\Lambda_{\CC},\scrQCoh_{\frG})^K$ more concretely as follows:
	 let $\frG\sslash K$ be the set of $K$-orbits on $\frG$, and let $[\frG\sslash K]\subset \frG$ be a set of representatives of each orbit, and denote $K_c$ the stabilizer of $K$ at $c \in \frG$. 
	 Then we have equivalence of categories:  
	 \beq\label{Kinv section QCoh} \Gamma(\Lambda_{\CC},\scrQCoh_\frG)^K \simeq \prod_{c \in [\frG\sslash K]} (\Sym(\frz_{f(c)}[1]) \textup{-mod})^{K_c}
	  \simeq \prod_{c \in [\frG\sslash K]} k[K_c]\#\Sym(\frz_{f(c)}[1]) \textup{-mod}.
	 \eeq
\end{remark}

Let  $K' \subset K$ be a subgroup, and $\frG' \subset \frG$ be a $K'$-stable subset. We have a pair of adjoint functors:
\beq  \label{eq:induction_restriction}   
  \Ind_{\frG'/K'}^{\frG/K} : \Gamma(\Lambda_{\CC}, \scrQCoh_{\frG'})^{K'} \rightleftarrows           \Gamma(\Lambda_{\CC}, \scrQCoh_{\frG})^K  :\Res^{\frG/K}_{\frG'/K'}         
\eeq
And when the context is clear, we shall simply denote them by $\Ind, \Res$.

 Now we take $\Lambda=\xcoch(T)$, and hence $\Lambda_\CC = \frt $. In Notation~\ref{section:notations}, the groups $W, \tilW$ and $\hatW$ naturally act on $\frS, \tilfrS$ and $\hatfrS$ respectively. For any $w \in W, \tilW$ or $\hatW$, denote its image in $W$ by $\ov w$. Therefore  $\Ad(\ov w)$ gives idenfications $L_{\epsilon} \simeq L_{w(\epsilon)}$, and $C_{L_{\epsilon}} \simeq C_{L_{w(\epsilon)}}$.  This gives an action of $W,\tilW$ and $\hatW$ on $\frC, \tilfrC$ and $\hatfrC$ respectively. Let  $f:\frC\to \frS$, $\wt f: \tilfrC\to \tilfrS$ and $\wh f:  \hatfrC\to  \hatfrS$ be the projections. Therefore the categories
$$\Gamma(\frt, \scrQCoh_{\frC})^W, \Gamma(\frt, \scrQCoh_{\wt\frC})^{\wt W}, \mbox{ and }\Gamma(\frt, \scrQCoh_{\wh\frC})^{\wh W}$$
are defined.

We have the following combinatorial description of character sheaves on $G$ and $\frg$:

\begin{theorem} 
	\label{thm:combinatorial_descrpition_character_sheaves} Let $G$ be a connected reductive group over $\CC$. There are equivalences of categories:
	\begin{enumerate}
		\item  $\LL_{\frg}: Sh_\calN(\frg/G)   \simeq   \Gamma(\frt, \scrQCoh_{\frC})^W; $
		\item  $\LL_{G}:   Sh_\calN(G/G) \simeq \Gamma(\frt,\scrQCoh_{\widetilde{\frC}})^{\widetilde{W}}. $
	\end{enumerate}
	
\end{theorem}
\begin{proof}
 (1) This is essentially the generalized Springer decomposition.  We use notations from Section~\ref{sss: g notation}; in particular, $I$ denotes the set of simple roots of $G$ with respect to $T$ and $B$. Let
\begin{equation}\label{def frD}
\frD:=\{(J,F): J \subset I,  F \in C_J=C_{L_{J}}\}.
\end{equation}
By \cite[Theorem 1.2]{liDerivedCategoriesCharactera},  we have an equivalence of categories:
	\begin{equation}
		\label{eq:character_lie_algebra_facet_description}
		Sh_\calN(\frg/G) \simeq \bigoplus_{(J,F)\in \frD} k[W^J]\#\Sym(\frz_J[1]) \textup{-mod}.
	\end{equation}
Let $\wt f:  \frD\to \frC$ be the map sending $(J,F)$ to $(\e_{J}, B\cap L_{J}, F)$, and let $f$ be the composition $\frD \xr{\wt f} \frC \to \frC\sslash W$. Then $f$ is bijective by \cite[Lemma 3.1]{liDerivedCategoriesCharactera} (the content here is: if $L_{J}$ carries a cuspidal local system, then any two parabolics containing $L_{J}$ as a Levi subgroup are conjugate in $G$). One can further identify $W^{J}$ with $W_{c}$, the stabilizer of $c = f(J,F)$ under the $W$-action on $\frC$. Therefore, (1) holds by \eqref{Kinv section QCoh}.  

(2) is proved in \cite[Theorem 1.1.1]{liDerivedCategoriesCharacterb}.
\end{proof}

In view of Remark~\ref{rmk:invariant_category}, Theorem~\ref{thm:combinatorial_descrpition_character_sheaves} implies the following explicit description of the category of character sheaves on the group $G$, analogous to \eqref{eq:character_lie_algebra_facet_description}.

\begin{cor}\label{c:comb CS G} For any connected reductive group $G$ over $\CC$, there is a canonical equivalence
	\begin{equation}\label{CS G blocks}
		Sh_\calN(G/G) \simeq \bigoplus_{\{(J,F): J \sft I^a, F \in C_J\}/\Om} k[\tilW^J]\#\Sym(\frz_J[1]) \textup{-mod}.
	\end{equation}
\end{cor}

\begin{remark}\label{r:supp block} Let $\cF\in Sh_\calN(G/G)$ be an object whose image on the right side of \eqref{CS G blocks} lies in a single summand indexed by $(J,F)$. Then the support of $\cF$ is contained in the closed subset of $G$ consisting of elements whose semisimple parts lie in $\Ad(G)(\exp(\e_{J}))$, where the exponential map $\exp:\frt\to T$ normalized to have kernel $\xcoch(T)$. 
\end{remark}

Now we explain the relationship between character sheaves on the group and on the Lie algebra. There is a natural functor $\beta_1: Sh_\calN(G/G) \to  Sh_\calN(\frg/G) $ defined as follows: let $D$ be a $\Ad(G)$-invariant star-shaped open subset of $\frg$ such that $\exp: D\to G$ is an isomorphism onto its image. Let $j: D/G \hookrightarrow \frg/G$ be the open inclusion. 
Then $j^*: \Sh_\calN(\frg/G) \to \Sh_\calN(D/G)$ is an equivalence of categories, and we put 
\begin{equation}
	\label{eq:beta}
	\beta_1= (j^*)^{-1} \circ \exp|_{D/G}^*: { \Sh_{\cN}(G/G)\to \Sh_{\cN}(\frg/G).}
\end{equation}
The functor $\beta_1$ perserves limit, and therefore { admits a left adjoint $\alpha_1$}.

\begin{prop}
	\label{prop:restriction_G_to_g}
	Under the equivalence in Theorem~\ref{thm:combinatorial_descrpition_character_sheaves},
	the diagrams naturally { commute}:
	$$\xymatrix{
		Sh_\calN(\frg/G)  \ar@<-.5ex>[d]_{\alpha_1}  \ar[r]^-{\simeq}_-{\LL_{\frg}}  &    \Gamma(\frt, \scrQCoh_{\frC})^W  \ar@<-.5ex>[d]_{\Ind} \\
		Sh_\calN(G/G)  \ar[r]^-{\simeq}_-{\LL_{G}}  \ar@<-.5ex>[u]_{\beta_1}  &  \Gamma(\frt, \scrQCoh_{\tilfrC})^{\tilW} \ar@<-.5ex>[u]_{\Res} 
	}$$
\end{prop}
\begin{proof}
	This follows from the proof of \cite[Theorem 4.4.3]{liDerivedCategoriesCharacterb}: 
	under the notation \textit{loc. cit.}, the functor $Sh_\calN(G/G) \to Sh_\calN(\frg/G)$ is identified with the natural functor $\lim_{I  \in \frF_G} \textsf{Ch}_G \to  \textsf{Ch}_G(I=0)$, which is identified with the functor $\lim_{I  \in \frF_G} \textsf{QCoh}_{\tilfrC}^{\tilW}  \to  \textsf{QCoh}_{\tilfrC}^{\tilW} (I=0),$ which is then further identified with $\Res  :\Gamma(\frt, \scrQCoh_{\tilfrC})^{\tilW} \to \Gamma(\frt, \scrQCoh_{\frC})^W $. The other commutative square follows by adjunction.
\end{proof}

Now let $P \supset T$ be a parabolic subgroup of $G$, with Levi subgroup $L\supset T$. Then we have a pair of adjoint functors by parabolic induction and restriction:
$$  \Ind_{L \subset P}^G : Sh_\calN(L/L) \rightleftarrows Sh_\calN(G/G) : \Res_{L \subset P}^G $$

\begin{prop}[\cite{liDerivedCategoriesCharacterb},Theorem 1.1.2]
	\label{prop:restriction_G_to_L}
	Under the equivalence in Theorem~\ref{thm:combinatorial_descrpition_character_sheaves}(2),
	the diagrams naturally commute:
	$$\xymatrix{
		Sh_\calN(L/L)  \ar@<-.5ex>[d]_{\Ind_{L \subset P}^G}  \ar[r]^-{\simeq}_-{\LL_{L}}   &    \Gamma(\frt, \scrQCoh_{\tilfrC_L})^{\tilW_L}  \ar@<-.5ex>[d]_{\Ind} \\
		Sh_\calN(G/G)  \ar[r]^-{\simeq}_-{\LL_{G}}  \ar@<-.5ex>[u]_{\Res_{L \subset P}^G}  &  \Gamma(\frt, \scrQCoh_{\tilfrC_G})^{\tilW_G} \ar@<-.5ex>[u]_{\Res} 
	}$$
\end{prop}

Next we will prove a double affine version of Theorem~\ref{thm:combinatorial_descrpition_character_sheaves}. Instead of considering character sheaves on the Lie algebra or the group, we consider the colimit $\colim_{J \in \cD} Sh_\cN(L_J/L_J)$. Recall in Section~\ref{sec:semi_orth_decomp} we have identified the  initial filtered piece $hh(\cH_\cG)_0$ of $hh(\cH_\cG)$ with this colimit. From the point of view of Betti geometric Langlands, this colimit can be thought of as an elliptic version of character sheaves, i.e., nilpotent sheaves on the semistable locus of $\Bun_{G}(E)$ for an elliptic curve.

\begin{theorem}
	\label{thm_gluing_character_sheaves}
	There is an equivalence of categories 
	$$\LL_{\cG}:  \colim_{J \in \cD} Sh_\calN(L_J/L_J)  \simeq  \Gamma(\frt,\mathscr{QC}\textup{oh}_{\widehat{\frC}})^{\widehat{W}} $$
\end{theorem}
where the colimit is formed using the parabolic induction functors as in Section~\ref{sss:colim CS}.\begin{proof}
We may rewrite the colimit on the left side as a limit
\begin{equation*}
\lim_{J\in \cD^{\opp}} Sh_\calN(L_J/L_J)
\end{equation*}
using the parabolic restriction functors that are right adjoint to the induction functors.

	We view $\widetilde{W}$ as a subgroup of  $ \widehat{W}$ in two ways:    $\widetilde{W}= (\xcoch(T) \times 0) \rtimes W$, and $\widetilde{W}^\tau= (0 \times \xcoch(T)) \rtimes W$. Denote $W^a \subset \widetilde{W}, W^{a,\tau} \subset \widetilde{W}^\tau$ the corresponding affine Weyl group. Let $\frA:=X_*(T)_\RR$. Recall we have fixed an imaginary $\t$ so that $\frt=\frA_{\CC}=\frA\times \frA \t$. For a subset $X\subset \frA$, $X \t$ denotes the corresponding subset of $\frA \t$, which in turn is a subset of $\frt$.  	
	
	The affine space $\frA$ is stratified by facets given by connected components of intersections of affine roots hyperplanes. For any $J \sft I^{a}$, let $\frA_{J}$ be the corresponding facet in the fundamental alcove, i.e.,  $\frA_J= \{x \in \frA | \alpha(x) =0, \beta(x) >0,  \mbox{ for all }\alpha \in J, \beta \in I^a \setminus  J \}$. Let $\str(\frA_J) \subset \frA$ be the star of $\frA_J$, namely, $\str(\frA_J)$ is the union of facets whose closures contain $\frA_J$. Denote the open subset:
	$${V}_J:= \frA \times \str(\frA_J)\tau \subset \frA \times \frA\tau = \frA_{\CC}.$$
	We have an isomorphism of topological groupoids
	\begin{equation*}
		\colim_{J \in \cD^\c}  \str(\frA_J)\tau/W_J= \frA \tau/ W^{a,\tau}. 
	\end{equation*}
Dividing $\Omega$ on both side, we get an isomorphism of topological groupoids
\begin{equation}\label{STJD}
\colim_{J \in \cD}  \str(\frA_J)\tau/W_J \simeq \frA\t/ \widetilde{W}^\tau. 
\end{equation}
	Let $\xcoch(T)  \rtimes W^{\t}_J\subset \wh W$ be the subgroup generated by $\xcoch(T)\times 0$ and $W^{\t}_{J}\subset W^{a,\t}$. Abstractly, it is isomorphic to the semi-direct product  $\xcoch(T)\rtimes W_{J}\simeq \tilW_{J}$, the extended affine Weyl group of $L_{J}$.
	Then \eqref{STJD} induces an isomorphism of topological groupoids
	\begin{equation*}
\colim_{J \in \cD} V_J/ (\xcoch(T)  \rtimes {W}^{\t}_J) \simeq \frt/\widehat{W}.
\end{equation*}
	Since all arrows in the above colimit are open embeddings, it further induces an equivalence
	\begin{equation}
		\label{eq:limit_QCoh_WC}
		\begin{split}
			\Gamma(\frt,\mathscr{QC}\textup{oh}_{\widehat{\frC}})^{\widehat{W}}  &   \simeq  \Gamma(\frt/\widehat{W} ,\mathscr{QC}\textup{oh}_{\widehat{\frC}})  \\
			& \simeq \lim_{J \in \cD^{\opp}}  \Gamma({V}_J/ (\xcoch(T)  \rtimes {W}^{\t}_J), \mathscr{QC}\textup{oh}_{\widehat{\frC}})  \\
			& \simeq \lim_{J \in \cD^{\opp}}  \Gamma(V_J, \mathscr{QC}\textup{oh}_{ \widehat{\frC}})^{ \xcoch(T)  \rtimes {W}^{\t}_J}
		\end{split}
	\end{equation}  
	
	
	On the other hand, choose any { $x \in \frA_{J} \tau$}, and put $t_x:\frt \to \frt$ the translation by $x$. Then there is an canonical isomorphism of sheaves on $V_J$
	\begin{equation}\label{QCoh VJ}
t_{x,*}(\mathscr{QC}\textup{oh}_{\widetilde{\frC}_J})|_{V_J} \simeq \mathscr{QC}\textup{oh}_{\widehat{\frC}}|_{V_J}.
	\end{equation}
 	This is because affine spaces in $\hatfrC$ having nonempty intersections with $V_J$ are precisely those of the form $t_x(\epsilon)$, for $\epsilon$ in $\tilfrC_J$. The isomorphism \eqref{QCoh VJ} is naturally $\xcoch(T) \rtimes W^{\t}_J\simeq\widetilde{W}_J$ equivariant, and induces an equivalence on global sections:
	\begin{equation}
		\label{eq:isomorphism_of_sections}
		\Gamma( {V_J},  \mathscr{QC}\textup{oh}_{\widehat{\frC}})^{\xcoch(T) \rtimes W^{\t}_J} \simeq \Gamma(t^{-1}_x(V_J), \mathscr{QC}\textup{oh}_{\widetilde{\frC}_J} )^{\widetilde{W}_J} \simeq 
		\Gamma(\frt, \mathscr{QC}\textup{oh}_{\widetilde{\frC}_J} )^{\widetilde{W}_J}.
	\end{equation}
	Here, the second equivalence is because $t^{-1}_x(V_J)$ is convex 
	and has nonempty intersection with each affine space in $\tilfrC_J$ 
	 (and hence the intersection is contractible).
	Moreover for $J \subset J'\sft I^{a}$, under (\ref{eq:isomorphism_of_sections}) the restriction map:
	$$ \Gamma(V_J, \mathscr{QC}\textup{oh}_{\widehat{\frC}})^{ \xcoch(T)  \rtimes {W}^{\t}_J}   \to \Gamma(V_{J'}, \mathscr{QC}\textup{oh}_{ \widehat{\frC}})^{\xcoch(T)  \rtimes {W}^{\t}_{J'}}$$
	is identified with the restriction map defined in \eqref{eq:induction_restriction}:
	$$ \Res^{\wt{\frC}_{J'}/\tilW_{J'}}_{\wt{\frC}_{J}/\tilW_{J}}: \Gamma(\frt, \mathscr{QC}\textup{oh}_{\widetilde{\frC}_J} )^{\widetilde{W}_J} \to \Gamma(\frt, \mathscr{QC}\textup{oh}_{\widetilde{\frC}_{J'}} )^{\widetilde{W}_{J'}}  $$
	It is also clear that \eqref{eq:isomorphism_of_sections} is equivariant under the action of $\Om$. Therefore, the functor $\cD^{\opp}\to \St^{R}_{k}$ defined by $J\mapsto \Gamma( {V_J},  \mathscr{QC}\textup{oh}_{\widehat{\frC}})^{\xcoch(T) \rtimes W^{\t}_J}$ can be identified with the functor $J\mapsto \Gamma(\frt, \mathscr{QC}\textup{oh}_{\widetilde{\frC}_J} )^{\widetilde{W}_J}$ (using the restriction functors and the $\Om$-action).
	In summary, we have 
	\begin{equation}
		\label{eq:gluing_character_sheaves}
		\begin{split}
			\Gamma(\frt,\scrQCoh_{\hatfrC})^{\hatW} & \simeq  \lim_{J \in \cD^{\opp}}  \Gamma(V_J, \mathscr{QC}\textup{oh}_{\widehat{\frC}})^{ \xcoch(T)  \rtimes {W}^{\t}_J}     \qquad  \textup{by } \eqref{eq:limit_QCoh_WC}  \\
			& \simeq \lim_{J \in \cD^{\opp}}\Gamma(\frt,\scrQCoh_{\tilfrC_J})^{\tilW_J} \qquad \textup{by } \eqref{eq:isomorphism_of_sections} \\
			& \simeq \lim_{J \in \cD^{\opp}} Sh_\calN(L_J/L_J)   \qquad \textup{by Theorem~\ref{thm:combinatorial_descrpition_character_sheaves} (2) and Prop~\ref{prop:restriction_G_to_L}}
		\end{split}
	\end{equation}
	
\end{proof}

\begin{cor}
	\label{cor:induction_and_alpha}	
	Under Theorem~\ref{thm:combinatorial_descrpition_character_sheaves} and Theorem~\ref{thm_gluing_character_sheaves},
	the diagrams naturally commutes:
	\beq \label{eq:identify_induction_functor}
	\xymatrix{ Sh_\calN(\frg/G) \ar[r]^-{\alpha_1}\ar[d]^{\simeq}_{\LL_{\frg}} & Sh_\calN(G/G) \ar[r]^-{\alpha_2} \ar[d]^{\simeq}_{\LL_{G}} & \colim_{J\in \cD} Sh_\calN(L_J/L_J) \ar[d]^{\simeq}_{\LL_{\cG}} \\
		\Gamma(\frt,\scrQCoh_{\frC})^W 	\ar[r]^{\Ind_1}  & \Gamma(\frt,\scrQCoh_{\tilfrC})^{\tilW} \ar[r]^{\Ind_2}   & 	\Gamma(\frt,\scrQCoh_{\hatfrC})^{\hatW} 			 }    
	\eeq
	where $\alpha_2$ is the natural map into the colimit corresponding to $J=I$, and $\Ind_1, \Ind_2$ are the induction functors defined in \eqref{eq:induction_restriction}.
\end{cor}
\begin{proof}
	The first square is commutative by Proposition~\ref{prop:restriction_G_to_g}. For the second square, by \eqref{eq:gluing_character_sheaves}, the right adjoint of $\a_{2}$ corresponds to the restriction map $\Gamma(\frt,\scrQCoh_{\hatfrC})^{\hatW}\to \Gamma(V_{I},\scrQCoh_{\wt\frC})^{\tilW}$, and the latter can be identified with $\Gamma(\frt,\scrQCoh_{\wt\frC})^{\tilW}$. Passing to left adjoints gives the commutativity of the second square.
\end{proof}

\subsection{Fourier-Sato transform and Whittaker sheaf}\label{ss:FS Wh} In this section, we study the Lie algebra version of the Whittaker sheaf, and identify its image under the equivalence in Proposition~\ref{prop:restriction_G_to_g}. The main tool here is the Fourier-Sato transform. 

Consider the diagram
\begin{equation}\label{fru}
\xymatrix{       \GG_a     &        \fru^{-}/U^{-}   \ar[l]_-{f}  \ar[r]^-{r_{-}} & \frg/G   }
\end{equation}
where $r_{-}$ is induced by the inclusion $\fru^{-}\incl \frg$, and $f=d\chi: \fru^{-}\to \GG_a$ is the differential of the generic character $\chi: U^{-}\to \Ga$ used to define the Whittaker functor on the group, see Section~\ref{sss:Whit G}. 

\begin{defn}\label{def:Wg}
\begin{enumerate}
\item The Whittaker functor on $\Sh_{\cN}(\frg/G)$ is the functor 
\begin{equation*}
\cW_{\frg/G}: \Sh_{\cN}(\frg/G)\to k\lmod, \quad \calW_{\frg/G}(\cF)= \varphi_{f,0} \circ r_{-}^!\cF.
\end{equation*}

\item The Whittaker sheaf on the Lie algebra is the object $\Wh_{\frg/G}\in \Sh_{\cN}(\frg/G)$ corepresenting the Whittaker functor $\cW_{\frg/G}$, i.e.,
\begin{equation*}
\cW_{\frg/G}(\cF)\simeq \Hom_{\Sh_{\cN}(\frg/G)}(\Wh_{\frg/G}, \cF), \quad \mbox{ for all }\cF\in \Sh_{\cN}(\frg/G).
\end{equation*}
\end{enumerate}
\end{defn}


Let 
$$ \mathbb{T}:  Sh(\calN/G) \simeq  Sh_\calN(\frg/G) $$ 
be the Fourier-Sato transform, normalized to be $t$-exact with respect to the perverse $t$-structures.

\begin{prop}
	\label{prop:Fourier_transform} 
	Let $e \in \fru$ be the regular nilpotent element corresponding to $f\in (\fru^{-})^{*}$ under the invariant form on $\frg$.  Let $i_e:\{e\} \to \calN/G$ be the natural map. 
	Then there is a natural commutative diagram:
	$$ \xymatrix{ Sh_\calN(\frg/G)  \ar@{<-}[r]^{\mathbb{T}} \ar[rd]^{\calW_{\frg/G}} &  Sh(\calN/G) \ar[d]^{i^*_e[r]}  \\
		&		k\lmod   }  
	$$
\end{prop}

\begin{proof}
	Since $f$ is linear, we have 
	\beq \label{vanf}
	\varphi_{f,0}(\cK)= \varphi_{id,0} (f_*(\cK))
	\eeq
	 for any $\GG_m$-monodromic sheaf $\cK$ on $\fru^{-}$ (where $\GG_m$ acts by dialition). 
	Moreover, for any $\GG_m$-monodromic sheaf $\cE$ on $\GG_a$, we have 
	\beq \label{vanid}
	\varphi_{id,0}(\TT_{\Ga}(\cE))[-1] \simeq i^!_{1}(\cE)
	\eeq
	where $i_{1}: \{1\}\incl \Ga$,  and $\TT_{\Ga}$ is the Fourier-Sato transform for $\Sh(\Ga)$.

	We have the dual maps to \eqref{fru}  $$\xymatrix{       \GG_a  \ar[r]^-{i}    &     \fru   & \frg  \ar[l]_-{\pi} } ,  $$
	where $\pi$ is the quotient map to $\frg/\frb^{-}\cong \fru$, and $i(c)= c e$. 
	For any $\cF \in Sh(\calN/G)$, we have 
	\begin{equation*}
	\calW_{\frg/G} \circ \TT (\cF) 
	\simeq \varphi_{id,0} \circ f_* \circ r_{-}^! \circ \TT (\cF)
	\simeq \varphi_{id,0} \circ \TT_{\Ga} \circ i^! \circ \pi_*(\cF)[2\nu + r -1]
	\end{equation*}
$$	
	\simeq i^!_1 \circ i^! \circ \pi_*(\cF)[2\nu +r ] \simeq R\G( i^!_{\pi^{-1}(e)} \cF)[2 \nu +r ]
	\simeq  R\G(i^!_{\pi^{-1}(e) \cap \calN} \cF)[2\nu + r]
	\simeq i^*_{e} \cF [r] 
$$
for $r=\dim T$, and $\nu= \dim U$,	where the first isomorphism uses \eqref{vanf}, the second uses the standard property of the Fourier-Sato transform under linear maps,  and the third uses \eqref{vanid}. The last isomorphism is because $\pi^{-1}(e) \cap \calN = (e +\frb^{-}) \cap \calN$ is the free orbit of $e$ under the adjoint action of $U^{-}$, which is contractible and along which $\cF$ is contant.
\end{proof}

\begin{prop}
	\label{prop:wh_G_and_wh_g}
	The diagram naturally commutes:
	$$\xymatrix { Sh_\calN(G/G) \ar[r]^{\beta_1} \ar[dr]^{\calW_{G/G}} &   Sh_\calN(\frg/G) \ar[d]^{\calW_{\frg/G}}  \\
		&	k\lmod	 }$$
\end{prop}
\begin{proof}
	Let $D \subset \frg$ as in the definition of $\beta_1$ (c.f (\ref{eq:beta})). Put $D':= \exp(D) \subset G$. Let $f'=\chi:U^{-} \to \GG_a$ and $f=d\chi :\fru^{-} \to \GG_a$, then the diagram commutes:
	$$\xymatrix{   \GG_a  \ar[d]^{=}   &        \fru^{-}\cap D \ar[d]^{\exp}_{\simeq}  \ar[l]_-{f}  \ar[r]^-{r_0} & D/G \ar[d]^{\exp}_{\simeq}  \ar[r]^{j} & \frg/G \\
		\GG_a     &        U^{-}\cap D'   \ar[l]_-{f'}  \ar[r]^-{r'_0} & D'/G   \ar[r]^{j'} & G/G
	}$$
	The Whittaker functor $\cW_{{\frg/G}}$ on the Lie algebra $\frg$ can be identified with $\ph_{f,0}\c r_{0}^{!}\c j^{*}$,  therefore 
	$$\calW_{\frg/G} \circ \beta_1 =  \varphi_{f,0} \circ r_0^! \circ j^* \circ (j^*)^{-1} \circ  \exp|_{D/G}^* \circ j'{}^* =  \varphi_{f',1} \circ r'_0{ }^! \circ j'{}^* = \calW_{G/G}.$$
\end{proof}

\sss{Generalized Springer correspondence}
\label{sec:generalized_springer}
 We recall the generalized Springer correspondence due to Lusztig \cite[Theorem 6.5]{lusztigIntersectionCohomologyComplexes1984}.
For a nilpotent orbit $O$,  let $A_{G}(O)=C_{G}(e)/C_{G}(e)^{\c}$ for $e\in O$, which is a finite group well-defined up to inner automorphisms. For a finite group $\G$, let $\Irr(\G)$ be the set of irreducible representations of $\G$ over $k$. Then the generalized Springer correspondence is a bijection 
\begin{equation*}
\xymatrix{\frP:=\{(O,\chi): O \textup{ a nilpotent orbit of } G, \chi \in \Irr(A_{G}(O))\}\ar[d]^{\GSpr}_{1:1}\\
\wt\frD:=\{(J,F,\theta): J \subset I,  F \in C_J, \theta \in \Irr(W^J)\}.
}
\end{equation*}
Let $\frP^{\reg}\subset\frP$ be the subset of pairs $(O,\chi)$ where $O=O^{\reg}$ is the regular nilpotent orbit. Since $A_G(O^{\reg})=Z(G)/Z(G)^{\c}$, we can identify $\frP^{\reg}$ with $\Irr(Z(G)/Z(G)^{\c})$. 

On the other hand, recall the set $\frD$ from \eqref{def frD}. Let $\frD^{\reg}\subset\frD$ be the subset of those $(J,F)\in\frD$ such that $\Supp(F)=\calN_{L_J}$. 

\begin{lemma}\label{l:IrrZ Dreg} Under the map
\begin{equation}\label{PD}
\frP\xr{\GSpr}\wt\frD\to\frD
\end{equation}
the subset $\frP^{\reg}$ maps bijectively to $\frD^{\reg}$. In particular, it induces a bijection
\begin{equation}\label{IrrZ Dreg}
\Irr(Z(G)/Z(G)^{\c})\xr{1:1} \frD^{\reg}
\end{equation}
sending $\chi$ to the unique pair $(J_{\chi}, F_{\chi})\in \frD^{\reg}$ such that $\GSpr(O^{\reg},\chi)=(J_{\chi}, F_{\chi}, \th_{\chi})$ for some $\th_{\chi}\in \Irr(W^{J_{\chi}})$.
\end{lemma}
\begin{proof}
Let $(J,F,\th)=\GSpr(O^{\reg}, \chi)$. We claim that $\Supp(F)=\cN_{L_{J}}$. By the construction of the generalized Springer correspondence, $\IC(O^{\reg}, \chi)$ is a direct summand of $\Ind_{L_{J}}^{G}(F)$, where $\Ind_{L_{J}}^{G}$ is the pull-push functor along the diagram (and $P_{J}$ is a parabolic subgroup of $G$ containing $L_{J}$ as a Levi subgroup)
\begin{equation*}
\xymatrix{ \cN_{G}/G & \cN_{P_{J}}/P_{J}\ar[r]^{q_{J}}\ar[l]_{p_{J}} & \cN_{L_{J}}/L_{J}
}
\end{equation*}
If $\Supp(F)$ was properly contained in $\cN_{L_{J}}$, $p_{J}(q^{-1}_{J}\Supp(F))$ would not meet $O^{\reg}$. Therefore we must have $\Supp(F)=\cN_{L_{J}}$. This shows that the map \eqref{PD} sends $\frP^{\reg}$ to $\frD^{\reg}$, and the map \eqref{IrrZ Dreg} is defined. The map \eqref{IrrZ Dreg} is bijective because for each $(J,F)\in \frD^{\reg}$, $\Ind_{L_{J}}^{G}(F)$ has full support on $\cN_{G}$, and its restriction to $O^{\reg}$ has rank one, hence it has a unique direct summand of the form $\IC(O^{\reg}, \chi)$ for some $\chi\in \Irr(Z(G)/Z(G)^{\c})$.
\end{proof}


\begin{prop}\label{p:Wg image}
	Under the equivalence $(\ref{eq:character_lie_algebra_facet_description})$, the { Lie algebra Whittaker sheaf} $$\Wh_{\frg/G} \in Sh_\calN(\frg/G)$$ 
	is identified with 
	$$\bigoplus_{(J_{\chi},F_{\chi}) \in \frD^{\reg}} \Sym(\frz_{J_\chi}[1]) [-r]\in  \bigoplus_{(J,F) \in \frD} k[W^J]\#\Sym(\frz_J[1]) \textup{-mod}\simeq \Sh_{\cN}(\frg/G).$$
	Here we use the bijection in Lemma \ref{l:IrrZ Dreg} to index elements in $\frD^{\reg}$ by $\chi\in \Irr(Z(G)/Z(G)^{\c})$.
	
\end{prop}
\begin{proof} 
	Let $A_{J}=k[W^J]\#\Sym(\frz^*_J[-2])$. Let $k_{\triv}$ be the $1$-dimensional $A_{J}$-module in degree $0$ where $W^{J}$ acts trivially. Then $W^{J}$-equvariant Koszul duality gives an equivalence
	\begin{equation}\label{KD}
	\Hom_{\Sym(\frz^*_J[-2])}(k_{\triv}, -): A_{J}\textup{-perf}\simeq k[W^J]\#\Sym(\frz_J[1]) \textup{-mod}_{\textup{fd}} 
	\end{equation}
	where we identify $\Sym(\frz_J[1])$ with $\End_{\Sym(\frz^*_J[-2])}(k_{\triv})^{\opp}$, and $\textup{-mod}_{\textup{fd}}$ stands for the category of finite dimensional modules. Under this equivalence, $k_{\triv}$ corresponds to $\Sym(\frz_J[1])$.

	Let $Sh_{c,\calN}(\frg/G) \subset \Sh_{\cN}(\frg/G)$ be the full subcategory of constructible sheaves. Using \eqref{eq:character_lie_algebra_facet_description} and \eqref{KD}, we have equivalences
	 \beq\label{constr gG}
	 \xymatrix{Sh_{c,\calN}(\frg/G) \simeq   \bigoplus_{(J,F) \in \frD}  k[W^J]\#\Sym(\frz_J[1]) \textup{-mod}_{\textup{fd}}  \ar[r]^-{\eqref{KD}}_-{\sim} & \bigoplus_{(J,F) \in \frD}  A_{J} \textup{-perf}.}
	 \eeq
	 Moreover, the composition
	\begin{equation*}
	Sh_{c}(\calN/G) \xrightarrow{\TT} Sh_{c,\calN}(\frg/G) 	\simeq \bigoplus_{(J,F) \in \frD}  A_{J}\textup{-perf}
	\end{equation*}
	is identified with the functor $\Hom_{\calN/G}(  \oplus_{(J,F) \in \frD} \Ind_{L_J \subset P_J}^G F,-  )$, where we use 
	\begin{equation*}
	\End(\Ind_{L_J \subset P_J}^G F)\simeq A_{J}^{\opp}.
	\end{equation*}

	By Prop~\ref{prop:Fourier_transform}, $\TT^{-1}\Wh_{\frg/G}$ is isomorphic to $(i_{e})_{!}[r]$ where $i_{e}: \{e\}\to \cN/G$ is the inclusion of a regular nilpotent element $e$. In particular, $\TT^{-1}\Wh_{\frg/G}\in \Sh_{c}(\cN/G)$, hence $\Wh_{\frg/G}$ corresponds to a collection of perfect $A_{J}$-modules $M_{J,F}$, one for each $(J,F)\in\frD$.
	
	For any $(J,F)\in \frD$, put $e_J \subset \calN_{L_J}$ a regular nilpotent element. Then by Prop~\ref{prop:Fourier_transform}
	\begin{equation*}\label{eq:Hom_Wh_Spr}
	\Hom_{\calN_G/G}(\TT^{-1}\Wh_{\frg/G},\Ind_{L_J \subset P_J}^G (F)) \simeq i^*_e(\Ind_{L_J \subset P_J}^G (F))[r] = i^*_{e_J}(F)[r]=\begin{cases}
		k_{\sign},  \qquad (J,F) \in \frD^{\reg} \\
		0,    \qquad (J,F) \notin \frD^{\reg}
	\end{cases}
	\end{equation*}
	where $k_{\sign}$ is the $1$-dimensional $\End(\Ind_{L_J \subset P_J}^G (F))$-module where $W^{J}$ acts by the sign representation (this is because we identify the Weyl group actions on Springer fibers via Fourier transform, therefore the regular Springer fiber corresponds to the sign representation). 
	Therefore we have an isomorphism of right $A_{J}$-modules (the right $A_{J}$-action comes from post-composing with the right $A_{J}$-action on itself)
	\begin{equation*}
	\Hom_{A_{J}}(M_{J,F}, A_{J})\simeq \begin{cases}
		k_{\sign},  \qquad (J,F) \in \frD^{\reg} \\
		0,    \qquad (J,F) \notin \frD^{\reg}.
	\end{cases}
\end{equation*}

		Note also that there is an equivalence of categories between left and right $A_{J}$-modules:
    \begin{equation*}
\xymatrix{ \Hom_{A_{J}}(-, A_{J}) : A_{J}\textup{-perf}  \ar[r]^-{\sim} &    A_{J}^{\opp} \textup{-perf}}
\end{equation*}	
under which the $1$-dimensional trivial module $k_{\triv}$ corresponds to $k_{\sign}[r]$ by a Koszul resolution calculation. From this we conclude that
	\begin{equation*}
	M_{J,F}\simeq \begin{cases}
		k_{\triv}[-r],  \qquad (J,F) \in \frD^{\reg} \\
		0,    \qquad (J,F) \notin \frD^{\reg}.
	\end{cases}
	\end{equation*}
	Therefore, under the equivalence \eqref{constr gG}, $\Wh_{\frg/G}$ corresponds to
	\begin{equation*}
	\oplus_{(J,F) \in \frD^{\reg}} k_{\triv}[-r] \in  \bigoplus_{(J,F) \in \frD}  k[W^J]\#\Sym(\frz^*_J[-2]) \textup{-perf},
	\end{equation*}
	which is further identified under Koszul duality \eqref{KD}with  
	\begin{equation*}
	\oplus_{(J,F) \in \frD^{\reg}} \Sym(\frz_J[1])[-r] \in \bigoplus_{(J,F) \in \frD} k[W^J]\#\Sym(\frz_J[1]) \textup{-mod}_{\textup{fd}} \subset \bigoplus_{(J,F) \in \frD} k[W^J]\#\Sym(\frz_J[1]) \textup{-mod}.  
	\end{equation*}
	
\end{proof}

\subsection{Calculation of endormorphisms of the Whittaker sheaf}

Recall from Section~\ref{sss:des tr Wh} the descended trace $\Wh_{\cG/\cG}\in hh(\cH_{\cG})$ of the universal affine Whittaker sheaf $\Wh_{\cG}$. We can now calculate its endomorphism algebra explicitly.


For $\chi \in \textup{Irr}(Z(G)/Z(G)^{\c})$, denote $T_\chi=Z(L_\chi)^{\c}, W_\chi=W^{J_\chi}$. Let $T^\vee_\chi$ the torus over $k$ dual to $T_{\chi}$ (i.e., $\xch(T^{\vee}_{\chi})=\xcoch(T_{\chi})$), and $\frt^\vee_\chi=\Lie(T^\vee_\chi)=\frz^{*}_{J_\chi}$.

Recall the functor $\alpha_2: \Sh_{\cN}(G/G)\to \colim_{J \in\cD} Sh_\calN(L_J/L_J)$ corresponding to $J=I$. Define
\begin{equation*}
\Wh_{G,I}:=\a_{2}(\Wh_{G/G})\in \colim_{J \in\cD} Sh_\calN(L_J/L_J).
\end{equation*}

\begin{theorem}
	\label{thm:end_of_Wh} 
	There is a canonical isomorphism of dg algebras: 
	$$ \End_{\colim_{J \in\cD} Sh_\calN(L_J/L_J)}(\Wh_{G,I}) \simeq \bigoplus_{\chi \in \textup{Irr}(Z(G)/Z(G)^{\c})} \calO(T^\vee_\chi \times T^\vee_\chi \times \frt^\vee_\chi[-1])^{W_\chi} $$
\end{theorem} 

\begin{proof} 
	By Proposition~\ref{prop:wh_G_and_wh_g}, we have
	\begin{equation*}
	\Wh_{G,I}=\alpha_2(\Wh_{G/G}) \simeq \alpha_2 \circ \alpha_1 (\Wh_{\frg/G}).
\end{equation*}
	By Corollary~\ref{cor:induction_and_alpha}, the image of $\Wh_{G,I}$ in $\G(\frt,\scrQCoh_{\hatfrC})^{\hatW}$ is 
	\begin{equation}\label{IndIndWhg}
	\Ind_2 \circ \Ind_1(\LL_{\frg}(\Wh_{\frg/G})) \in \G(\frt,\scrQCoh_{\hatfrC})^{\hatW}.
	\end{equation}

	For any $\epsilon \in \frS$, put $\Lambda_\epsilon= \xcoch(T) \cap \epsilon$. 
	For $c=(\epsilon,B,F) \in \frC$, the summand corresponds to $c$ in
	$\Gamma(\frt,\scrQCoh_{\frC})^W$ is equivalent to the category
	$k[W^{\epsilon}]\#\Sym(\frz_\epsilon[1])\lmod$.
	Similarly the summand corresponds to $c$ in $\Gamma(\frt,\scrQCoh_{\tilfrC})^{\tilW} $ 
	(resp. in $\Gamma(\frt,\scrQCoh_{\hatfrC})^{\hatW} $) is equivalent to $k[\tilW^{\epsilon}]\#\Sym(\frz_\epsilon[1])\lmod$ 
	(resp. $k[\hatW^{\epsilon}]\#\Sym(\frz_\epsilon[1])\lmod$).
	We also have $\tilW^\epsilon= \Lambda_\epsilon \rtimes W^\epsilon$, and $\hatW^\epsilon= (\Lambda_\epsilon \times \Lambda_\epsilon)\rtimes W^\epsilon$. Therefore, $\Ind_2 \circ \Ind_1$ can be identified with the direct sum of the following induction functors over $(J,F)\in \frD$:
	\begin{equation*}
	\Ind_{W^{\e}}^{\hatW^{\e}}: k[W^{J}]\#\Sym(\frz_J[1])\lmod\to k[\hatW^{J}]\#\Sym(\frz_J[1])\lmod.
	\end{equation*}
	
	By Proposition~\ref{p:Wg image} and \eqref{IndIndWhg}, the image of $\Wh_{G,I}$ in $\G(\frt,\scrQCoh_{\hatfrC})^{\hatW}$ is
	$$\Ind_2 \circ \Ind_1\left( \bigoplus_{\chi \in \textup{Irr}(Z(G)/Z(G)^{\c})} \Sym(\frz_{J_\chi}[1])\right) =  \bigoplus_{\chi \in \textup{Irr}(Z(G)/Z(G)^{\c})} \Ind_{W^{J_{\chi}}}^{\hatW^{J_{\chi}}}\Sym(\frz_{J_\chi}[1]) \in \G(\frt,\scrQCoh_{\hatfrC})^{\hatW}.$$

	Therefore (all direct sums are over $\chi\in\Irr(Z(G)/Z(G)^{\c})$) 
	\begin{eqnarray*}
			\End(\Wh_{G,I})  & \simeq & 
			\bigoplus_{\chi} \End_{k[\hatW^{J_\chi}]\#\Sym(\frz_{J_\chi}[1])} 
			\left(\Ind_{W^{J_{\chi}}}^{\hatW^{J_{\chi}}}\Sym(\frz_{J_\chi}[1]) \right) \\
			& \simeq & \bigoplus_{\chi} \Hom_{k[W^{J_{\chi}}]\#\Sym(\frz_{J_\chi}[1])}\left(\Sym(\frz_{J_\chi}[1]), \Res^{\hatW^{J_{\chi}}}_{W^{J_{\chi}}}\Ind_{W^{J_{\chi}}}^{\hatW^{J_{\chi}}}\Sym(\frz_{J_\chi}[1])\right) \\
%
%
			& \simeq & \bigoplus_{\chi} \left(k[\Lambda_{J_\chi} \times \Lambda_{J_\chi}  ] \otimes \Sym(\frz_{J_\chi}[1])\right)^{W^{J_\chi}} \\
			& = & \bigoplus_{\chi} \calO(T^\vee_\chi \times T^\vee_\chi \times \frt^\vee_\chi[-1])^{W_\chi}.
	\end{eqnarray*}
	In the last identity we use $\L_{J_{\chi}}=\xcoch(T_{\chi})=\xch(\dT_{\chi})$,  $\frt^{\vee}_{\chi}\cong \frz_{J_\chi}^{*}$ and $W_{\chi}=W^{J_{\chi}}$. 
\end{proof}

\begin{cor}\label{c:end_of_Wh}
There is a canonical isomorphism of dg algebras: 
	$$ \End_{hh(\calH_\calG)}(\Wh_{\calG/\cG}) \simeq \bigoplus_{\chi \in \textup{Irr}(Z(G)/Z(G)^{\c})} \calO(T^\vee_\chi \times T^\vee_\chi \times \frt^\vee_\chi[-1])^{W_\chi}. $$
\end{cor}
\begin{proof}
By Theorem~\ref{thm: dtr of whit}, we have $\Wh_{\calG/\cG} \simeq a(\Wh_{G/G})$ where $a$ is the natural map $Sh_\cN(G/G) \incl \cH_{G,I}\to hh(\cH_{\cG})$. By Theorem~\ref{th:ff}, we have a fully-faithful embedding
\begin{equation}\label{colimCS emb hh}
\colim_{J \in\cD} Sh_\calN(L_J/L_J) \hookrightarrow hh(\calH_\calG).
\end{equation}
Under this embedding, we can identify $\alpha_2: \Sh_{\cN}(G/G)\to \colim_{J \in\cD} Sh_\calN(L_J/L_J)$ (corresponding to the term $J=I$) with $a: Sh_\cN(G/G)\to hh(\cH_{\cG})$. Hence $\Wh_{\cG/\cG}$ can be identified with the image of $\Wh_{G,I}$ under the embedding \eqref{colimCS emb hh}. Therefore
\begin{equation*}
\End_{hh(\cH_{\cG})}(\Wh_{\cG/\cG})\simeq \End_{\colim_{J \in\cD} Sh_\calN(L_J/L_J)}(\Wh_{G,I}).
\end{equation*}
The desired statement now follows from Theorem~\ref{thm:end_of_Wh}.
\end{proof}

%
%
%

\subsection{Additional calculations}\label{ss:additional app}

In this section, we use similar method to compute two additional endomorphism algebras of interest: (i)
the derived spherical Hecke algebra and (ii) endomorphisms of parabolic inductions (for simplicity, we will restrict to inductions from a torus). 
The results in this subsection are only used in Section~\ref{s:further apps}, in which we will explain their spectral counterparts.

For each of the categories $\G(\frt, \scrQCoh_{\frC})^{W}, \G(\frt, \scrQCoh_{\wt\frC})^{\tilW}$ and $\G(\frt, \scrQCoh_{\wh\frC})^{\hatW}$, there is a ``principal block'' (direct summand) corresponding to the affine space $\e=\frt$ and the skyscraper sheaf on $T$. They are described as follows:
\begin{eqnarray*}
k[W]\#\Sym((\dt)^{*}[1])\lmod\simeq k[W]\#\cO(\dt[-1])\lmod\subset \G(\frt, \scrQCoh_{\frC})^{W},\\
k[\tilW]\#\Sym((\dt)^{*}[1])\lmod\simeq k[W]\#\cO(\dT\times\dt[-1])\lmod\subset \G(\frt, \scrQCoh_{\wt\frC})^{\tilW},\\
k[\hatW]\#\Sym((\dt)^{*}[1])\lmod\simeq k[W]\#\cO(\dT\times \dT\times\dt[-1])\lmod\subset \G(\frt, \scrQCoh_{\wh\frC})^{\hatW}.
\end{eqnarray*}

\begin{lemma}\label{l:criterion princ}
A simple perverse sheaf $\cF\in \Sh_{\cN}(G/G)$ lies in the principal block if and only if the following two conditions hold:
\begin{enumerate}
\item The support of $\cF$ (which is closed by definition) contains $1\in G$;
\item Every simple constituent of the Fourier transform $\TT^{-1}(\b_{1}\cF)\in \Sh(\cN/G)$ is a summand of the Springer sheaf $\Spr\in \Sh(\cN/G)$ (normalized to be perverse, i.e., stalks along regular nilpotent elements lie in degree $r$). 
\end{enumerate}
Moreover, when $\cF\in \Sh_{\cN}(G/G)$ is in the principal block, $\LL_{\frg}(\b_{1}\cF)\in k[W]\#\Sym((\dt)^{*}[1])\lmod$ is Koszul dual to
\begin{equation*}
\Hom_{\cN/G}(\Spr, \TT^{-1}(\b_{1}\cF))\in \End(\Spr)^{\opp}\lmod\simeq k[W]\#\Sym(\dt[-2])\lmod.
\end{equation*}
\end{lemma} 
\begin{proof}
First, assume $\cF$ lies in the principal block. By Proposition~\ref{prop:restriction_G_to_g}, $\beta_1 \cF$ lies in the principal block and is nonzero because the restriction functor $\Res: k[\tilW]\#\Sym((\dt)^{*}[1])\lmod\to k[W]\#\Sym((\dt)^{*}[1])\lmod$ is faithful. Being in the principal block, all simple constituents of $\beta_1 \cF$ are summands of the Grothendieck-Springer sheaf (direct image of constant sheaf under $\frb/B\to \frg/G$). This implies that all simple constituents of $\TT^{-1}(\beta_1 \cF)$ are summands of $\Spr$, which verifies condition (2).
Since all simple summands of the Grothendieck-Springer sheaf have full support, and $\b_{1}\cF\ne0$, its support contains $0$. In particular, the support of $\cF$ contains $1$, which verifies condition (1).


Conversely, assume the conditions (1) and (2) hold. Using Remark \ref{r:supp block} and condition (1), we see that $\LL_{G}(\cF)$ lies in a summand in \eqref{CS G blocks} indexed by $(J,F)$ such that $1\in \exp(\e_{J})$. This implies that, up to changing $J$ by the action of $\Om$, we may arrange $J\subset I$. Under the restriction functor $\Res: \G(\frt, \scrQCoh_{\wt\frC})^{\tilW}\to \G(\frt, \scrQCoh_{\frC})^{W}$, the block indexed by $(J,F)$ with $J\subset I$ on the source maps into the block indexed by the same $(J,F)$ on the target, under the decomposition \eqref{eq:character_lie_algebra_facet_description}.
Therefore by Proposition~\ref{prop:restriction_G_to_g}, $\LL_{G}(\cF)$ lies in the principal block if and only if $\LL_{\frg}(\beta_1 \cF)$ lies in the principal block, which is guaranteed by condition (2).  
\end{proof}

Recall the natural monoidal embedding $i_{!}: \cH_G \to \cH_{\cG}$, and let $a: Sh_\calN(G/G) \simeq hh(\cH_G) \to hh(\cH_{\cG})$ denote the induced functor. 

\begin{lemma}\label{l:endo princ}
Suppose $\cF\in \Sh_{\cN}(G/G)$ lies in the principal block, with the corresponding $k[W]\#\cO(\dT\times\dt[-1])$-module $\LL_{G}(\cF)$. Then there is a canonical equivalence of dg algebras
\begin{equation*}
\End_{hh(\cH_{\cG})}(a(\cF))\simeq \left(\cO(\dT)\ot \End_{\cO(\dT\times\dt[-1])}(\LL_{G}(\cF))\right)^{W}.
\end{equation*}
\end{lemma}
\begin{proof}
Note that $a(\cF)$ can be identified with $\a_{2}(\cF)\in \colim_{J\in\cD}\Sh_{\cN}(L_{J}/L_{J})$, the latter identified with the full subcategory $hh(\cH_{\cG})_{0}$ by Theorem~\ref{th:ff}. By Corollary~\ref{cor:induction_and_alpha}, we have
\begin{equation}\label{SigcG ind}
\LL_{\cG}(\a_{2}(\cF))\simeq \Ind_{\tilW}^{\hatW}(\LL_{G}(\cF))\cong \cO(\dT)\otimes \LL_{G}(\cF)\in k[W]\#\cO(\dT\times \dT\times \dt[-1])\lmod.
\end{equation}
From this we see
\begin{eqnarray*}
\End_{hh(\cH_{\cG})}(a(\cF))&\simeq &\End_{k[W]\#\cO(\dT\times\dT\times\dt[-1])}(\LL_{\cG}(\a_{2}(\cF)))\\
&\simeq& \End_{k[W]\#\cO(\dT\times\dT\times\dt[-1])}(\cO(\dT)\ot \LL_{G}(\cF))\\
&\simeq&\left(\cO(\dT)\ot \End_{\cO(\dT\times\dt[-1])}(\LL_{G}(\cF))\right)^{W}.
\end{eqnarray*}
\end{proof}


Write 
 $\tr_G: \cH_{G} \to hh(\cH_{G})$,  $\tr_\cG: \cH_{\cG} \to hh(\cH_{\cG})$ for the trace maps, and recall the natural isomorphism $\tr_\cG |_{\cH_{G}}  \simeq a \circ \tr_G$. 
 

\begin{theorem}\label{thm:more endo}
\begin{enumerate}
\item Let $k_{G/G}$ denote the constant sheaf on $G/G$. There is a canonical equivalence of dg algebras
\begin{equation*}
\End_{hh(\cH_{\cG})} (a(k_{G/G}))  \simeq   \cO(T^\vee \times { (\dt)^{*}[1] \times (\dt)^{*}[2]} )^W.
\end{equation*}

\item Let $e_{\cG}$ be the monoidal unit of the universal affine Hecke category $\cH_{\cG}$. There is a canonical equivalence of dg algebras
\begin{equation*}
\End_{hh(\cH_{\cG})} (\tr_\cG(e_\cG))  \simeq   k[W] \#   \cO(T^\vee \times T^\vee \times \frt^\vee[-1]).
\end{equation*}
\end{enumerate}
\end{theorem}

\begin{proof}

		
	In both calculations, we are computing $\End(a(\cF))$ for some $\cF\in \Sh_{\cN}(G/G)$. We shall show in both cases $\cF$ lies in the principal block of $Sh_\calN(G/G)$, and identify the corresponding $k[W]\#\cO(\dT\times\dt[-1])$-module $\LL_{G}(\cF)$. Then we conclude the calculation by invoking Lemma \ref{l:endo princ}.

	(1) Note that $\TT^{-1}(\b_{1}k_{G/G})=\TT^{-1}(k_{\frg/G})\simeq k_{0/G}[-\dim G]$, which is a direct summand of $\Spr$. By the criterion in Lemma \ref{l:criterion princ}, $\LL_{G}(k_{G/G})$ lies in the principal direct summand. Under Koszul duality,  $\LL_{\frg}(\b_{1}k_{G/G})=\LL_{\frg}(k_{\frg/G})\in k[W]\#\Sym(\frt[1])\lmod$ corresponds to
	\begin{equation*}
		\Hom_{\cN/G}(\Spr, k_{0/G}[-\dim G])\in  k[W]\#\Sym(\frt^{\vee}[-2])\lmod.
	\end{equation*}
Using the Cartesian diagram
\begin{equation*}
\xymatrix{ \{0\}/B\ar[d]^{\nu_{0}} \ar@{^{(}->}[r]^{i'_{0}} & \fru/B\ar[d]^{\nu}\\
\{0\}/G \ar@{^{(}->}[r]^{i_{0}} & \cN/G
}
\end{equation*}
we have by adjunction and proper base change
\begin{eqnarray*}
\Hom_{\cN/G}(\Spr, i_{0*}k_{0/G}[-\dim G])&=&\Hom(\nu_{!}k_{\fru/B}[-r], k_{0/G}[-\dim G])\\
&\simeq&\Hom(k_{\fru/B}[-r], \nu^{!}i_{0*}k_{0/G}[-\dim G])\\
&\simeq&\Hom(k_{\fru/B}[-r], i'_{0*}\nu_{0}^{!}k_{0/G}[-\dim G])\\
&=&\Hom(k_{\fru/B}, i'_{0*}k_{0/B})\\
&\simeq& \upH^{*}_{B}(\pt)\simeq \Sym(\frt^{\vee}[-2]).
\end{eqnarray*}
This shows that the Koszul dual of $\LL_{\frg}(\b_{1}k_{G/G})$ is $\Sym(\frt^{\vee}[-2])$ as a natural $k[W]\#\Sym(\frt^{\vee}[-2])$-module, hence $\LL_{\frg}(\b_{1}k_{G/G})\simeq k_{\triv}$ as a $k[W]\#\Sym((\dt)^{*}[1])$-module.

In other words, $\LL_{G}(k_{G/G})\in k[\tilW]\#\Sym((\dt)^{*}[1])$ is one-dimensional over $k$ with the trivial action of $k[W]\#\Sym((\dt)^{*}[1])$. Using the fact that $k_{G/G}$ appears as a direct summand of $\Ind_{T\subset B}^{G}(k_{T/T})$, and the compatibility of $\LL_{G}$ with parabolic induction in Proposition \ref{prop:restriction_G_to_L}, we see that the action of the lattice part $\xcoch(T)\subset \tilW$ on $\LL_{G}(k_{G/G})$ is also trivial. We conclude that $\LL_{G}(k_{G/G})\simeq k_{\triv}$ as an object in $k[\tilW]\#\Sym((\dt)^{*}[1])\lmod=k[W]\#\cO(\dT\times\dt[-1])\lmod$. Therefore we have a $W$-equivariant equivalence of dg algebras
\begin{equation*}
\End_{\cO(\dT\times\dt[-1])}(\LL_{G}(k_{G/G}))\simeq \cO((\dt)^{*}[1] \times (\dt)^{*}[2])
\end{equation*}


By Lemma \ref{l:endo princ}, we get
\begin{equation*}
	\End_{hh(\cH_{\cG})} (a(k_{G/G})) \simeq  (\cO(\dT) \otimes \End_{\cO(\dT\times\dt[-1])}(\LL_{G}(k_{G/G})))^W\simeq \cO(T^\vee {\times (\dt)^{*}[1] \times (\dt)^{*}[2]} )^W. 
	\end{equation*}
	

(2) Let $e_{G}$ be the monoidal unit of $\cH_{G}$, then $\tr_{\cG}(e_\cG) \simeq a(\tr_G(e_G))$. We claim that $\tr_G(e_G)$ is in fact the ``universal Grothendieck-Springer sheaf". 
{ Indeed, consider the following commutative diagram with a Cartesian square on the left
\begin{equation*}
\xymatrix{ H \ar[d]^{p_{H}} & \ar[l]_-{q_{H}} U\bs B/U & \ar[l]_-{\d}\f{B}{U}\ar[d] \ar[r]^-{\pi}& \f{G}{G}\ar@{=}[d]\\
\f{H}{H} &  & \ar[ll]_-{q'_{H}}\f{B}{B}\ar[r]^-{\pi'} & \f{G}{G}}
\end{equation*}
By Theorem~\ref{thm:CS G}, $\tr_{G}=\g$ is the horocycle functor.  Therefore
\begin{equation*}
\tr_{G}(e_{G})=\pi_{!}\d^{*}(e_{G})=\pi_{!}\d^{*}q^{*}_{H}(\exp_{!}k_{\frh})[r].
\end{equation*}
By proper base change, we can identify it with
\begin{equation}\label{tr e ind}
\tr_{G}(e_{G})\simeq \pi'_{!}q'^{*}_{H}(p_{H!}\exp_{!}k_{\frh})[r]\simeq \Ind^{G}_{T\subset B}(p_{H!}\exp_{!}k_{\frt}[r]).
\end{equation}
Under the equivalence $\Sh_{\cN}(T/\Ad(T))=\Sh_{0}(T/\Ad(T))\simeq \cO(\dT\times\dt[-1])\lmod$, $p_{H!}\exp_{!}k_{\frt}[r]$ corresponds to the free rank one module $\cO(\dT\times\dt[-1])$. }
By Proposition~\ref{prop:restriction_G_to_L}, $\LL_{G}(\tr_G(e_{G}))\cong k[W]\# \cO(\dT \times \dt[-1])$ is also the free rank one $k[W]\# \cO(\dT \times \dt[-1])$-module. By \eqref{SigcG ind}, $\LL_{\cG}(a(\tr_G(e_{G})))$ is again the free rank one $k[W]\# \cO(\dT \times \dt[-1])$-module. Therefore  $\End(\tr_\cG(e_{\cG}))\simeq \End(a(\tr_G(e_{G})))\simeq \End(\LL_{\cG}(a(\tr_G(e_{G}))))$ is the dg algebra $k[W]\# \cO(T^\vee \times T^\vee \times \frt^\vee[-1]) $.  
\end{proof}

\begin{remark}
The functor $a$ is analogous to the compact induction functor for $p$-adic representations $\textup{c-ind}: G(\cO_{K})\textup{-rep} \to G(K)\textup{-rep}$. { Here $K$ is a local non-archimedean field with valuation ring $\cO_{K}$, and $G$ is a split reductive group over $K$.} In particular, $\End_{hh(\cH_{\cG})} (a(k_{G/G})) $ is analogous to the spherical Hecke algebra $\End_{G(K)}(\textup{c-ind}(k))$.
\end{remark}


\section{Functions on commuting stacks}\label{s:comm}

\subsection{Functions on derived commuting stacks} 
 To deduce spectral applications of our prior automorphic results, we will invoke here  Ansatz \eqref{ans:intro univ aff equiv}. 
 Recall it states in particular that 
 there is a monoidal equivalence of universal affine Hecke categories
\begin{equation*}
\xymatrix{
\Phi:\Ind\Coh(\St_{\dG}) \ar[r]^-\sim &   \cH_{\cG}
}
\end{equation*}
such that 
 $\Phi$ identifies the structure sheaf $\cO_{\St_{\dG}}$ with the universal affine Whittaker object $\Wh_{\cG}$ as coalgebra objects.
 Taking cocenters of both sides of $\Phi$ we get an equivalence
\begin{equation*}
hh(\Phi): hh(\Ind\Coh(\St_{\dG}))\simeq hh(\cH_{\cG}).
\end{equation*}
Combined with \eqref{spec cocenter}, we get an equivalence
\begin{equation}\label{Z2hh}
\Ind\Coh_{\cN}(Z^{2}_{\dG})\simeq  hh(\cH_{\cG}).
\end{equation}

Now we may combine the spectral and automorphic calculations of
Corollary~\ref{c:OZ OSt} and Theorem~\ref{thm: dtr of whit} to obtain:

\begin{cor}\label{c:OZ Wh} Assume Ansatz \eqref{ans:intro univ aff equiv}. Under the equivalence \eqref{Z2hh}, the structure sheaf $\cO_{Z^{2}_{\dG}}\in \Ind\Coh_{\cN}(Z^{2}_{\dG})$ corresponds to  $\Wh_{\cG/\cG}\in hh(\cH_{\cG})$.
\end{cor}

And as an immediate consequence of 
Corollary~\ref{c:end_of_Wh}, we further obtain:

\begin{cor}\label{c:derived fun comm}
	Assume Ansatz~\ref{ans:intro univ aff equiv}. Then there is a canonical isomorphism of dg algebras:
	$$ \calO(Z^2_{{G^\vee}})\simeq  \bigoplus_{\chi \in \textup{Irr}(Z(G)/Z(G)^{\c})} \calO(T^\vee_\chi \times T^\vee_\chi \times \frt^\vee_\chi[-1])^{W_\chi}.$$
	In particular, $H^0(\calO(Z^2_{{G^\vee}}))\simeq \bigoplus_{\chi \in \textup{Irr}(Z(G)/Z(G)^{\c})} \calO(T^\vee_\chi \times T^\vee_\chi )^{W_\chi}$ is reduced.
\end{cor}

The goal of most of the remainder of this section is to deduce from 
Corollary~\ref{c:derived fun comm} the assertions stated in 
Corollary~\ref {intro main cor}, Theorem~\ref{thm:intro classical functions}, and Theorem~\ref{thm:intro Lie mix}. 
In a final section, we deduce further spectral consequences of our other automorphic calculations.


\subsection{Almost commuting pairs of semisimple elements}\label{ss:alm comm} 
Let $\dG$ be a connected reductive group over $\CC$. We first recall some results of Borel, Friedman and Morgan \cite[\S4]{borelAlmostCommutingElements} concerning almost commuting pairs with compact simple and simply-connected Lie groups. We will state an extension of their results to almost commuting pairs of {\em semisimple elements} in complex reductive groups. It is easy to adapt their proof to this situation.

Let $\dGsc$ be the simply-connected cover of $\dGder$.  Let $\wt\dG=\dGsc\times (ZG^{\vee})^{\c}$. Then the natural map $\pi: \wt\dG\to \dG$ is a finite central isogeny whose kernel {\em contains} $\ker(\dGsc\to \dGder)=\pi_{1}(\dGder)$.  For any $c\in \pi_{1}(\dGder)$, let $Z^{2}_{\dG}(c)\subset Z^{2}_{\dG}$ be the open and closed substack of pairs $(x,y)$ in $\dG$ such that, for some (any) liftings $\wt x, \wt y$ of $x$ and  $y$ to $\wt\dG$, $[\wt x,\wt y]=c$, all up to $\dG$-conjugation. An almost commuting pair in $\wt\dG$ refers to a pair $(\wt x, \wt y)\in \wt\dG^{2}$ as above. We have a decomposition
\begin{equation*}
Z^{2}_{\dG}=\coprod_{c\in \pi_{1}(\dGder)}Z^{2}_{\dG}(c).
\end{equation*}

Let $(x,y)\in Z^{2}_{\dG}(c)$ with $x,y$ {\em semisimple}. Consider the simulteneous centralizer $C_{\dG}(x,y)$, which is a reductive (possibly disconnected) subgroup of $\dG$. Let $S\subset C_{\dG}(x,y)$ be a maximal torus. It is shown in \cite[Proposition 4.2.1]{borelAlmostCommutingElements} that the $\dG$-conjugacy class of $S$ is independent of the choice of $(x,y)$. Let us fix such a torus $S$ and denote it by $S_{c}$. 

Let $\dL_{c}=C_{\dG}(S_{c})$, a Levi subgroup of $\dG$. It is known that $S_{c}=Z(\dL_{c})^{\c}$. Let $\dT_{c}=\dL_{c}/\dLder_{c}$ be the quotient torus. We have $\dT_{c}=S_{c}/(S_{c}\cap \dLder_{c})$ ($\dT_{c}$ is denoted $\ov S_{c}$ in \cite{borelAlmostCommutingElements}). Let $W_{c}=W(S_{c}, \dG)=N_{\dG}(S_{c})/\dL_{c}$ be the relative Weyl group of $\dG$ with respect to $S_{c}$.  Then $W_{c}$ also acts on $\dT_{c}$. Fix a pair $(x_{c}, y_{c})\in (\dLder_{c})^{2}$ such that $(x_{c},y_{c})\in Z^{2}_{\dG}(c)$. We have a map
\begin{equation*}
\wt\io_{c}: Z^{2}_{S_{c}}\to Z^{2}_{\dG}(c)
\end{equation*}
sending $(t_{1},t_{2})\in Z^{2}_{\dT_{c}}$ to $(t_{1}x_{c}, t_{2}y_{c})$. 

Now we state the generalized form of results in \cite{borelAlmostCommutingElements} replacing elements in compact groups to semisimple elements in reductive groups; the proofs in {\em loc.cit.} works without change.

\begin{prop}[{\cite[Proposition 4.2.1, Corollary 4.2.2]{borelAlmostCommutingElements}}]\label{p:BFM}
\begin{enumerate}
\item Any other choice of $(x'_{c}, y'_{c})\in (\dL_{c})^{2}$ such that $(x'_{c},y'_{c})\in Z^{2}_{\dG}(c)$ is $\dL_{c}$-conjugate to $(x_{c},y_{c})$. Moreover, $C_{\dLder_{c}}(x_{c},y_{c})$ is finite.
\item The map $\wt\io_{c}$ factors through $Z^{2}_{\dT_{c}}$ and is $W_{c}$-invariant for the diagonal action, inducing a map of derived stacks
\begin{equation*}
\io_{c}: Z^{2}_{\dT_{c}}/W_{c}\to Z^{2}_{\dG}(c).
\end{equation*}
\item The map $\io_{c}$ restricts to a bijection on the set of semisimple closed points:
\begin{equation*}
(\dT_{c}(\CC)\times \dT_{c}(\CC))/W_{c}\bij |Z^{2}_{\dG}(c)(\CC)^{ss}|.
\end{equation*}
Here $|Z^{2}_{\dG}(c)(\CC)^{ss}|$ denotes the set of $\dG$-conjugacy classes of semisimple pairs $(x,y)\in Z^{2}_{\dG}(c)(\CC)$.
\end{enumerate}
\end{prop}


\subsection{Chevalley restriction theorem for the commuting stack}\label{ss:Chevalley res}
We finish the proof of Theorem \ref{thm:intro classical functions}, which is restated as follows.
\begin{theorem}\label{thm:function c part} Assume Ansatz~\ref{ans:intro univ aff equiv}. Let $c\in \pi_{1}(\dGder)$.
\begin{enumerate}
\item $Z^{2}_{\dG}(c)\ne\vn$. \footnote{Although implicit in \cite{borelAlmostCommutingElements} but we did not find an explicit statement of the non-emptiness of $Z^{2}_{\dG}(c)$. }
\item The map $\io_{c}$ defined in Proposition \ref{p:BFM}(2) induces an isomorphism on classical functions:
\begin{equation*}
\io_{c}^{*}: H^{0}\cO(Z^{2}_{\dG}(c))\isom H^{0}\cO(Z^{2}_{\dT_{c}})^{W_{c}}=\cO(\dT_{c}\times \dT_{c})^{W_{c}}.
\end{equation*}
\end{enumerate}
\end{theorem}
\begin{proof} 
(1) Let $X_{c}=\Spec H^{0}\cO(Z^{2}_{\dG}(c))$, which is understood to be empty if $Z^{2}_{\dG}(c)$ is empty. First, $X_{c}$ is connected if non-empty. Indeed, assuming $Z^{2}_{\dG}(c)\ne\vn$, and suppose for the contrary that $X_{c}$ decomposes into a disjoint union of open and closed non-empty subschemes $X'_{c}\coprod X''_{c}$. Let $\frC^{2}_{\dG}(c)\subset \frC^{2}_{\dG}$ be the preimage of $Z^{2}_{\dG}(c)$ in the commuting scheme $\frC^{2}_{\dG}$. Let $Z^{2}_{\dG}(c)=Z'_{c}\coprod Z''_{c}$ and $\frC^{2}_{\dG}(c)=\frC'_{c}\coprod \frC''_{c}$ be the corresponding decompositions by taking preimages of $X'_{c}$ and $X''_{c}$. Now $\frC'_{c}$ contains a closed $\dG$-orbit, which then consists of semisimple pairs $(x,y)\in \frC^{2}_{\dG}(c)$. This implies $Z'_{c}$ contains a semisimple pair $(x,y)$. By Proposition \ref{p:BFM}(2)(3), $Z'_{c}$ meets the image of $\io_{c}: (\dT_{c}\times \dT_{c})/W\to Z^{2}_{\dG}(c)$. Similarly, $Z''_{c}$ also contains a semisimple pair, hence it also meets the image of $\io_{c}$. But $\dT_{c}\times \dT_{c}$ is irreducible, we get a contraction. This proves that $X_{c}$ is connected if non-empty.

Corollary \ref{c:derived fun comm} shows that $\Spec H^{0}\cO(Z^{2}_{\dG})=\coprod_{c\in \pi_{1}(\dGder)}X_{c}$ has exactly $\#\Irr(Z(G)/Z(G)^{\c})$ connected components. Since $\#\Irr(Z(G)/Z(G)^{\c})=\#\pi_{1}(\dGder)$, and each $X_{c}$ is connected if non-empty, all of them must be non-empty and connected.

(2) We claim that $X_{c}$ is irreducible, reduced and normal. Above we have shown that $X_{c}$ is a connected component of $\Spec H^{0}\cO(Z^{2}_{\dG})$, hence isomorphic to $\Spec\cO(\dT_{\chi}\times \dT_{\chi})^{W_{\chi}}$ for a unique $\chi\in\Irr(Z(G)/Z(G)^{\c})$. In particular, $X_{c}$ is irreducible, reduced and normal.


The map $\io_{c}$ induces a map on coarse moduli spaces
\begin{equation*}
\io'_{c}: (\dT_{c}\times \dT_{c})\sslash W_{c}\to X_{c}.
\end{equation*}
We next claim that $\io'_{c}$ induces a bijection on closed points. Indeed, by a result of Richardson \cite{richardsonConjugacyClassesNtuples1988}, the closed points in $X_{c}$ are in bijection with $\dG$-orbits of $(x,y)\in Z^{2}_{\dG}(c)$ such that the Zariski closure $A(x,y)\subset \dG$ of the subgroup generated by $x$ and $y$ is reductive. Since $x$ commutes with $y$, $A(x,y)$ is a commutative; it is reductive if and only if $x$ and $y$ are both semisimple. Therefore by Proposition \ref{p:BFM}(3), $\io'_{c}$ induces a bijection on closed points.

Since $(\dT_{c}\times \dT_{c})\sslash W $ and $X_{c}$ are reduced and irreducible of finite type over $\CC$, and $\io'_{c}$ is a bijection on closed points, to show $\io'_{c}$ is an isomorphism if suffices to show that $\io'_{c}$ is finite.

We have a map $\wt s=(\wt s_{1},\wt s_{2}): Z^{2}_{\dG}(c)\to (\dG\sslash\dG)^{2}=(\dT\sslash W)^{2}$ by recording the $\dG$-invariants of $x$ and $y$ in a commuting pair $(x,y)$. This induces a map $s=(s_{1},s_{2}): X_{c}\to (\dT\sslash W)^{2}$. We have maps of affine schemes
\begin{equation*}
S_{c}\times S_{c}\xr{p}
(\dT_{c}\times \dT_{c})\sslash W_{c} \xr{\io'_{c}} X_{c}\xr{s}(\dT\sslash W)^{2}
\end{equation*}
where the composition is the square of the natural projection $S_{c}\to \dT\sslash W$, hence finite. Since $p$ is surjective,  $s\c \io'_{c}: (\dT_{c}\times \dT_{c})\sslash W_{c}\to (\dT\sslash W)^{2}$ is also finite, therefore $\io'_{c}$ is finite as well. This finishes the proof.
\end{proof}

\begin{remark} Let $\dG_{\QQ}$ be the split form of $\dG$ over $\QQ$. Then $Z^{2}_{\dG}$ has a $\QQ$-form $Z^{2}_{\dG,\QQ}$. It is easy to deduce from Theorem \ref{thm:intro classical functions} a description of global functions on $Z^{2}_{\dG,\QQ}$. Indeed, we have a decomposition of $Z^{2}_{\dG,\QQ}$ into open and closed  substacks $Z^{2}_{\dG,\QQ}([c])$ indexed by Galois orbits $[c]\in \pi_{1}(\dGder)/\Gal(\ov\QQ/\QQ)$. If $\QQ_{[c]}$ is the fixed field of the stabilizer of any $c\in [c]$ under $\Gal(\ov\QQ/\QQ)$ (so $\QQ_{[c]}$ is the smallest cyclotomic extension of $\QQ$ over which $Z^{2}_{\dG}(c)$ is defined), we have an isomorphism $Z^{2}_{\dG,\QQ}([c])\cong Z^{2}_{\dG,\QQ_{[c]}}(c)$, the latter being the descent of $Z^{2}_{\dG}(c)$ to $\QQ_{[c]}$. For any $c\in [c]$, $\dT_{c}$ has a $\QQ_{[c]}$-form $\dT_{c,\QQ_{[c]}}$. We conclude
\begin{equation*}
H^{0}\cO(Z^{2}_{\dG,\QQ})\isom\bigoplus_{[c]\in \pi_{1}(\dGder)/\Gal(\ov\QQ/\QQ)} \cO(\dT_{c,\QQ_{[c]}}\times \dT_{c,\QQ_{[c]}})^{W_{c}}.
\end{equation*}
\end{remark}

\subsubsection{Identifying $\chi$ and $c$.}
There is a canonical isomorphism
\begin{equation*}
\delta: \Irr(Z(G)/Z(G)^{\c}) \simeq\pi_1(G^{\vee,\der})
\end{equation*}
because both sides are canonically dual to the torsion part of $\xch(T)/\mbox{root lattice of $G$}$.

On the other hand, by Corollary~\ref{c:derived fun comm} and Theorem~\ref{thm:function c part}, we have two expressions of the ring $H^0(\calO(Z^2_{{G^\vee}}))$:
\begin{equation}\label{two cal O}
 \bigoplus_{\chi \in \textup{Irr}(Z(G)/Z(G)^{\c})} \calO(T^\vee_\chi \times T^\vee_\chi )^{W_\chi} \simeq  H^0(\calO(Z^2_{{G^\vee}})) \simeq  \bigoplus_{c \in \pi_1(G^{\vee,\der}) }  \calO(T^\vee_c \times T^\vee_c )^{W_c}.
\end{equation}

Fratila \cite{fratilaStackSemistableGbundles2016} and Bonnaf\'e \cite{bonnafeElementsUnipotentsReguliers2004} observe a mysterious coincidence that the pairs $(\dT_{c},W_{c})$ that appear on the right side of \eqref{two cal O} are exactly those coming from Levi subgroups of $G$ that support cuspidal local systems on the regular nilpotent orbit, i.e., those pairs $(\dT_{\chi}, W_{\chi})$ that appear on the left side of \eqref{two cal O}. They verified this fact by a case-by-case explicit computation. Below we give a uniform conceptual proof of this combinatorial observation.




\begin{prop}
	\label{prop:chi=c} 
	Assume Ansatz~\ref{ans:intro univ aff equiv}.  Let $\chi\in \Irr(Z(G)/Z(G)^{\c})$ and $c=\d(\chi)\in \pi_1(G^{\vee,\der})$. Then the following hold:
\begin{enumerate}
	\item Put $L_\chi := L_{J_\chi}$ as in Section~\ref{sec:generalized_springer}, then $\dL_{c}$ (defined in Section~\ref{ss:alm comm}) is the dual Levi of $L_\chi$ (up to $\dG$-conjugacy).
	\item There is an isomorphism of pairs $(\dT_{\chi}, W_\chi)\simeq (\dT_{c},W_{c})$ compatible with their respective actions. Moreover, such an isomorphism is canonical up to conjugation by $W_{\chi}$.
\end{enumerate}
\end{prop}
\begin{proof} 
(1) Let $\dL_{\chi}\subset \dG$ be the dual Levi to $L_{\chi}$. Note that the canonical map $\pi_{1}(\dLder_{\chi})\to \pi_{1}(\dGder)$ is always injective, and $c=\d(\chi)$ lies in its image. Therefore it makes sense to form $Z^{2}_{\dL_{\chi}}(c)\subset Z^{2}_{\dL_{\chi}}$, which is non-empty by Theorem~\ref{thm:function c part}(1). Viewing $c$ as an element in $\pi_{1}(\dGder)$, we can form the Levi $\dL_{c}$ of $\dG$ by taking a maximal torus $S_{c}\subset C_{\dG}(x,y)$ and take its centralizer. Now $Z(\dL_{\chi})^{\c}$ is a torus in $C_{\dG}(x,y)$, hence we may take $S_{c}$ to contain $Z(\dL_{\chi})^{\c}$. This implies $\dL_{c}=Z_{\dG}(S_{c})\subset Z_{\dG}(Z(\dL_{\chi})^{\c})=\dL_{\chi}$. Since $\dT_{c}$ and $\dT_{\chi}$ are the abelianizations of $\dL_{c}$  and $\dL_{\chi}$ respectively, we have a surjection $\dT_{c}\surj \dT_{\chi}$. In particular, we have the inequality
\begin{equation}\label{dimTc ineq}
\dim \dT_{c}\ge\dim\dT_{\chi}, \mbox{ for all $\chi\in \Irr(Z(G)/Z(G)^{\c})$ and $c=\d(\chi)$}.
\end{equation}
We remark that, since $\dL_{c}$ is a Levi subgroup of $\dL_{\chi}$ up to conjugation, $\dL_{c}$ is conjugate to $\dL_{\chi}$ if and only if $\dim \dT_{c}=\dim \dT_{\chi}$. 

On the other hand, by taking the Krull dimensions of all direct summands on both sides of \eqref{two cal O}, we get
\begin{equation*}
\sum_{\chi \in \textup{Irr}(Z(G)/Z(G)^{\c})}2\dim \dT_{\chi}=\sum_{c \in \pi_{1}(\dGder)}2\dim \dT_{c}.
\end{equation*}
Combined with the inequality \eqref{dimTc ineq}, we are forced to have an equality $\dim \dT_{\chi}=\dim \dT_{\d(\chi)}$ for every $\chi$. As we remarked earlier, this implies that $\dL_{\d(\chi)}$ is conjugate to $\dL_{\chi}$ in $\dG$.


(2) Let $c=\d(\chi)$. Note that $W_\chi= N_{\dG}(\dL_\chi)/\dL_\chi$, $\dT_{\chi}$ is the abelianization of $\dL_{\chi}$, and the same is true if $\chi$ is replaced with $c$. Any $g\in \dG$ that conjugates $\dL_{\chi}$ to $\dL_{c}$ induces an isomorphism of pairs $\io_{g}: (\dT_{\chi}, W_{\chi})\simeq (\dT_{c}, W_{c})$ compatible with the actions. Different choices of  $g$ differ by right multiplication by $N_{\dG}(\dL_\chi)$,  therefore the resulting $\io_{g}$ differ by $W_{\chi}$-conjuation. 
\end{proof}

Combining Corollary~\ref{c:derived fun comm} and Proposition~\ref{prop:chi=c}, we get the following description of the dg ring of derived functions on $Z^2_{{G^\vee}}$ promised in Theorem~\ref{intro main cor}.

\begin{cor}\label{c:derived fun comm c}
	Assume Ansatz~\ref{ans:intro univ aff equiv}. Then there is a canonical isomorphism of dg algebras:
	$$ \calO(Z^2_{{G^\vee}})\simeq  \bigoplus_{c \in  \pi_1(G^{\vee,\der})} \calO(T^\vee_c \times T^\vee_c \times \frt^\vee_c[-1])^{W_c}.$$
\end{cor}

\subsection{From groups to Lie algebras}

We will deduce the Lie algebra analogue of the Chevalley restriction theorem from the group case.

Let $\dG$ be a connected reductive group with Lie algebra $\dg$. Let $\frC^{2}_{\dg}$ be the commuting scheme for $\dg$, i.e., it is the (classical) fiber over $0$ of the Lie bracket map $[\cdot,\cdot]: \frg\times\frg\to \frg$. Then $\dG$ acts on $\frC^{2}_{\dg}$ by diagonal adjoint action. We are interested in calculating the $\CC$-algebra $\cO(\frC^{2}_{\dg})^{\dG}$. 

Let $\dT\subset \dG$ be a maximal torus and $\dt=\Lie\dT$.  The embedding $\dt\times\dt\incl \frC^{2}_{\dg}$ induces a map of stacks
\begin{equation*}
\io_{\frg^{\vee}}: (\dt\times\dt)/W\to \frC^{2}_{\dg}/\dG
\end{equation*}

\begin{theorem}\label{thm:Lie alg} Assume Ansatz~\ref{ans:intro univ aff equiv}. The map $\io_{\dg}$ induces an isomorphism on global (i.e.~invariant) functions
\begin{equation}\label{io dg}
\io_{\dg}^{*}: \cO(\frC^{2}_{\dg})^{\dG}\isom \cO(\dt\times\dt)^{W}.
\end{equation}
\end{theorem}
\begin{proof}
For the statement of the theorem we may replace $\dG$ by another group isogenous to it. In particular, we may assume the derived group $\dGder$ is simply-connected.

We have the natural map 
\begin{equation*}
\frs: \frC^{2}_{\dg}\to (\dt\sslash W)^{2}
\end{equation*}
by taking the invariants of both elements in the commuting pair. Let $\wh\frC^{2}_{\dg}$ be the formal completion of $\frC^{2}_{\dg}$ along the fiber $\frs^{-1}(0,0)$ (which classifies commuting pairs of nilpotent elements ).

Analogously, let $\frC^{2}_{\dG}\subset \dG\times\dG$ be the commuting scheme for $\dG$. We have the natural map
\begin{equation*}
s: \frC^{2}_{\dG}\to (\dT\sslash W)^{2}
\end{equation*}
by taking the invariants of both elements in the commuting pair. Let $\wh\frC^{2}_{\dG}$ be the formal completion of $\frC^{2}_{\dG}$ along the fiber $s^{-1}(1,1)$ (which classifies commuting pairs of unipotent elements ).

By Lemma \ref{l:formal completion} below, the exponential map gives a $\dG$-equivariant isomorphism
\begin{equation*}
\exp: \wh\frC^{2}_{\dg}\isom \wh\frC^{2}_{\dG}.
\end{equation*}
Moreover, letting $\wh{\dt\times\dt}$ and $\wh{\dT\times \dT}$ be the formal completions at $(0,0)$ and $(1,1)$ respectively, we have a commutative diagram of formal schemes
\begin{equation*}
\xymatrix{\wh{\dt\times\dt}\ar[r]^-{\io_{\dg}}\ar[d]^{\exp}  & \wh\frC^{2}_{\dg}\ar[d]^{\exp}\\
\wh{\dT\times\dT}\ar[r]^-{\io_{\dG}}  & \wh\frC^{2}_{\dT}}
\end{equation*}
where $\io_{\dG}$ is the obvious analogue of $\io_{\dg}$. The vertical maps are isomorphism. This gives a commutative diagram of algebras
\begin{equation}\label{compl io}
\xymatrix{\cO(\wh\frC^{2}_{\dg})^{\dG}\ar[r]^-{\wh\io_{\dg}^{*}} & \cO(\wh{\dt\times\dt})^{W}\\
\cO(\wh\frC^{2}_{\dG})^{\dG}\ar[r]^-{\wh\io_{\dG}^{*}} \ar[u]^{\cong}& \cO(\wh{\dT\times\dT})^{W}\ar[u]^{\cong}}
\end{equation}
Now the top arrow is obtained from the map of $\frS=\cO(\dt)^{W}\ot \cO(\dt)^{W}$-algebras \eqref{io dg} by completing along the maximal ideal of $\cO(\dt)^{W}\ot \cO(\dt)^{W}$ corresponding to $(0,0)$. Similarly, the bottom row is obtained from the map of $S=\cO(\dT)^{W}\ot \cO(\dT)^{W}$-algebras
\begin{equation*}
\io_{\dG}^{*}: \cO(\frC^{2}_{\dG})^{\dG}\to \cO(\dT\times\dT)^{W}
\end{equation*}
by completing along the maximal ideal of $\cO(\dT)^{W}\ot \cO(\dT)^{W}$ corresponding to $(1,1)$. 

By Theorem \ref{thm:function c part}, $\io_{\dG}^{*}$ is an isomorphism (using that $\dGder$ is simply-connected this follows directly, but in fact this is true without assuming $\dGder$ is simply-connected), hence so is its completed version $\wh\io_{\dG}^{*}$. By the commutative diagram \eqref{compl io} where both vertical arrows are isomorphisms, we conclude that $\wh\io_{\dg}^{*}$ is an isomorphism. 

We would like to deduce that $\io_{\dg}^{*}$ is an isomorphism from the fact that $\wh\io_{\dg}^{*}$ is an isomorphism. Note that we have a $\Gm$-action on $\frC^{2}_{\dg}$ by simultaneously scaling the elements in the commuting pair. This gives a grading $\frR:=\cO(\frC^{2}_{\dg})^{\dG}=\op_{n\ge0} \frR_{n}$. Similarly, $\frT=\cO(\dt\times\dt)^{W}$ has a grading $\frT=\op_{n\ge0}\frT_{n}$ as well as $\frS=\op_{n\ge0} \frS_{n}$. The map $\io_{\dg}^{*}$ is a map of graded $\frS$-algebras:
$$\io_{\dg}^{*}: \frR=\op_{n\ge0}\frR_{n}\to \frT=\op_{n\ge0}\frT_{n}.$$
The map $\wh\io^{*}_{\dg}$ is obtained from $\io^{*}_{\dg}$ by $\frS_{+}$-adic completion. It is clear that the $\frS_{+}$-adic completion of $\frT$ is given by $\wh\frT=\prod_{n\ge0}\frT_{n}$. 

Note that $\frR/\frS_{+}\frR=\cO(\frC^{2}_{\cN^{\vee}})^{\dG}$ where $\frC^{2}_{\cN^{\vee}}=\frs^{-1}(0,0)\subset \cN^{\vee,2}$ is the scheme of commuting nilpotent elements in $\dg$. Similarly, for $R=\cO(\frC^{2}_{\dG})^{\dG}$ and $S_{+}$ the augmentation ideal of $S$, $R/S_{+}R=\cO(\frC^{2}_{\cU^{\vee}})^{\dG}$ where $\frC^{2}_{\cU^{\vee}}=s^{-1}(1,1)\subset \cU^{\vee,2}$ is the the scheme of commuting unipotent elements in $\dG$. Lemma \ref{l:formal completion} below (or rather a simpler version of the same argument) shows that $\frR/\frS_{+}\frR\cong R/S_{+}R$. By Theorem \ref{thm:function c part}, $R/S_{+}R$ is the coordinate ring of the central fiber of $(\dT\times\dT)\sslash W\to (\dT/\sslash W)^{2}$, hence finite dimemsional. We conclude that  $\frR/\frS_{+}\frR$ is finite-dimensional over $\CC$. This implies that $\frS_{+}\frR$ and $\frR_{+}=\op_{n>0}\frR_{n}$ define the same topology on $\frR$. Therefore the $\frS_{+}$-adic completion $\wh\frR$ is $\prod_{n\ge0}\frR_{n}$. The fact that 
$$\io^{*}_{\dg}: \wh\frR=\prod_{n\ge0}\frR_{n}\to \wh\frT=\prod_{n\ge0}\frT_{n}$$ 
is an isomorphism implies its restriction to each degree $\frR_{n}\to \frT_{n}$ is bijective, hence $\io^{*}_{\dg}$ is an isomorphism. This finishes the proof.
\end{proof}

\begin{lemma}\label{l:formal completion}
The exponential map $(x,y)\mapsto (\exp(x), \exp(y))$ induces a $\dG$-equivariant isomorphism of formal schemes
\begin{equation*}
\exp^{2}: \wh\frC^{2}_{\dg}\isom\wh\frC^{2}_{\dG}.
\end{equation*}
\end{lemma}
\begin{proof} 
Let $\Nil$ be the category of pairs $(R,I)$ where $R$ is a $\CC$-algebra with a nilpotent ideal $I$ (i.e. $I^{N}=0$ for some $N>0$). We write $\ov R=R/I$. Then the category of formal schemes over $\CC$ full faithfully embeds into the category of functors $\Fun(\Nil, \Sets)$. 

The formal completion $\wh\dg$ of $\dg$ along the nilpotent cone $\cN^{\vee}$ is the functor that sends $(R,I)$ to the set of $x\in \dg\ot R$ such that $\ov x$ (the image of $x$ in $\dg\ot \ov R$) lies in $\cN^{\vee}(\ov R)$. Similarly, the  formal completion $\wh\dG$ of $\dG$ along the unipotent variety $\cU^{\vee}$ sends $(R,I)$ to the set of $g\in \dG(R)$ such that $\ov g\in\cU^{\vee}(\ov R)$.

Similarly, the completion $\wh\frC^{2}_{\dg}$ sends $(R,I)$ to the set of commuting pairs $(x,y)\in \frC^{2}_{\dg}(R)$ such that $\ov x,\ov y\in \cN^{\vee}(\ov R)$. The completion $\wh\frC^{2}_{\dG}$ sends $(R,I)$ to the set of commuting pairs $(g,h)\in \frC^{2}_{\dG}(R)$ such that $\ov g,\ov h\in \cU^{\vee}(\ov R)$.

We first check that the exponenital map and logarithm map give inverse isomorphisms between the functors $\wh\dg$ and $\wh\dG$. To construct these maps, we choose a faithful representation $V$ of $\dG$ and embed $\dG$ into $\GL(V)$. We define $\wh\gl(V)$ and $\wh\GL(V)$ as the formal completions of $\gl(V)$ and $\GL(V)$ along the closed subsets of nilpotent and unipotent matrices, viewed as functors on $\Nil$.  Note that any $x\in \wh\gl(V)(R,I)$ is nilpotent (because $x^{n}$ has entries in $I$, hence nilpotent). Similarly, any $g\in \wh\GL(V)(R,I)$ is unipotent (i.e., $g-1$ is nilpotent). The exponential map $\exp: \wh\gl(V)(R,I)\to \wh\GL(V)(R,I)$ is defined using the usual formula $\exp(x)=\sum_{n\ge0}\frac{x^{n}}{n!}$ which is a finite sum since $x$ is nilpotent. The logarithm map $\log: \wh\GL(V)(R,I)\to \wh\gl(V)(R,I)$ is defined using the usual formula $\log(g)=-\sum_{n\ge1}\frac{(1-g)^{n}}{n}$ which is a finite sum since $g-1$ is nilpotent.  Clearly $\exp$ and $\log$ give inverse isomorphisms between the functors $\wh\gl(V)$ and $\wh\GL(V)$.

We claim that $\exp(\wh\dg(R,I))\subset \wh\dG(R,I)$ and $\log(\wh\dG(R,I))\subset \wh\dg(R,I)$. Once this is shown,  $\exp$ and $\log$ then restrict to give inverse bijections between $\wh\dg(R,I)$ and $\wh\dG(R,I)$ functorially for $(R,I)\in \Nil$, hence giving inverse isomorphisms (clearly $\dG$-equivariant) between the formal schemes $\wh\dg$ and $\wh\dG$. To check $\exp(\wh\dg(R,I))\subset \wh\dG(R,I)$, let $\{v_{i}\}$ be a collection of tensors on $V$ such that $\dG$ is defined as the simultaneous stabilizer of $\{v_{i}\}$ in $\GL(V)$. Then $\dg$ is the simultaneous annihilator of $\{v_{i}\}$ in $\gl(V)$. If $x\in \wh\dg(R,I)$, then $xv_{i}=v_{i}$, hence $\exp(x)v_{i}=v_{i}$ and therefore $\exp(x)\in \wh\dG(R,I)$. This proves $\exp(\wh\dg(R,I))\subset \wh\dG(R,I)$. The other inclusion $\log(\wh\dG(R,I))\subset \wh\dg(R,I)$ is proved in the same way.

Now we have a $\dG\times\dG$-equivariant isomorphism $\exp^{2}: \wh\dg\times\wh\dg\isom \wh\dG\times\wh\dG$ with inverse $\log^{2}$. We claim that for any $(R,I)\in \Nil$, $\exp^{2}$ sends $\wh\frC^{2}_{\dg}(R,I)$ to $\wh\frC^{2}_{\dG}(R,I)$, and $\log^{2}$ sends $\wh\frC^{2}_{\dG}(R,I)$ to $\wh\frC^{2}_{\dg}(R,I)$. Once this is checked, $\exp^{2}$ and $\log^{2}$ then restrict to give inverse bijections between $\wh\frC^{2}_{\dg}(R,I)$ and $\wh\frC^{2}_{\dG}(R,I)$ functorially for $(R,I)\in \Nil$, hence giving inverse isomorphisms between the formal schemes $\wh\frC^{2}_{\dg}$ and $\wh\frC^{2}_{\dG}$, as desired.  Now suppose $(x,y)\in \wh\frC^{2}_{\dg}(R,I)$. Since $\exp(x)$ and $\exp(y)$ are polynomials in $x$ and $y$, and $[x,y]=0$, we see that $\exp(x)$ commutes with $\exp(y)$ as elements in $\dG(R)$, hence a fortiori in $\wh\dG(R,I)$. This verifies that $\exp^{2}$ sends $\wh\frC^{2}_{\dg}(R,I)$ to $\wh\frC^{2}_{\dG}(R,I)$. The statement about $\log^{2}$ is proved similarly, using that $\log(g)$ and $\log(h)$ are polynomials in $g$ and $h$. This finishes the proof.  
\end{proof}

Same argument as in the proof of Theorem \ref{thm:Lie alg} proves also the Chevalley restriction theorem for the  Lie algebra-group commuting scheme. Let $\frC_{\dg,\dG}$ be the subscheme of $(x,g)\in \dg\times \dG$ such that $\Ad_{g}(x)=x$. We have the obvious map of stacks
\begin{equation*}
\io_{\dg,\dG}: (\dt\times \dT)/W\to  \frC_{\dg,\dG}/\dG.
\end{equation*}

\begin{theorem}\label{thm:mix} Assume Ansatz~\ref{ans:intro univ aff equiv}. The map $\io_{\dg,\dG}$ induces an isomorphism on global (i.e.~invariant)  functions
\begin{equation*}
\io_{\dg,\dG}^{*}: \cO(\frC_{\dg,\dG})^{\dG}\isom \cO(\dt\times \dT)^{W}.
\end{equation*}
\end{theorem}
\begin{proof}[Sketch of proof] We let $\cS=\cO(\dt)^{W}\ot \cO(\dT)^{W}$ equipped with the grading coming from the scaling action on $\dt$. Then $\cS/\cS_{+}\cong \cO(\dT)^{W}$. Then $\cR=\cO(\frC_{\dg,\dG})^{\dG}=\op_{n\ge0}\cR_{n}$ and $\cT=\cO(\dt\times\dT)^{W}=\op_{n\ge0}\cT_{n}$ are both graded $\cS$-algebras with gradings coming from scaling actions on $\dg$ and $\dt$. The same argument as in Theorem \ref{thm:Lie alg} shows that the map induced by $\io_{\dg,\dG}^{*}$ on the $\cS_{+}$-adic completions $\wh\cR\to \wh\cT$ is an isomorphism. It is easy to see that $\wh\cT=\prod_{n\ge0}\cT_{n}$. We claim that $\wh\cR=\prod_{n\ge0}\cR_{n}$. Indeed, we have $\cR/\cS_{+}\cR=\cO(\frC_{\cN^{\vee}, \dG})^{\dG}$ where, $\frC_{\cN^{\vee}, \dG}$ is the scheme of commuting pairs in $\cN^{\vee}\times \dG$. The exponenital map gives a $\dG$-equivariant isomorphism $\frC_{\cN^{\vee}, \dG}\isom \frC_{\cU^{\vee}, \dG}$ where the nilpotent cone is replaced with the unipotent variety $\cU^{\vee}$. By Theorem \ref{thm:function c part}, $\cO(\frC_{\cU^{\vee}, \dG})^{\dG}\cong \cO(\dT)^{W}$ is free of rank one over $\cS/\cS_{+}=\cO(\dT)^{W}$. Therefore $\cR/\cS_{+}\cR\cong \cO(\frC_{\cN^{\vee}, \dG})^{\dG}\cong \cO(\dT)^{W}\cong \cR_{0}=\cO(\dG)^{\dG}$. This implies $\cS_{+}\cR=\cR_{+}=\op_{n>0}\cR_{n}$, which implies $\wh\cR=\prod_{n\ge0}\cR_{n}$. Therefore the  isomorphism $\wh\cR\isom \wh\cT$ implies the isomorphism before completion.
\end{proof}


\subsection{Further spectral consequences}\label{s:further apps}
In this section, we apply the results of Section~\ref{ss:additional app} to deduce further spectral consequences.
We continue with the notation of Section~\ref{ss:additional app}, and first introduce  the  spectral objects corresponding to $a(k_{G/G}), \tr_\cG(e_\cG)\in hh(\cH_\cG)$ introduced therein.

Introduce the natural closed embedding  $i_{2}: G^\vee/G^\vee  \to Z^2_{G^\vee}$ induced by $g \mapsto (g,1)$, and the coherent sheaf $i_{2*}\cO_{G^\vee/G^\vee}$ on $Z^2_{G^\vee}.$  Introduce as well the natural projection $p: Z^2_{B^\vee} \to Z^2_{G^\vee}$, and the coherent sheaf $p_* \cO_{Z^2_{B^\vee}}$ 
on $Z^2_{G^\vee}$.


Consider the closed embedding of derived stacks
\begin{equation*}
\xymatrix{\s: \dB\bs \dG/\dB\simeq \BB \dB\times_{\BB\dG}\BB \dB\ar@{^{(}->}[r] & \St_{\dG}.}
\end{equation*}
Taking direct image of ind-coherent sheaves gives a monoidal functor
\begin{equation}\label{indcoh sigma}
\xymatrix{\s_{*}: \Ind\Coh(\dB\bs \dG/\dB)\ar[r] & \Ind\Coh(\St_{\dG}).}
\end{equation}

Let $2\r^{\vee}\in\xch(\dT)$ be the sum of positive roots of $\dG$ with respect to $\dB$. Let $N$ be the number of positive roots of $\dG$. When $\r^{\vee}\in\xch(\dT)$, we consider the sheaf 
$$\cA:=\cO_{B^\vee\bs G^\vee/B^\vee }(-\rho^{\vee}, -\rho^{\vee})[N]\in \Ind\Coh(\dB\bs \dG/\dB)$$ 
obtained as the tensor product of pullbacks of $\cO(-\r^{\vee})$ from both factors $\BB\dB$. When $\r^{\vee}$ is not a character of $\dT$, let $\nu:\dG_{1}\to \dG$ be the double cover for which $\r^{\vee}$ is  a character of $\dT_{1}=\nu^{-1}(\dT)$, then we have $B^\vee\bs G^\vee/B^\vee =\dG_{1}/(\dB_{1}\times\dB_{1}/\D(\ker(\nu)))$ (where $\dB_{1}=\nu^{-1}(\dB)$, two factors of $\dB_{1}$ act on $\dG_{1}$ by left and right translations). Since $(-\r^{\vee}, \r^{\vee})$ is always character of $\dB_{1}\times\dB_{1}/\D(\ker(\nu))$, it defines a line bundle on $\dG_{1}/(\dB_{1}\times\dB_{1}/\D(\ker(\nu)))$ and hence on $B^\vee\bs G^\vee/B^\vee$, which we still denote by $\cO_{B^\vee\bs G^\vee/B^\vee}(-\r^{\vee},-\r^{\vee})$. In any case, we have the object $\cA:=\cO_{B^\vee\bs G^\vee/B^\vee }(-\rho^{\vee}, -\rho^{\vee})[N]\in \Ind\Coh(\dB\bs \dG/\dB)$, and the object $\s_{*}\cA\in \Ind\Coh(\St_{\dG})$.

\begin{lemma}\label{l:alg A}
The object $\s_{*}\cA\in \Ind\Coh(\St_{\dG})$ carries a canonical algebra structure for which its descended trace is
\begin{equation}\label{des tr A}
\ov\tr(\s_{*}\cA)\simeq i_{2*}\cO_{\dG/\dG}\in \Ind\Coh(Z_{\dG}^{2}).
\end{equation}
\end{lemma}
\begin{proof} We prove the lemma in the case $\r^{\vee}\in \xch(\dT)$; the general case can be deduced by passing to double cover $\dG_{1}$. 

We construct an algebra structure on $\cA$. Let $\cV_{\dG}=\dB/\dB\times_{\dG/\dG}\BB\dG$. Consider the correspondence
\begin{equation*}
\xymatrix{\St_{\dG}=\dB/\dB\times_{\dG/\dG}\dB/\dB & \ar[l]_-{c}\dB/\dB\times_{\dG/\dG}\BB\dB \ar[r]^-{d}& \dB/\dB\times_{\dG/\dG}\BB\dG=\cV_{\dG}}
\end{equation*}
where the maps are induced by the natural maps $\dB/\dB\hookleftarrow \BB\dB\to \BB\dG$ on the second factors. Consider the adjoint functors
\begin{equation*}
\xymatrix{\Pi^{L}=d_{*}(c^{*}\ot p_{2}^{*}\cO(-\r^{\vee})[N]): \Ind\Coh(\St_{\dG})\ar@<1ex>[r]  &  \ar@<1ex>[l]\Ind\Coh(\cV_{\dG}): \Pi=c_{*}(d^{*}\ot p_{2}^{*}\cO(-\r^{\vee}))}
\end{equation*}
From these we get a natural isomorphism of endo-functors of $\Ind\Coh(\St_{\dG})$:
\begin{equation}\label{Amonad}
\Pi\c\Pi^{L}\simeq (-)\star\s_{*}\cA.
\end{equation}
The monad structure on $\Pi\c\Pi^{L}$ then gives an algebra structure on $\cA$.

From \eqref{Amonad} and the Barr-Beck-Lurie theorem, we conclude that the category $\RMod_{\s_{*}\cA}(\Ind\Coh(\St_{\dG}))$ of right $\s_{*}\cA$-modules in $\Ind\Coh(\St_{\dG})$ can be identified with $\Ind\Coh(\cV_{\dG})$. Similarly, $\Bimod_{\cA}(\Ind\Coh(\St_{\dG}))$ can be identified with $\Ind\Coh(\cW_{\dG})$, where 
$$\cW_{\dG}=\BB\dG\times^{\bR}_{\dG/\dG}\BB\dG$$ 
is the spectral stack hosting the derived spherical Hecke category. Under this equivalence, the regular bimodule $\cA$ itself corresponds to the monoidal unit $\D_{*}\cO_{\BB\dG}\in \Ind\Coh(\cW_{\dG})$. By Definition~\ref{def:desc tr}, we have
\begin{equation}\label{dtr A tr O}
\ov\tr(\s_{*}\cA)\simeq \tr(\D_{*}\cO_{\BB\dG})\in \Ind\Coh(Z^{2}_{\dG}).
\end{equation}
Here the right side a priori lies in $hh(\Ind\Coh(\cW_{\dG}))$, which is a full subcategory of $\Ind\Coh(Z^{2}_{\dG})$.

To compute $\tr(\D_{*}\cO_{\BB\dG})$, we apply Example \ref{ex:tr O} to the map $f:X=\BB \dG\to \dG/\dG=Y$. The induced map on the loop spaces $Lf: LX=\dG/\dG\to Z^{2}_{\dG}=LY$ can be identified with $i_{2}$. Therefore Example \ref{ex:tr O} gives
\begin{equation*}
\tr(\D_{*}\cO_{\BB\dG})\simeq i_{2*}\cO_{\dG/\dG}.
\end{equation*}
Combined with \eqref{dtr A tr O} we get \eqref{des tr A}.

\end{proof}



To state the next result, we need a variant 
of Ansatz~\ref{ans:intro univ aff equiv}. 
\begin{ansatz}\label{ans:enh univ aff equiv} The monoidal equivalence~(\ref{univ aff equiv}) in  Ansatz~\ref{ans:intro univ aff equiv} holds:
\begin{equation*}	
\xymatrix{
		\Phi:\Ind\Coh(\St_{\dG}) \ar[r]^-\sim &   \cH_{\cG},
	}
\end{equation*}
Moreover, $\Phi$ takes  $\s_{*}\cA$ to $i_{!}k_{U\bs G/U } \in \cH_{\cG}$ (where $k_{U\bs G/U}\in \cH_{G}$ is the constant sheaf and $i_{!}: \cH_{G}\incl \cH_{\cG}$), compatibly with their respective algebra structures.
\end{ansatz}

Recall assuming equivalence~(\ref{univ aff equiv}), we have the equivalence (see \eqref{Z2hh})
\begin{equation}\label{Z2hh'}
\xymatrix{\Ind\Coh_{\cN}(Z^2_{G^\vee})\ar[r]^-{\sim} & hh(\cH_{\cG})}.
\end{equation}

\begin{prop} 
\begin{enumerate}
\item Assume Ansatz~\ref{ans:enh univ aff equiv} holds. Then under the equivalence \eqref{Z2hh'},  $i_{2*}\cO_{G^\vee/G^\vee}$ corresponds to $a(k_{G/G})$.
\item Assume the monoidal equivalence (\ref{univ aff equiv}) holds. Under the equivalence \eqref{Z2hh'}, $ p_*\cO_{Z^2_{B^\vee}}$ corresponds to $\tr(e_\cG)$.
\end{enumerate}
 \end{prop}
 \begin{proof}
(1) To see the first identification, we first have 
 an analogous (but simpler) version of Theorem~\ref{thm: dtr of whit}. In place of the universal finite Whittaker sheaf $\Wh_G\in \cH_G$, we consider the constant sheaf $k_{U\bs G/U } \in \cH_G$ as an algebra object.  Then it is straightforward to check its descended trace  $\ol \tr_G(k_{U\bs G/U }) \in hh(\cH_G) \simeq \Sh_\cN(G/G)$ (as a bimodule over itself) is the constant sheaf $k_{G/G}\in  \Sh_\cN(G/G)$.  

By the functoriality of descended trace in Lemma~\ref{l: dtr inv} applied to $i_{!}: \cH_{G}\to \cH_{\cG}$, we have
 \begin{equation}\label{aut ov tr const}
\ov\tr_{\cG}(i_{!}k_{U\bs G/U })\simeq a(k_{G/G})\in hh(\cH_{\cG}).
\end{equation}
Comparing \eqref{aut ov tr const} with \eqref{des tr A}, we conclude that $i_{2*}\cO_{\dG/\dG}$ matches with $a(k_{G/G})$ under \eqref{Z2hh'}.

%

(2) Under the monoidal equivalence \eqref{univ aff equiv}, the monoidal unit $\D_{*}\cO_{\dB/\dB}\in \Ind\Coh(\St_{\dG})$ corresponds to the monoidal unit $e_{\cG}\in \cH_{\cG}$. Therefore $\tr(e_{\cG})$ corresponds to $\tr(\D_{*}\cO_{\dB/\dB})$. By Example \ref{ex:tr O} applied to the map $f: X=\dB/\dB\to \dG/\dG=Y$, we conclude that $\tr(\D_{*}\cO_{\dB/\dB})\simeq p_{*}\cO_{Z^{2}_{\dB}}$, as desired. 
 \end{proof}

 \begin{cor} 
\begin{enumerate}
\item Assume Ansatz~\ref{ans:enh univ aff equiv} holds. Then there is a canonical equivalence of dg algebras
\begin{equation*}
\End_{Z^2_{G^\vee}} (i_{2*}\cO_{G^\vee/G^\vee})  \simeq   \cO(T^\vee \times (\dt)^{*}[1] \times (\dt)^{*}[2] )^W.
\end{equation*}

\item Assume the monoidal equivalence~(\ref{univ aff equiv}) holds. Then there is a canonical equivalence of dg algebras
\begin{eqnarray*}
\End_{Z^2_{G^\vee}} (p_*\cO_{Z^2_{B^\vee}})  \simeq   k[W] \#   \cO(T^\vee \times T^\vee \times \frt^\vee[-1]).
\end{eqnarray*}
\end{enumerate}
 \end{cor}


\appendix


\section{Calculating some colimits}\label{s:str sh} 

This appendix collects some results about colimits of categories called upon in Section~\ref{ss:cocenter recoll}. 
We will consider colimits of  stable presentable $\oo$-categories viewed as  objects of the $\oo$-category {$\St^L$} of stable presentable $\oo$-categories with morphisms given by left adjoints.  (The discussion extends verbatim if we additionally work in the {$k$-linear} context.) 
 
The main new result here is a categorical analog of the contraction principle for sheaf cohomology. It allows us to reduce the calculation of a colimit indexed by a complicated diagram to a simpler one, assuming a certain contractibility condition. This is a key categorical ingredient in the proof of Theorem~\ref{th:semi-orth hh G neutral}.



\subsection{Recollements and stratifications}

\sss{Recollements}\label{sss:rec}
For recollement in the setting of  $\infty$-categories, we recommend the reference~\cite[A.8]{lurieHigherAlgebra2012}. In our setting of stable presentable $\infty$-categories, we will largely follow the reference  
~\cite[1.1]{ayalaStratifiedNoncommutativeGeometry}.

Recall a {\em recollement}  ~\cite[Definition 1.1.1]{ayalaStratifiedNoncommutativeGeometry} of stable presentable $\oo$-categories is a diagram of adjoint triples
\begin{equation*}
\xymatrix{
\cC_0 \ar@<3ex>[r]^{j_!} \ar@<-3ex>[r]^{j_*} &  \ar[l]_{j^! = j^*}  \cC \ar@<3ex>[r]^{i^*}  \ar@<-3ex>[r]^{i^!}  &   \ar[l]_{i_* = i_!}   \cC_1
}
\end{equation*}
where $j_!, j_*$, and $i_* = i_!$ are fully faithful with $\textup{Im}(j_!)= \textup{Ker}(i^*)$,
$\textup{Im}(j_*)= \textup{Ker}(i^!)$, and $\textup{Im}(i_* = i_!)= \textup{Ker}(j^! = j^*)$.
Note the recollement is determined by the subdiagram
\begin{equation*}
\xymatrix{
\cC_0 &  \ar[l]_{j^! = j^*}  \cC  &   \ar[l]_{i_* = i_!}   \cC_1
}
\end{equation*}
with the rest of the notion comprising properties to be verified.

In our topological setting, we call $\cC_0$ the open subcategory and  $\cC_1$ the closed subcategory
of the recollement.\footnote{The terminology is intended to model  sheaves on a space with an open-closed decomposition. Note this is opposite to the convention 
of \cite{ayalaStratifiedNoncommutativeGeometry} 
 where $\cC_0$ is called a closed subcategory since it models quasi-coherent sheaves  on a scheme supported along a closed subscheme.}
We refer to $\cC_0$  and  $\cC_1$ as the recollement complements of each other.
 
Suppose we have a commutative diagram in $\St^{L}$
\begin{equation}\label{morphism rec}
\xymatrix{\cC_0 \ar[d]^{f_{0}} &  \ar[l]_{j^! = j^*}  \cC  \ar[d]^{f} &   \ar[l]_{i_* = i_!}   \cC_1\ar[d]^{f_{1}}\\
\cD_0  &  \ar[l]_{j^! = j^*}  \cD   &   \ar[l]_{i_* = i_!}   \cD_1
}
\end{equation}
such that both rows extend to recollements. We say the above diagram (or the functors $(f_{0},f,f_{1})$) is a {\em morphism of recollements} if the square on the left is both left and right adjointable with respect to the horizontal arrows $j^{!}=j^{*}$ (equivalently, the square on the right is both left and right adjointable with respect to the horizontal arrows $i_{!}=i_{*}$).

If in a morphism of recollement as in \eqref{morphism rec}, both $f_{0}$ and $f_{1}$ are equivalences, then $f$ is also an equivalence (this follows from \cite[Prop. A.8.14]{lurieHigherAlgebra2012}).

\sss{Stratifications}\label{sss:strat}
We will use the more general notion of a recollement of stable presentable $\oo$-categories indexed by a poset  $P$ as introduced in~\cite[1.3]{ayalaStratifiedNoncommutativeGeometry} under the name {\em stratification}. Since we work in the topological rather than algebraic setting, we will adapt the notation and terminology in the definitions to reflect this.
 For a finite totally ordered poset,  a stratification  also goes  by the name {\em semi-orthogonal decomposition}~\cite{bondalRepresentableFunctorsSerre1990}.
 When we apply the below definitions, our poset will be  the {\em opposite} of the non-negative natural numbers $\mathbb Z_{\geq 0}$ with their usual total order.
 
Given a  stable presentable $\infty$-category, an {\em open  subcategory} (called a {\em closed subcategory} in ~\cite[Definition 1.3.1]{ayalaStratifiedNoncommutativeGeometry}) 
 is an adjoint triple
\begin{equation*}
\xymatrix{
\cC_0 \ar@<3ex>[r]^{j_!} \ar@<-3ex>[r]^{j_*} &  \ar[l]_{j^! = j^*}  \cC
}
\end{equation*}
where $j_!,  j_*$ are fully faithful. We write $\cC_{open}$ for the poset of open subcategories of $\cC$ with respect to inclusion under $j_!$.

 Given a poset $P$,  a {\em $P$-stratification} (following \cite[Definition 1.3.2]{ayalaStratifiedNoncommutativeGeometry})  of a stable presentable $\oo$-category $\cC$  is a functor
 \begin{equation*}
\xymatrix{
\cZ_{\bullet}:P \ar[r] & \cC_{open}}
\end{equation*}
such that $\cC = \bigcup_{p \in P} \cZ_{ p} $ and
for any $p, q\in P$, there exists a factorization
\begin{equation*}
\xymatrix{
\bigcup_{r\leq p, q} \cZ_{ r} \ar[r]^-{j_!} & \cZ_{ p}\\
\ar@{-->}[u] \cZ_{ q} \ar[r]^-{j_!} & \cZ\ar[u]_-{j^! = j^*}
}
\end{equation*}

 Given a  $P$-stratification $\cZ_\bullet$, its {\em $p$th stratum}~\cite[Definition 1.3.3]{ayalaStratifiedNoncommutativeGeometry}   is the quotient
 \begin{equation*}
 \cC_p = \cZ_{ p}/ \bigcup_{q< p} \cZ_{ q}
 \end{equation*}
 By construction, there is a recollement
 \begin{equation*}
\xymatrix{
 \bigcup_{q< p} \cZ_{ q} \ar@<3ex>[r]^{j_!} \ar@<-3ex>[r]^{j_*} &  \ar[l]_{j^! = j^*}  \cZ_{ p} \ar@<3ex>[r]^{i^*}  \ar@<-3ex>[r]^{i^!}  &   \ar[l]_{i_* = i_!}   \cC_p
}
\end{equation*}

\begin{ex}
Suppose $X = \cup_{p \in \ZZ_{\geq 0}} X_{p}$ is a topological space written as the increasing union of closed subspaces. Suppose $\cC = \colim_{p\in  \ZZ_{\geq 0}} \Sh(X_p)$ where the transition maps are given by extension by zero.

This fits into the above formalism as follows. For $p \in \ZZ_{\geq 0}$, set $U_p = X \setminus X_{p-1}$ (convention: $X_{-1} = \vn$) viewed as 
 the increasing union of closed subspaces $U_p =  \cup_{q\in \ZZ_{\geq 0}} U_p \cap X_{q}$.
 Set $\cZ_{ p} = \colim_{q\in  \ZZ_{\geq 0}} \Sh(U_p \cap X_{q}) $ where the transition maps are given by extension by zero.

Then we have 
a stratification $\cC = \cZ_{ 0} \hookleftarrow \cZ_{ 1} \hookleftarrow \cZ_{ 2} \cdots$  with respect to the  poset  $P = ( \ZZ_{\geq 0})^{op}$
with open subcategories 
\begin{equation*}
\xymatrix{
\cZ_{ p} \ar@<3ex>[r]^{j_!} \ar@<-3ex>[r]^{j_*} &  \ar[l]_{j^! = j^*}  \cC
}
\end{equation*}
with inclusions given by $!$-extensions.
 The strata  are given 
by $\cC_p = \cZ_{ p}/\cZ_{p+1} =  \Sh(X_p \setminus X_{p-1})$ noting that $U_p \setminus  U_{p+1} = X_p \setminus X_{p-1}$.
 By construction, there are recollements
 \begin{equation*}
\xymatrix{
\cZ_{ p+1} \ar@<3ex>[r]^{j_!} \ar@<-3ex>[r]^{j_*} &  \ar[l]_{j^! = j^*}  \cZ_{ p} \ar@<3ex>[r]^{i^*}  \ar@<-3ex>[r]^{i^!}  &   \ar[l]_{i_* = i_!}   \cC_p
}
\end{equation*}

The structure is equivalent to giving the  inclusions
$ \cY_{ 0} \hookrightarrow \cY_{ 1} \hookrightarrow \cY_{ 2} \cdots$
of the
complementary closed categories $\cY_p = \Sh(X_p)$ with their recollements
 \begin{equation*}
\xymatrix{
 \cC_p  \ar@<3ex>[r]^{j_!} \ar@<-3ex>[r]^{j_*} &  \ar[l]_{j^! = j^*}  \cY_{ p} \ar@<3ex>[r]^{i^*}  \ar@<-3ex>[r]^{i^!}  &   \ar[l]_{i_* = i_!}   \cY_{p-1} 
}
\end{equation*}

\end{ex}

\subsection{Interaction of  colimits and recollement}

The following gives natural hypotheses for when a colimit  of
stable presentable $\oo$-categories with recollement inherits a recollement.

\begin{prop}\label{p:recoll} Let $\cI$ be a small category,  $\cC'_\bullet, \cC_\bullet :\cI \to \St^L$ 
 functors to stable presentable $\oo$-categories (with  {morphisms given by left adjoints}),
and $i_\bullet:\cC'_\bullet\to \cC_\bullet$  a map of functors.


\begin{enumerate}
	\item  Assume the following:

\begin{enumerate}

\item For each $J\in \cI$, the functor $i_{J}$ extends to a recollement in $\St^{L}$:
\begin{equation*}
\xymatrix{\cC''_{J} \ar@<3ex>[r]^{j_{J!}} \ar@<-3ex>[r]^{j_{J*}}&  \ar[l]_{j_{J}}\cC_{J} \ar@<3ex>[r]^{i_{J}^{*}} \ar@<-3ex>[r]^{i_{J}^{!}} & \ar[l]_{i_{J}} \cC'_{J}}
\end{equation*}
In particular, $i_{J}$ is fully faithful and $\cC''_J = \cC_J/i_J(\cC'_J)$ (hence the assignment $J\mapsto \cC''_{J}$ extends naturally to a functor $\cC''_\bullet: \cI\to \St^{L}$). 


\item
For  all maps $I \to J$ in $\cI$, the natural commutative diagram
\beq
\label{eq:adjointable_square}
\xymatrix{\cC''_{I}\ar[d] &  \ar[l]_{j_{I}}\cC_{I}\ar[d] & \ar[l]_{i_{I}} \cC'_{I}\ar[d]\\
\cC''_{J} &  \ar[l]_{j_{J}}\cC_{J} & \ar[l]_{i_{J}} \cC'_{J}
}
\eeq
is a morphism of recollements. Here the vertical morphisms are given as part of the data of $\cC''_{\bu}, \cC_{\bu}$ and $\cC'_{\bu}$ as functors.
\end{enumerate}

Then the map 
\begin{equation*}
\xymatrix{
\fri=\colim i_{\bu}:\colim_{\cI} \cC'_\bullet \ar[r] &  \colim_{\cI} \cC_\bullet }
\end{equation*}
extends to a recollement
\begin{equation*}
\xymatrix{ \colim \cC''_\bullet \ar@<3ex>[r]^{\frj_!} \ar@<-3ex>[r]^{\frj_*} & \ar[l]_{\frj} \colim \cC_\bullet \ar@<3ex>[r]^{\fri^{*}} \ar@<-3ex>[r]^{\fri^{!}}& \ar[l]_{\fri} \colim \cC'_\bullet }
\end{equation*}
with $\fri^! \simeq \colim_{\cI} i_J^! $, $\fri^* \simeq \colim_{\cI} i_J^*$, $\frj=\colim_{\cI} j_J, \frj_!=\colim_{\cI} j_{J!}$ and $\frj_*=\colim_{\cI} j_{J*}$.

\item Assume further we are given a functor $f: \cI_0 \to \cI$, denote by $\cC'_{0\bullet} = \cC'_{\bullet} \circ f , \cC_{0\bullet} = \cC_{\bullet} \circ f$ and $\cC''_{0\bullet} = \cC''_{\bullet} \circ f$. Then the natural commutative diagram is a morphism of recollements:
\begin{equation*}
\xymatrix{\colim \cC''_{0\bullet} \ar[d] &  \ar[l] \colim\cC_{0\bullet}\ar[d] & \ar[l] \colim\cC'_{0\bullet}\ar[d]\\
\colim	\cC''_{\bullet} &  \ar[l] \colim\cC_{\bullet} & \ar[l] \colim \cC'_{\bullet}
}
\end{equation*}
\end{enumerate}
\end{prop}

\begin{proof}
(1) By assumption (a), we have adjoint triples   $(i^*_J, i_J ,i^!_J)$  in $St^L$, so that the unit  $ u : \id \to i^!_J  i_J$ and counit $\epsilon: i^*_J i_J \to \id$ are isomorphisms. Assumption (b) implies ($\fri^{*}:=\colim i^*_J, \fri=\colim i_J ,\fri^!:=\colim i^!_J)$ is also an adjoint triple, with counit $ \fri_* \fri \to \id$ an isomorphism. Therefore $\fri$ is fully faithful, and $\colim \cC_J$ is a recollement of $\colim \cC'_J$ and $\colim \cC_J/ \colim \cC'_J\simeq \colim \cC''_J$. Moreover, the natural map $\frj: \colim \cC_\bullet \to \colim \cC''_\bullet$ is naturally isomorphic to $\colim_{\cI} j_J$, by the functoriality of cokernel (cokernel of a functor $f: \cA \to \cB$ is by definition the pushout of $0 \leftarrow \cA \xrightarrow{f} \cB $). Since taking colimit preserves adjoint pairs, so we also have $\frj_! = \colim_{\cI} j_{J!}$ and $\frj_* = \colim_{\cI} j_{J*}$. 

(2) Follows directly from the functoriality of colimits.
\end{proof}

\begin{cor}\label{c:coprod}
Suppose we have a morphism of recollement as in \eqref{morphism rec} in which $f_{0}$ is an equivalence, then the functor
\begin{equation*}
f\coprod i_{*}: \cC\coprod_{\cC_{1}}\cD_{1}\to \cD
\end{equation*}
is an equivalence.
\end{cor}
\begin{proof}
Applying Prop. \ref{p:recoll} to the pushout of the recollements $(\cC_{0}\leftarrow\cC\leftarrow\cC_{1})$ and $(0\leftarrow\cD_{1}\leftarrow\cD_{1})$ along $(0\leftarrow\cC_{1}\leftarrow\cC_{1})$, we get a recollement $(\cC_{0}\leftarrow \cC\coprod_{\cC_{1}}\cD_{1}\leftarrow \cD_{1})$ that fit into a commutative diagram
\begin{equation*}
\xymatrix{\cC_{0}\ar[d]^{f_{0}} & \ar[l] \cC\coprod_{\cC_{1}}\cD_{1}\ar[d]^{f\coprod i_{*}} & \ar[l]\cD_{1}\ar@{=}[d]\\
\cD_{0} &\ar[l] \cD & \ar[l]\cD_{1}
}
\end{equation*}
It is easy to check the left square is both left and right adjointable, hence the above diagram is a morphism of recollements. Since both of the vertical arrows at the two ends are equivalences, so is the middle vertical arrow.
\end{proof}


\subsection{Posets and cosheaves}
\sss{Generalities on posets}\label{sss:poset}
For a poset (partially ordered set) $\cI$, let $N(\cI)$ denote its nerve, viewed as a simplicial set and hence a quasi-category. We say a subset  $\cJ\subset \cI$ is {\em up-closed} if $i\in \cJ$, $i\le i'$ implies $i'\in \cJ$; similarly, we say a subset $\cJ\subset \cI$ is {\em down-closed} if $i\in \cJ$, $i'\le i$ implies $i'\in \cJ$.

For a poset $\cI$, let $|\cI|$ be the geometric realization of $\cI$, defined as the usual geometric realization of the simplicial set $N(\cI)$ (hence a simplicial complex). In particular, the $r$-dimensional faces of $|\cI|$ are in bijection with non-degenerate $r$-dimensional simplices in $N(\cI)$, which are parametrized by strictly increasing chains $i_{0}<i_{1}<\cdots< i_{r}$ in $\cI$.

A poset $\cI$ is called {\em acyclic} if $|\cI|$ is weakly contractible.

An order-preserving map of posets $f: \cI\to \cJ$ is called {\em coCartesian}, if for any pair  $j\le j'$ in $\cJ$, and any $i\in f^{-1}(j)$, the set $\{i'\in \cI|i\le i', j'\le f(i')\}$ has a unique minimal element $\partial_{j\to j'}(i)$ that maps to $j'$. 

An order-preserving map of posets $f: \cI\to \cJ$  is called a {\em left covering map}, if for each pair $j\le j'$ in $\cJ$, and any $i\in f^{-1}(j)$, there is a unique $i'\in f^{-1}(j')$ such that $i\le i'$. A left covering map is in particular coCartesian. If $f$ is a left covering map, then each fiber $f^{-1}(j)$ is a discrete poset (any pair of distinct elements are incomparable), and $\partial_{j\to j'}$ gives a map $f^{-1}(j)\to f^{-1}(j')$ for each pair $j\le j'$ in $\cJ$, compatible with compositions. Such a structure is simply a functor $\cJ\to \Sets$ (called a $\cJ$-set in Section~\ref{sss:Pset}). Conversely, given a functor $X: \cJ\to \Sets$, if we let $\cI=\coprod_{j\in \cJ}X_{j}$ and define the partial order on $\cI$ to be generated by $i\le X_{j\to j'}(i)$ for all $j\le j'\in \cJ$ and $i\in X_{j}$,  then the natural projection $f: \cI\to \cJ$  is a left covering map of posets. We conclude that a left covering map of posets $f: \cI\to \cJ$ is the same datum as a $\cJ$-set, i.e., a functor $\cJ\to \Sets$.

\sss{Generalities on cosheaves}\label{sss:cosh}
Let $\cI$ be a poset with nerve $N(\cI)$. We call a functor $\cF: \cI\to \St^{L}$ a {\em cosheaf} on $\cI$. In other words, for each $i\in I$, $\cF$ assigns a stable presentable $\oo$-category $\cF_{i}$, and for each arrow $\ph: i\to j$ in $\cI$, there is a functor $\cF_{\ph}: \cF_{i}\to \cF_{j}$, together with higher compatibilities. We call $\colim_{\cI}\cF\in\St^{L}$ the (category of) cosections of $\cF$. 

Let $f: \cI\to \cJ$ be an order-preserving map of posets. Let $\cG$ be a cosheaf on $\cJ$, then we define $f^{*}\cG$ to be the cosheaf on $\cI$ given by the composition $\cI\xr{f}\cJ\xr{\cG}\St^{L}$. If $f$ is injective, then we write $f^{*}\cF$ as $\cF|_{\cI}$.

In the following, we gather some useful tools for calculating colimits. Most of them can be found in Lurie \cite[Section 4]{lurieHigherToposTheory2009a}.

\sss{Pushforward}\label{sss:push cosh}
Let $f: \cI\to \cJ$ be an order-preserving map of posets and $\cF$ be a cosheaf on $\cI$. We denote by $f_{!}\cF: \cJ\to \St^{L}$ a left Kan extension of $\cF$ along the map $f$. Then $\colim_{\cI}\cF\isom \colim_{\cJ}f_{!}\cF$ by the general properties of left Kan extensions. 

Suppose $f:\cI\to \cJ$ is coCartesian, then for $j\le j'$ in $\cJ$, there is a natural functor
\begin{equation*}
\colim_{i\in f^{-1}(j)}\cF_{i}\to \colim_{i'\in f^{-1}(j')}\cF_{i'}
\end{equation*}
induced by the functors (for $i\in f^{-1}(j)$)
\begin{equation*}
\cF_{i}\to \cF_{\partial_{j\to j'}(i)}\to \colim_{i'\in f^{-1}(j')}\cF_{i'}
\end{equation*}
where the second map uses that $\partial_{j\to j'}(i)\in f^{-1}(j')$. These maps give a cosheaf $\int_{f}\cF$ on $\cJ$ whose value at $j$ is $\colim_{i\in f^{-1}(j)}\cF_{i}$.

\begin{prop}[Lurie, special case of {\cite[Proposition 4.3.3.10]{lurieHigherToposTheory2009a}}]\label{p:push Kan}
Suppose $f:\cI\to \cJ$ is a coCartesian map of posets and $\cF$ is a cosheaf on $\cI$. Then $\int_{f}\cF$ is a left Kan extension of $\cF$ along $f$. In particular, there is an equivalence $\colim_{\cI}\cF\simeq \colim_{\cJ}\int_{f}\cF$.
\end{prop}

Another instance of pushforward is extension by zero. Let $\im: \cI_{1}\subset \cI$ be the inclusion of an up-closed subset. Then for a cosheaf $\cF_{1}$ on $\cI_{1}$, the cosheaf $\im_{!}\cF_{1}$ on $\cI$ is obtained by extension by zero: it assigns to each $i\in \cI\setminus \cI_{1}$ the zero category.  

Dually, let $\jmath: \cI_{0}\subset \cI$ be the inclusion of a down-closed subset. Let $\cF_{0}$ be a cosheaf on $\cI_{0}$. We define $j_{*}\cF_{0}$ to be extension by zero of $\cF_{0}$ as a cosheaf on $\cI$. Then $j_{*}: \Fun(\cI_{0}, \St^{L})\to \Fun(\cI,\St^{L})$ is a right adjoint to the restriction $j^{*}: \Fun(\cI,\St^{L})\to \Fun(\cI_{0},\St^{L})$.

\sss{Subdivision} Let $sd(\cI)$ be the set of non-degenerate simplices in the nerve $N(\cI)$: they are strictly increasing non-empty chains $\bi=(i_{0}< \cdots < i_{r})$ in $\cI$. Equip $sd(\cI)$ with a partial order by defining $\bi\le\bi'$ if and only if $\bi'$ is part of $\bi$. We have a map of posets $\pi: sd(\cI)\to \cI$ sending $\bi=(i_{0}< \cdots < i_{r})$ to $i_{0}$. 

For a cosheaf $\cF$ on $\cI$, let $sd(\cF)=\pi^{*}\cF$ be the pullback cosheaf on $sd(\cI)$.

\begin{lemma}\label{l:sd}
The map $\pi: sd(\cI)\to \cI$ is cofinal. In particular, for any cosheaf $\cF$ on $\cI$, there is a canonical equivalence
\begin{equation*}
\colim_{sd(\cI)}sd(\cF)\simeq \colim_{\cI}\cF.
\end{equation*}
\end{lemma}
\begin{proof}
To check $\pi$ is cofinal, we apply Quillen's Theorem A as in \cite[Theorem 4.1.3.3]{lurieHigherToposTheory2009a}. For $i\in \cI$, the category $M=sd(\cI)\times_{\cI}\cI_{i/}$ consists of strictly increasing non-empty chains $i_{0}<\cdots <i_{r}$ in $\cI$ such that $i\le i_{0}$. We need to show that the geometric realization $|M|$ of $M$ is contractible. Let $L\subset M$ be the subposet where $i_{0}=i$; let $K\subset M$ be the subposet where $i<i_{0}$. Then $M=L\sqcup K$, and we can identify $L\cong K^{\rhd}$ (the right cone of $K$) because $L$ has a maximal element that is $i$ itself as a chain, and its complement is isomorphic to $K$ via $\bi=(i_{0}<\cdots <i_{r})\to \bi'=(i_{1}<\cdots <i_{r})$. Also we have $\bi\le \bi'$ (where $\bi'\in K$ and $\bi\in L$ is obtained from $\bi'$ by adding $i$), and all the other order relations are consequences of it and the fact that $L$ and $K$ are subposets. We conclude that $M\cong (K\times [1])\coprod_{K\times\{0\}}K^{\rhd}$ (where $[1]$ is the poset $\{0,1\}$). The geometric realization $|M|$ is thus homeomorphic to the cone over $|K|$, hence contractible. 
\end{proof}

\sss{Double colimits} In some situations we want to rewrite $\colim_{\cI} \cF$ as a colimit of colimits, where the inner colimit is taken over a subposet of $\cI$. Lurie's theorem \cite[Corollary 4.2.3.10]{lurieHigherToposTheory2009a} gives a general framework for doing this. We recall the statement of {\em loc.cit.}  in the simplified situation where the colimits are parametrized by posets. Suppose $\cK$ is a poset and $\cF: \cK\to \St^{L}$ is a cosheaf on $\cK$. Let $\cJ$ be another poset with an assignment $\cJ\ni j\mapsto \cK_{j}\subset \cK$ such that $j\le j'$ in $\cJ$ implies $\cK_{j}\subset \cK_{j'}$. 

\begin{prop}[Lurie, special case of {\cite[Corollary 4.2.3.10]{lurieHigherToposTheory2009a}}]\label{p:double colim} Under the above notation, assume: $\cK_{j}$ is up-closed for each $j\in \cJ$, and for each $k\in \cK$, the subposet $\cJ_{k}:=\{j\in \cJ|k\in \cK_{j}\}$ is acyclic. Then the natural maps $\colim_{\cK_{j}}\cF\to \colim_{\cK} \cF$ give an equivalence 
\begin{equation*}
\colim_{j\in \cJ}(\colim_{\cK_{j}}\cF)\simeq \colim_{\cK} \cF.
\end{equation*}
\end{prop}

As an immediate application, we get a Mayer-Vietoris type theorem for colimits.
\begin{cor}\label{c:MV} Let $\cI$ be a poset.  Suppose $\cI=\cI_{0}\cup \cI_{1}$, where $\cI_{0}$ and $\cI_{1}$ are both up-closed. Let $\cF$ be a cosheaf on $\cI$. Let $\cI_{01}=\cI_{0}\cap \cI_{1}$.  Then the natural maps give an equivalence
\begin{equation*}
 \colim_{\cI_{0}}\cF|_{\cI_{0}}\coprod_{\colim_{\cI_{01}}\cF|_{\cI_{01}}}\colim_{\cI_{1}}\cF|_{\cI_{1}}\isom \colim_{\cI}\cF.
\end{equation*}
\end{cor}
\begin{proof}
Let $\cD_{1}$ be the poset of proper subsets of $\{0,1\}$ under inclusion. Apply Proposition \ref{p:double colim} to the assignment $\vn\mapsto \cI_{01}$, $\{0\}\mapsto \cI_{0}$ and $\{1\}\mapsto \cI_{1}$.
\end{proof}

\sss{Locally constant cosheaves}
A {\em locally constant cosheaf} on $\cI$ is one such that $\cF_{\ph}$ is an equivalence for all arrows $\ph$ in $\cI$. 

\begin{prop}[Lurie, {\cite[Corollary 4.4.4.9]{lurieHigherToposTheory2009a}}]\label{p:loc} Let $f: \cI\incl \cJ$ be an order-preserving map of posets and $\cF$ a locally constant cosheaf on $\cJ$. If $|\cI|\incl |\cJ|$ is a weak homotopy equivalence, then the natural functor
\begin{equation*}
\colim_{\cI}f^{*}\cF\to \colim_{\cJ}\cF
\end{equation*}
is an equivalence.

In particular, if $\cF$ is a cosheaf on an acyclic poset $\cI$, then for any $i\in \cI$, the natural map $\cF_{i}\to \colim_{\cI}\cF$ is an equivalence. 
\end{prop}

\sss{Recollement}\label{sss:cosh rec}
Given maps of cosheaves on a poset $\cI$
\begin{equation*}
\xymatrix{ \cF'' & \ar[l]_{\b} \cF & \ar[l]_{\a} \cF'
}
\end{equation*}
we say it is a recollement of cosheaves, if
\begin{enumerate}
\item For each $i\in \cI$, evaluating $\a$ and $\b$ at $i$ gives a recollement in $\St^{L}$:
\begin{equation*}
\xymatrix{\cF''_{i} \ar@<3ex>[r]^{\b_{i!}} \ar@<-3ex>[r]^{\b_{i*}} &  \ar[l]_{\b_{i}}  \cF_{i} \ar@<3ex>[r]^{\a_{i}^*}  \ar@<-3ex>[r]^{\a_{i}^!}  &   \ar[l]_{\a_{i}}  \cF'_{i}}
\end{equation*}
\item For $i\le i'$ in $\cI$, the commutative diagram where the vertical maps are the transition functors for $\cF,\cF'$ and $\cF''$
\begin{equation*}
\xymatrix{\cF''_{i} \ar[d]  &  \ar[l]_{\b_{i}}  \cF_{i}  \ar[d]  &   \ar[l]_{\a_{i}}  \cF'_{i}\ar[d] \\
\cF''_{i'}  &  \ar[l]_{\b_{i'}}  \cF_{i'}    &   \ar[l]_{\a_{i'}}  \cF'_{i'}}
\end{equation*}
is a morphism of recollements.
\end{enumerate}
In this situation, Proposition~\ref{p:recoll} is applicable to conclude that
\begin{equation*}
\xymatrix{\colim_{\cI}\cF''& \ar[l]_{\colim\b} \colim_{\cI}\cF & \ar[l]_{\colim \a} \colim_{\cI}\cF'}
\end{equation*}
extends to a recollement.

\begin{remark} A caution to those who bring intuition of cosheaves with values in a stable category (the $\oo$-category $\St^L$ is itself not stable): one  needs the hypothesis 
of a recollement of cosheaves in the statement of Proposition~\ref{p:recoll}. Else a ``triangle of cosheaves" need not lead to a ``triangle of cosections". 
For a toy example of what could go wrong,
consider the diagram of cosheaves
$$
\xymatrix{
\cQ  & \ar[l]  \cF &\ar[l] \cK
}
$$
on the poset $\D_1^\opp = \{ 0 \leftarrow  \vn \to  1\}$, where each cosheaf is given by a horizontal row of the diagram
\begin{equation}\label{eq: example of cosh}
\xymatrix{
\ar[d]  0  & \ar[l]   0  \ar[d]   \ar[r] &  \Vect  \oplus 0 \ar[d]_-{i} \\
\ar[d] 0  & \ar[l] \ar[d]^-=  \Vect    \ar[r]^-{d} &   \ar[d]^-p  \Vect \oplus \Vect \\
0 & \ar[l]    \Vect     \ar[r]^-\sim  &  0 \oplus \Vect 
}
\end{equation}
Here $d$ is the diagonal embedding, $i$ is inclusion to the first factor, and $p$ is projection to the second factor.  
Note at each element of $\D_1^\opp$, the corresponding vertical sequence in \eqref{eq: example of cosh} is a recollement.
But note the  right squares are not adjointable with respect to the vertical maps, so we do not have a recollement of cosheaves.

Then we have 
\begin{equation*}
\xymatrix{
\ar@{=}[d] \colim_{\D_1^\opp} \cQ & \ar@{=}[d]\ar[l]   \colim_{\D_1^\opp} \cF   & \ar@{=}[d]\ar[l]   \colim_{\D_1^\opp} \cK  \\
0 & \ar[l] 0 & \ar[l] \Vect
}
\end{equation*}
So while $\colim_{\D_1^\opp}  \cQ $ is the cokernel of 
$ \colim_{\D_1^\opp} \cF  \leftarrow \colim_{\D_1^\opp} \cK$, we see $\cK$ is not the kernel of 
$ \colim_{\D_1^\opp} \cQ \leftarrow \colim_{\D_1^\opp} \cF$. Thus we do not obtain a recollement on cosections.
\end{remark}

\subsection{Expansion of cosheaves}
In this section, we prove the invariance of colimits of cosheaves under a certain modification called expansion.

\sss{Expansion of posets}

\begin{defn}\label{d:poset exp} Let $\cI$ be a poset and $\cI'\subset \cI$. 
\begin{enumerate}
\item We call $\cI$ a {\em weak elementary expansion} of $\cI'$ if:
\begin{itemize}
\item $\cI\setminus \cI'$ contains a unique minimal element, which we denote by $0$.
\item Let $\cI'_{>0}=\{i\in \cI'|0<i\}$. Then $\cI'_{>0}$ is up-closed in $\cI$ and acyclic (in particular non-empty).
\end{itemize}
We denote the above situation by $\cI'\nearrow\cI$ or $\cI'\nearrow_{0}\cI$ if want to indicate the minimal element in $\cI\setminus \cI'$.

\item We call $\cI$ an {\em expansion} of $\cI'$ if there is a finite sequence of {\em up-closed} subsets $\cI_{n}\subset \cI$ ($1\le  n\le N$) such that $\cI_{0}=\cI', \cI_{N}=\cI$ and $\cI_{n}$ is a weak elementary expansion of  $\cI_{n-1}$: 
\begin{equation}\label{seq elem exp}
\cI'=\cI_{0}\nearrow \cI_{1}\nearrow \cI_{2}\nearrow\cdots \nearrow \cI_{N}=\cI.
\end{equation}
\end{enumerate}
\end{defn}

Note that a weak elementary expansion may not be an expansion; it is an expansion if $\cI'$ is up-closed in $\cI$.

We give an example of a weak elementary expansion that we use in the paper. Fix $n\geq 0$, and set $[n] =  \{0, \ldots, n\}$. 
Let $\cD_n$ be the partially ordered set of {\em proper} subsets $J \sne[n] $ with respect to inclusion. 


\begin{lemma}\label{lem:contractinle ex} Fix a nonempty $J \subset [n] $. Set $\cD_n^{J}\subset \cD_n$ to be the subsets of $[n]$ {\em not} contained in $J$.  
\begin{enumerate}
\item The inclusion $\cD_{n}^{J}\subset\cD_{n}$ is a weak elementary expansion. 
\item More generally, for any poset $\cI$ with an embedding $\cD_{n}\incl \cI$ whose image is up-closed, $\cI$ is a weak elementary expansion of $\cI':=\cI\setminus (\cD_{n}\setminus\cD_{n}^{J})$. 
\end{enumerate}
\end{lemma}
\begin{proof} 
(1) Clearly $\cD_{n}^{J}$ is up-closed in $\cD_{n}$ and $\cD_{n}\setminus \cD^{J}_{n}$ has the minimal element $\vn$. We need to show that $\cD_{n}^{J}$ is acyclic. We can identify the geometric realization of the nerve $N(\cD_{n})$ with the standard $n$-simplex $\D_{n}$ with vertices $\wh 0,\cdots, \wh n$ in such a way that $J\in \cD_{n}$ corresponds to the barycenter of the face $\s^{J}$ whose vertices are $\{\wh i|i\in [n]\setminus J\}$ (in particular, $\s^{\vn}$ is  the maximal face of $\D_{n}$, and $J_{1} \subset J_{2}$ if and only if $\s^{J_{1}}\supset \s^{J_{2}}$). Under this identification, the geometric realization of $N(\cD_{n}^{J})$ is $\D_{n}^{J}:=\cup_{J'\in \cD^{J}_{n}}\s^{J'}=\cup_{J'\sne[n], J'\not\subset J}\s^{J'}$. We need to show that $\D_{n}^{J}$ is contractible.

Without loss of generality, we may assume $J = \{0,1,\cdots, j\}$, for $j\geq 1$. Let $\D_{j}\subset \D_{n}$ be the face with vertices $\wh 0,\cdots, \wh j$. Then $\D_{n}^{J}$ is the union of faces whose vertices do not contain all vertices $\wh i$ with $j+1\le i\le n$. Define a deformation retraction $\D_n \to \D_j$ by linearly extending the map on vertices given by  $\wh i \mapsto \wh i$, if $i\leq j$, and $\wh i \mapsto \wh j$, if $i> j$. Restriction of this deformation retraction to  $\D_n^J \subset \D_n$ gives a  deformation retraction $\D^J_n \to \D_j$, proving that $\D_{n}^{J}$ is contractible.

(2) Let $\vn\in \cD_{n}$ be the unique minimal element. Then $\cI'_{>\vn}=\cD_{n}^{J}$ is up-closed in $\cD_{n}$ hence also in $\cI$ by assumption, and is acyclic by (1).

%
%
\end{proof}

\sss{Expansion of cosheaves}
\begin{defn}\label{d:cosh exp} Let $\cI$ be a poset and $\cI'\subset \cI$. Let $\ph: \cE\to\cF$ be a map of cosheaves on $\cI$. 
\begin{enumerate}
\item If $\cI'\nearrow_{0} \cI$, we say that $\ph$ is a {\em weak elementary expansion} along $\cI'\nearrow_{0} \cI$ if
\begin{itemize}
\item $\ph|_{\cI'}: \cE|_{\cI'}\to \cF|_{\cI'}$ is an equivalence.
\item For $i\in \cI^{\new}=\cI\setminus \cI'$, the following square is a pushout
\begin{equation*}
\xymatrix{ \cE_{0}\ar[d] \ar[r]^{\ph_{0}} & \cF_{0}\ar[d]\\
\cE_{i}\ar[r]^{\ph_{i}}& \cF_{i}}
\end{equation*}
\end{itemize}
\item Suppose $\cI$ is an expansion of $\cI'$. We call $\ph$ an {\em expansion} along $\cI'\subset \cI$ if there exists a sequence of elementary expansions as in \eqref{seq elem exp}, where $\cI_{n}$ are up-closed, such that:
\begin{itemize}
\item $\ph|_{\cI'}: \cE|_{\cI'}\to \cF|_{\cI'}$ is an equivalence.
\item For each $1\le n\le N$, let $0_{n}\in \cI_{n}\setminus \cI_{n-1}$ be the unique minimal element. Then for any $i\in \cI_{n}\setminus \cI_{n-1}$, the following square is a pushout
\begin{equation*}
\xymatrix{ \cE_{0_{n}}\ar[d] \ar[r]^{\ph_{0_{n}}} & \cF_{0_{n}}\ar[d]\\
\cE_{i}\ar[r]^{\ph_{i}}& \cF_{i}}
\end{equation*}
\end{itemize}
\end{enumerate}
\end{defn}

The key property of an expansion of cosheaves is the invariance of colimit.

\begin{prop}\label{p:pushout}
\begin{enumerate}
\item Let $\cI'\nearrow_{0}\cI$ be a weak elementary expansion of posets, and let $\ph:\cE\to \cF$ be a weak elementary expansion of cosheaves on $\cI$ along $\cI'\nearrow_{0}\cI$. Then $\ph$ induces an equivalence on colimits
\begin{equation}\label{colim unchanged}
\colim\ph: \colim_{\cI}\cE\isom \colim_{\cI}\cF.
\end{equation}
\item The same holds if $\cI'\subset \cI$ is an expansion and $\ph$ is an expansion of cosheaves on $\cI$ along $\cI'\subset \cI$.
\end{enumerate}
\end{prop}
\begin{proof}
(1) Let $\cI^{\old}=\cI'$ and $\cI^{\new}=\cI\setminus \cI^{\old}$. Let $\cI^{+}=\cI\sqcup \{0'\}$, and extend the partial order from $\cI$ to $\cI^{+}$ by adding the relations $0<0'<i$ for all $i\in \cI^{\old}_{>0}$.  

We first claim that the map of posets $\cI\to \cI^{+}$ is cofinal. Indeed, we apply Quillen's Theorem A  to the map $i$, and note that for any $i\in \cI^{+}$, $\cI\times_{\cI^{+}}\cI^{+}_{\ge i}$ either has a unique minimal element $i$ if $i\in \cI$, or when $i=0'$, $\cI\times_{\cI^{+}}\cI^{+}_{\ge0'}=\{i\in \cI^{\old}|0\le i\}=\cI^{\old}_{>0}$ which is acyclic by assumption. Therefore the assumption for applying Quillen's Theorem A is satisfied. 

Extend $\cE$ to a cosheaf $\cE^{+}$ on $\cI^{+}$ by assigning $\cE^{+}_{0'}:=\cF_{0}$ and the functors $\cE^{+}_{0'}\to \cE^{+}_{i}$ for $i\in \cI^{\old}$ are given by the composition $\cF_{0}\to \cF_{i}\xr{\ph_{i}^{-1}}\cE_{i}$. The cofinality of $\cI$ in $\cI^{+}$ implies
\begin{equation}\label{colim+}
\colim_{\cI} \cE \isom\colim_{\cI^{+}} \cE^{+}.
\end{equation}

Next we would like to apply Proposition \ref{p:double colim} to write $\colim_{\cI^{+}} \cE^{+}$ as a double colimit where the outer colimit is parametrized by $\cI$ again. We apply it to $\cK=\cI^{+}$, $\cJ=\cI$ and the following assignment $\cI\ni i\mapsto \cI^{+}_{i}$: if $0\le i$, let $\cI^{+}_{i}=\{j\in\cI|j\le i\}\cup \{0'\}$; otherwise let $\cI^{+}_{i}=\{j\in\cI|j\le i\}$. We check that the acyclicity condition holds. Let $k\in \cI^{+}$. If $k\in \cI$, then $\{i\in \cI|k\in \cI^{+}_{i}\}$ contains $k$ as the unique minimal element hence is acyclic. If $k=0'$, then $\{i\in \cI|0'\in \cI^{+}_{i}\}$ contains $0'$ as the unique minimal element, hence again is acyclic. Proposition \ref{p:double colim} gives  a natural equivalence
\begin{equation}\label{double colim}
\colim_{i\in \cI}(\colim_{\cI^{+}_{i}}\cE^{+})\simeq \colim_{\cI^{+}}\cE^{+}.
\end{equation}
For each $i\in \cI^{\old}$, $\cI^{+}_{i}$ has a unique maximal element $i$, hence $\colim_{\cI^{+}_{i}}\cE^{+}\simeq \cE^{+}_{i}=\cE_{i}\simeq \cF_{i}$. For $i\in \cI^{\new}$, $\cI^{+}_{i}$ does not intersect $\cI^{\old}_{>0}$ since $\cI^{\old}_{>0}$ is up-closed. This implies $\{0,0', i\}$ is cofinal in $\cI^{+}_{i}$, hence $\colim_{\cI^{+}_{i}}\cE^{+}\simeq \cE_{0'}\coprod_{\cE_{0}}\cE_{i}\simeq \cF_{i}$ by assumption. Combining these calculations we see that the right side of \eqref{double colim} is canonically equivalent to $\colim_{\cI}\cF$. Combined with \eqref{colim+} we conclude that \eqref{colim unchanged} holds.

(2) We would like to reduce to the situation of (1). Let $\cI'=\cI_{0}\nearrow_{0_{1}} \cI_{1}\nearrow_{0_{2}} \cI_{2}\nearrow\cdots \nearrow _{0_{n}}\cI_{N}=\cI$ be a sequence of elementary expansions, with each $\cI_{n}$ up-closed in $\cI$ such that $\ph$ satisfies the conditions in Definition \ref{d:cosh exp}(2). For $1\le n\le N$, let $\cE_{n}$ be the cosheaf on $\cI$ that is $\cF|_{\cI_{n}}$ on $\cI_{n}$ and is $\cE|_{\cI\setminus \cI_{n}}$ on $\cI\setminus \cI_{n}$. For $\cI\setminus \cI_{n}\ni i\le i'\in \cI_{n}$, the transition functor $(\cE_{n})_{i}=\cE_{i}\to (\cE_{n})_{i'}=\cF_{i'}$ is defined to be the composition $\cE_{i}\xr{\ph_{i}}\cF_{i}\to\cF_{i'}$. Then $\cE_{0}=\cE$ and $\cE_{N}=\cF$ and there is a natural map $\ph_{n}: \cE_{n-1}\to \cE_{n}$. It suffices to show that $\ph_{n}$ induces an equivalence on colimits. 

Let $\cI^{\new}_{n}=\cI_{n}\setminus \cI_{n-1}$ and $\cI^{\old}=\cI\setminus \cI^{\new}_{n}$. We claim that $\cI$ is an elementary expansion of $\cI^{\old}$. Indeed, $\cI^{\new}_{n}$ has a unique minimal element $0_{n}$ by assumption, and $\cI^{\old}_{>0_{n}}=(\cI_{n-1})_{>0_{n}}$ because $\cI_{n-1}$ is up-closed in $\cI$. By construction, $\ph_{n}:\cE_{n-1}\to \cE_{n}$ is a weak elementary expansion along $\cI^{\old}\nearrow\cI$ in the sense of Definition \ref{d:cosh exp}(1). Therefore we reduce to the case of (1).
\end{proof}

\subsection{A categorical contraction principle}
If $Y$ is a manifold with a flow $\{\Phi_{t}\}_{t>0}$ that contracts $Y$ to a subspace $Z\subset Y$ as $t\to 0$, (i.e., the flow $\{\Phi_{t}\}_{t>0}$, viewed as an $\RR_{>0}$-action, extends to an action $\{\Phi_{t}\}_{t\geq 0}$ of the monoid $\RR_{\geq 0}$ with $Z = \Phi_0(Y)$) and if $\cF$ is a constructible sheaf on $Y$ locally-constant (hence constant) on flow lines,  then the restriction map gives an isomorphism $R\G(Y,\cF)\cong R\G(Z,\cF|_Z)$. This is usually called the {\em contraction principle} for sheaves.  Below we formulate and prove a categorical analog of this principle.


\begin{theorem}\label{th:contracting cosheaf}
	Let $\cJ \subset \cI$ be a subposet. Denote $s : \cJ \to \cI$ the inclusion and $\cI_0 = \cI \backslash \cJ$. Let $\cF: \cI \to \St^L$ be a functor, denote by $\cL : \cI_0 \to \St^L$ the cokernel of the natural map $(s_! s^* \cF)|_{\cI_0} \to \cF|_{\cI_0}$. Assume that 
	\begin{enumerate}
		\item  The maps of cosheaves on $\cI_{0}$
		\begin{equation*}
			\xymatrix{\cL & \ar[l]_-{\b} \cF|_{\cI_{0}} & \ar[l]_-{\a} (s_!s^{*}\cF)|_{\cI_{0}} 
			}
		\end{equation*}
	is a recollement of cosheaves in the sense of Section~\ref{sss:cosh rec}.
		\item $\cL$ is a locally constant cosheaf on $\cI_0$.
		\item The inclusion $|s|: |\cJ|\incl |\cI|$ is a homotopy equivalence.
	\end{enumerate}
Then the natural map is an equivalence:
$$\colim_{\cJ}s^{*}\cF\simeq \colim_{\cI}\cF.$$
\end{theorem}

Before proving the theorem, we give a situation where it is easy to compute the left Kan extension $s_{!}s^{*}\cF$ of $s^{*}\cF$. Let $f: \cI\to \cJ$ be a coCartesian map of finite posets such that each fiber $f^{-1}(j)$ contains a unique minimal element $s(j)$. Then $s: \cJ\to \cI$ is a section to $f$. Let $\cI_{0}=\cI\setminus s(\cJ)$. Let $\cF$ be a cosheaf on $\cI$. 

\begin{lemma}\label{l:pull is s!} The pullback $f^{*}s^{*}\cF$ is the left Kan extension $s_{!}(s^{*}\cF)$ of $s^{*}\cF$ along $s: \cJ\to \cI$. 
\end{lemma}
\begin{proof}
For each $i\in \cI$, $\{j\in \cJ|s(j)\le i\}$ has a unique maximal element $f(i)$. Therefore the value of $s_{!}(s^{*}\cF)$ at $i$ is $(s^{*}\cF)_{f(i)}=\cF_{s(f(i))}=(f^{*}s^{*}\cF)_{i}$.
\end{proof}


Since the map $f$ is coCartesian, $f_{!}\cF$ is equivalent to $\int_{f}\cF$ by Proposition~\ref{p:push Kan}.  We get the following corollary from Theorem \ref{th:contracting cosheaf}.



\begin{cor}\label{c:contracting cosheaf} Under the above notations, assume the conditions in Theorem \ref{th:contracting cosheaf} hold. Then the natural maps are equivalences
$$\colim_{\cJ}s^{*}\cF\simeq \colim_{\cI}\cF\simeq \colim_{\cI}\int_{f}\cF.$$
\end{cor}
%

The proof of Theorem \ref{th:contracting cosheaf} will be given after recalling some facts about simplicial complexes.

\sss{Simplicial complexes and exit poset}
For a poset $\cI$ we have a simplicial complex  $|\cI|$. For a simplicial complex $Y$ we define its {\em exit poset} $\frP(Y)$ to be the set of (closed) faces $\s$ of $Y$ with the partial order given by inclusions. From the definition, faces in $|\cI|$ are parametrized by non-degenerate simplices in $N(\cI)$, therefore we get a canonical isomorphism of posets
\begin{equation}\label{sdI}
sd(\cI)\cong \frP(|\cI|)^{\opp}.
\end{equation}

For any subset $\cI'\subset \cI$ with the inherited poset structure, $|\cI'|\subset |\cI|$ is a closed subcomplex. The complement $|\cI|\setminus|\cI'|$ is the open star $\str^{\c}(|\cI\setminus \cI'|)$.

For a simplicial complex $Y$, let $sd(Y)$ be its barycentric subdivision. Then $sd(|\cI|)$ can be identified with $|sd(\cI)|$ as simplicial complexes.

For a simplicial complex $X$, by {\em a cosheaf} $\cF$ on $X$ we mean a cosheaf $\cF$ on $\frP(X)^{\opp}$, i.e., a functor $\cF: \frP(X)^{\opp}\to \St^{L}$ (equivalently, by passing to right adjoints, it is the same datum as a functor $\frP(X)\to \St^{R}$). We will write $\int_{X}\cF$ for $\colim_{\frP(X)^{\opp}}\cF$. We write $\cF_{Z}=\cF|_{\frP(Z)^{\opp}}$ for the restriction of a cosheaf to a subcomplex $Z\subset Y$.

Let $\cF$ be a cosheaf on a poset $\cI$. By \eqref{sdI}, $sd(\cF)$ is a cosheaf on $\frP(|\cI|)^{\opp}$, hence a cosheaf on $|\cI|$. By Lemma \ref{l:sd}, there is a canonical equivalence
\begin{equation*}
\int_{|\cI|}sd(\cF)\simeq \colim_{\cI}\cF.
\end{equation*}

\sss{Simplicial collapse}
Recall a simplex $\tau \subset K$ in a simplicial complex $K$ is called a {\em free face} if there exists a unique maximal simplex $\sigma \subset K$ such that $\tau\sne \sigma$. In this case,  the {\em simplicial collapse} of $K$ along the free face $\tau \subset K$ is defined to be the subcomplex $K' = K \setminus (\cup_{ \tau \subset\gamma \subset \sigma} \gamma^{\c}) \subset K$. Note the pair $K'  \subset K$  implicitly knows the free face $\tau$ as the unique minimal face of $K$ removed to obtain $K'$. 
 Given general  simplicial complexes $K_0\subset K$, by a {\em simplicial collapse} of $K$ to $K_0$, we will mean a  finite sequence of subcomplexes $K_0 \subset \cdots \subset  K_{n-1} \subset K_{n} \subset \cdots \subset K_N =  K$ such that  $K_{n-1}$ results from $K_n$ via a simplicial collapse along a free face  $\tau_n \subset K_n$.
   
\begin{lemma}[Whitehead {\cite[Corollary on page 252]{whiteheadSimplicialSpacesNuclei1939}}]\label{l:collapse}
   For an inclusion of finite simplicial complexes   $Z_0\subset Y_0$,
   after the second barycentric subdivision $Z = sd^2(Z_0) \subset sd^2(Y_0) = Y$, there is always a simplicial collapse of 
  the closed star   $U = \str(  Z) \subset  Y$ to $Z$.
 \end{lemma}
   
   \begin{proof}
   We give a sketch here and refer to \cite{whiteheadSimplicialSpacesNuclei1939} for more details.
   
   After the first barycentric subdivision, $Z_1 = sd(Y_0)$ will be a {\em full subcomplex} of $Y_1 = sd(Y_0)$, i.e.~if the vertices of a simplex of $Y_1$ lie in $Z_1$, then the simplex itself lies in $Z_1$. To see this, recall the $k$-simplices $\s\subset Y_1$ are given by $k$-flags of simplices $\s_0 \subset \s_1 \subset \cdots \subset \s_k \subset Y_0$. 
 The vertices of $\s$  are the  $0$-flags $\s_0,  \s_1, \ldots, \s_k \subset Y_0$ that occur in the flag. If these $0$-flags  satisfy $\s_0,  \s_1, \ldots, \s_k \subset Z_1$,  then the simplices satisfy  $\s_0,  \s_1, \ldots, \s_k \subset Z_0$. 
  Thus the $k$-flag representing $\s$ indeed satisfies  $\s_0 \subset \s_1 \subset \cdots \subset \s_k \subset Z_0$, and hence $\s\subset Z_1$. 
 
 Now starting from the full subcomplex  $Z_1 \subset Y_1$,
 after a barycentric subdivision 
  $Z = sd(Z_1) \subset sd(Y_1) = Y$, we can always find a simplicial collapse 
  of  the closed star  $U = \str(  Z) \subset  Y$ to $Z$ as follows. 
  Such a simplicial collapse will be determined by  a total order on the simplices of $U \setminus Z$. We can take any total order compatible with the following partial order. 
  
  First, consider     
   the prior open star $U_1^\circ = \str^\circ(  Z_1) \subset  Y_1$, and the complement  $U_1^\circ \setminus Z_1$.
    Note every relatively open simplex $\s^\circ \subset U \setminus Z$  lies in a unique relatively open  simplex $\hat \s^\circ \subset U_1^\circ \setminus Z_1$.
    
   Partially order the simplices of $U_1^\circ \setminus Z_1$ opposite to their natural order by closure inclusion; so the partial order begins with the maximal simplices of $U_1^\circ \setminus Z_1$ and proceeds  to the minimal. 
  (Note $U_1^\circ \setminus Z_1$ has no vertices -- its simplex closures must contain vertices from both $Z_1$ and its complement -- so its minimal simplices will at least be edges).

Partially  order the simplices $\s^\circ \subset U \setminus Z$  by the above partial order on the simplices $\hat\s^\circ\subset  U_1^\circ \setminus Z_1$ containing them.
  From here, it suffices to  independently define for each $\tau^\circ \subset U_1^\circ \setminus Z_1$  with partial closure $\tau =\ol { \tau^\circ} \subset U_1^\circ$, 
    a simplicial collapse of the subcomplex $\tau \cap U$ to the subcomplex $\partial \tau \cap U$.  
    
    Thus it remains to solve the following general local problem: given a closed simplex $\tau_1$ and a proper closed facet $\rho_1 \subset \tau_1$, consider the barycentric subdivision $\tau = sd(\tau_1)$ and its subcomplex $ \rho = sd(\rho_1) \subset \tau$. Let $U = \str(\rho) \subset \tau$ be the closed star of $\rho \subset \tau$. Then we need a simplicial collapse of $U$ to $U \cap \partial \tau$. To achieve this, we can take any total order on the simplices of $U \setminus (U \cap \partial \tau)$ compatible with the following partial order. First, partially order the simplices of $U \setminus (U \cap \partial \tau)$ opposite to their natural order by closure inclusion; so the partial order begins with the maximal simplices of $U \setminus (U \cap \partial \tau)$ and proceeds  to the minimal. Then refine this partial ordering, by further ordering the simplices of $U \setminus (U \cap \partial \tau)$ starting with those with the maximal number of vertices not in $\rho$ and ending with
    those with the minimal number of vertices not in $\rho$. It is elementary to check this indeed gives 
    a simplicial collapse of $U$ to $U \cap \partial \tau$.
\end{proof}

\sss{Proof of Theorem \ref{th:contracting cosheaf}}
Let $\cE=s_{!}s^{*}\cF$. We need to show that the adjunction map $\a: \cE\to \cF$ induces an equivalence $\colim_{\cI}\cE\simeq \colim_{\cI}\cF$.

Let $Y_{0}=|\cI|$ and $Z_{0}=|\cI_{1}|=|s(\cJ)|$, which are finite simplicial complexes by assumption, and $Z_{0}\subset Y_{0}$ is a closed subcomplex. We denote $Y=sd^2(Y_0)$ and $Z = sd^2(Z_0)$. Then $sd(\cF)$ and $sd(\cE)$ define cosheaves on $Y_{0}$ and $Z_{0}$. We abuse the notation to denote the pullback of $sd(\cF),sd(\cE)$ to $Y$ and $Z$ still by $\cF$ and $\cE$. By Lemma \ref{l:sd}, it suffices to show the natural map $\int_{Y}\cE\to \int_{Y}\cF$ is an equivalence. 

The locally constant cosheaf $\cL$ on $\cI_{0}$ defines a locally constant cosheaf $sd(\cL)$ on $sd(\cI)\setminus sd(s(\cJ))$ via pullback by the first vertex map $sd(\cI)\setminus sd(s(\cJ))\to \cI\setminus s(\cJ)=\cI_{0}$. In other words, $\cL$ pullbacks to a locally constant cosheaf $sd(\cL)$ on $\frP(Y_{0})^{\opp}\setminus \frP(Z_{0})^{\opp}$. Iterating this procedure, we get a locally constant cosheaf $sd^{3}(\cL)$ on $\frP(Y)^{\opp}\setminus \frP(Z)^{\opp}$. For simplicity we still denote it by $\cL$. For a cosheaf $\cF'$ on $Y$, we denote its restriction to $\frP(Y)^{\opp}\setminus \frP(Z)^{\opp}$ by $\cF'_{Y\setminus Z}$.  The assumptions on the original $\cF$ imply the maps $\a$ and $\b$ restricted to $Y\setminus Z$ extend to a recollement of cosheaves on $\frP(Y)^{\opp}\setminus \frP(Z)^{\opp}$: 
\begin{equation*}
\xymatrix{\cL \ar@<2ex>[r]^{} \ar@<-2ex>[r]^{} &  \ar[l]_-{\b}  \cF_{Y\setminus Z} \ar@<2ex>[r]^{}  \ar@<-2ex>[r]^{}  &   \ar[l]_-{\a}  \cE_{Y\setminus Z}}
\end{equation*}

Let $U^\circ  = \str^\circ(  Z) \subset  Y$ be the open star of $Z$, and $U = \str(  Z) \subset  Y$  the closed star of $Z$. Set $V =  Y \setminus U^\circ$ and $W = U \cap V$.   By Corollary \ref{c:MV}, we have
\begin{equation}\label{MV for F}
\int_{U}\cF_{U}\coprod_{\int_{W}\cF_{W}}\int_{V}\cF_{V}\isom \int_{Y}\cF.
\end{equation}
The same applies to $\cE$ in place of $\cF$. We will prove the following claims:
\begin{enumerate}
\item The diagram
\begin{equation}\label{pushout W to V}
\xymatrix{ \int_{W}\cE_{W}\ar[r] \ar[d]& \int_{W}\cF_{W}\ar[d]\\
\int_{V}\cE_{V}\ar[r] & \int_{V}\cF_{V}}
\end{equation}
is a pushout square.
\item The natural map $\int_{U}\cE_{U}\to \int_{U}\cF_{U}$ is an equivalence.
\end{enumerate}
We first show that these two claims imply the proposition. Let $\cD_{1}$ be the poset of proper subsets of $\{0,1\}$. We define a cosheaf $\cA$ on $\cD_{1}$ by
\begin{equation*}
\cA_{\vn}=\int_{W}\cE_{W}, \quad \cA_{\{1\}}=\int_{U}\cE_{U}, \quad \cA_{\{0\}}=\int_{V}\cE_{V}.
\end{equation*}
Define another cosheaf $\cA'$ on $\cD_{1}$ by 
\begin{equation*}
\cA'_{\vn}=\int_{W}\cF_{W}, \quad \cA'_{\{1\}}=\int_{U}\cF_{U}, \quad \cA'_{\{0\}}=\int_{V}\cF_{V}.
\end{equation*}
Then $\a$ induces a map of cosheaves $\cA\to \cA'$. The inclusion $\{\{0\}\}\subset \cD_{1}$ is a weak elementary expansion in the sense of Definition \ref{d:poset exp}(1), and $\cA\to \cA'$ is a weak elementary expansion along $\{\{0\}\}\nearrow \cD_{1}$. Then Proposition \ref{p:pushout}(1) applies to the situation to conclude that $\colim_{\cD_{1}}\cA\isom\colim_{\cD_{1}}\cA'$. By \eqref{MV for F} and its analog for $\cE$, we conclude that $\int_{Y}\cE\isom\int_{Y}\cF$.

We prove (1). By the second assumption, over $V$, the termwise recollement 
\begin{equation}\label{ab rec}
\xymatrix{\cL_{V} & \ar[l]_{\b_{V}} \cF_{V} & \ar[l]_{\a_{V}} \cE_{V}}
\end{equation}
extends to a recollement of cosheaves, i.e., all six-term diagrams relating the values of \eqref{ab rec} at two faces $\s\subset \t$ in $V$ are morphisms of recollements. By Proposition~\ref{p:recoll}(1), we conclude that the colimits also fit into  recollements
\begin{equation*}
\xymatrix{\int_{V}\cL_{V}\ar@<2ex>[r]\ar@<-2ex>[r] & \ar[l]_{\b_{V}} \int_{V}\cF_{V} \ar@<2ex>[r]\ar@<-2ex>[r] & \ar[l]_{\a_{V}} \int_{V}\cE_{V}}
\end{equation*}
The same holds for $W$ in place of $W$, so we have a diagram where the two rows extend to recollements
\begin{equation}\label{int W to V rec}
\xymatrix{\int_{W}\cL_{W}\ar[d]& \ar[l]_{\b_{W}} \int_{W}\cF_{W}\ar[d]  & \ar[l]_{\a_{W}} \int_{W}\cE_{W}\ar[d]\\
\int_{V}\cL_{V} & \ar[l]_{\b_{V}} \int_{V}\cF_{V}  & \ar[l]_{\a_{V}} \int_{V}\cE_{V}}
\end{equation}
Moreover, by Proposition~\ref{p:recoll}(2) this is a morphism of recollements.
Now consider the map
\begin{equation}\label{int W to V}
\int_{W}\cL_{W}\to \int_{V}\cL_{V}.
\end{equation}
We claim that $W\incl V$ is a homotopy equivalence. Indeed, since  the inclusions $Z\subset Y$,  $Z\subset U$ are homotopy equivalences, the inclusion
$U\subset Y$ is as well. Therefore there is a deformation retraction of $Y$ to $U$. (See for example~\cite[Chapter 0]{hatcherAlgebraicTopology2002}.)  Removing  $U^\circ$ then gives a deformation retraction 
 of $V= Y \setminus  U ^\circ$ to $W = U \setminus  U ^\circ$; in particular, the inclusion
 $W\incl V$ is a homotopy equivalence. Since $\cL_{V}$ is locally constant, we can then apply Proposition \ref{p:loc} to conclude that \eqref{int W to V} is an equivalence. Then by Corollary \ref{c:coprod} applied to the morphism of recollements \eqref{int W to V rec}, we conclude that \eqref{pushout W to V} is a pushout diagram.
 
We prove (2). By Lemma \ref{l:collapse}, we have a sequence of subcomplexes $Z=K_{0}\subset K_{1}\subset\cdots K_{N}=U$ such that $K_{n-1}$ is a simplicial collapse of $K_{n}$ along a free face $\t_{n}$, $1\le n\le N$. We claim $\frP(K_{n})^{\opp}$ is a weak elementary expansion of $\frP(K_{n-1})^{\opp}$ with minimal element in $\frP(K_{n})^{\opp}\setminus \frP(K_{n-1})^{\opp}$ given by the unique maximal face $\s_{n}$ in $K_{n}$ containing $\t_{n}$. Indeed, we use notations from Lemma \ref{lem:contractinle ex}. Let $d=\dim \s_{n}$, we can identify $\frP(\s_{n})$ with $\cD_{d}$ after choosing an identification of the vertices of $\s_{n}$ with $[d]$, and then $\frP(K_{n-1}\cap \s_{n})^{\opp}$ is identified with $\cD_{d}^{J}$ for some proper subset $J\subset [d]$ corresponding to $\t_{n}$. Then the partition $\frP(K_{n})^{\opp}=\frP(K_{n-1})^{\opp}\sqcup (\cD_{d}\setminus \cD^{J}_{d})$ satisfies the conditions in Lemma \ref{lem:contractinle ex}(2), which verifies that $\frP(K_{n-1})^{\opp}\nearrow_{\s_{n}}\frP(K_{n})^{\opp}$. Therefore $\frP(Z)^{\opp}\subset \frP(U)^{\opp}$ is an expansion in the sense of Definition \ref{d:poset exp}.

We claim that $\a_{U}: \cE_{U}\to \cF_{U}$ is an expansion of cosheaves along $\frP(Z)^{\opp}\subset \frP(U)^{\opp}$ in the sense of Definition \ref{d:cosh exp}. Indeed,  by construction, $\a_{U}$ is an equivalence when restricted on $Z$. For any two faces $\t\subset \s$ of $U$ not contained in $Z$,  we have a morphism of recollements 
\begin{equation*}
\xymatrix{ \cL_{\s}\ar[d] & \ar[l]_-{\b_{\s}}\cF_{\s} \ar[d] & \ar[l]_-{\a_{\s}}\cE_{\s}\ar[d]\\
\cL_{\t} & \ar[l]_-{\b_{\t}}\cF_{\t} & \ar[l]_-{\a_{\t}}\cE_{\t}
}
\end{equation*}
where the left vertical map is an equivalence since $\cL$ is locally constant. By Corollary \ref{c:coprod}, the right square above is a pushout square. This verifies the second condition for an expansion of cosheaves. Claim (2) now follows from Proposition \ref{p:pushout}. \qed

\section{Functoriality of sheaves with Lagrangian singular support}

Let $k$ be a field of characteristic $0$, and let $X$ be a smooth algebraic stack over $\CC$ of finite type, $\cL \subset T^*X$ be a closed ($\CC^{\times}$-)conic Lagrangian substack. Denote $\Sh(X)$ the $\infty$-category of sheaves of $k$-modules on $X$, and $Sh_\cL(X) \subset Sh(X)$ the full subcategory of sheaves with singular support contained in $\cL$. It is known that $Sh_\cL(X)$ is compactly generated \cite[Cor.~G.7.8]{arinkinStackLocalSystems}, $Sh(X)$ is presentable, and the inclusion $Sh_\cL(X) \to Sh(X)$ preserves both limits and colimits \cite[Cor.~G.7.5]{arinkinStackLocalSystems}  
The goal of this section is to prove the following:

\begin{prop}\label{p:cont 4 functors} 
	 Let $f: X \to Y$ be a representable map between finite type smooth algebraic stacks over $\CC$.
	\begin{enumerate}
		\item Let $\calL_Y \subset T^*Y$ be a closed conic Lagrangian substack, then the functors $f^*,f^!: Sh_{\cL_Y}(Y) \to Sh(X)$ preserve both limits and colimits.
			\item Let $\calL_X \subset T^*X$ be a closed conic Lagrangian substack, then the functors $f_!,f_*: Sh_{\cL_X}(X) \to Sh(Y)$ preserve both limits and colimits.
	\end{enumerate}
\end{prop}	
\begin{proof}
(1) We show the statement for $f^*$, the one for $f^!$ follows from similar argument. Since $f^{*}: \Sh(Y)\to \Sh(X)$ is a left adjoint, it preserves colimits; hence $f^{*}|_{\Sh_{\cL_{Y}}(Y)}$ also preserves colimits since the inclusion $\io_{\cL_{Y}}: \Sh_{\cL_{Y}}(Y)\incl \Sh(Y)$ does.

Let $\cF: I\to \Sh_{\cL_{Y}}(Y)$ be a diagram of sheaves. We have the natural map
\begin{equation}\label{f*lim}
\phi: f^{*}(\lim_{i\in I} \cF_{i})\to \lim_{i\in I}f^{*}\cF_{i}
\end{equation}
in $\Sh(X)$. To show it is an equivalence, it suffices to show the same after pulling back to a smooth cover $p_{X}: \wt X\to X$, since $p^{*}_{X}$ is conservative. Let $p_{Y}: \wt Y\to Y$ be a smooth surjective map from a scheme $\wt Y$, and form the Cartesian diagram
\begin{equation*}
\xymatrix{ \wt X\ar[d]^{p_{X}}\ar[r]^{\wt f} & \wt Y\ar[d]^{p_{Y}}\\
X\ar[r]^{f} & Y}
\end{equation*}
Then $p^{*}_{X}\phi$ can be factored as
\begin{equation*}
p^{*}_{X}f^{*}(\lim \cF_{i})= \wt f^{*}p_{Y}^{*}(\lim \cF_{i}) \simeq \wt f^{*}\lim(p_{Y}^{*}\cF_{i}) \to\lim\wt f^{*}p_{Y}^{*}\cF_{i}=  \lim p^{*}_{X}f^{*}\cF_{i}\simeq p^{*}_{X}\lim f^{*}\cF_{i}.
\end{equation*}
Here we use that $p^{*}_{X}$ and $p^{*}_{Y}$ preserves limits (being smooth, they agree with $p^{!}_{X}$ and $p^{!}_{Y}$ up to a shift). To show the above composition is an equivalence, it suffices to show the middle arrow is an equivalence, i.e., $\wt f^{*}: \Sh_{\cL_{\wt Y}}(\wt Y)\to \Sh(\wt X)$ preserves limits (here $\cL_{\wt Y}\subset T^{*}\wt Y$ is the transport of $\cL_{Y}$ under the Lagrangian correspondence between $T^{*}\wt Y$ and $T^{*}Y$). 


Henceforth we assume both $X,Y$ are smooth schemes. For any closed conic Lagrangian $\cL_Y \subset T^*Y$, there is a Whitney stratification $\{{Y_\alpha}\}_{\alpha \in S}$ of $Y$, so that $\cL_{Y}  \subset \cup_{\alpha} T^*_{X_\alpha} X$. So suffices to prove the claim for $\cL_Y$ is conormal to a Whitney stratification. Now pick a Whitney stratification $\{{X_{\b}}\}_{\b\in T}$ of $X$, so that the images of $f^*$ land in $Sh_\cM(X) \subset Sh(X)$, for $\cM= \cup_{\b} T^*_{X_{\b}}X$. Now let $i_x: {x} \to X$ be the inclusion of a point. By \cite[Lemma A.1.11]{nadlerAutomorphicGluingFunctor}, the functor $i^*_x: Sh_\cM(X) \to k\lmod$ preserves limits. Now let $\cF: I\to \Sh_{\cL_{Y}}(Y)$ be a diagram of sheaves in $Sh_{\cL_Y}(Y)$, we need to show the natural map $\phi$ in \eqref{f*lim} is an equivalence. Equivalently, we need to check that $\phi$ induces an equivalence on each stalk. Taking stalks at $x\in X$,  $\phi_x$ can be identified with the composition of equivalences
 \begin{equation*} 
\phi_{x}:  (f^*(\lim \cF_i))_x \simeq (\lim \cF_i)_{f(x)} \simeq \lim (\cF_i)_{f(x)} \simeq \lim f^*(\cF_i)_x \simeq (\lim f^*(\cF_i))_x  
 \end{equation*}
Here the second (resp. last) equivalence uses the fact that $i^{*}_{f(x)}: Sh_{\cL_Y}(Y) \to k\lmod$ (resp. $i_{x}: \Sh_{\cM}(X)\to k\lmod$) preserves limits. This show $\phi$ is an equivalence.
 
 (2) We only prove that $f_!$ preserves limits, and the case of $f_{*}$ follows from a similar argument. By the same reduction using smooth covers as in (1),  we can assume $X,Y$ are smooth schemes. We can further assume $\cL_X$ is conormal to a Whitney stratification $X=\cup_{\a\in S}X_{\alpha}$. Now pick a relative compactification $\overline{X}$ of $f$, i.e, $f $ factors as $X \xrightarrow{j} \overline{X} \xrightarrow{p} Y$, where $j$ is an open embedding, and $p$ is proper. Now suffices to show that $j_!: Sh_{X_\alpha}(X) \to Sh(\overline{X})$ preserves limits, since $p_!=p_*$ preserves limits. Let $\ov X=\cup_{\a'\in S'}\overline{X}_{\alpha'}$ be a Whitney stratification on $\overline{X}$ that refines the partition $\overline{X}= (\cup X_{\alpha}) \cup (\overline{X} \backslash X)$. Define $\cL'_{X}$ and $\cL'_{\ov X}$ to be the conormals to the new stratifications of $X$ and $\ov X$. Now suffices to show that $j_!: Sh_{\cL'_{X}} (X) \to Sh_{\cL'_{\ov X}} (\overline{X})$ preserves limits. Let $\cF:I\to \Sh_{\cL'_{X}}(X)$ be a diagram of sheaves. Consider the natural map $\varphi: j_!(\lim \cF_i) \to \lim j_!(\cF_i)$. We only need to check that $\ph$ is an equivalence on each stalk. For $x \in X$, it is clear that $\varphi_x$ is an isomorphism. For $x \in \overline{X} \setminus X$, we have $(j_!(\lim \cF_i))_x=0$ and  $(\lim j_!(\cF_i))_x=  \lim (j_!(\cF_i))_x=0$ since $i_{x}^{*}$ preserves limits. Therefore $\varphi_x$ is also isomorphism, and hence $\varphi$ is an isomorphism.

\end{proof}

 
\section{Universal version of Tao-Travkin Theorem}\label{app:TT}

The method of  \cite{taoAffineHeckeCategory} can be applied to deduce the following theorem:
\begin{theorem}\label{thm:colimit_of_hecke}
	Let $\cG^{\c}$ be the neutral component of the loop group $\cG=G\lr{t}$ of a connected reductive group $G$ over $\CC$.  Let $\cH_{\cG^{\c}}\subset \cH_{\cG}$ be the full monoidal subcategory of the universal affine Hecke category consisting of sheaves supported on $\cI^{u}\bs \cG^{\c}/\cI^{u}$.
	Then the natural maps induce an equivalence of stable presentable monoidal categories:
	\begin{equation*}
	\xymatrix{
	\colim^{\ot}_{J \sft I^a} \cH_{L_J} \ar[r]^-\sim  & \cH_{\cG^{\c}}
	}
	\end{equation*}
	where the colimit is of stable presentable monoidal categories.
	\end{theorem}

	In this section, we shall adopt the notations in \textit{loc.cit.}. 
	Denote by $\scI$ the \textit{convolution Schubert 1-category} for the affine Weyl group $W^{a}$, which is denoted $\textup{Word}^{1}_{\mathrm{fr}}$ and defined in \cite[Definition 4.1.1]{taoAffineHeckeCategory}. Objects in $\scI$ are sequences $\bw=(w_{1},\cdots,w_{n})$ where each $w_{i}\in W^{a}$ lies in a finite standard parabolic subgroup.
		
	Recall $\cG_{\le w}= \overline{\cI w\cI} \subset \cG$. For $\bw=(w_1,..., w_n) \in \mathscr{I}$, put $X_{\bw}= \cI^u \backslash \cG_{\le w_{1}} \times^{\cI^u} ... \times^{\cI^u} \cG_{\le w_n}/\cI^u$. Let $X_{\bw}^\circ\subset X_{\bw}$ be the open stratum where $\cG_{\le w_{i}}$ are replaced with $\cG_{w_{i}}=\cI w_{i}\cI$, and let $\partial X_{\bw}=X_{\bw}\setminus X^{\c}_{\bw}$. Then $X_{\bw}$ is stratified into the union of $X_{\bw'}$, for $\bw'=(w'_{1},\cdots, w'_{n})\in\scI$ such that $w'_{i}\le w_{i}$ for each $i$. 
	
	Define $\Sh'(X^{\c}_{\bw})\subset \Sh(X^{\c}_{\bw})$  to be the full subcategory of locally constant sheaves. Denote by $\Sh'(X_{\bw}) \subset \Sh(X_{\bw})$ the full subcategory consisting of sheaves that are locally constant on each stratum $X_{\bw'}$. Similarly define $\Sh'(\pl X_{\bw})$. 	
	
We have a functor $\mathscr{H}:\mathscr{I}  \to \St^L_{k}$, 
which assigns to an object $\bw \in \mathscr{I}$ the category $Sh'(X_{\bw})$, and to a morphism  $\bw_1 \to \bw_2$ in $\mathscr{I} $,
the functor $ \mathscr{H}(\bw_1)  \to  \mathscr{H} (\bw_2) $ given by convolution. This is the universal monodromic analogue of \cite[Corollary 5.3.3]{taoAffineHeckeCategory}.

\begin{proof}[Sketch of proof of Theorem~\ref{thm:colimit_of_hecke}] The proof of \cite[Theorem 5.4.3]{taoAffineHeckeCategory} goes through, except we need to check the analogous statement for \cite[Corollary 5.4.2]{taoAffineHeckeCategory}: for every birational map $\bw_1 \to \bw_2$ in $\mathscr{I}$, the diagram 
\begin{equation*} 
	\xymatrix{  \colim_{\bw_1' \to \bw_1 \textup{ strict emb.}}	\mathscr{H}({\bw'_1})  \ar[r] \ar[d] & \mathscr{H}({\bw_1}) \ar[d]   \\
	\colim_{\bw_2' \to \bw_2 \textup{ strict emb.}}	\mathscr{H}({\bw'_2})  \ar[r]  & \mathscr{H}({\bw_2})  
	}
   \end{equation*}
is coCartesian. Now by same proof as in \cite[Corollary 5.4.2]{taoAffineHeckeCategory}, we have a natural equivalence for every $\bw\in\scI$:
\begin{equation*}
\colim_{\bw' \to \bw \textup{ strict emb.}}\mathscr{H}({\bw'}) \simeq  Sh'(\partial X_{\bw}).
\end{equation*}
Therefore we are left to check the following Lemma, analogous to \cite[Lemma 5.4.1]{taoAffineHeckeCategory}.
\end{proof}

\begin{lemma}
For any birational map $\bw_1 \to\bw_2$ in $\scI$, the natural diagram is coCartesian:
\beq  
\label{lemma:h-descent_Schubert_variety}
 \xymatrix{  Sh'(\partial X_{\bw_1})  \ar[r] \ar[d] & Sh'( X_{\bw_1})\ar[d]   \\
 Sh'(\partial X_{\bw_2})	 \ar[r]  & Sh'( X_{\bw_2}) 
 }
\eeq
\end{lemma}
\begin{proof} We have a morphism of recollements (see Section~\ref{sss:rec})
\begin{equation*}
\xymatrix{\Sh'(X_{\bw_{1}}^{\c}) \ar[d]\ar@<1ex>[r]\ar@<-1ex>[r]& \ar[l] Sh'( X_{\bw_1})\ar[d] \ar@<1ex>[r]\ar@<-1ex>[r] & \ar[l]Sh'(\partial X_{\bw_l})\ar[d]\\
\Sh'(X_{\bw_{2}}^{\c}) \ar@<1ex>[r]\ar@<-1ex>[r]& \ar[l]Sh'( X_{\bw_2}) \ar@<1ex>[r]\ar@<-1ex>[r]& \ar[l]Sh'(\partial X_{\bw_2})
}
\end{equation*}
Since $\bw_1 \to\bw_2$ is birational, $X_{\bw_{1}}^{\c}\cong X_{\bw_{2}}^{\c}$, the left vertical arrow above is an equivalence. The desired statement follows from Corollary \ref{c:coprod}.
\end{proof}

\bibliographystyle{alpha}
\bibliography{commuting_stack}

\end{document}